\documentclass{amsart}

\newtheorem{theorem}{Theorem}[section]
\newtheorem{lemma}[theorem]{Lemma}
\newtheorem{proposition}[theorem]{Proposition}

\theoremstyle{definition}

\newtheorem{example}[theorem]{Example}

\theoremstyle{remark}
\newtheorem{remark}[theorem]{Remark}

\numberwithin{equation}{section}

\usepackage{setspace}
\usepackage{graphicx,subcaption}
\usepackage{indentfirst} 
\usepackage{amsmath,amsfonts,amsthm,amscd,upref,amstext}	
\usepackage[titletoc]{appendix}
\usepackage{tikz-cd,tikz,stackrel,adjustbox}
\usepackage[all]{xy}

\definecolor{references}{rgb}{0,0,1}
\RequirePackage[pdftex,
colorlinks = true,
urlcolor = references, 
citecolor = references, 
linkcolor = references, 
]
{hyperref}

\begin{document}

\title{An Extension of the $sl(n)$ Polynomial to Knotted 4-Valent Graphs}

\author{Carmen Caprau}
\address{Department of Mathematics, California State University, Fresno}
\email{ccaprau@mail.fresnostate.edu}
\thanks{This work was partially supported by NSF-RUI grant DMS No. 2204386.}

\author{Victoria Wiest}
\address{Department of Mathematics, California State University, Fresno}
\email{wiest\_victoria@mail.fresnostate.edu}

\subjclass[2020]{57K12, 57K14}



\keywords{knots; generating sets; knotted graphs; invariants of knotted graphs}

\begin{abstract}
We use planar 4-valent graphs and a graphical calculus involving such graphs to construct an invariant for balanced-oriented, knotted 4-valent graphs. Our invariant is an extension of the $sl(n)$ polynomial for classical knots and links. We also provide a minimal generating set of Reidemeister-type moves for diagrams of balanced-oriented, knotted 4-valent graphs.
\end{abstract}

\maketitle
\par

\section{Introduction}
\label{Intro}
Knotted graphs are embeddings of graphs in 3-space and topologists are interested in distinguishing knotted graphs up to ambient isotopy, similar to the goal of knot theory. Isotopy of knotted graphs can be described combinatorially by working with diagrams of knotted graphs. A knotted graph diagram is a projection of a knotted graph into a plane, such that the projection contains only transverse double points; we also require that in such a projection, no interior point of an edge of the original graph  projects to the same point where a vertex projects. Specifically, a knotted graph diagram contains crossings (similar to knot or link diagrams) and the vertices of the graph.

In this paper, we work with 4-valent knotted graphs with rigid vertices. 
Equivalently, we will distinguish 4-valent knotted graphs up to \textit{rigid-vertex isotopy}.  It is known~\cite{Kauffman} that two diagrams of knotted 4-valent graphs represent rigid-vertex isotopic knotted graphs if and only if there is a finite sequence of planar isotopies and the classical Reidemeister moves $R1, R2$ and $R3$, together with the additional moves $R4$ and $R5$ involving a vertex - as depicted in Figure~\ref{fig:Reidemeister moves} - taking one diagram to another.

\begin{figure}[ht]
    \[\raisebox{-13pt}{\includegraphics[height=0.4in]{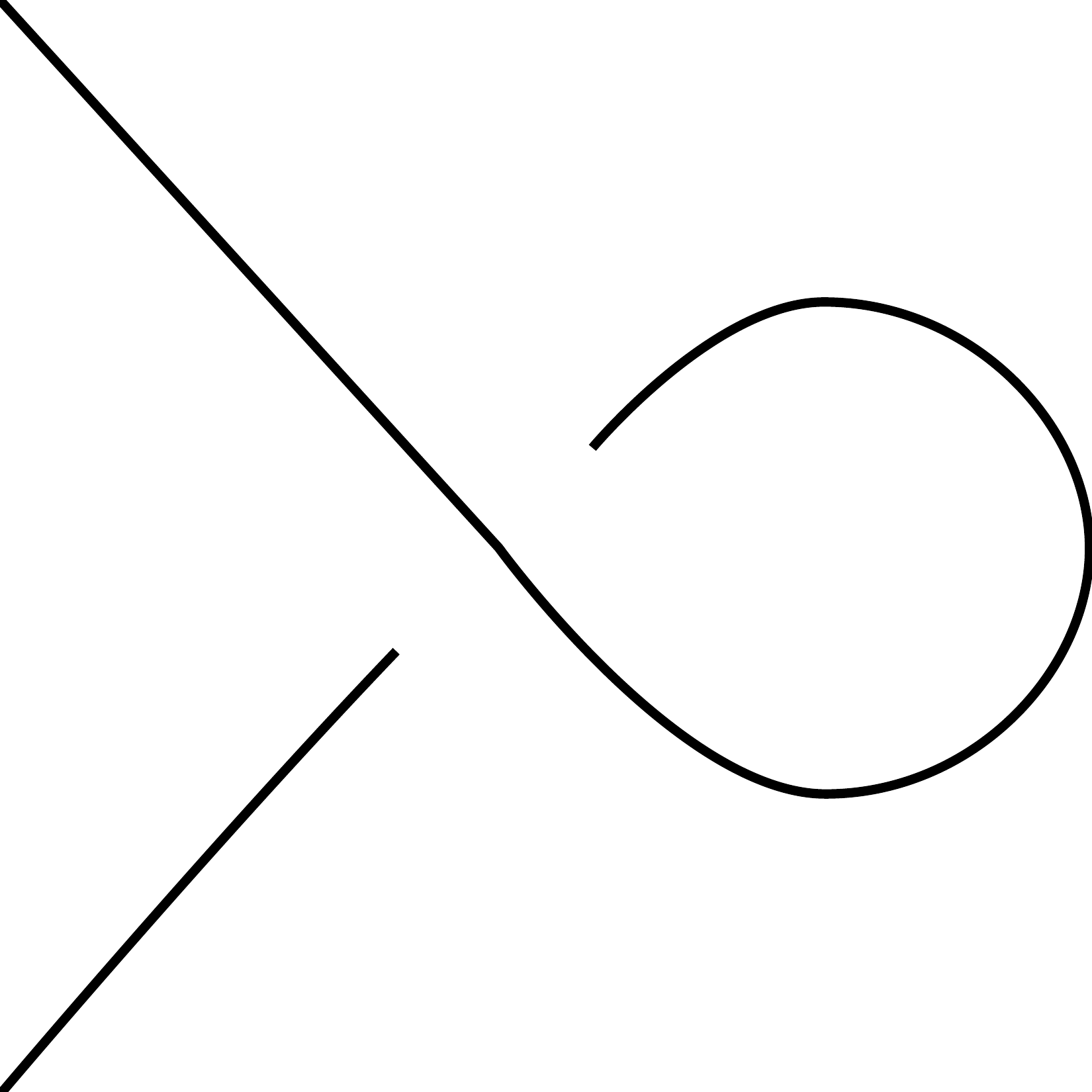}}\stackrel{R1}{\longleftrightarrow}\raisebox{-13pt}{\includegraphics[height =0.4in]{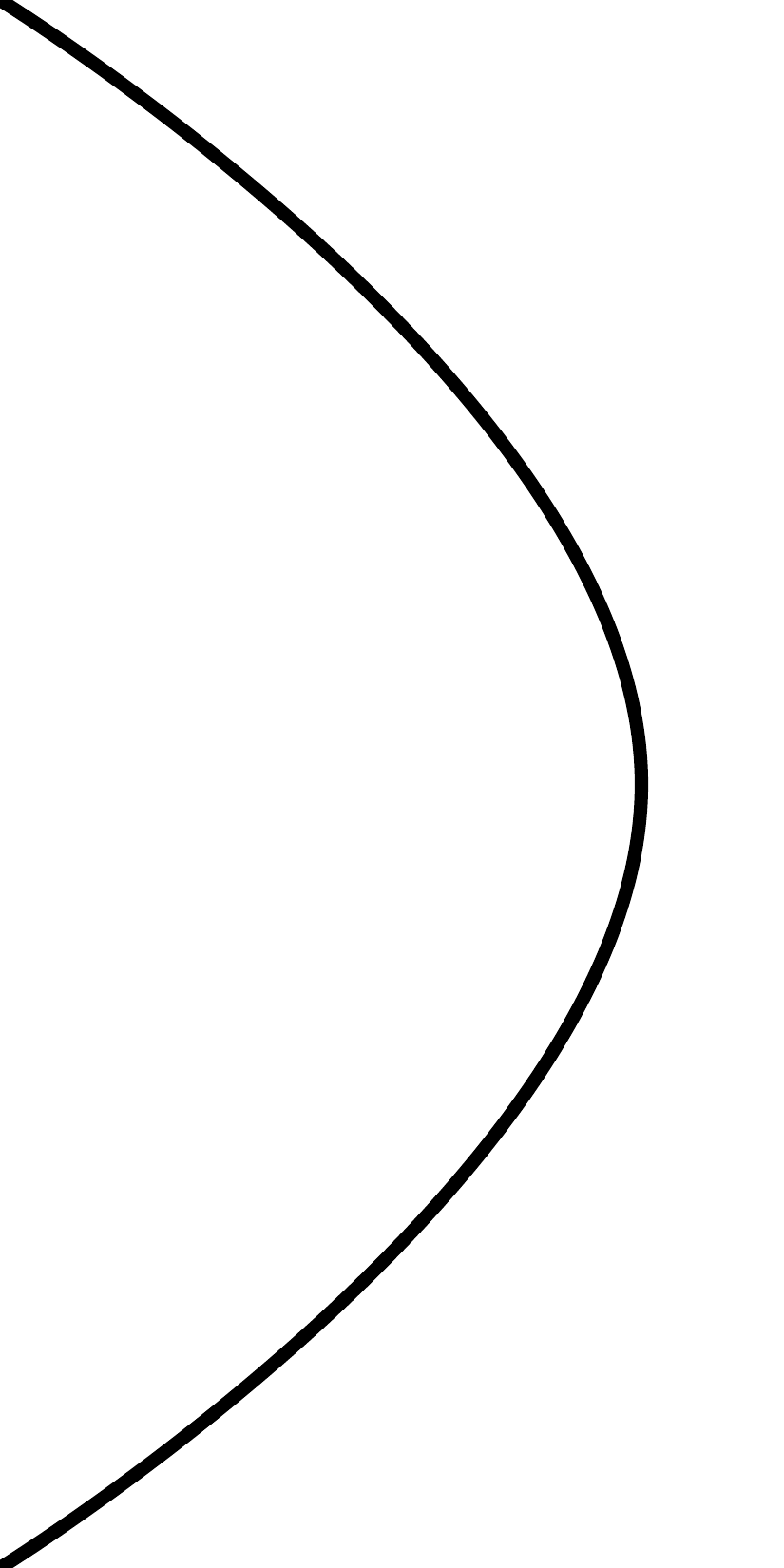}}\stackrel{R1}{\longleftrightarrow}\raisebox{-13pt}{\includegraphics[height=0.4in]{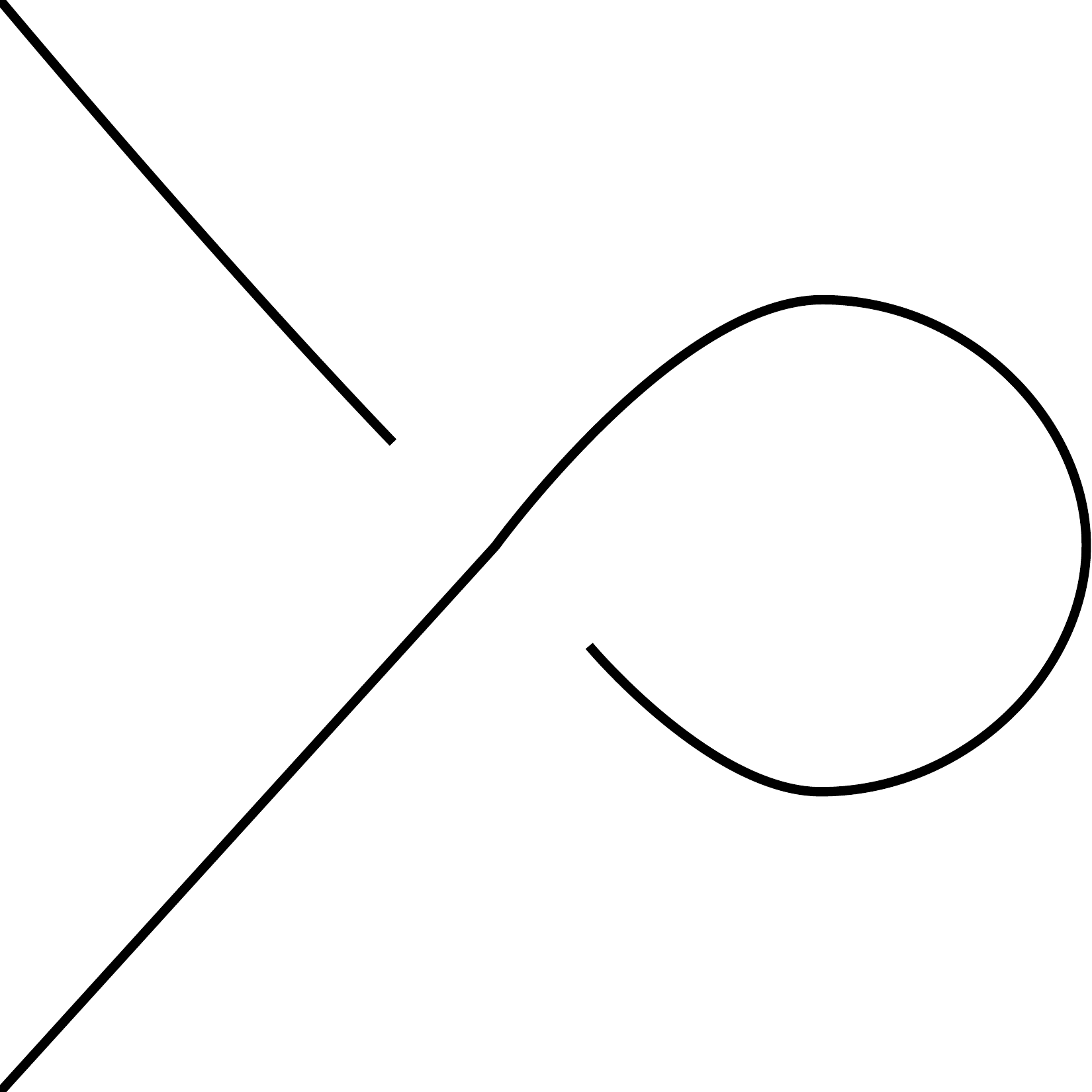}}\]  
    \[\raisebox{-13pt}{\includegraphics[height=0.4in]{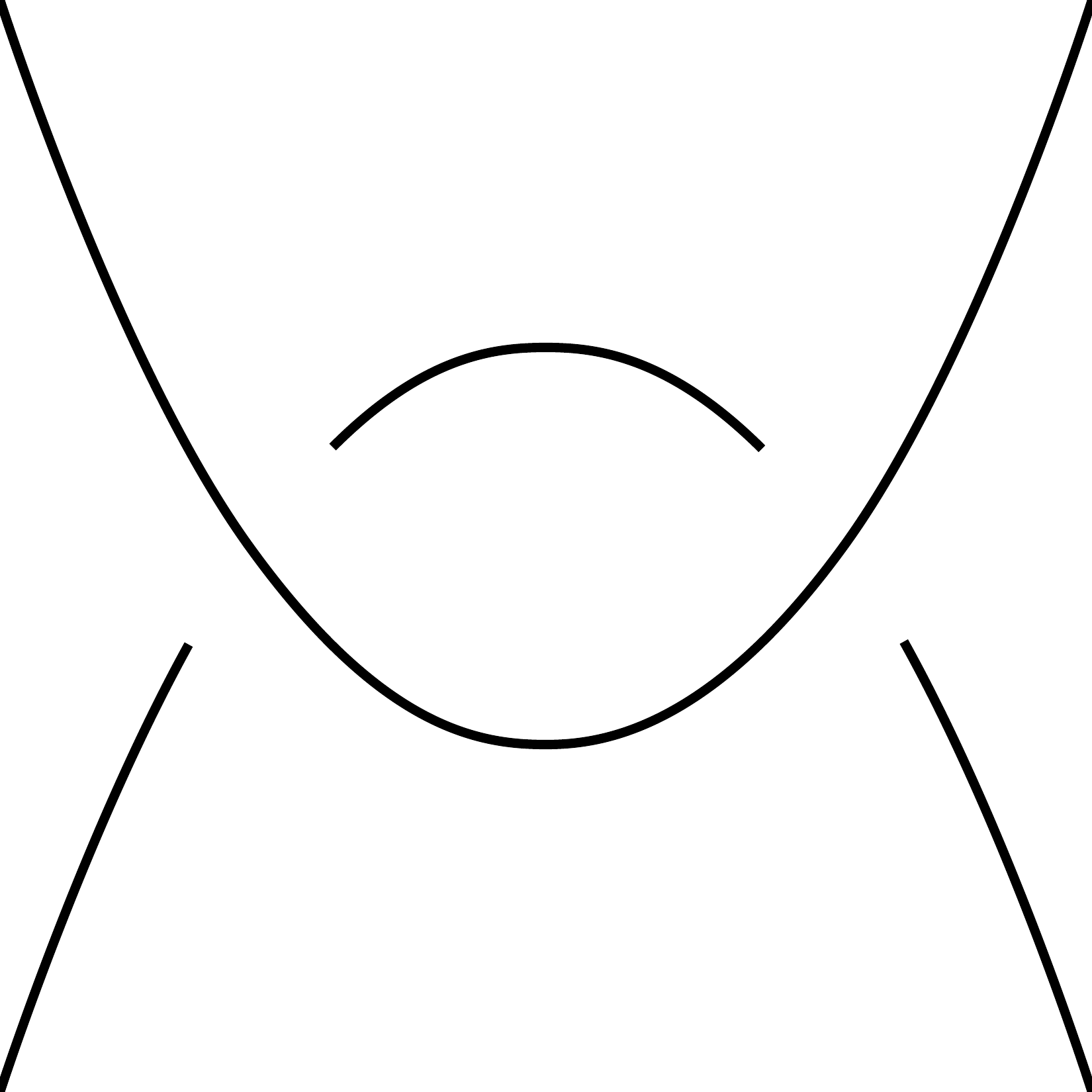}}\stackrel{R2}{\longleftrightarrow}\raisebox{-13pt}{\includegraphics[height =0.4in]{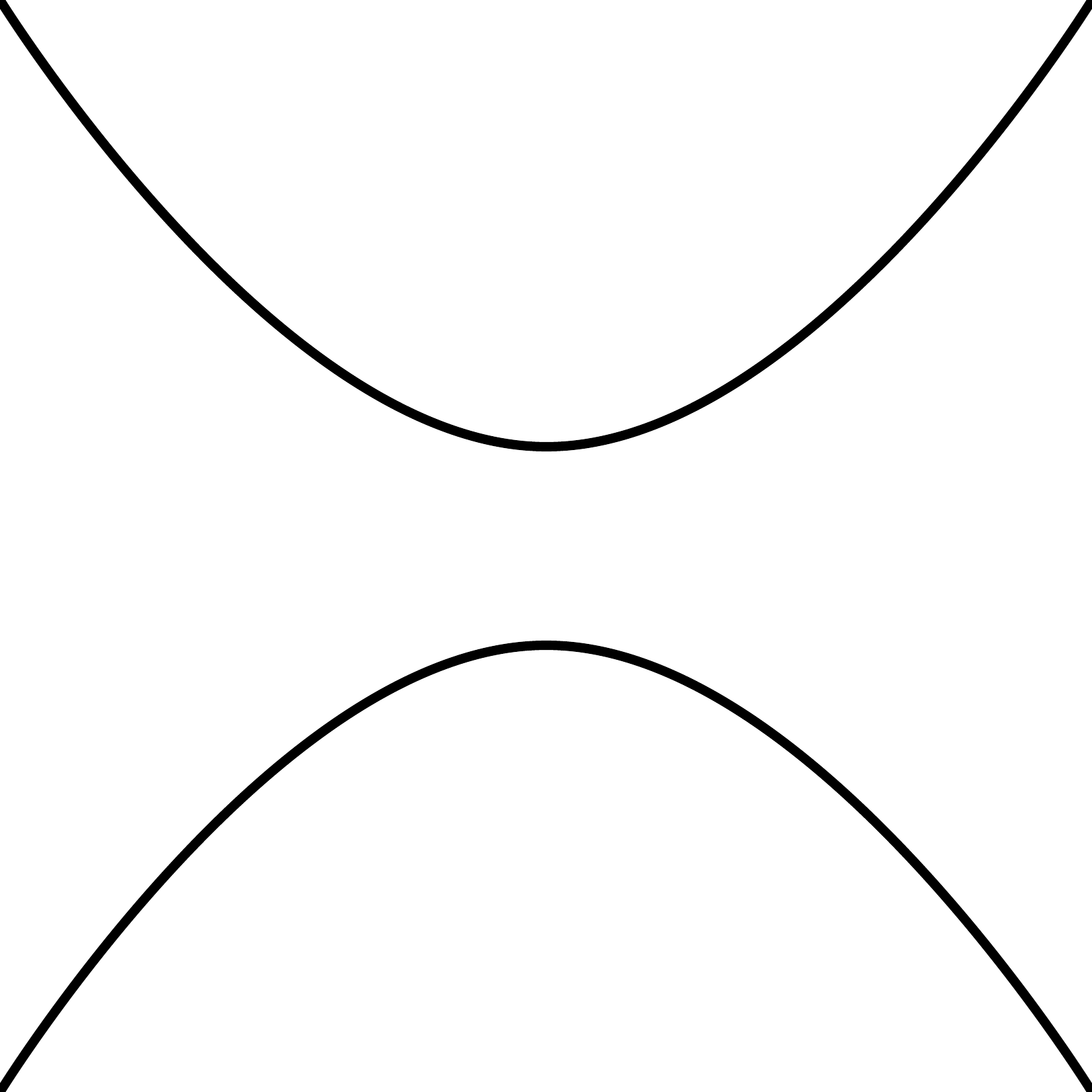}}\stackrel{R2}{\longleftrightarrow}\raisebox{-13pt}{\includegraphics[height=0.4in]{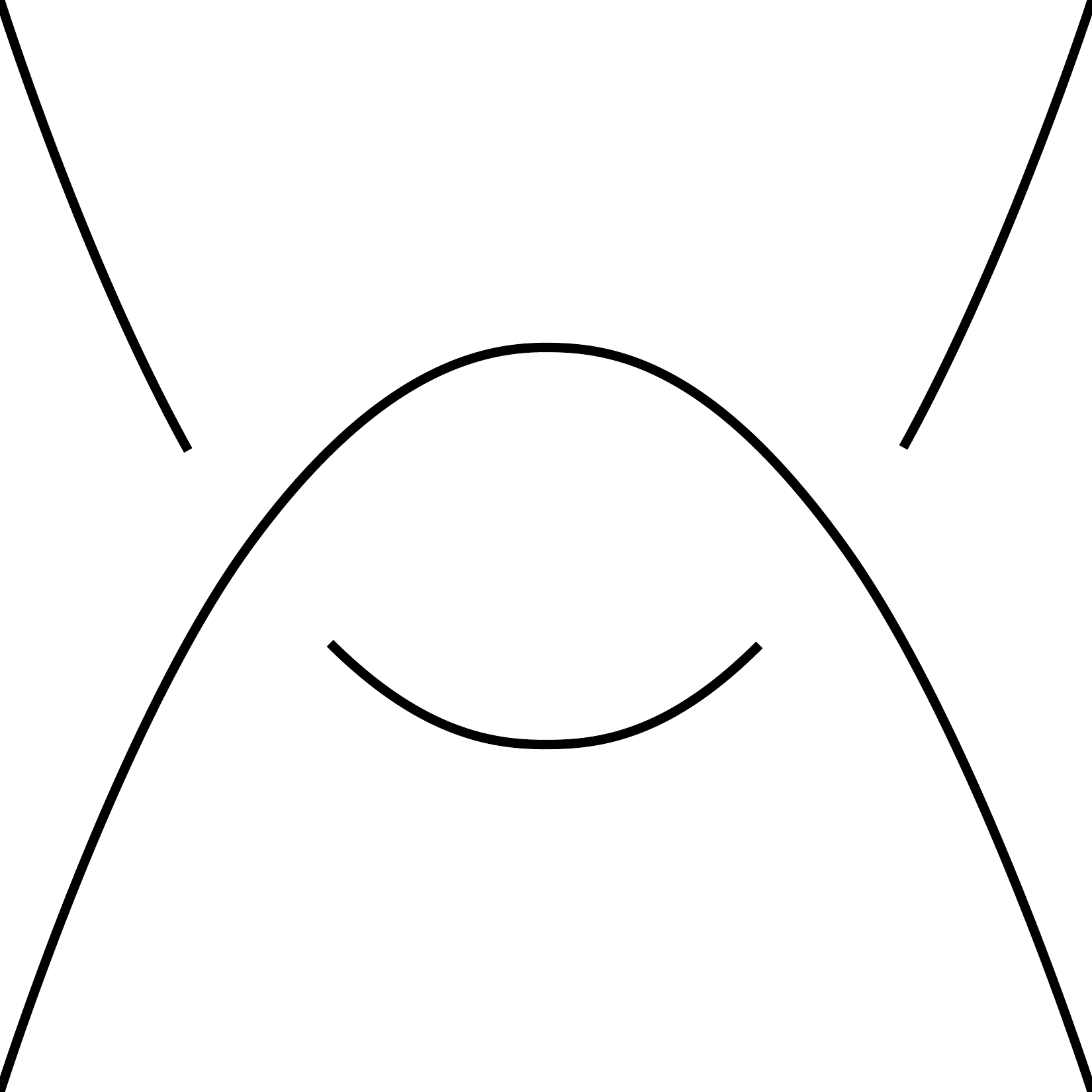}}
    \hspace{1cm}
    \raisebox{-13pt}{\includegraphics[height=0.4in]{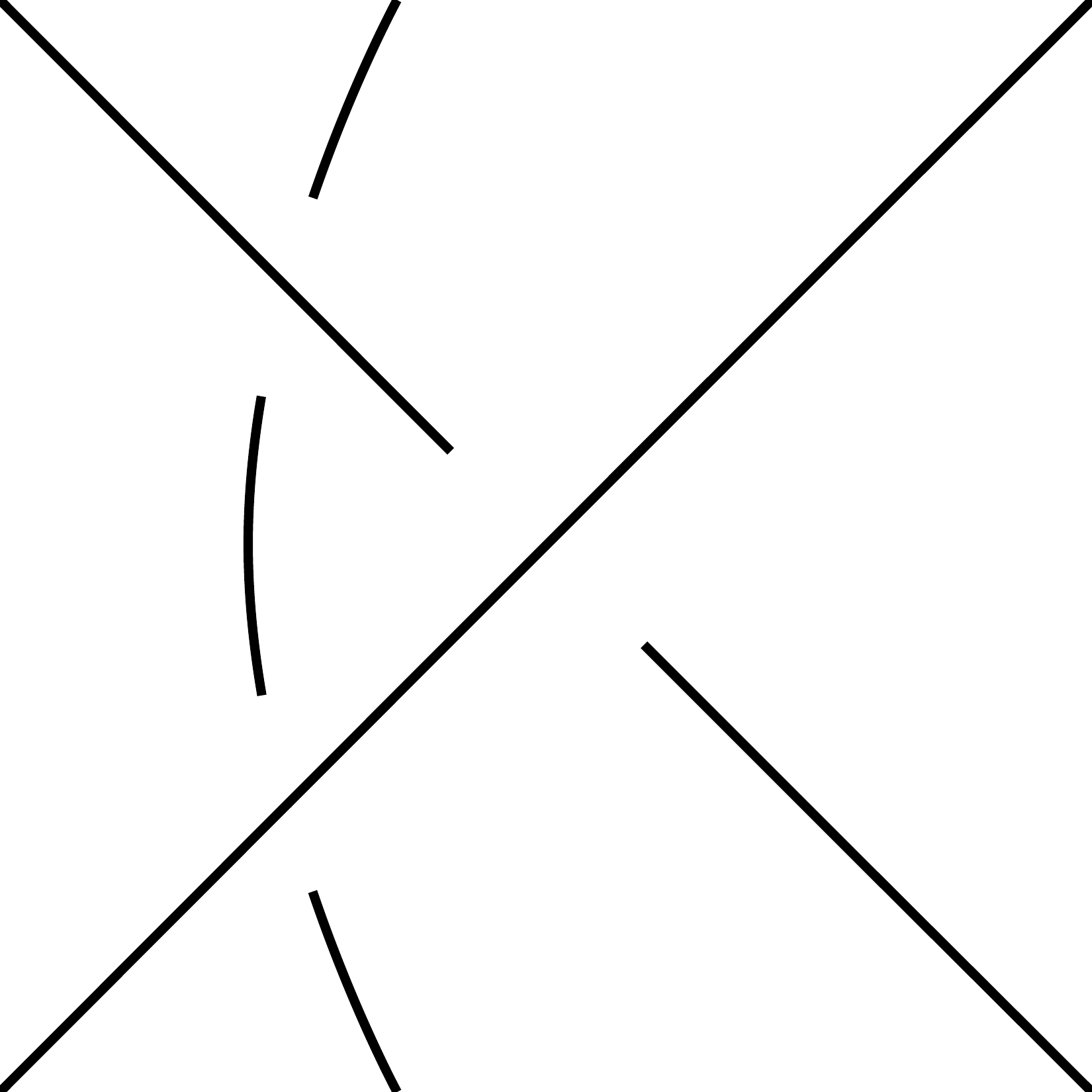}}\stackrel{R3}{\longleftrightarrow}\raisebox{-13pt}{\includegraphics[height=0.4in]{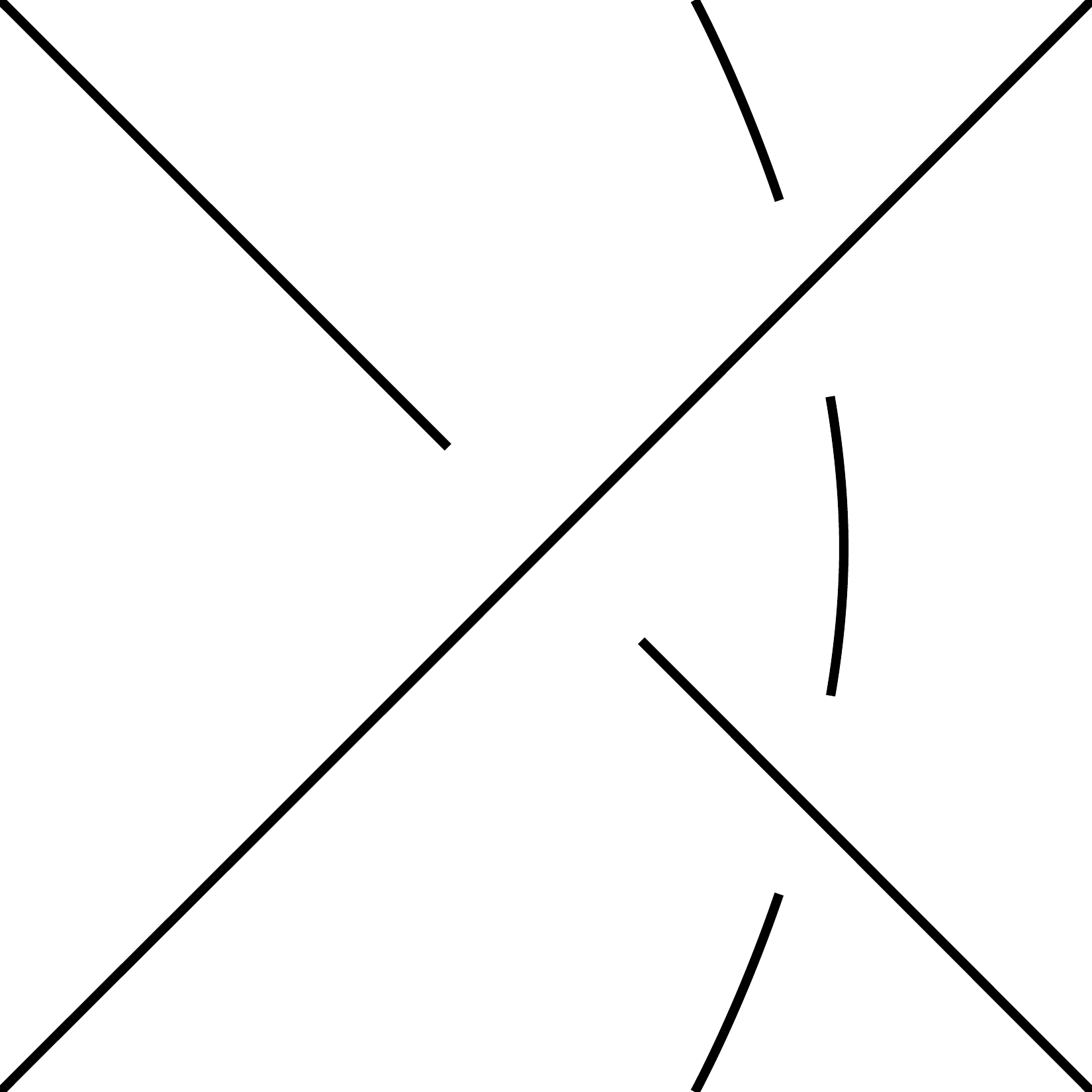}}\]
    \vspace{0.10cm}
     \[\raisebox{-12pt}{\includegraphics[height=0.4in]{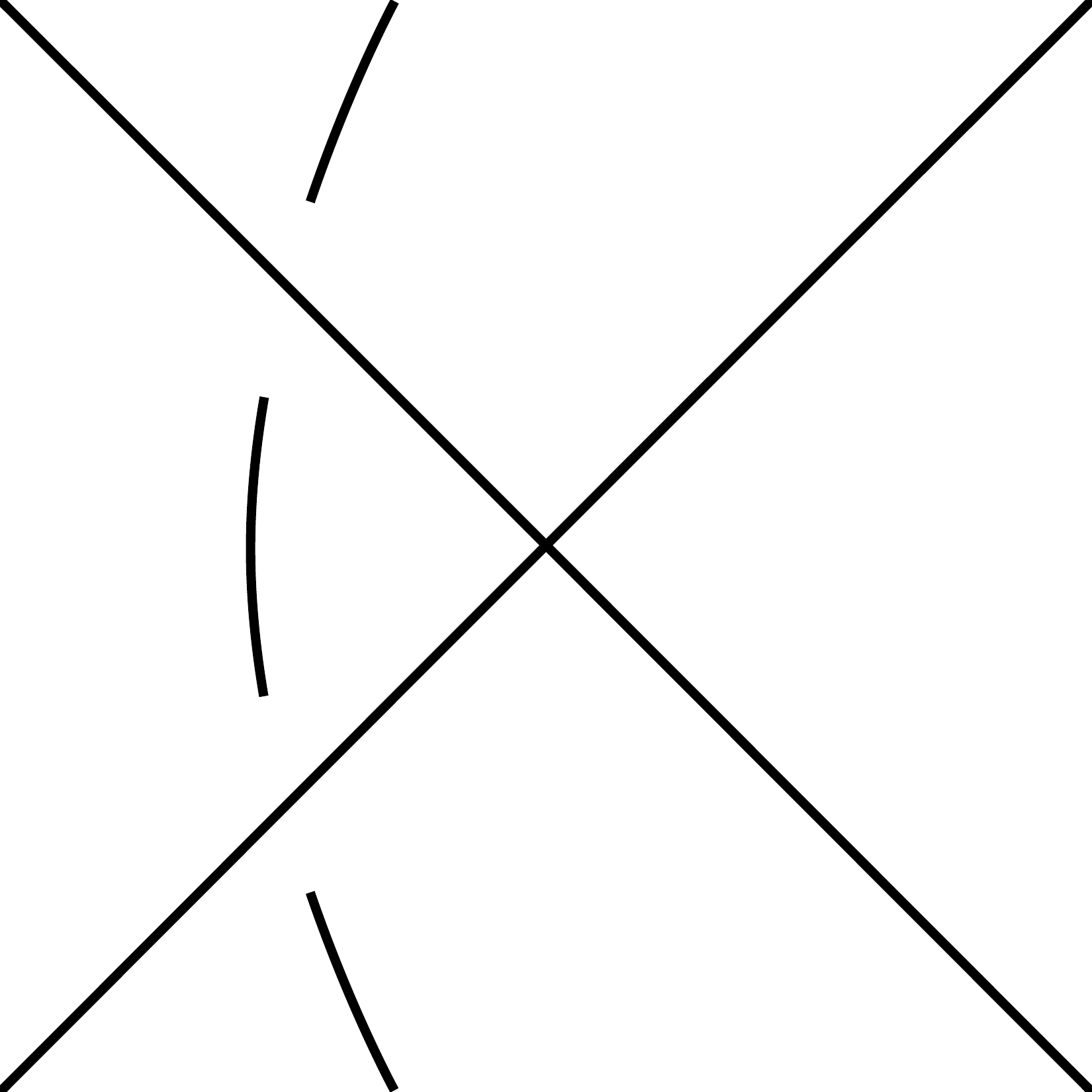}}\stackrel{R4}{\longleftrightarrow}\raisebox{-12pt}{\includegraphics[height=0.4in]{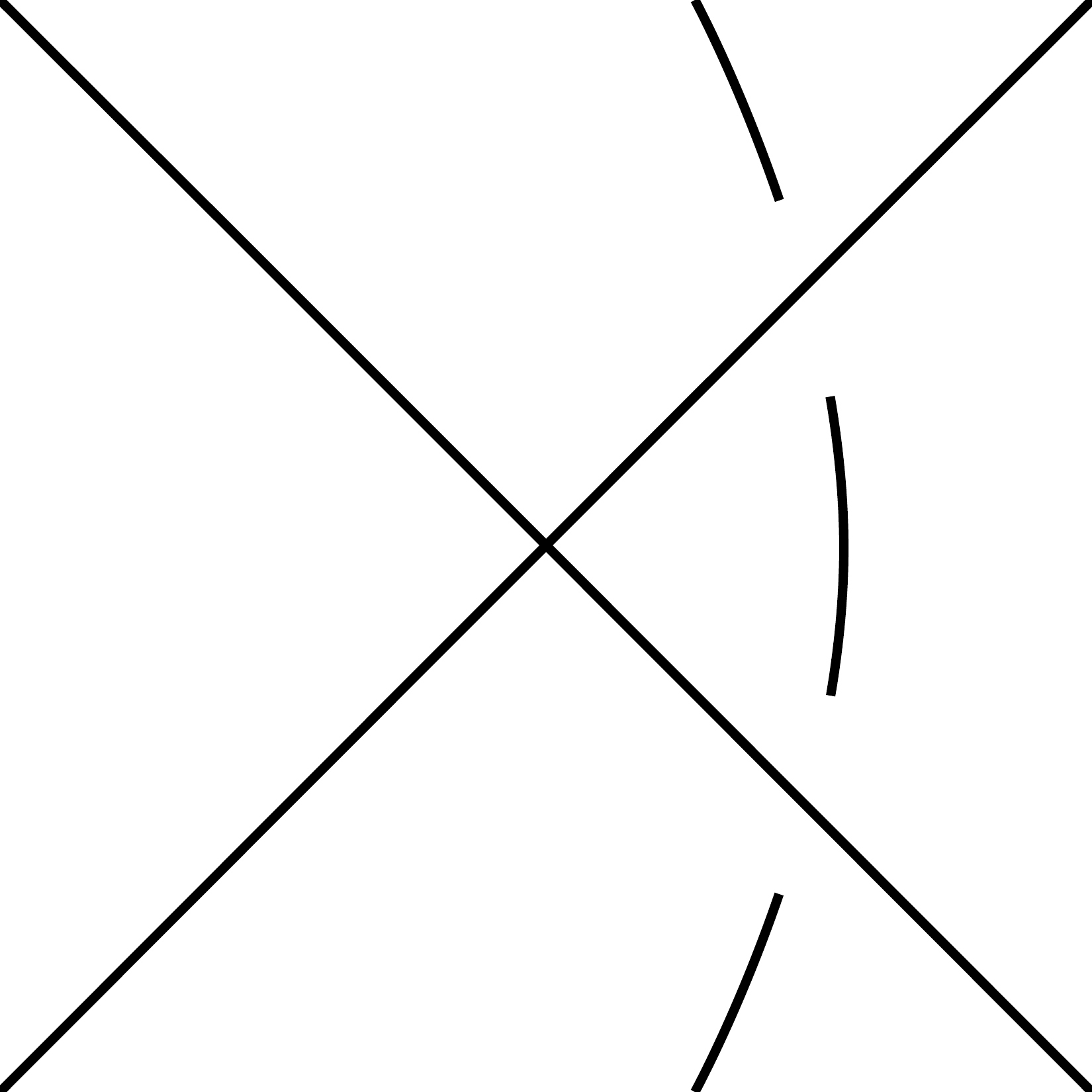}}\hspace{1cm}\raisebox{-12pt}{\includegraphics[height=0.4in]{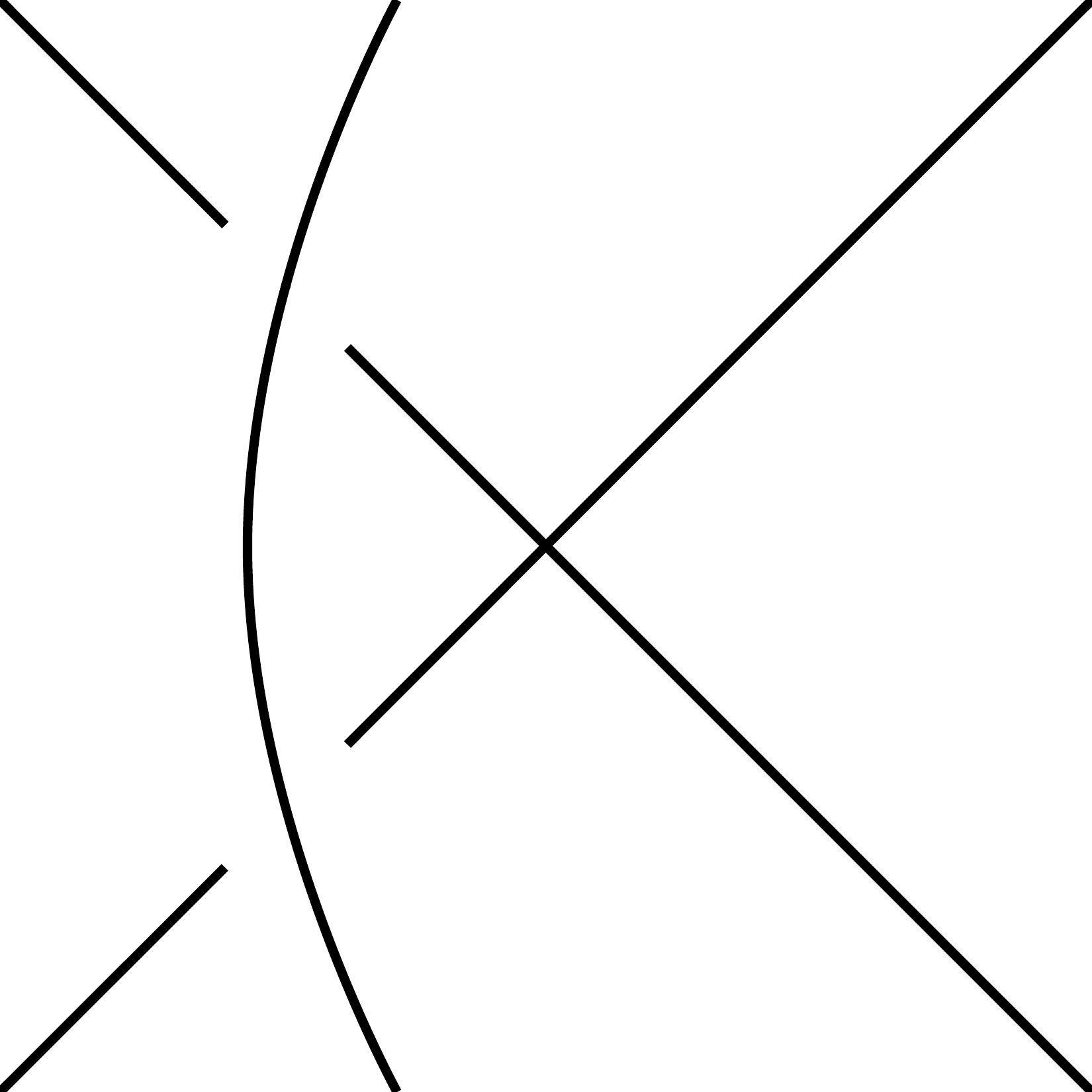}}\stackrel{R4}{\longleftrightarrow}\raisebox{-12pt}{\includegraphics[height=0.4in]{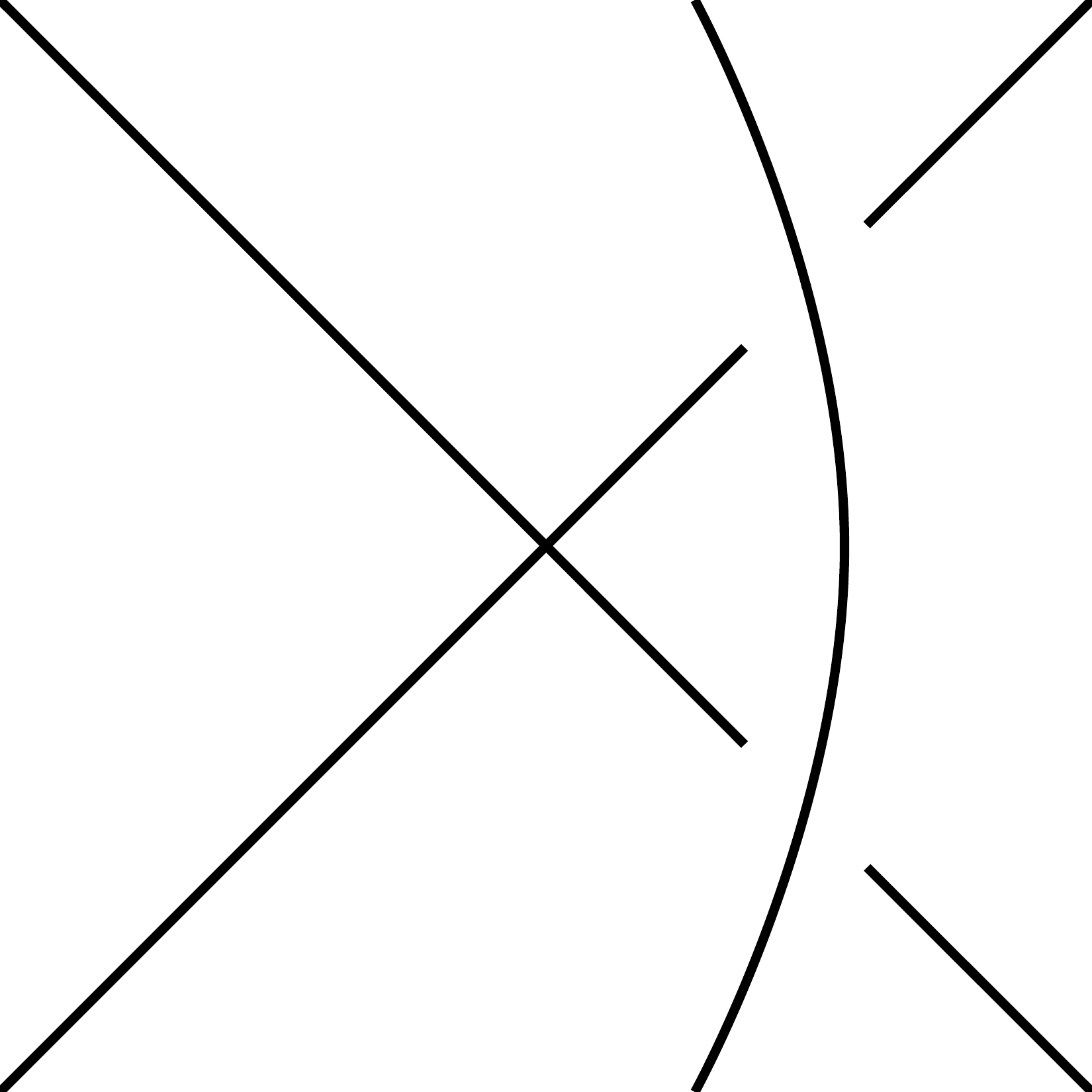}} \hspace{1cm} 
    \raisebox{-12pt}{\includegraphics[height=0.4in]{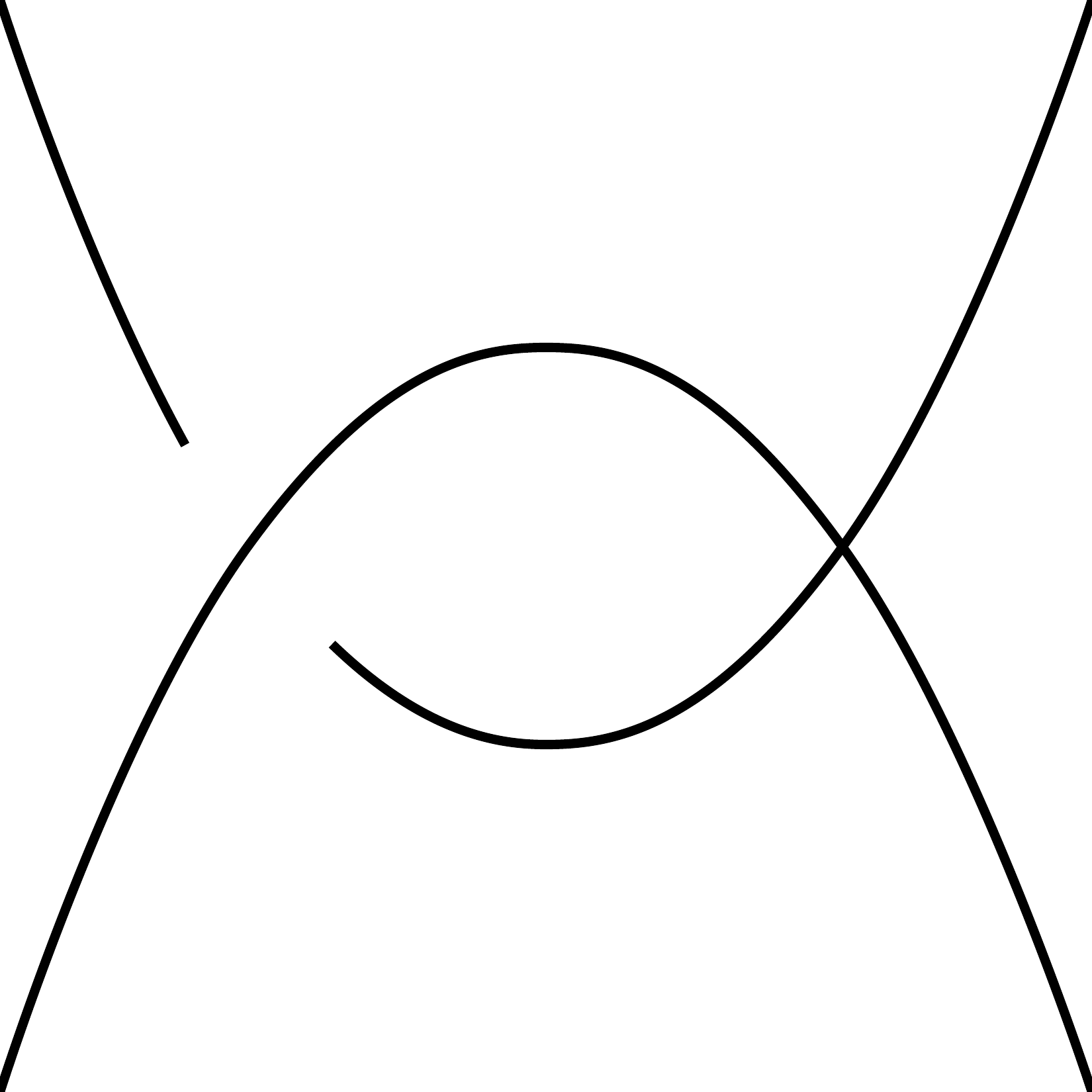}}\stackrel{R5}{\longleftrightarrow}\raisebox{-12pt}{\includegraphics[height=0.4in]{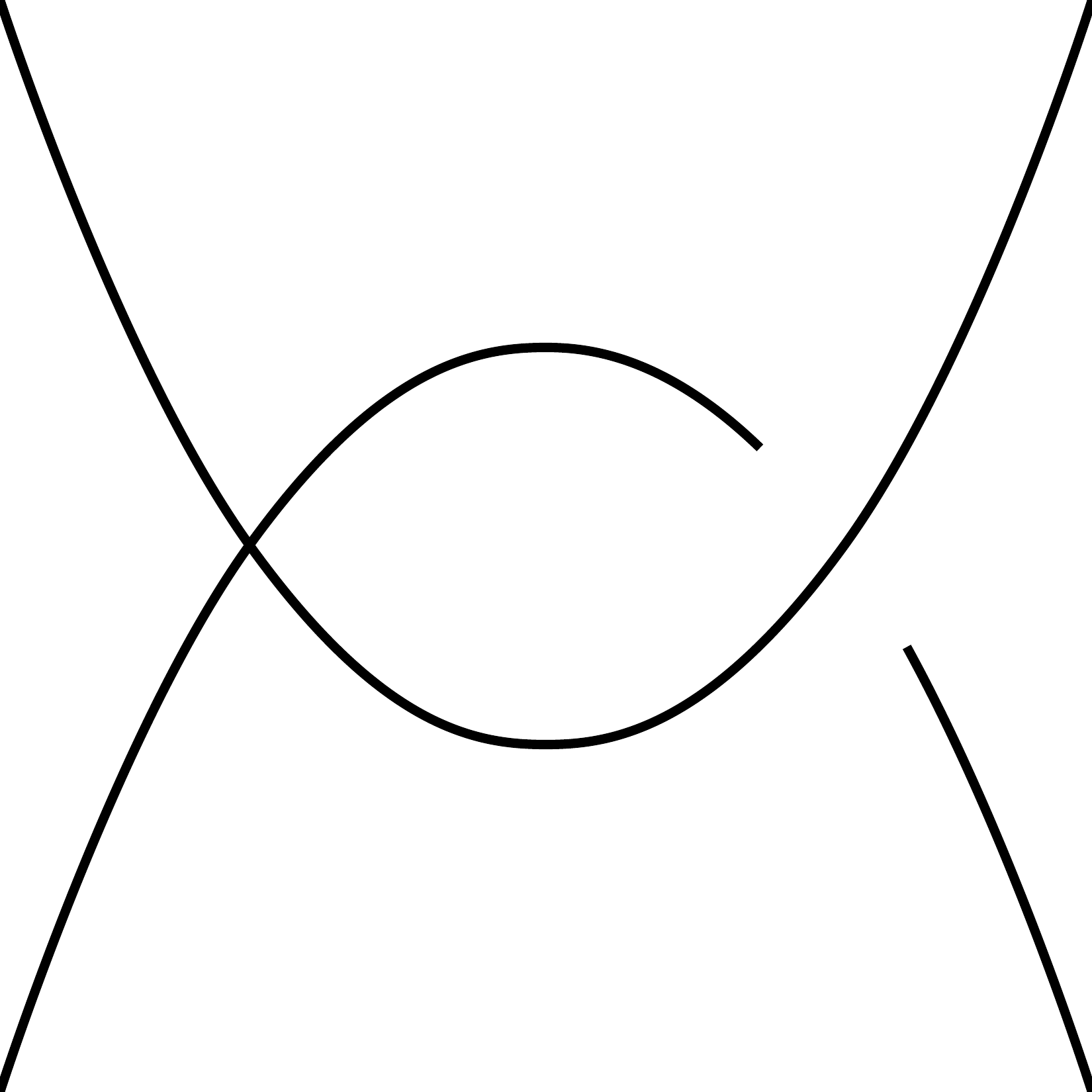}}\]
    \caption{Reidemeister-type moves for 4-valent knotted graphs}
    \label{fig:Reidemeister moves}
\end{figure}

The knotted 4-valent graphs that we consider in this work are oriented, meaning that each edge is given an orientation, such that the 4-valent vertices are \textit{balanced-oriented}; that is, the in-degree and out-degree at each vertex is two. Specifically, the vertices are of two orientation types, up to cyclic permutation: In-In-Out-Out and In-Out-In-Out (see Figure~\ref{fig:Vertices Orientation types}).

\begin{figure}[ht]
\centering
    \includegraphics[width=0.4in]{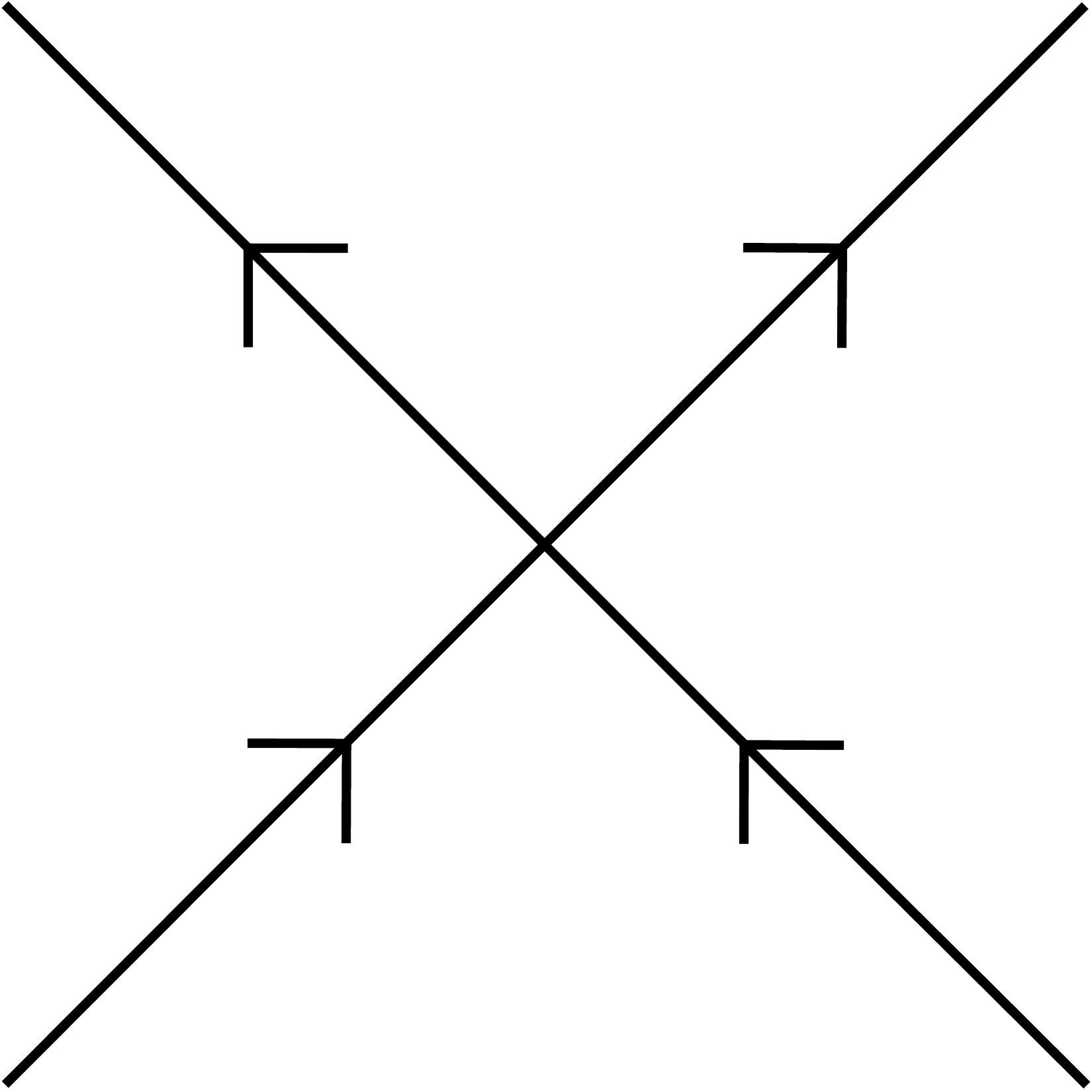}\hspace{1.5cm}\includegraphics[width=0.4in]{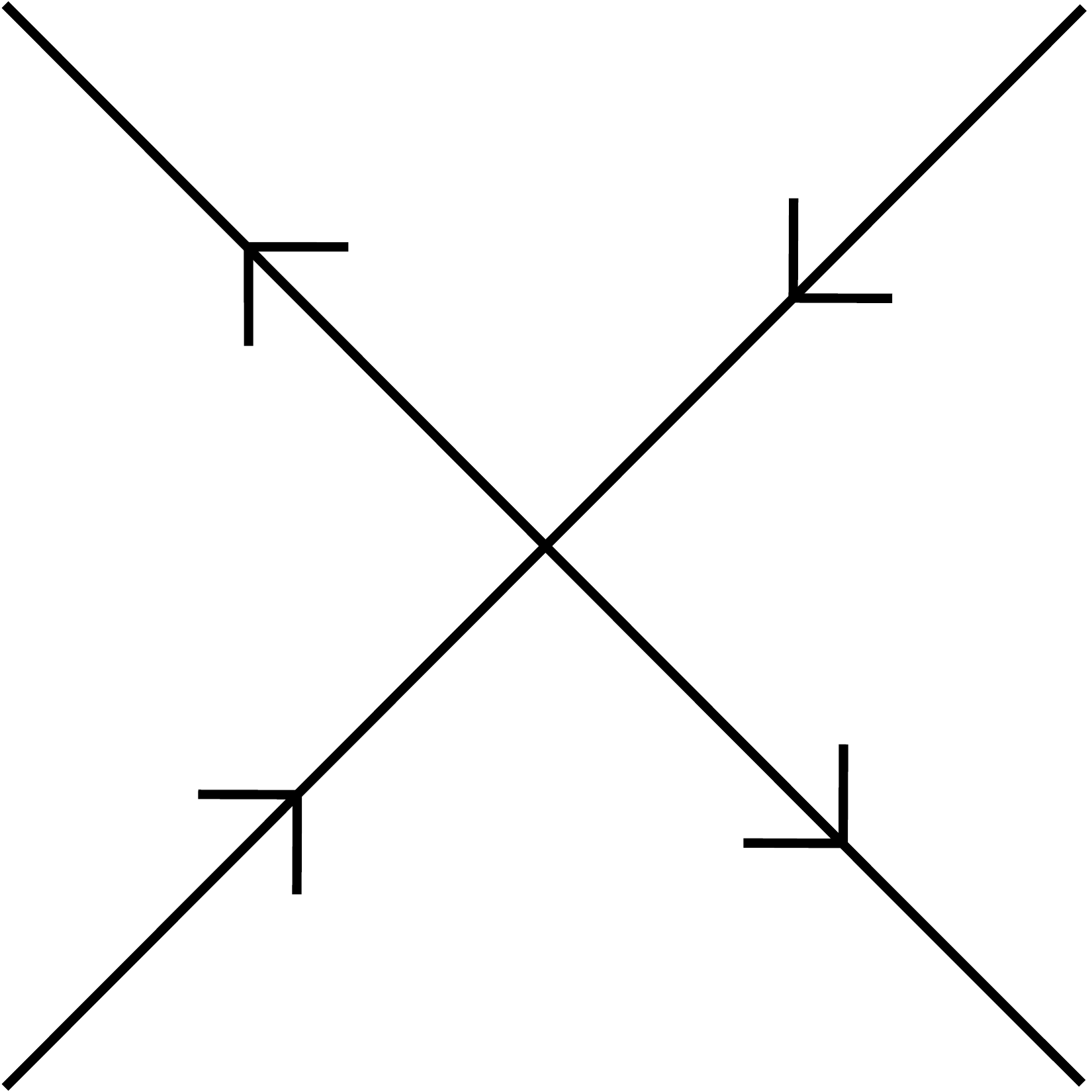}
\caption{Balanced-oriented 4-valent vertices}
\label{fig:Vertices Orientation types}
\end{figure}
The scope of this paper is to construct a polynomial invariant $P(L)$ for balanced-oriented, knotted 4-valent graphs $L$, which is an extension of the $sl(n)$ polynomial invariant for classical knots. We remind the reader that the $sl(n)$ polynomial is a one-variable specialization of the HOMFLY-PT polynomial invariant of oriented knots~\cite{HOMFLY, PT}. 

When working with oriented knotted graphs, there are several oriented versions of the Reidemeister-type moves for knotted graph diagrams. Hence, it is ideal to have at hand a generating set for all of the moves on knotted graph diagrams when proving that a certain quantity is an invariant of knotted graphs. In this paper, we also provide a minimal generating set of Reidemeister-type moves for diagrams of balanced-oriented, knotted 4-valent graphs with rigid vertices.

The paper is organized as follows: We provide the construction of the polynomial $P$ in Section~\ref{sec:polyinv}. In Section~\ref{sec:generatingset} we list all of the oriented versions of the Reidemeister-type moves for diagrams of balanced-oriented, knotted 4-valent graphs, and then provide a minimal generating set for these moves. We use this minimal generating set to prove in Section~\ref{sec:proofinv} that $P(L)$ is an invariant for balanced-oriented, knotted 4-valent graphs $L$ with rigid vertices.

\section{A Polynomial for Knotted, Balanced-Oriented, 4-Valent Graphs} \label{sec:polyinv}

The purpose of this paper is to construct a polynomial invariant for knotted 4-valent graphs by altering Caprau, Heywood, and Ibarra's ~\cite{Caprau_2014} modified MOY state model ~\cite{MOY} for the $sl(n)$ polynomial. Our polynomial invariant, denoted by $P(L)$ where $L$ is a balanced-oriented, knotted 4-valent graph with rigid vertices, is defined in terms of the skein relations and a graphical calculus involving planar graphs. 

Given a knotted 4-valent graph diagram $D$, we iteratively resolve the crossings using the skein relations in Figure~\ref{fig:SR Crossings}. The three diagrams in a skein relation are parts of knotted 4-valent graph diagrams that are identical, except in a small region where they differ as shown.

\begin{figure}[ht]
    \[P\Biggl(\raisebox{-10pt}{\includegraphics[height = .35in]{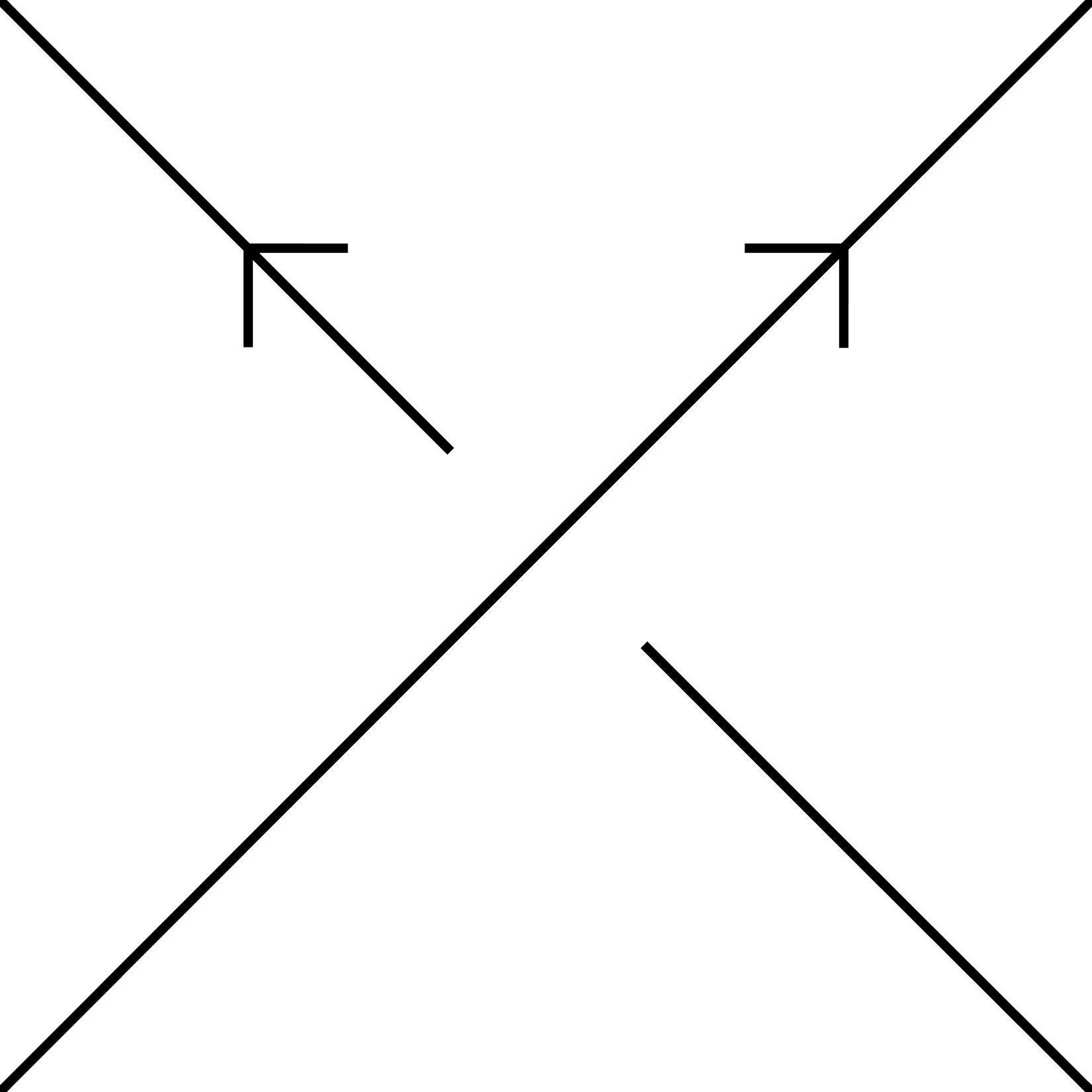}}\Biggr)= q^{n-1}P\Biggl(\raisebox{-10pt}{\includegraphics[height = .35in]{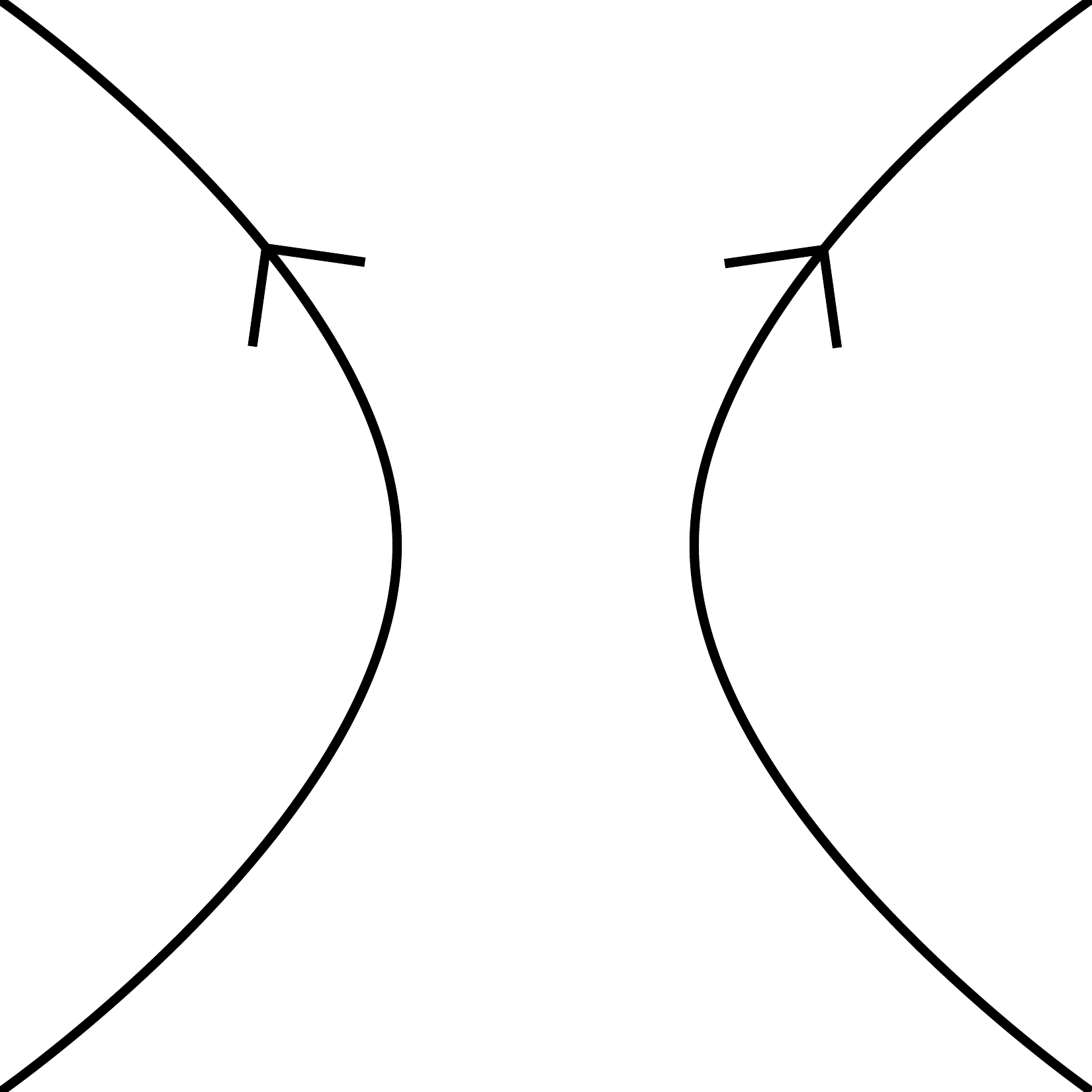}}\Biggr)-q^nP\Biggl(\raisebox{-10pt}{\includegraphics[height = .35in]{Figure-2/In-In-Out-Out.pdf}}\Biggr)\]

    \[P\Biggl(\raisebox{-10pt}{\includegraphics[height = .35in]{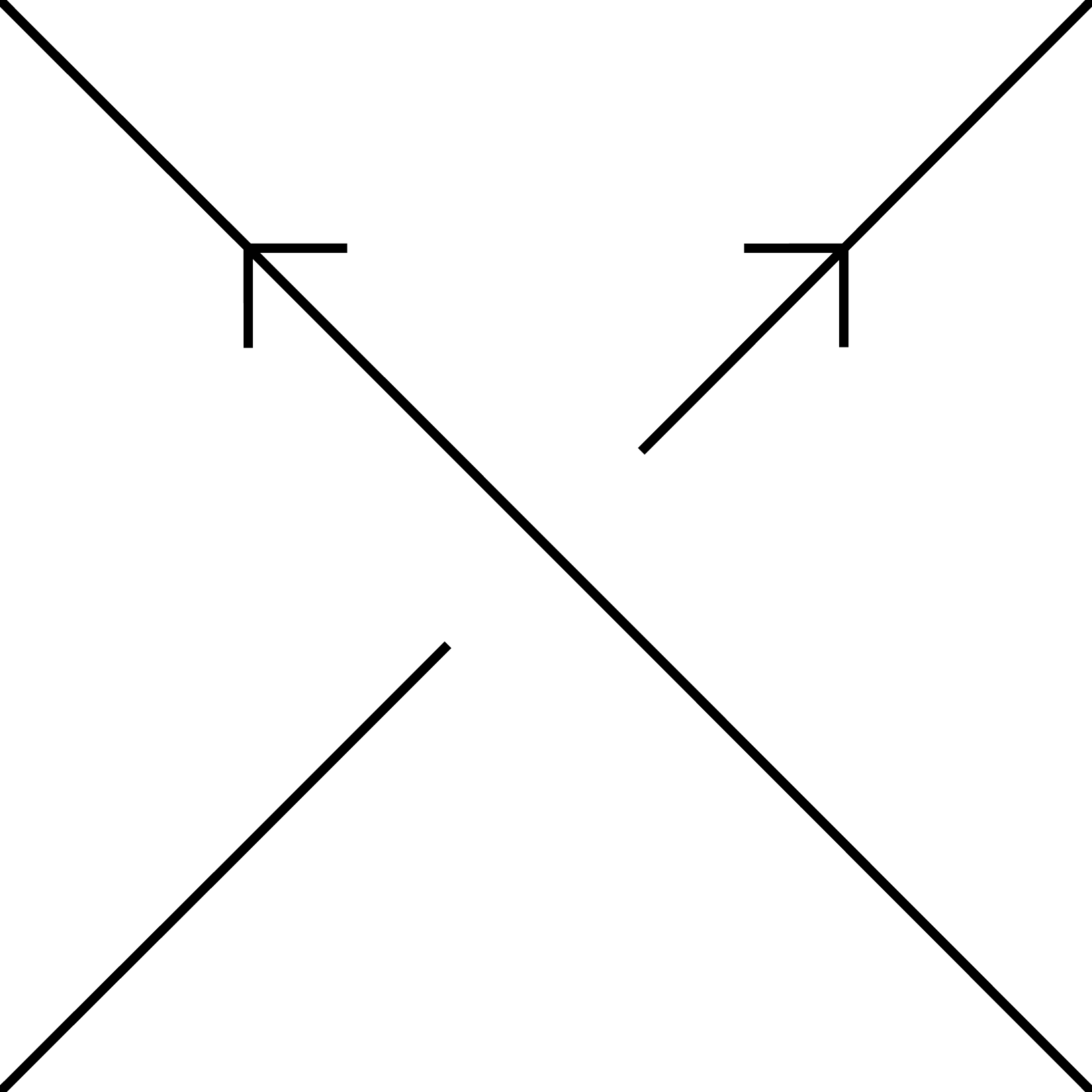}}\Biggr)= q^{1-n}P\Biggl(\raisebox{-10pt}{\includegraphics[height = .35in]{Skein-relations/oriented-strands-up.pdf}}\Biggr)-q^{-n}P\Biggl(\raisebox{-10pt}{\includegraphics[height = .35in]{Figure-2/In-In-Out-Out.pdf}}\Biggr)\]
    \caption{Skein relations for crossings}
    \label{fig:SR Crossings}
\end{figure}

After resolving all crossings, $P(D)$ is written as a linear combination of evaluations of planar 4-valent graphs; the resulting planar graphs are called the states of the original knotted 4-valent diagram. We evaluate a planar 4-valent graph using the skein relations for graphical states shown in Figure~\ref{fig:GSR}, where $[n]=\displaystyle\frac{{q^n}-q^{-n}}{q-q^{-1}}$ such that $n\in\mathbb{Z}$, $n\geq 2$, and $q$ is a formal parameter. Note that the symbol $P$ was omitted in Figure~\ref{fig:GSR} to avoid clutter.

\begin{figure}[ht]
    \[\raisebox{-12pt}{\includegraphics[height=0.4in]{Figure-2/In-Out-In-Out.pdf}}=\raisebox{-12pt}{\includegraphics[height =0.4in]{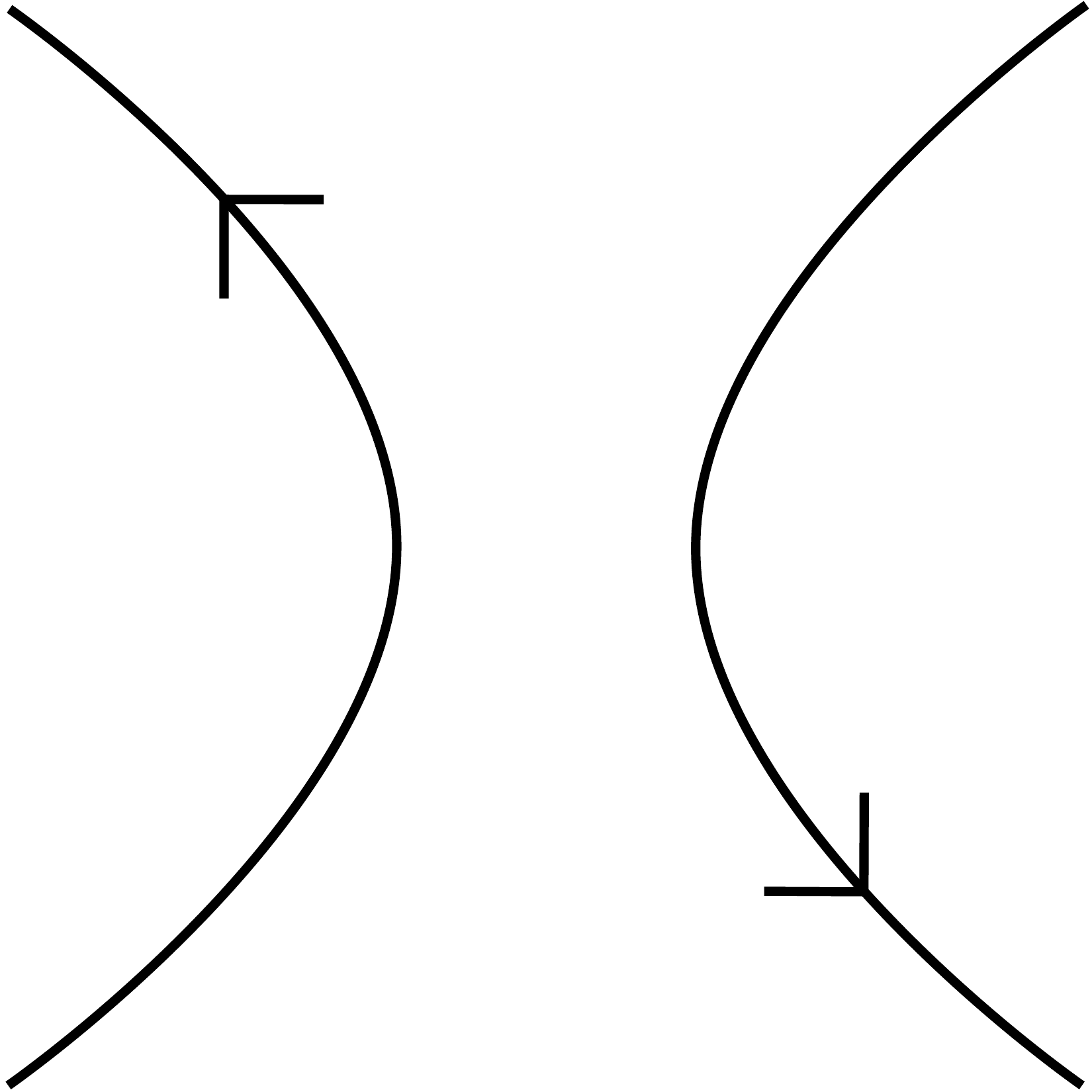}}+\raisebox{-12pt}{\includegraphics[height=0.4in]{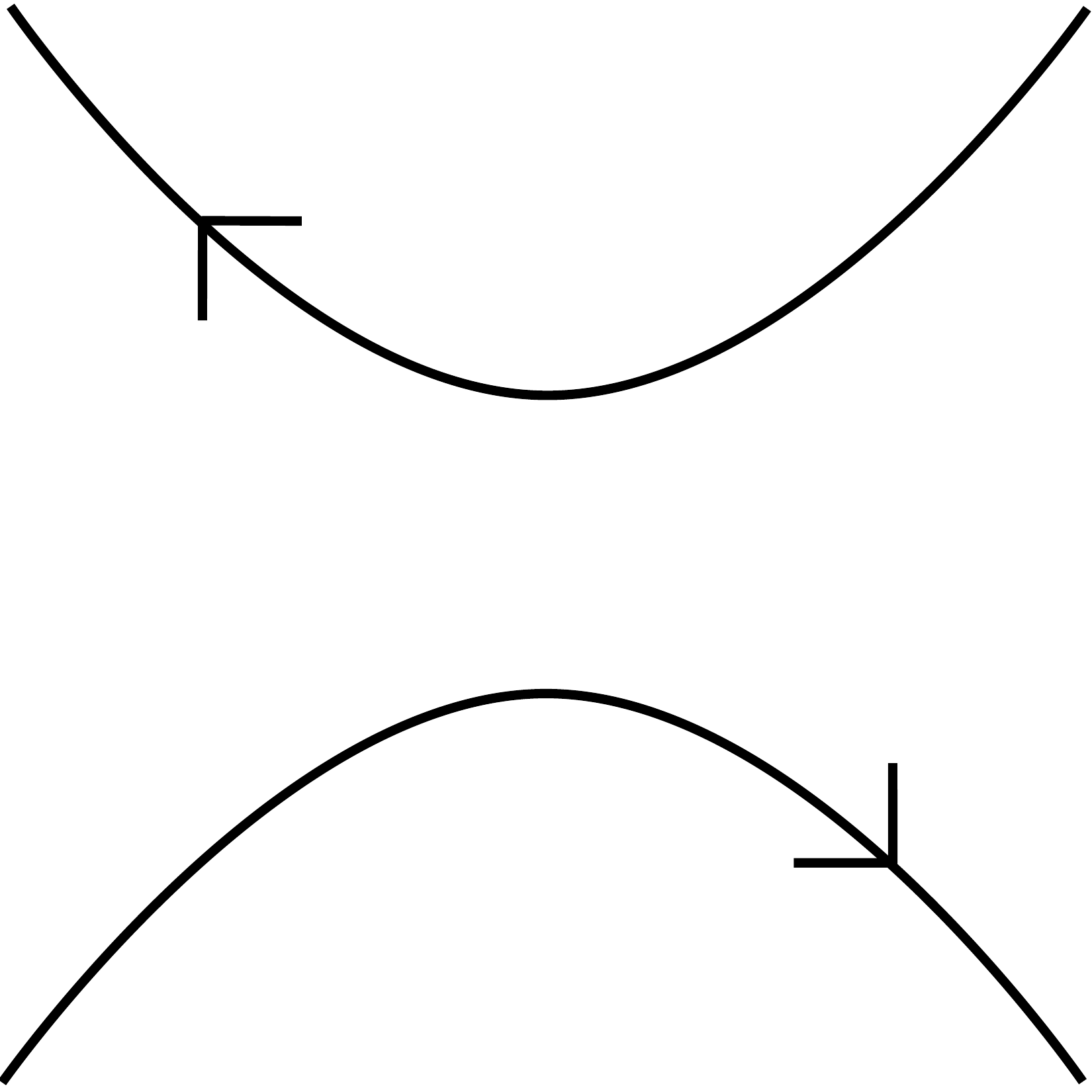}}\qquad
    \raisebox{-15pt}{\includegraphics[height=0.45in]{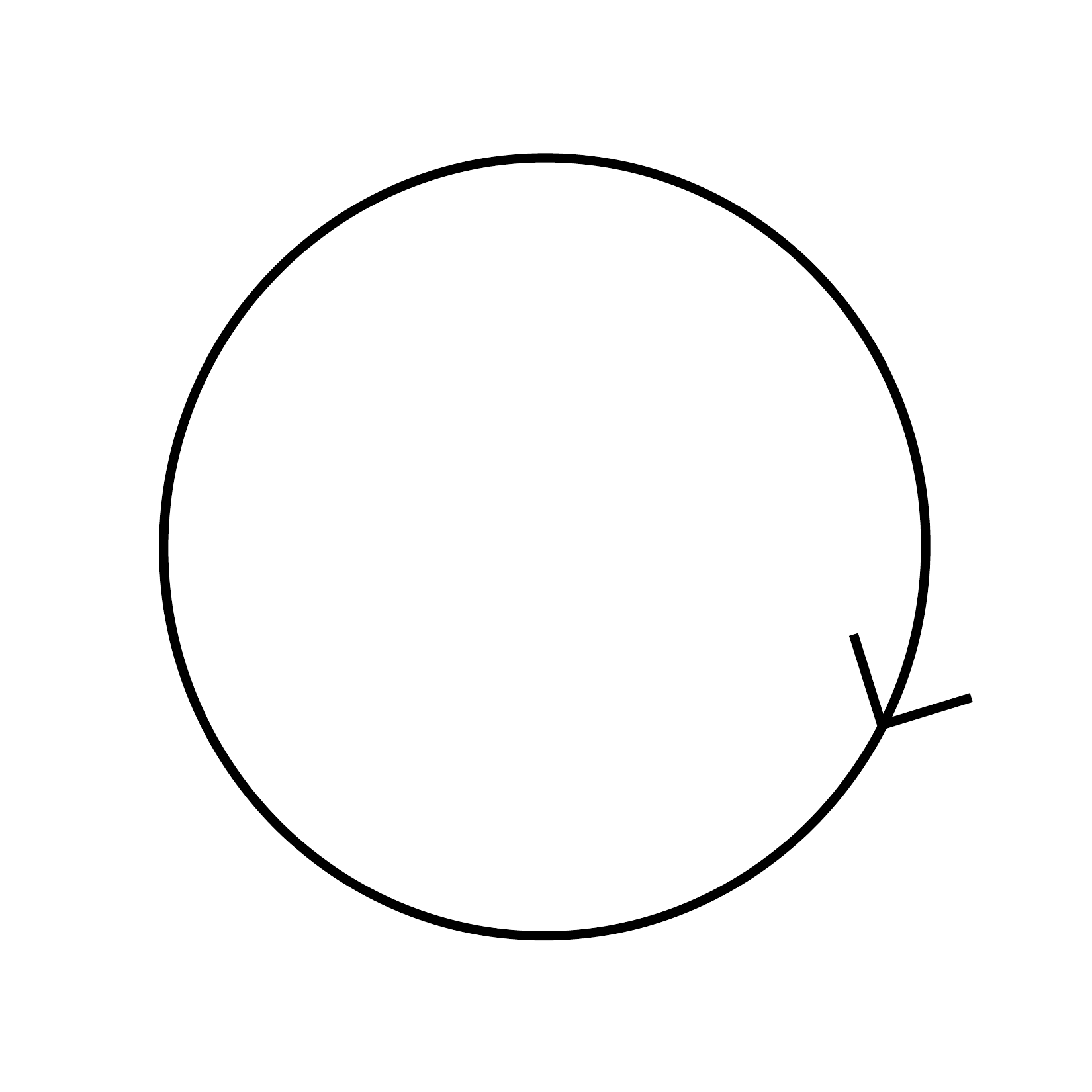}}=[n] \qquad \raisebox{-12pt}{\includegraphics[height=0.4in]{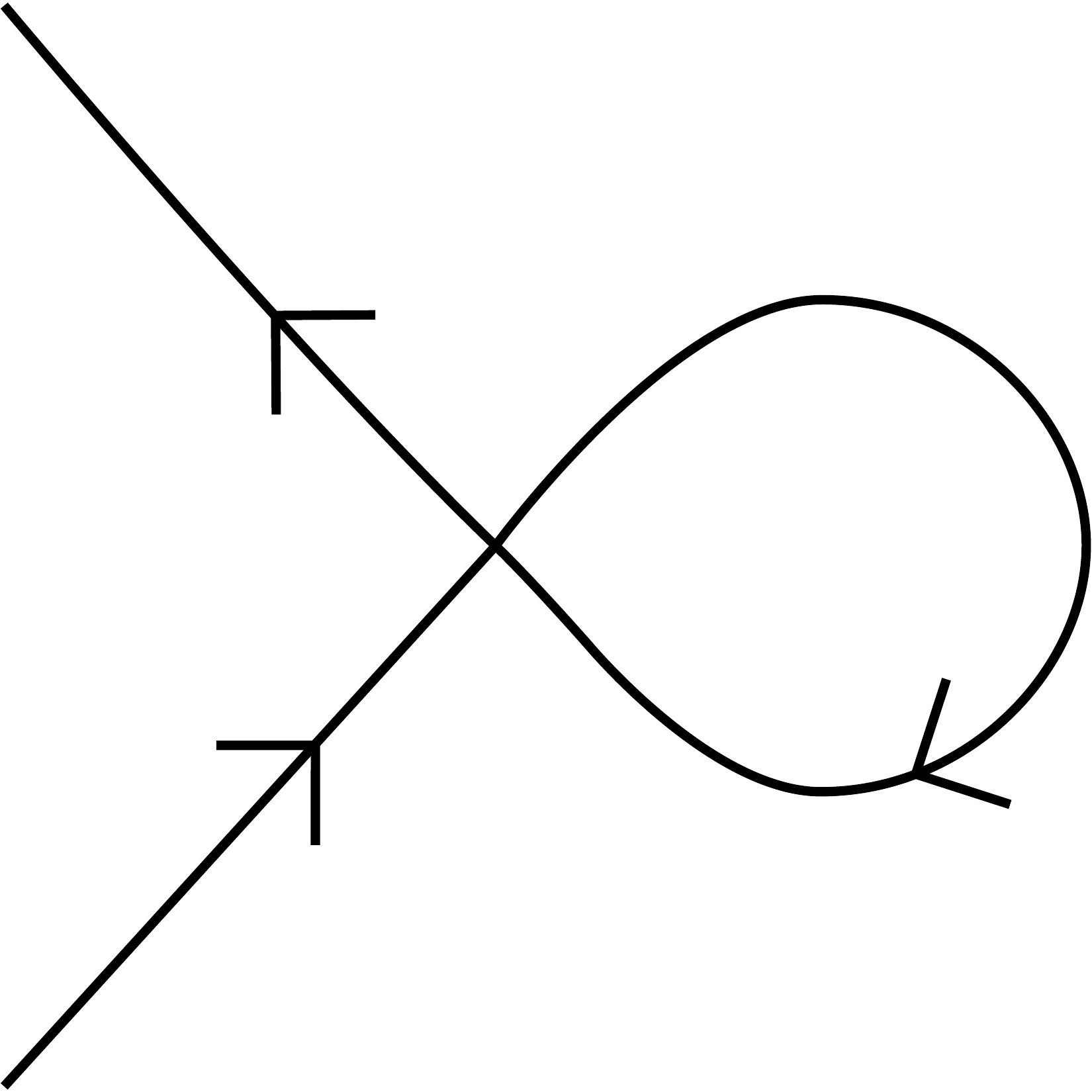}}=[n-1]\raisebox{-18pt}{\includegraphics[width=0.7in]{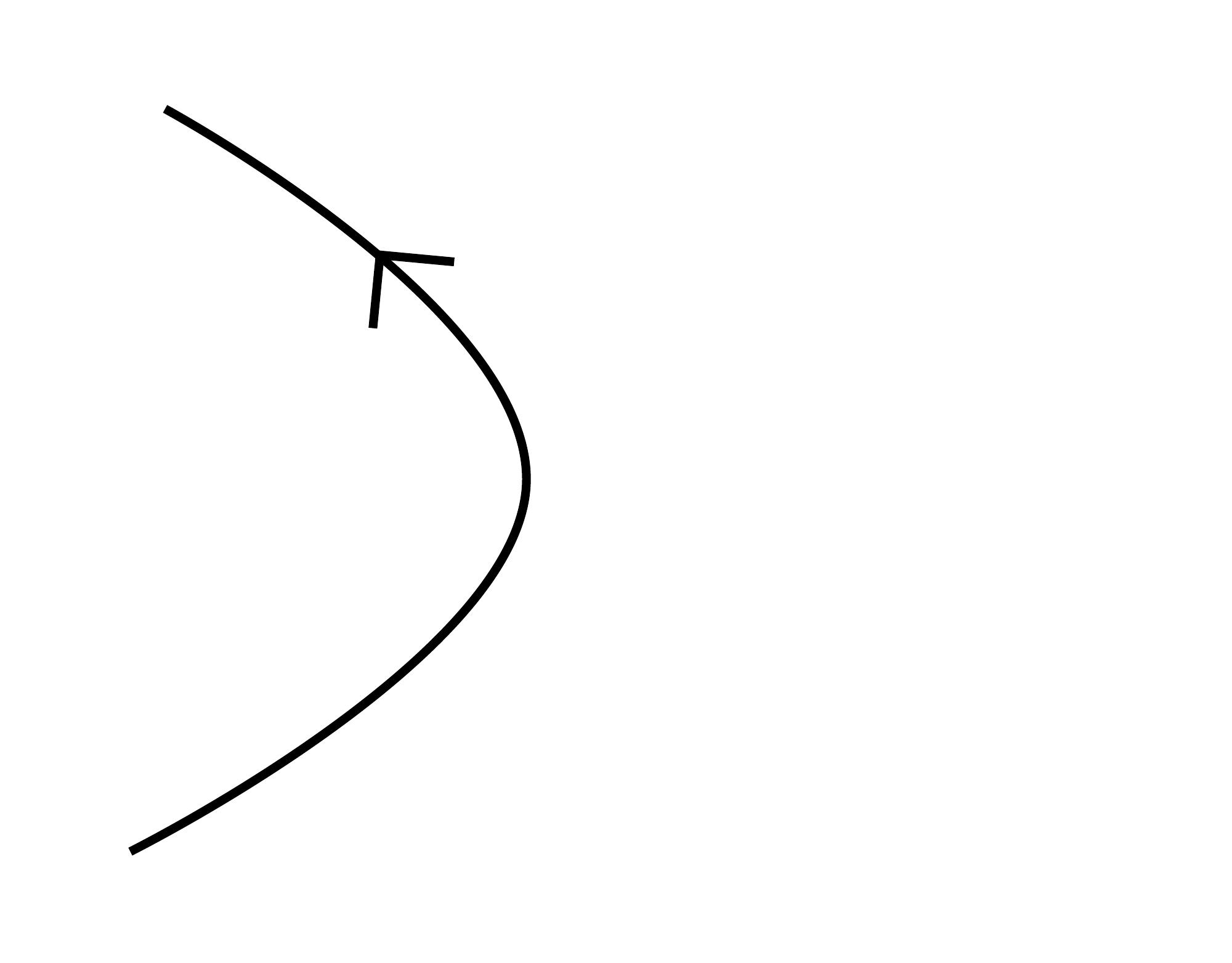}}
    \]

    \[
    \raisebox{-12pt}{\includegraphics[width=0.4in]{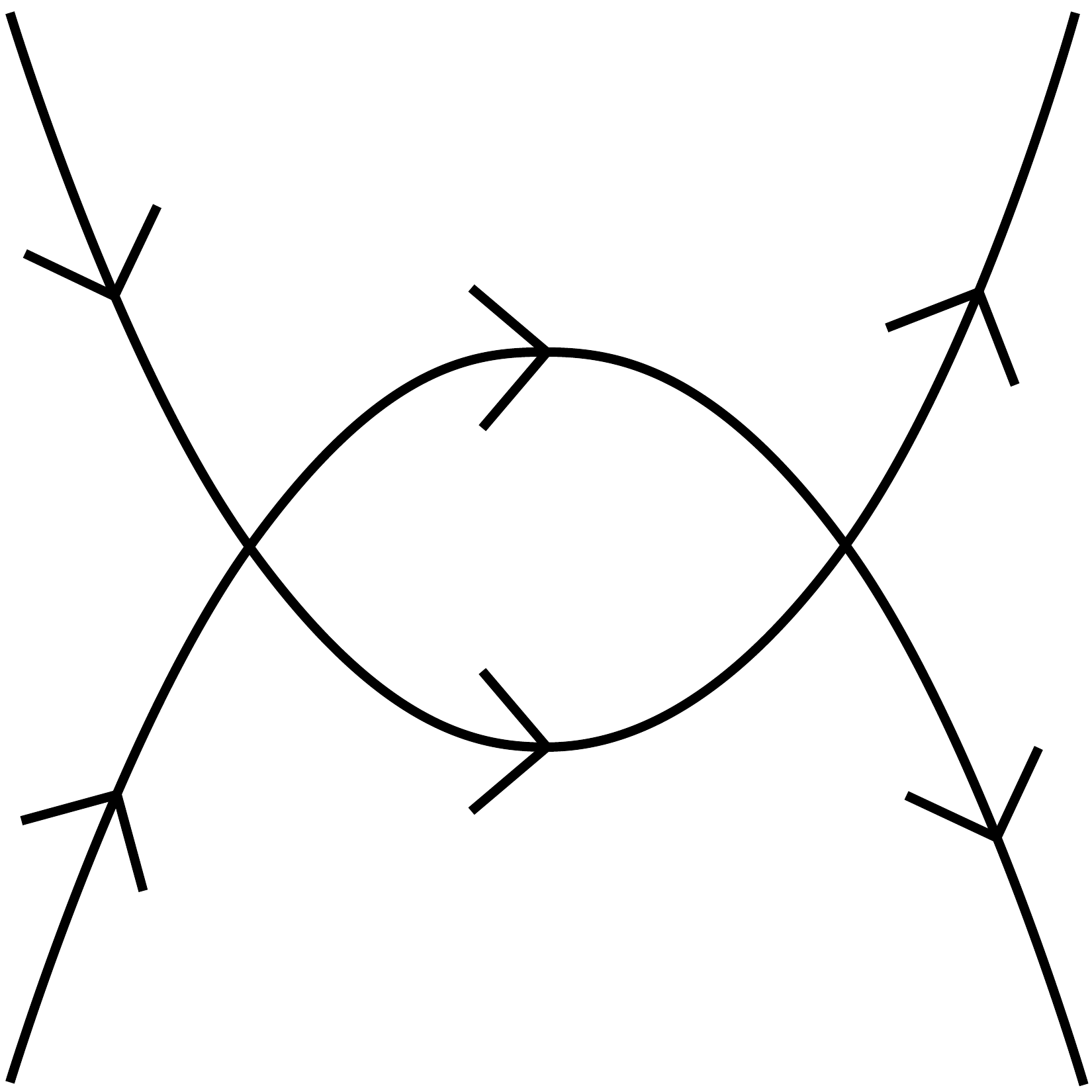}}=[2]\raisebox{16pt}{\rotatebox{270}{\includegraphics[height=0.4in]{Figure-2/In-In-Out-Out.pdf}}}\qquad\quad\raisebox{-12pt}{\includegraphics[width=.4in]{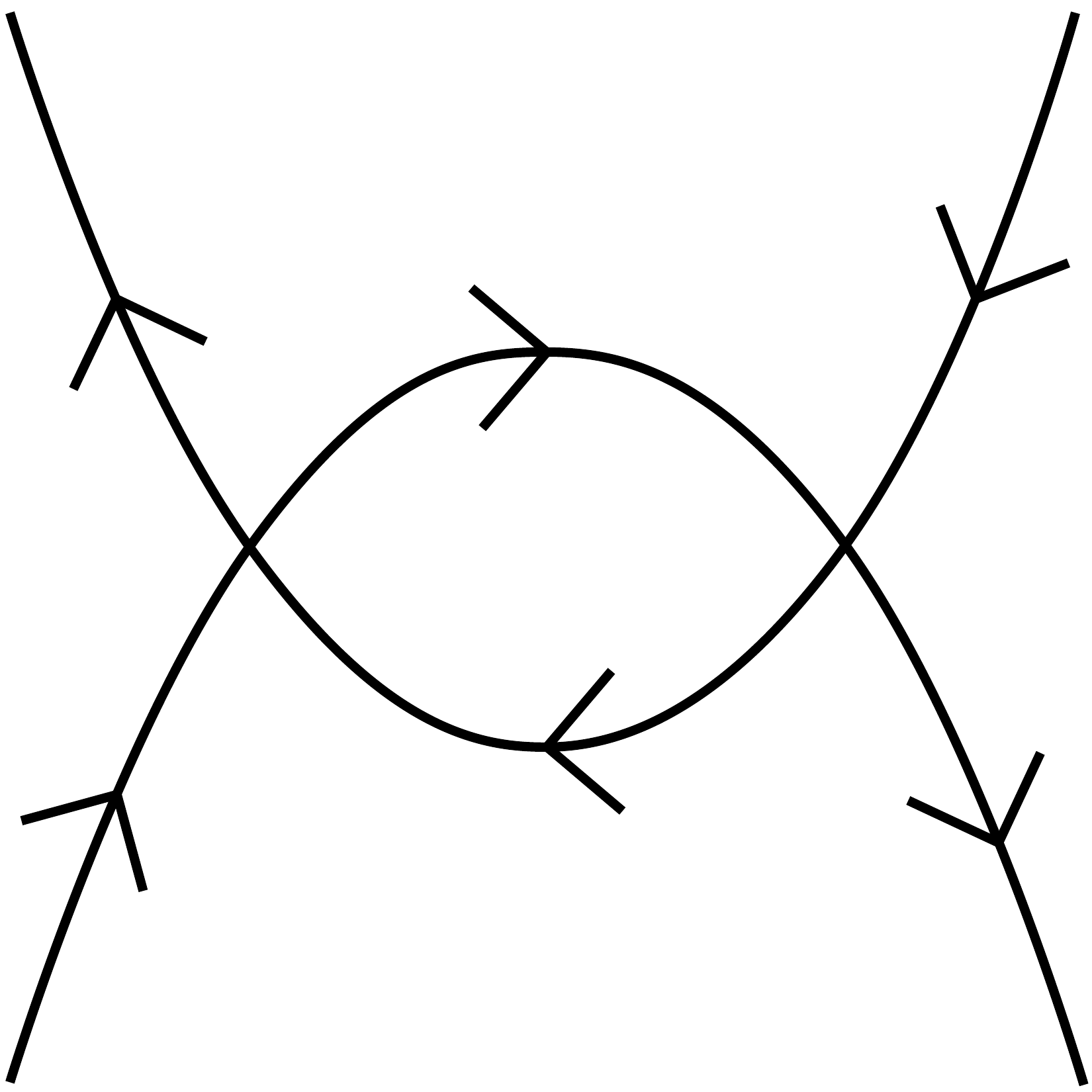}}=\raisebox{-10pt}{\includegraphics[width=.35in]{Graphical-Skein-Relations/two-oriented-strands-side.pdf}}+[n-2]\raisebox{-10pt}{\includegraphics[width=.35in]{Graphical-Skein-Relations/two-oriented-strands-up.pdf}}
    \]
    \[
    \raisebox{-23pt}{\includegraphics[width=.7in]{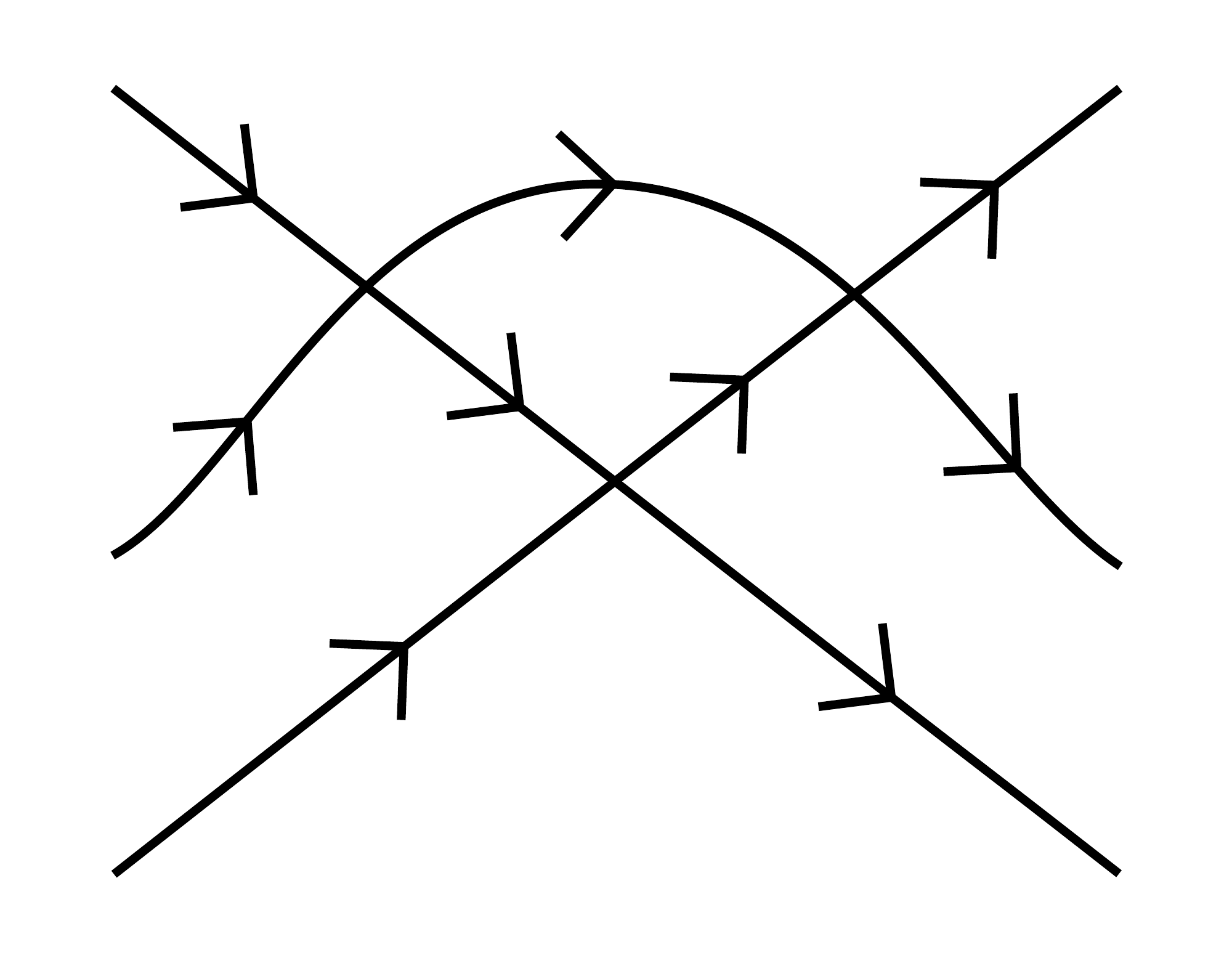}}\quad\raisebox{-6pt}{+}\quad\raisebox{-23pt}{\includegraphics[width=.7in]{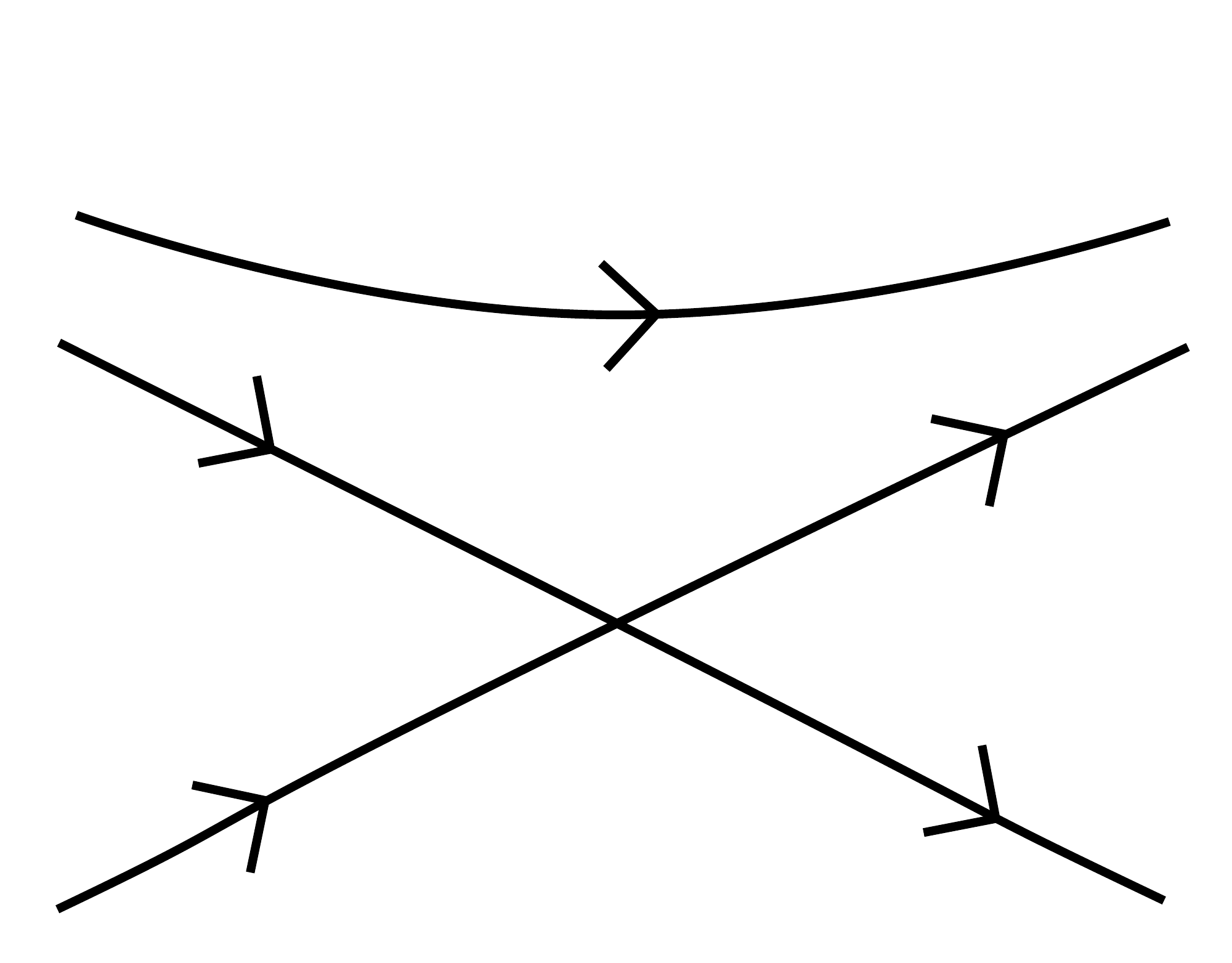}}\quad\raisebox{-6pt}{=}\quad\raisebox{-23pt}{\includegraphics[width=.7in]{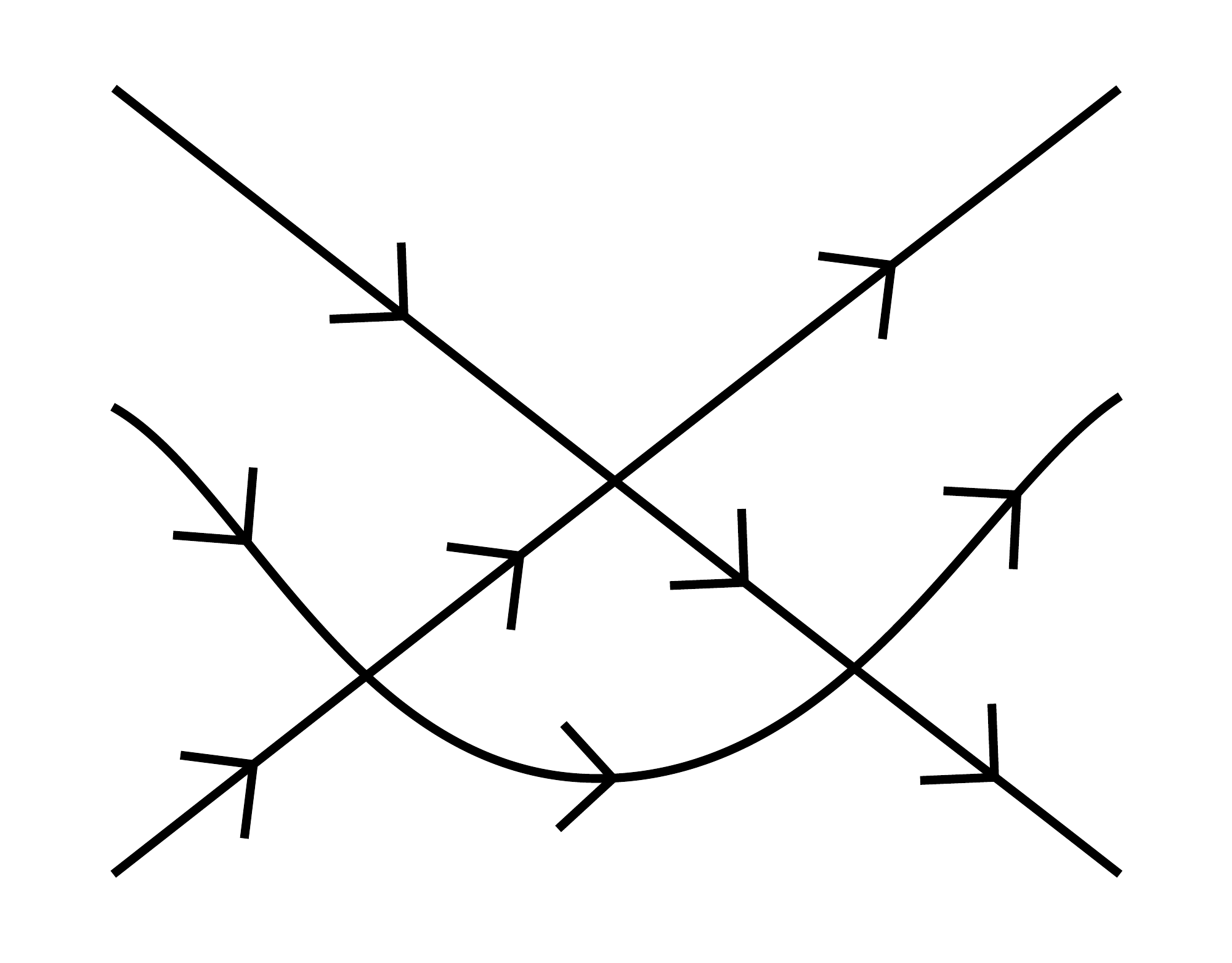}}\quad\raisebox{-6pt}{+}\quad\raisebox{-23pt}{{\includegraphics[width=.7in]{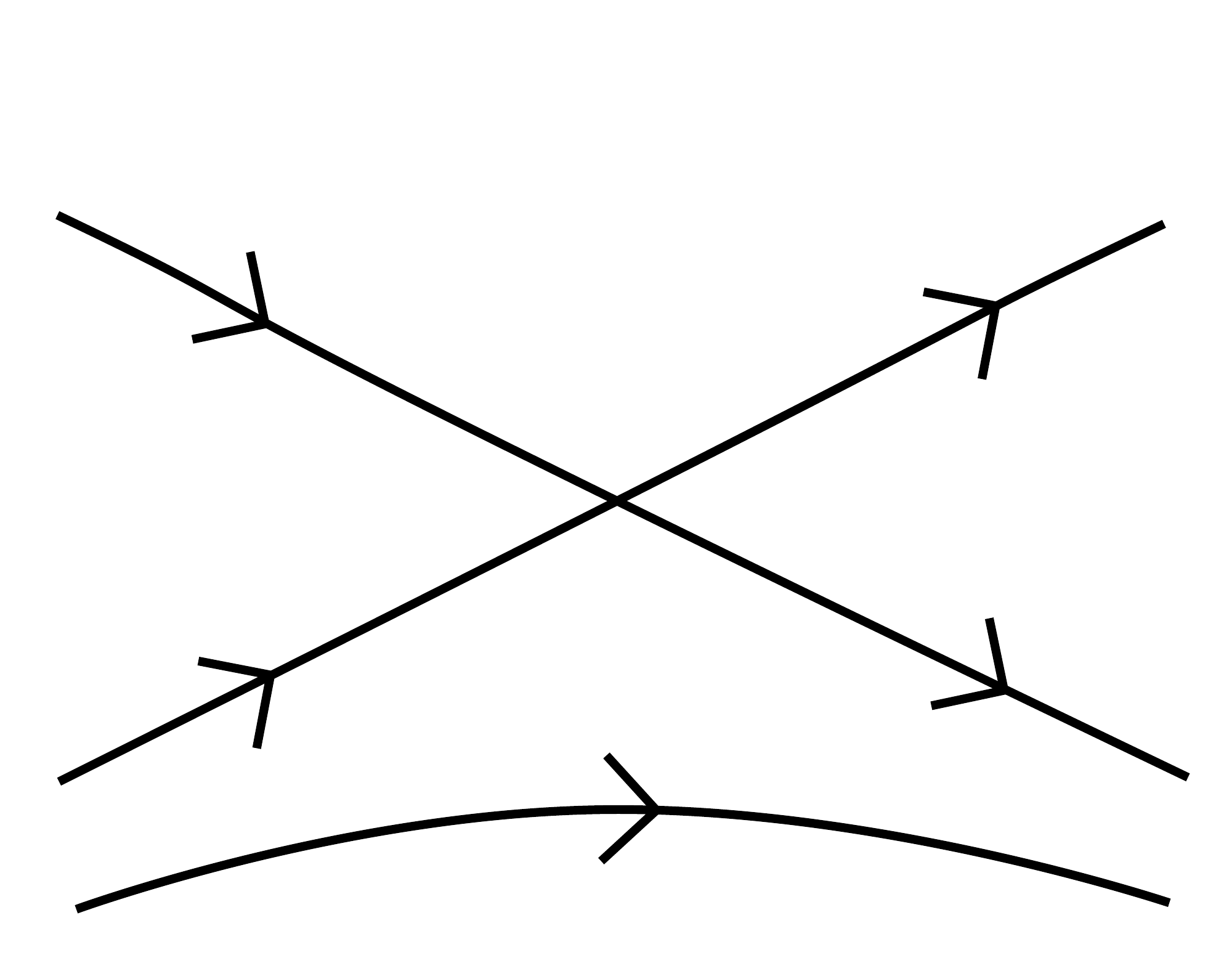}}}
    \]
    \[
    \raisebox{-30pt}{\includegraphics[width=.7in]{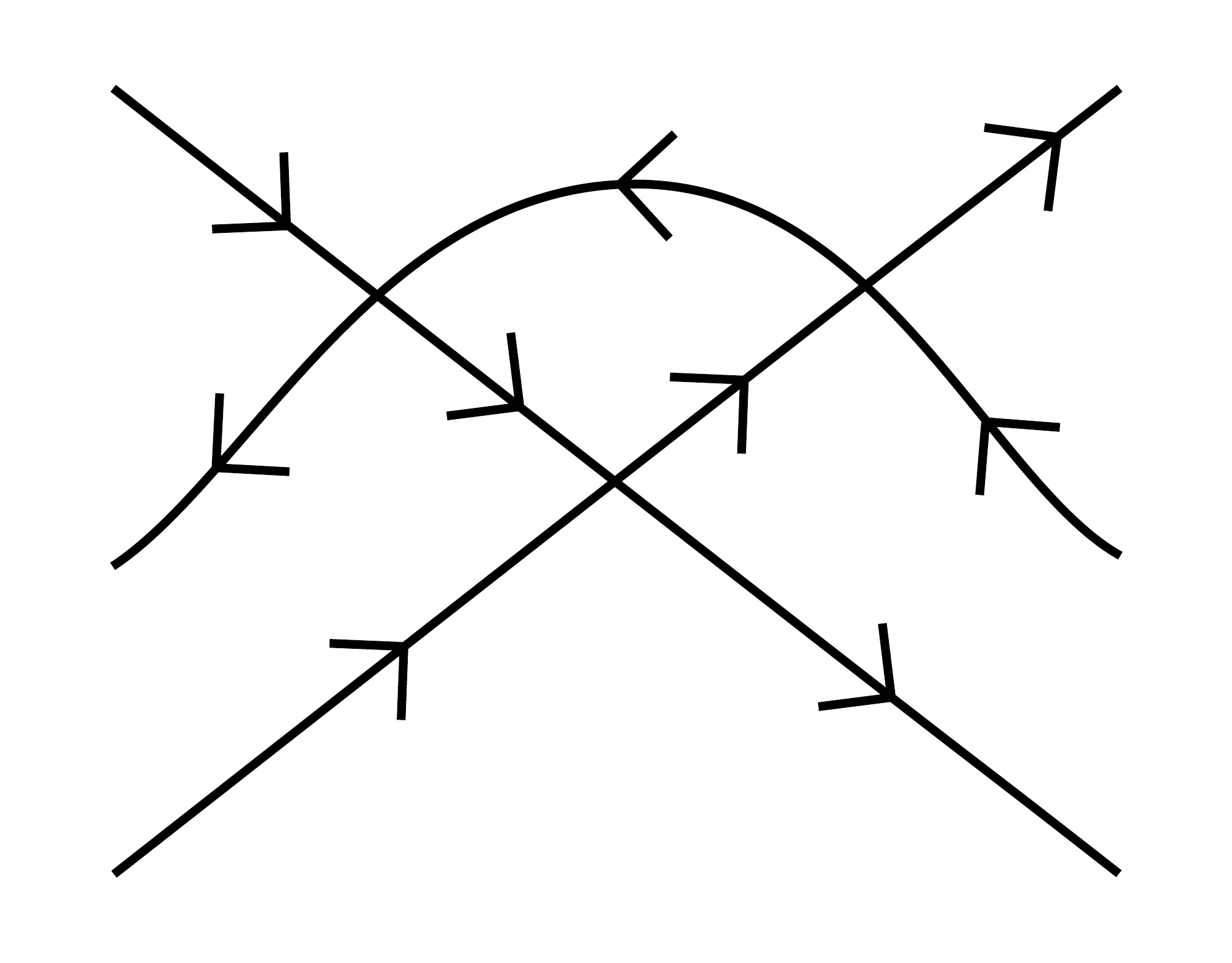}}\quad\raisebox{-13pt}{+\text{$[n-3]$}}\quad\raisebox{-30pt}{\includegraphics[width=.7in]{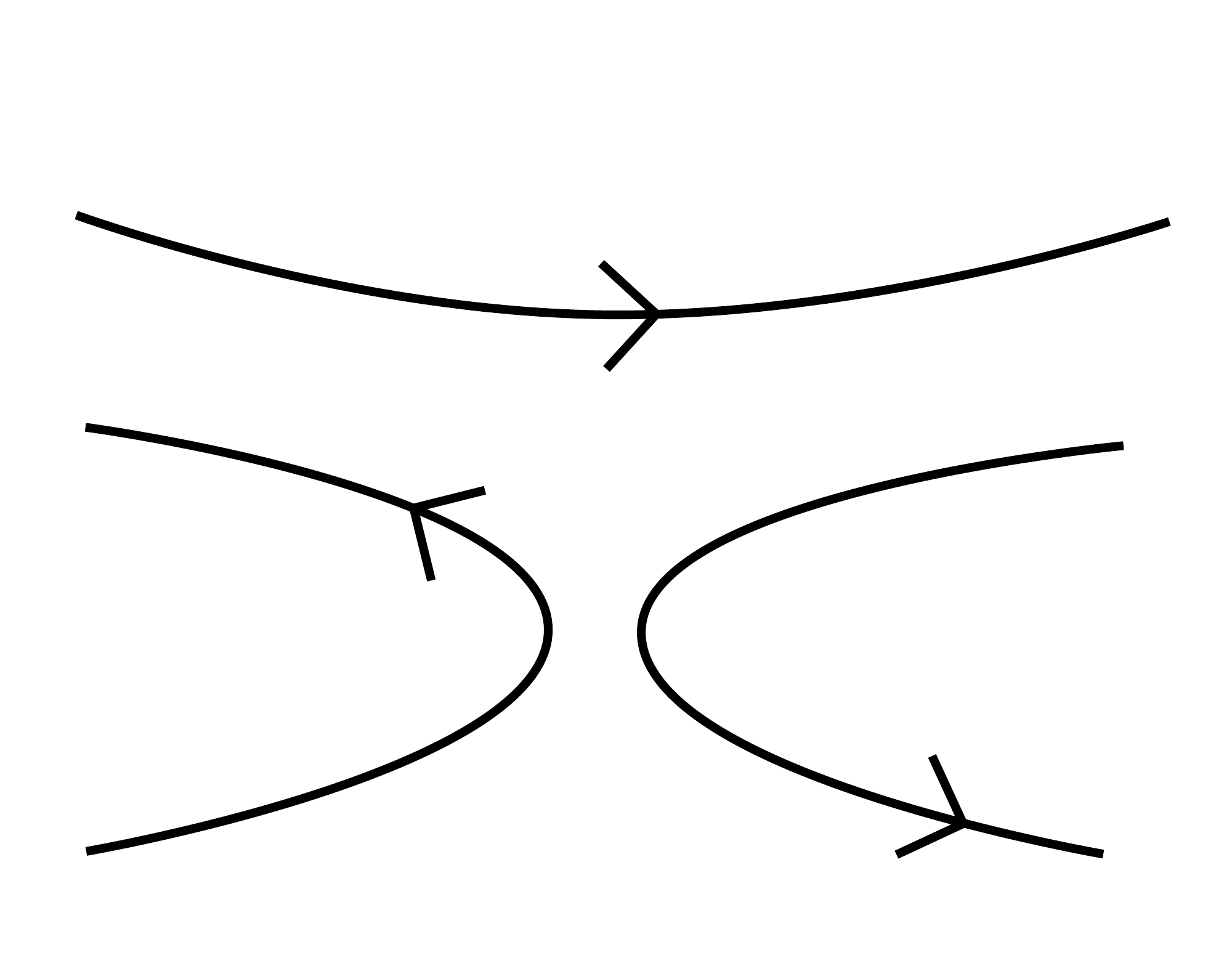}}\quad\raisebox{-12pt}{=}\quad\raisebox{-30pt}{\includegraphics[width=.7in]{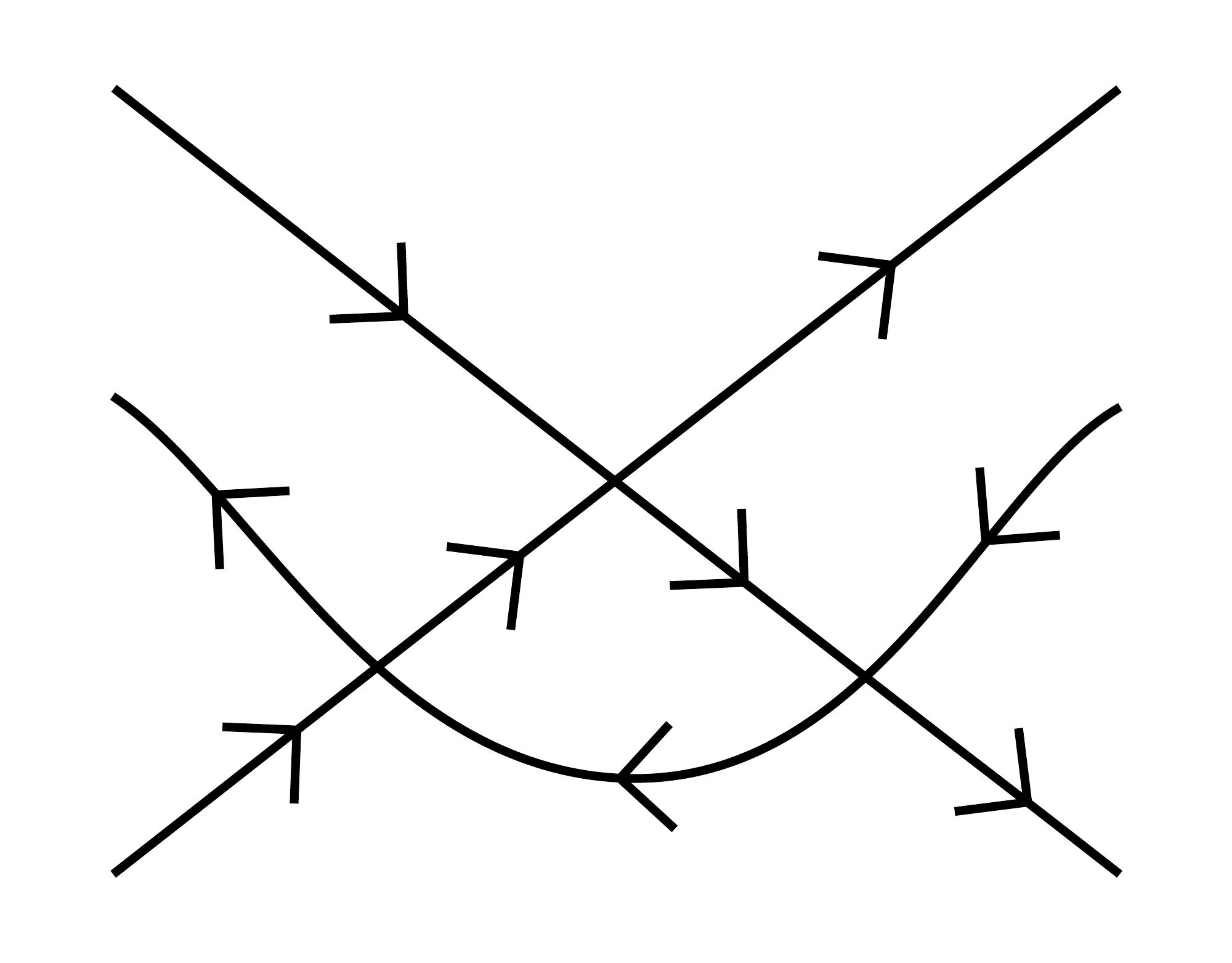}}\quad\raisebox{-13pt}{+\text{$[n-3]$}}\quad\raisebox{-30pt}{\includegraphics[width=.7in]{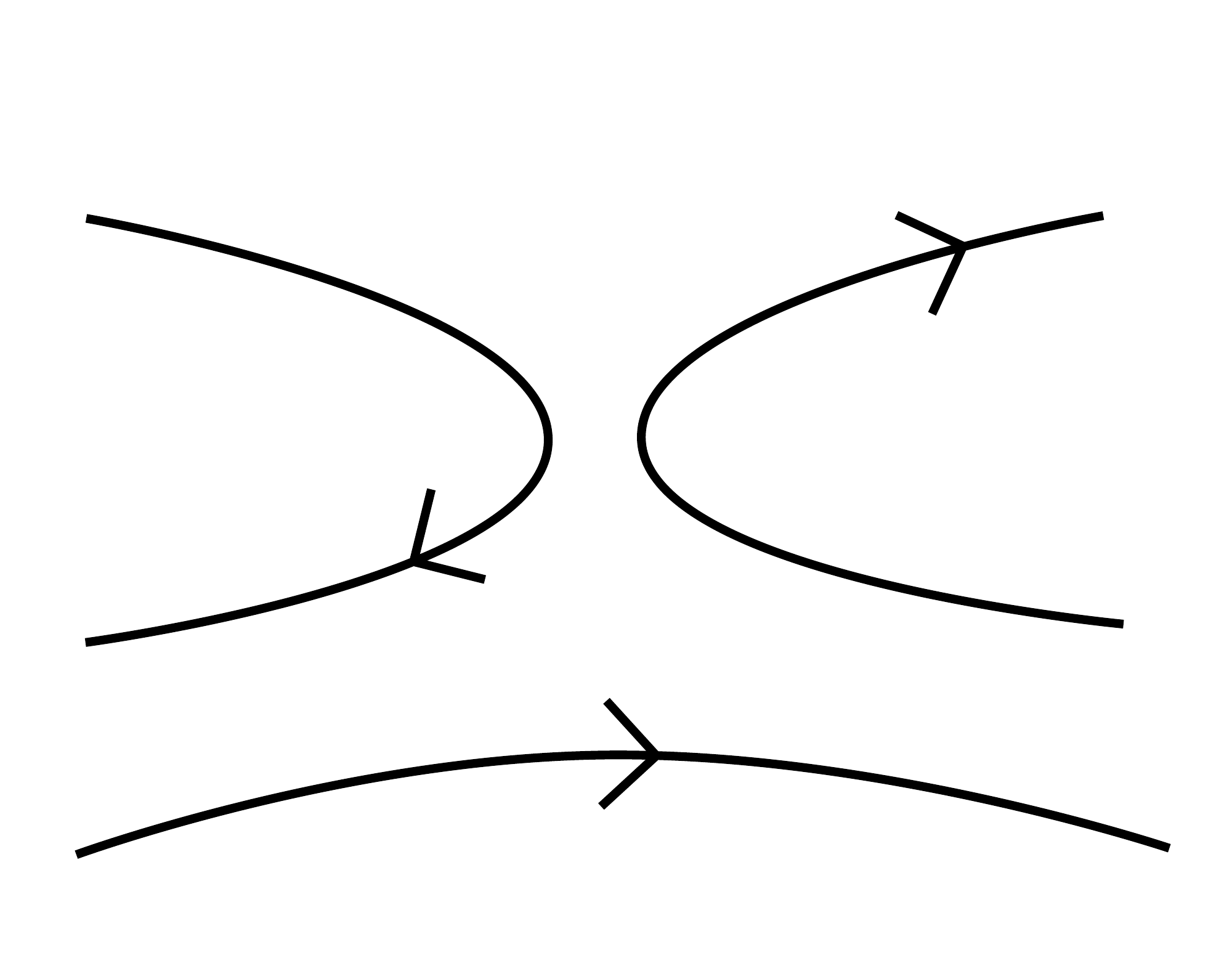}}
    \]
    \caption{Graphical skein relations}
    \label{fig:GSR}
\end{figure} 
The skein relations in Figure~\ref{fig:SR Crossings} imply that the polynomial $P$ satisfies the following skein relation:
\[q^{n}P\Biggl(\raisebox{-10pt}{\includegraphics[height = .35in]{Skein-relations/negcrossing.pdf}}\Biggr) - q^{-n}P\Biggl(\raisebox{-10pt}{\includegraphics[height = .35in]{Skein-relations/poscrossing.pdf}}\Biggr) = (q-q^{-1})P\Biggl(\raisebox{-8pt}{\includegraphics[height = .35in]{Skein-relations/oriented-strands-up.pdf}}\Biggr).  \]
This is combined with 
\[ P\left(D \cup \raisebox{-8pt}{\includegraphics[height=0.3in]{Graphical-Skein-Relations/oriented-circle.pdf }}\right) = [n]\, P(D), \]
where $D \cup \raisebox{-5pt}{\includegraphics[height=0.25in]{Graphical-Skein-Relations/oriented-circle.pdf }}$ is a disjoint union of any diagram $D$ and the standard diagram of the unknot. 
Therefore, when restricted to diagrams of knots (diagrams that contain only crossings and no vertices), our polynomial $P$ satisfies the defining skein relations for the $sl(n)$ polynomial for knots. Hence, $P$ can be regarded as an extension of the $sl(n)$ polynomial to balanced-oriented, knotted 4-valent graphs.

\begin{lemma}\label{lemma relations}
Let $G$ be a connected, oriented, 4-valent planar graph with vertices of type In-In-Out-Out. If $G$ contains at least one vertex and does not contain a loop $\raisebox{-7pt}{\includegraphics[height=0.25in]{Graphical-Skein-Relations/oriented-loop.pdf}}$ or a bigon $\raisebox{-7pt}{\includegraphics[width=0.25in]{Graphical-Skein-Relations/GSR-1.pdf}}$, $\raisebox{-7pt}{\includegraphics[width=0.25in]{Graphical-Skein-Relations/GSR-2.pdf}}$, then $G$ can be transformed into a 4-valent planar graph with vertices of type In-In-Out-Out and which contains a bigon, by using a finite sequence of the following moves:
\[ \raisebox{-8pt}{\includegraphics[width=.4in]{Graphical-Skein-Relations/GSR-5.pdf}} \leftrightarrow \raisebox{-8pt}{\includegraphics[width=.4in]{Graphical-Skein-Relations/GSR-7.pdf}} \quad \text{and} \quad \raisebox{-8pt}{\includegraphics[width=.4in]{Graphical-Skein-Relations/GSR-9.pdf}}  \leftrightarrow \raisebox{-8pt}{\includegraphics[width=.4in]{Graphical-Skein-Relations/GSR-11.pdf}} \]
\end{lemma}

\begin{proof}
This statement can be proved in a similar manner as in Carpentier's paper~\cite[Lemma 2]{Carpentier}.
\end{proof}

The graphical skein relations given in Figure~\ref{fig:GSR} allow us to evaluate any 4-valent planar graph $G$. Using the first skein relation, we can write the evaluation $P(G)$ as a linear combination of evaluations of 4-valent planar graphs with fewer vertices, all of which are of type In-In-Out-Out. By applying the skein relations involving a loop or a bigon, we obtain $P(G)$ written as a $\mathbb{Z}[q, q^{-1}]$-linear combination of evaluations of 4-valent planar graphs with fewer vertices of type In-In-Out-Out. Then by Lemma~\ref{lemma relations}, if a 4-valent planar graph with vertices of type In-In-Out-Out does not contain a loop or a bigon, we can use the last two graphical skein relations from Figure~\ref{fig:GSR} in a neighborhood of a triangular face, to write the evaluation of the graph in terms of evaluations of graphs with fewer vertices and the evaluation of a graph that contains a bigon. The latter graph is then evaluated using one of the graphical relations involving a bigon. We continue this process until $P(G)$ is written as a linear combination of evaluations of disjoint unions of circles, where each circle is then replaced by $[n]$.

\begin{proposition}
There is a unique polynomial $P(G)\in \mathbb{Z}[q, q^{-1}]$ associated to a planar, balanced-oriented, 4-valent graph $G$, such that the polynomial satisfies the graphical relations given in Figure~\ref{fig:GSR}.
\end{proposition}

\begin{proof} The existence part follows from Lemma~\ref{lemma relations} and the above paragraph. The uniqueness part is proved in a similar way as~\cite[Theorem 3]{Carpentier}, where the proof is by induction on the number of 4-valent vertices, as we now show.
    Let $P(G)$ and $P^*(G)$ be two polynomials associated to a planar, balanced-oriented, 4-valent graph $G$ and such that $P$ and $P^*$ satisfy the graphical relations in Figure~\ref{fig:GSR}.

    If $G$ has no vertices, then since $G$ is planar, it must be the disjoint union of unknotted circles. Hence, $P(G)=[n]^\alpha=P^*(G)$, where $\alpha$ is the number of unknotted circles in $G$.

    Suppose $P(G)=P^*(G)$ for all planar, balanced-oriented, 4-valent graphs $G$ with up to $n$ vertices. Let $G$ be a planar 4-valent graph containing $n+1$ vertices. If $G$ contains vertices of type In-Out-In-Out, we then resolve them using the first graphical relation in Figure~\ref{fig:GSR}, to write the evaluation  of $G$ in terms of planar graphs with fewer vertices, all of which are of type In-In-Out-Out. Thus, by the inductive hypothesis, $P(G) = P^*(G)$. On the other hand, if $G$ contains only vertices of type In-In-Out-Out and it contains a loop or a bigon, then using the graphical skein relations in Figure~\ref{fig:GSR} we can write the polynomials associated to $G$ as a sum of polynomials associated to graphs with fewer vertices than $G$. Thus by our inductive hypothesis, $P(G)=P^*(G)$. Now, if $G$ contains only vertices of type In-In-Out-Out and it does not contain a loop or a bigon, then by Lemma~\ref{lemma relations} there exists a finite sequence of graphs
    \[G=G_0\rightarrow G_1\rightarrow G_2\rightarrow\cdots\rightarrow G_m,\]
    where $G_m$ contains a bigon and where for each $i$ with $1 \leq i\leq m$, $G_i$ is obtained from $G_{i-1}$ via one of the moves below:
    \[ \raisebox{-8pt}{\includegraphics[width=.4in]{Graphical-Skein-Relations/GSR-5.pdf}} \leftrightarrow \raisebox{-8pt}{\includegraphics[width=.4in]{Graphical-Skein-Relations/GSR-7.pdf}} \quad \text{or} \quad \raisebox{-8pt}{\includegraphics[width=.4in]{Graphical-Skein-Relations/GSR-9.pdf}}  \leftrightarrow \raisebox{-8pt}{\includegraphics[width=.4in]{Graphical-Skein-Relations/GSR-11.pdf}} \]
 
     By the last two graphical relations in Figure~\ref{fig:GSR}, we know that $P\biggl(\raisebox{-8pt}{\includegraphics[width=.4in]{Graphical-Skein-Relations/GSR-5.pdf}}\biggr)-P\biggl(\raisebox{-8pt}{\includegraphics[width=.4in]{Graphical-Skein-Relations/GSR-7.pdf}}\biggr)$ and $P\biggl(\raisebox{-8pt}{\includegraphics[width=.4in]{Graphical-Skein-Relations/GSR-9.pdf}}\biggr)-P\biggl(\raisebox{-8pt}{\includegraphics[width=.4in]{Graphical-Skein-Relations/GSR-11.pdf}}\biggr)$, as well as $P^*\biggl(\raisebox{-8pt}{\includegraphics[width=.4in]{Graphical-Skein-Relations/GSR-5.pdf}}\biggr)-P^*\biggl(\raisebox{-8pt}{\includegraphics[width=.4in]{Graphical-Skein-Relations/GSR-7.pdf}}\biggr)$ and $P^*\biggl(\raisebox{-8pt}{\includegraphics[width=.4in]{Graphical-Skein-Relations/GSR-9.pdf}}\biggr)-P^*\biggl(\raisebox{-8pt}{\includegraphics[width=.4in]{Graphical-Skein-Relations/GSR-11.pdf}}\biggr)$, can be written as a sum of polynomials associated to graphs with fewer vertices than $G$. Then by the induction hypothesis,
     \[P(G_i)-P(G_{i-1})=P^*(G_i)-P^*(G_{i-1}),\] for all $i$ where $1\leq i \leq m$. Hence, \[P(G_m)-P(G_0)=P^*(G_m)-P^*(G_0).\] Since $G_m$ contains a bigon, the first part of the proof implies that $P(G_m)=P^*(G_m)$. Therefore, $P(G)=P^*(G)$, as desired. By the principle of mathematical induction, we conclude that the statement holds for any planar 4-valent graph with balanced-oriented vertices.  
\end{proof}

We will prove that $P(L)$ is an invariant for balanced-oriented,  knotted 4-valent graphs $L$ in Section~\ref{sec:proofinv}.

\begin{example}
We provide an example that evaluates the polynomial $P$ for a given diagram. In the first step below, we apply the skein relation for the vertex of type In-Out-In-Out. In the second step, we use the graphical relation to resolve the negative crossing in the second diagram, and use that $P$ is invariant under the first Reidemeister move to simplify the evaluation of the first diagram. In the third step we employ the first graphical relation involving a bigon followed by combining like terms. In the final steps, we use the graphical relation for the loop and that the value of the standard diagram of the unknot is $[n]$.
    \begin{align*}
        P\biggl(\raisebox{-8pt}{\includegraphics[height = .35in]{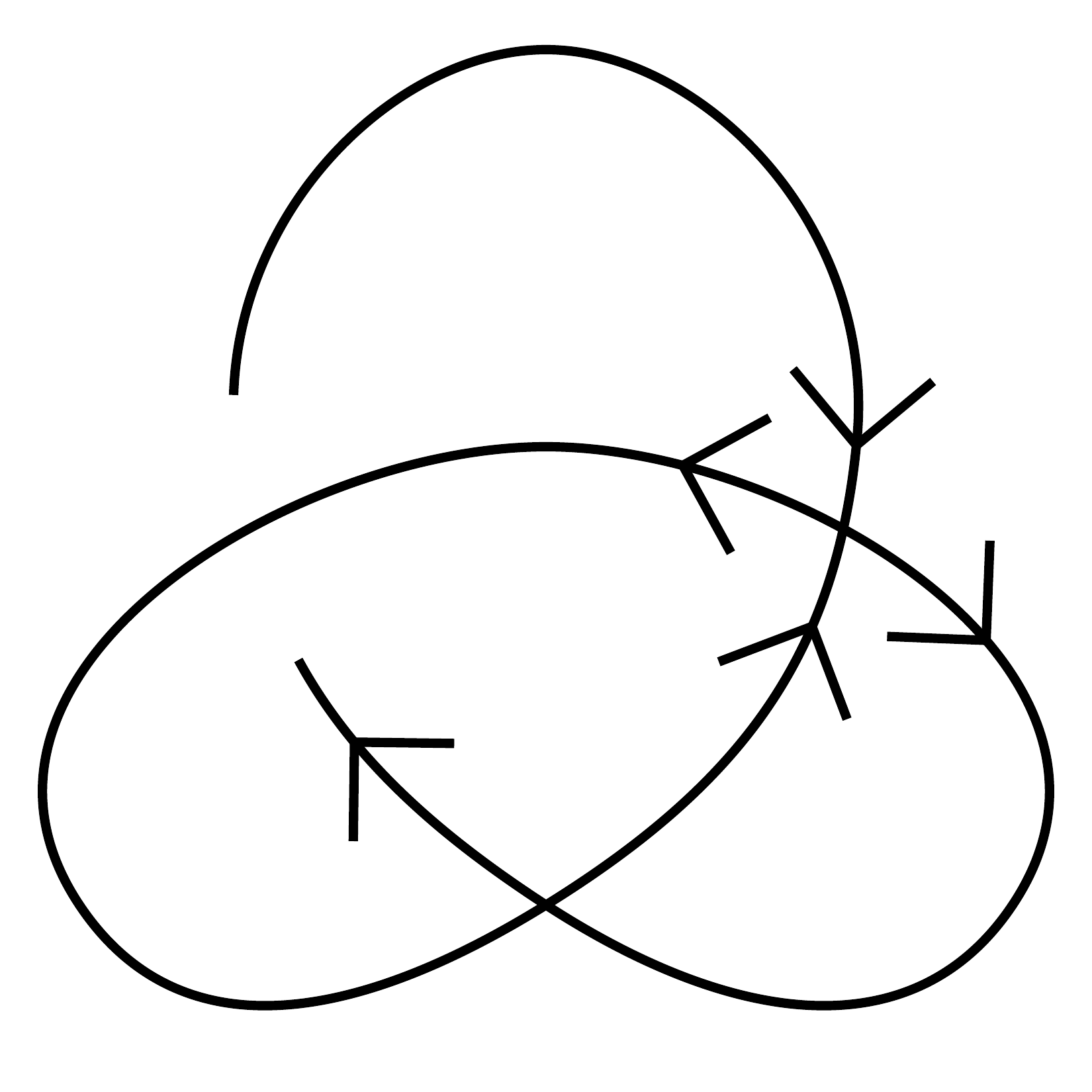}}\biggr)&=P\biggl(\raisebox{-8pt}{\includegraphics[height = .3in]{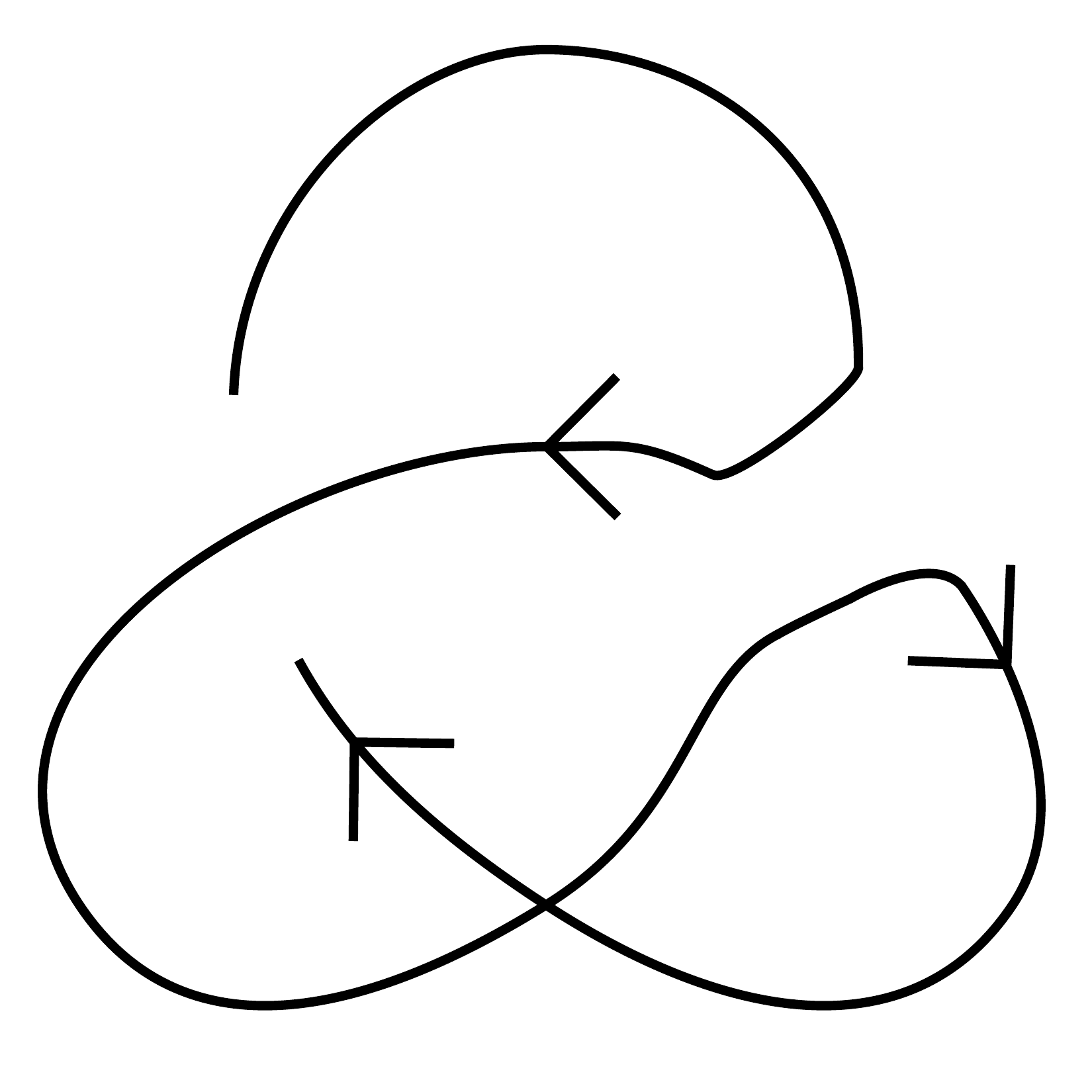}}\biggr)+P\biggl(\raisebox{-8pt}{\includegraphics[height = .3in]{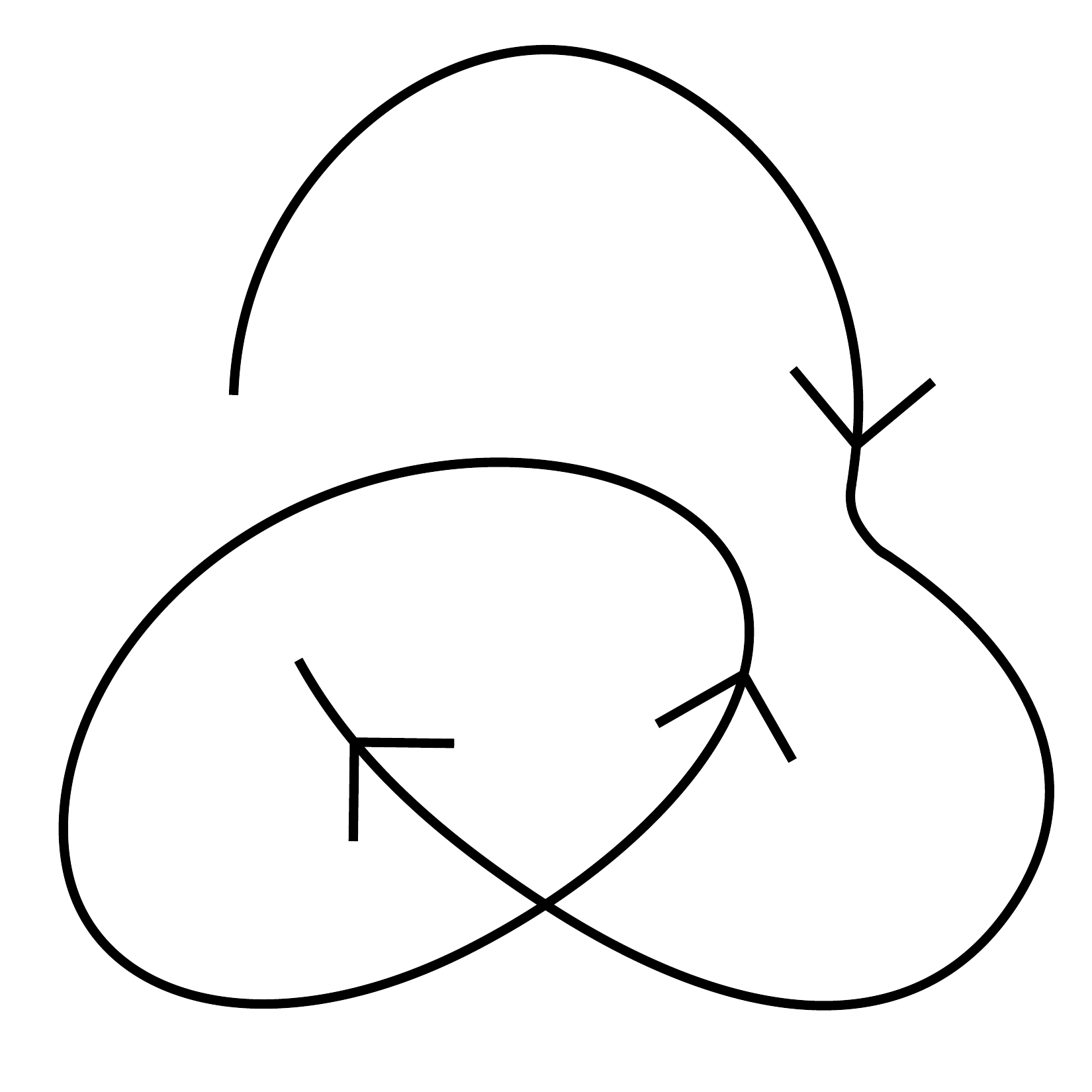}}\biggr) \\
        &=P\biggl(\raisebox{-8pt}{\includegraphics[height = .3in]{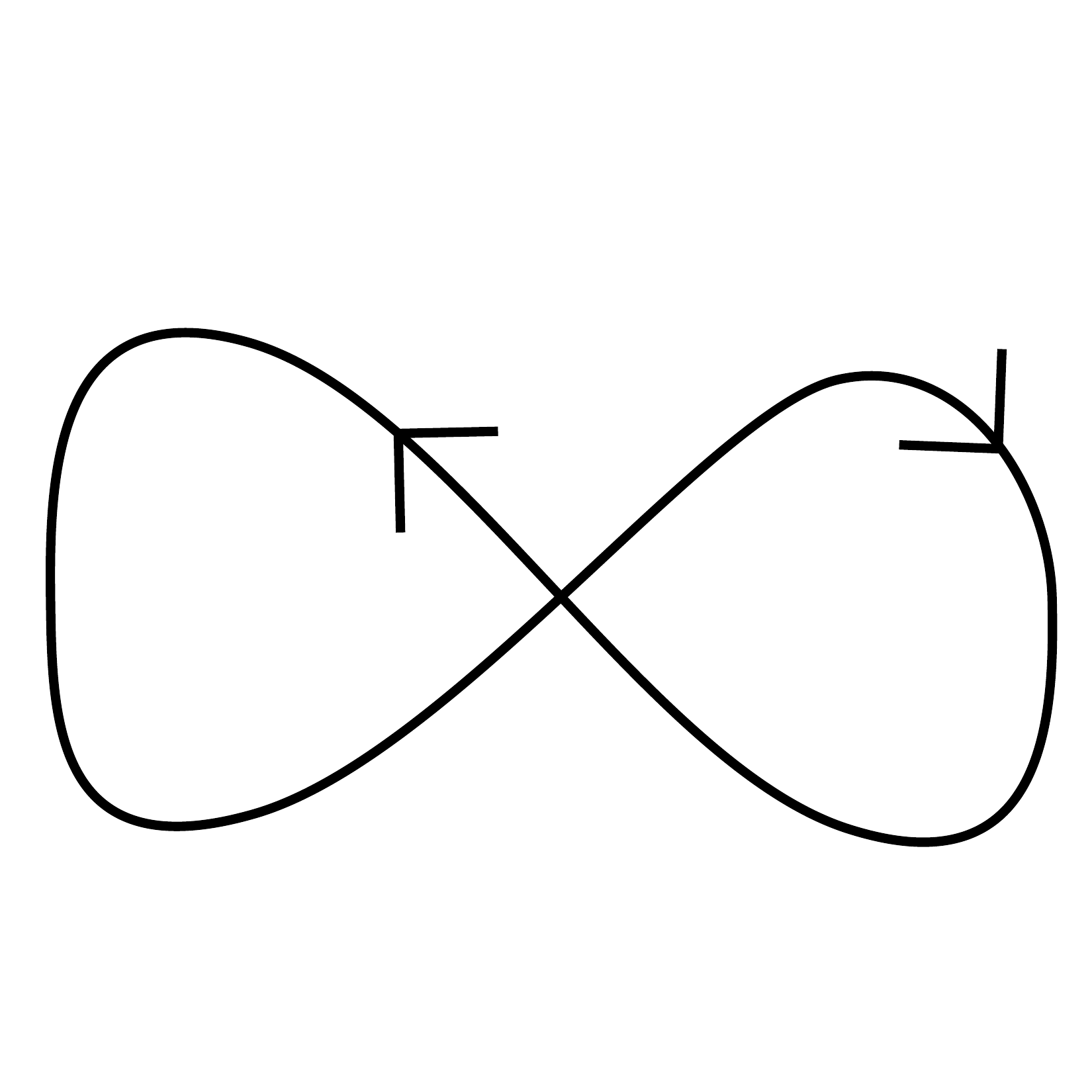}}\biggr)+\left[q^{1-n}P\biggl(\raisebox{-8pt}{\includegraphics[height = .3in]{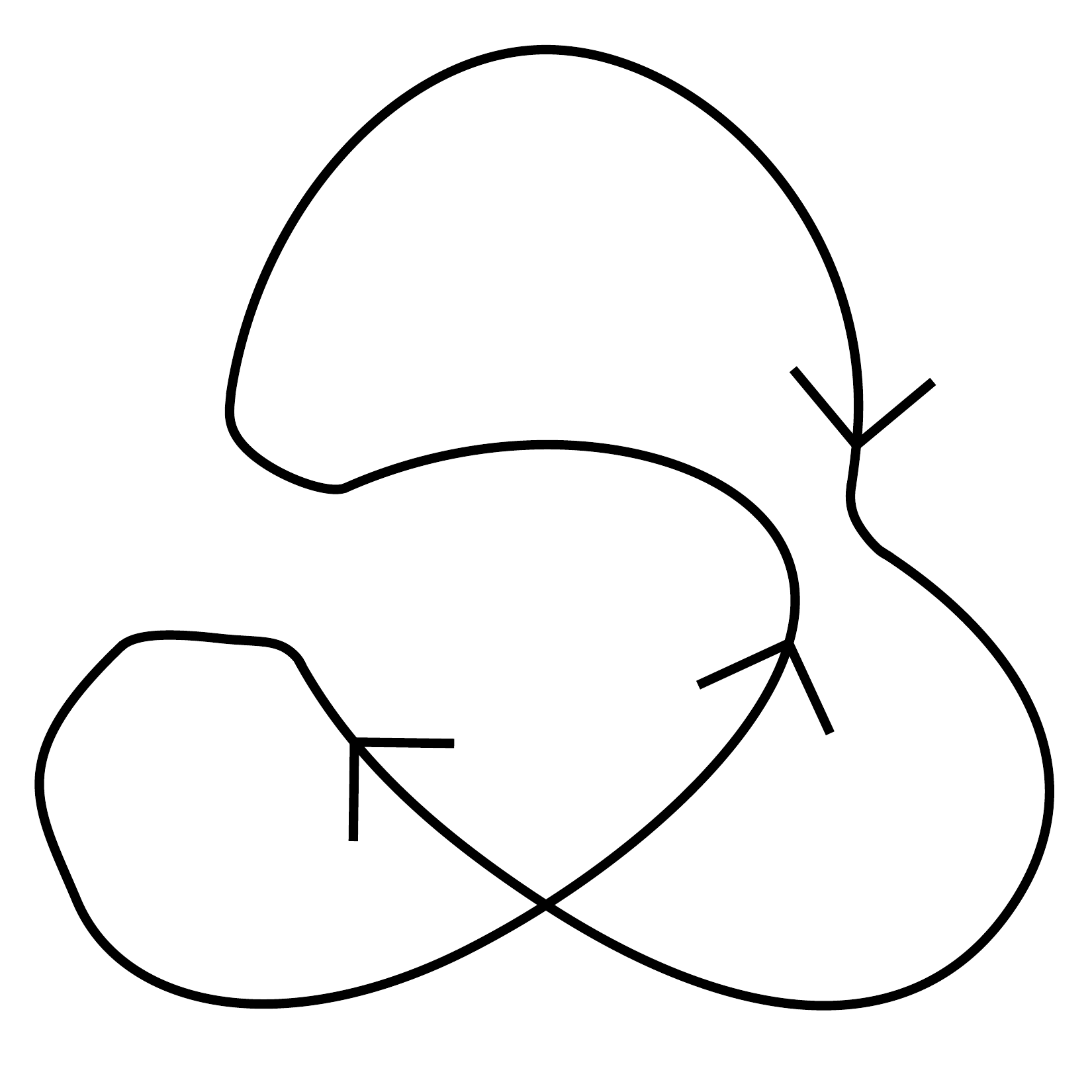}}\biggr)-q^{-n}P\biggl(\raisebox{-8pt}{\includegraphics[height = .3in]{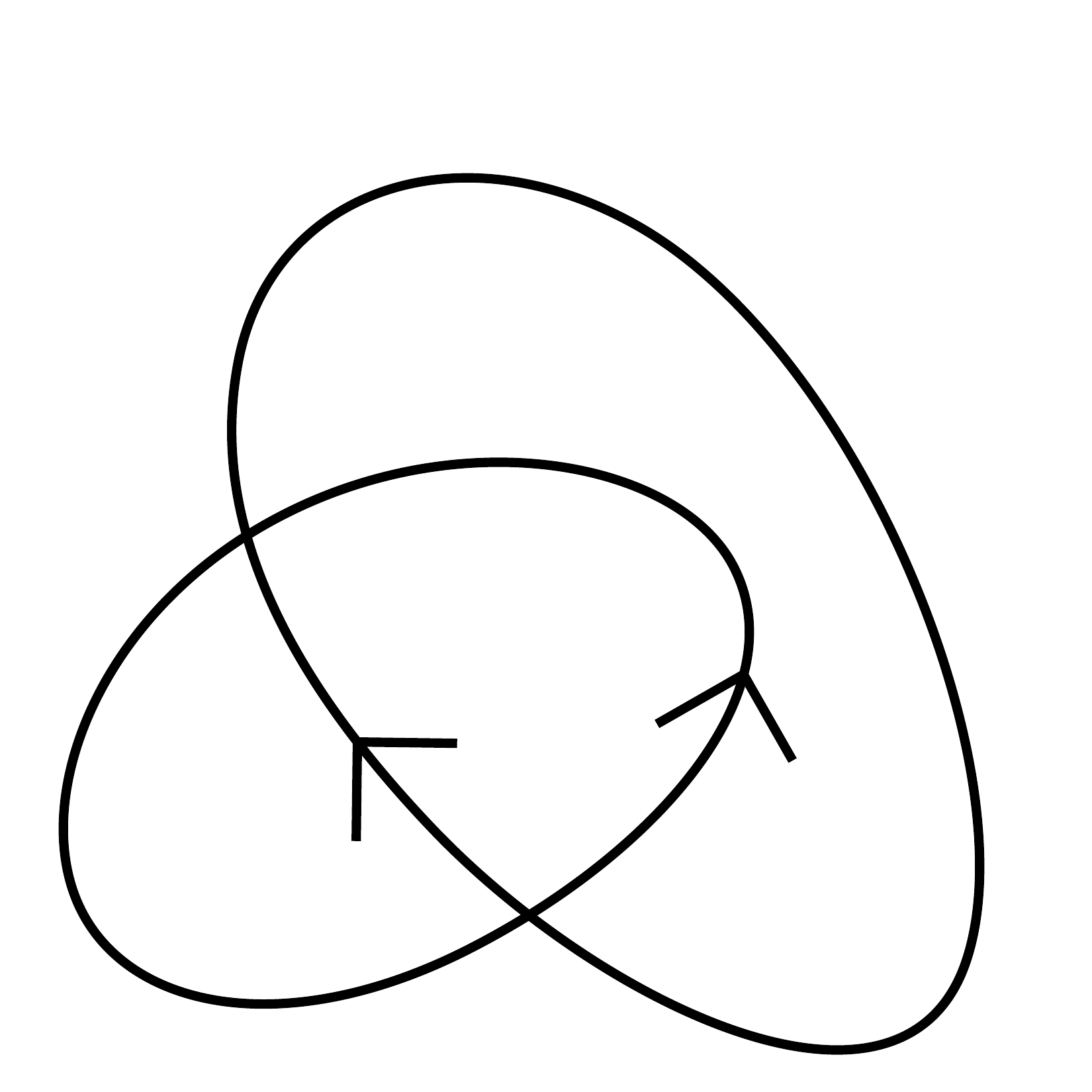}}\biggr)\right]\\
        &=(q^{1-n}+1)P\biggl(\raisebox{-8pt}{\includegraphics[height = .3in]{Example-2.3/Presentation-4.pdf}}\biggr)-q^{-n}[2]P\biggl(\raisebox{-8pt}{\includegraphics[height = .3in]{Example-2.3/Presentation-4.pdf}}\biggr)\\
        &=(q^{1-n}-q^{-n}[2]+1)P\biggl(\raisebox{-8pt}{\includegraphics[height = .3in]{Example-2.3/Presentation-4.pdf}}\biggr)\\
        &=(q^{1-n}-q^{-n}[2]+1)[n-1]P\biggl(\raisebox{-8pt}{\includegraphics[height = .3in]{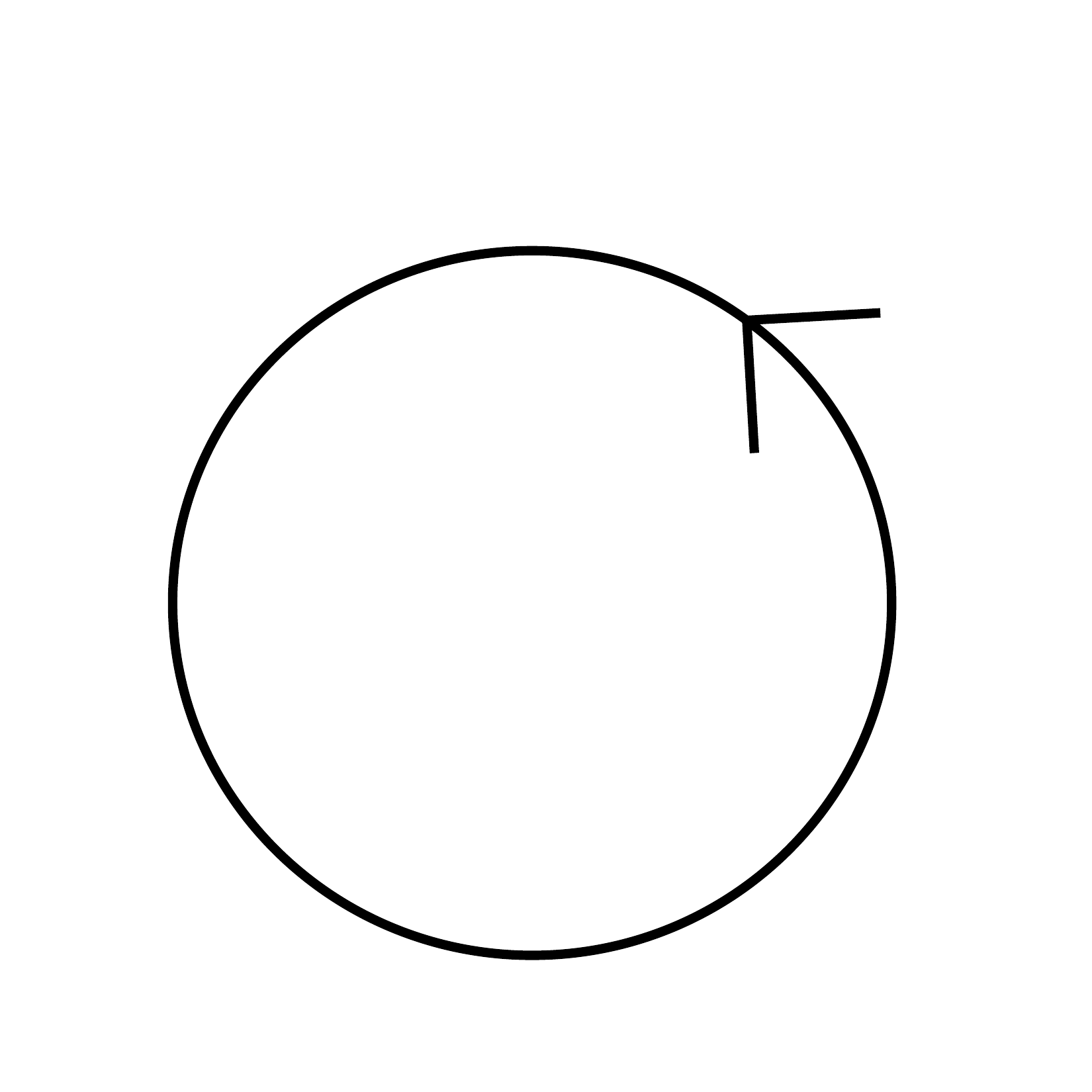}}\biggr)\\
        &=(q^{1-n}-q^{-n}[2]+1)[n-1][n].
        \end{align*}
\end{example}

\section{Reidemeister-type Moves for Knotted, Balanced-Oriented, 4-Valent Graphs}

\label{sec:generatingset}
In this section, we list all oriented versions of the Reidemeister-type moves for diagrams of balanced-oriented, knotted 4-valent graphs with rigid vertices. We also provide a minimal generating set for all of such moves, which will shorten the proof of invariance of the polynomial $P$.

Figures~\ref{fig:Omega1 Moves}, \ref{fig:Omega2 Moves}, and \ref{fig:Omega3 Moves} show all the oriented versions of the three classical Reidemeister moves $R1, R2$ and $R3$, but we employ Polyak's notation from~\cite{Polyak_2010}, using $\Omega1, \Omega2$ and respectively $\Omega3$ to denote these moves.
\begin{figure}[ht]
    \[\raisebox{-13pt}{\includegraphics[height=0.45in]{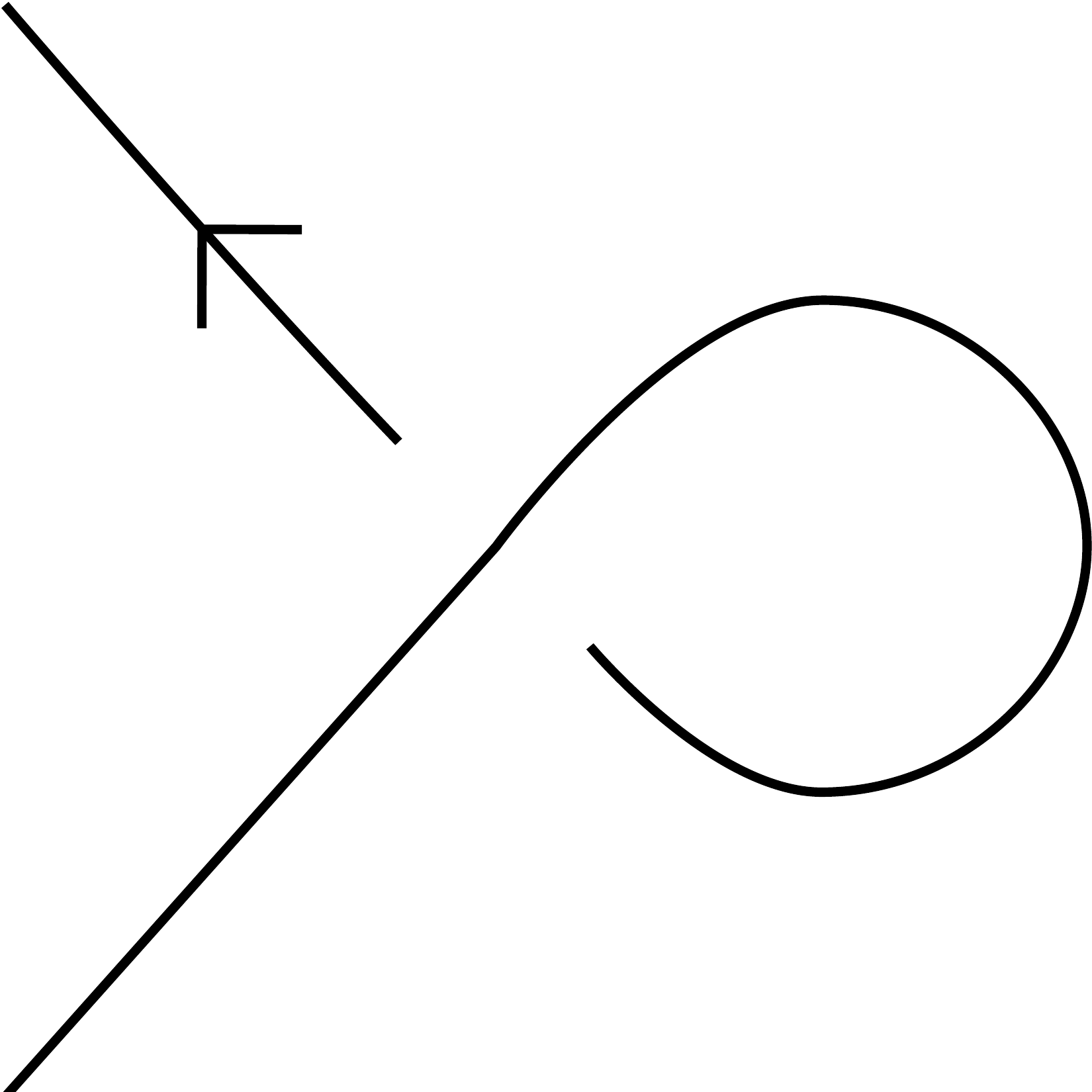}}\stackrel{\Omega 1a}{\longleftrightarrow}\raisebox{-13pt}{\includegraphics[height =0.45in]{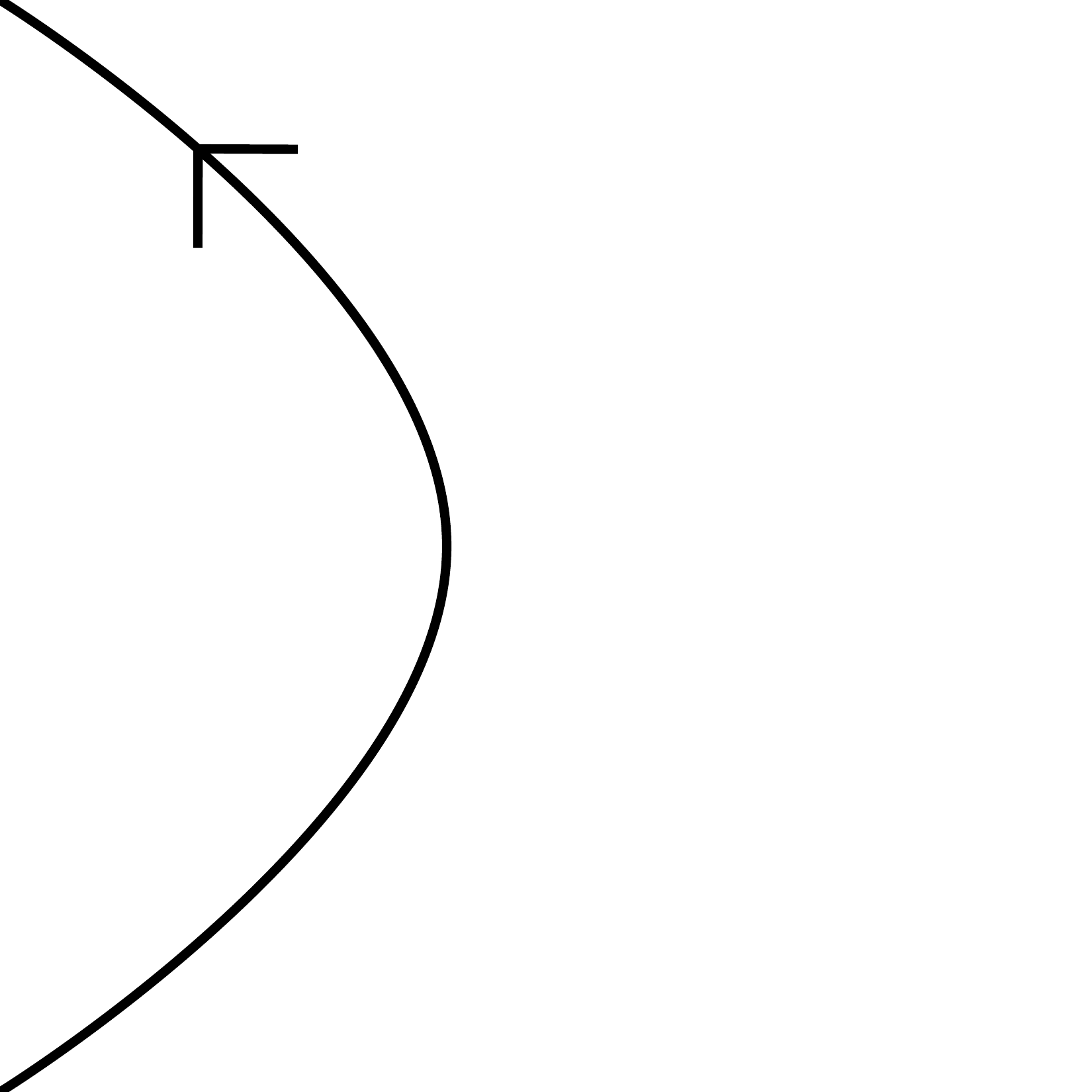}}\hspace{1.5cm}\raisebox{-13pt}{\includegraphics[height=0.45in]{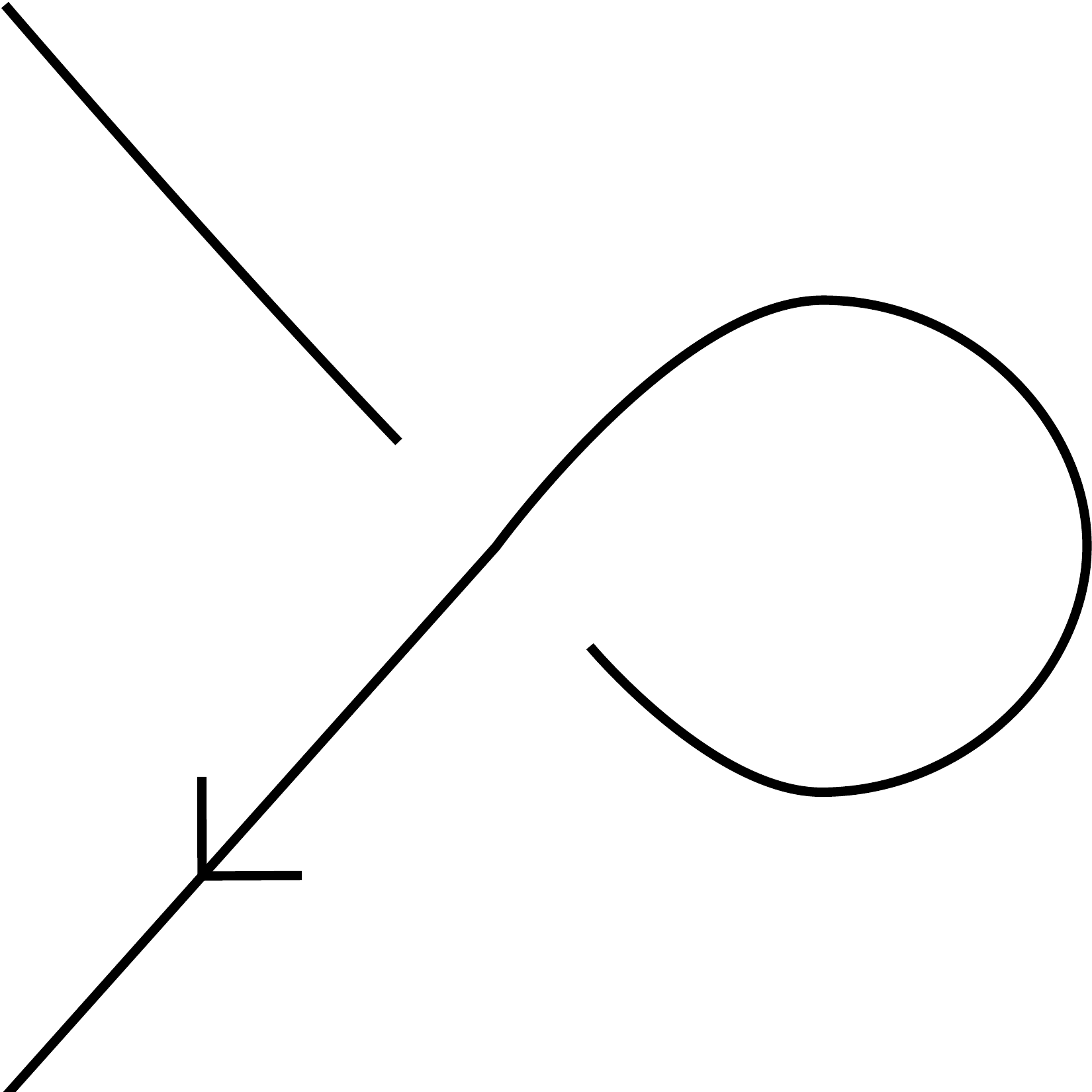}}\stackrel{\Omega 1b}{\longleftrightarrow}\raisebox{-13pt}{\includegraphics[height =0.45in]{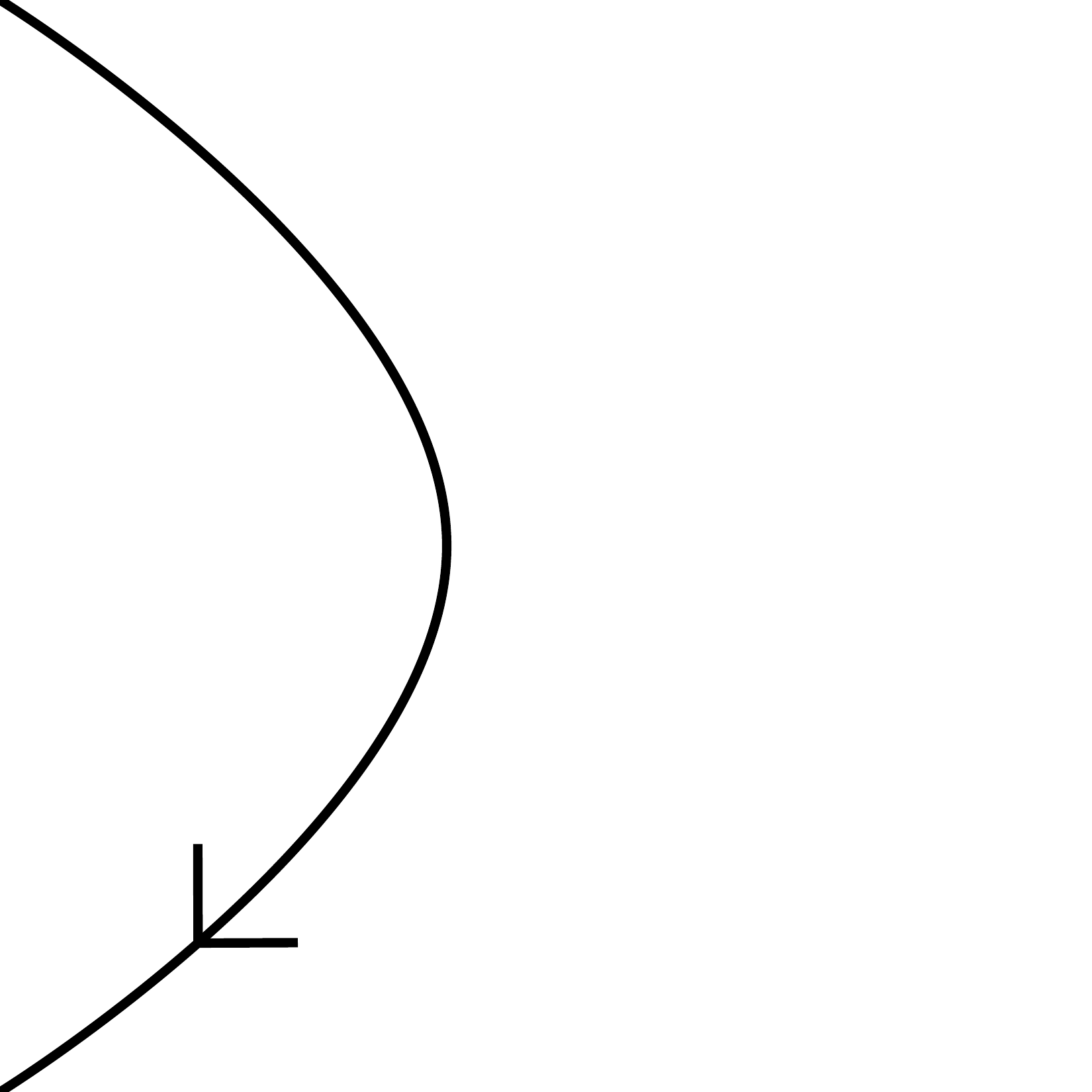}}\]
    \[\raisebox{-13pt}{\includegraphics[height=0.45in]{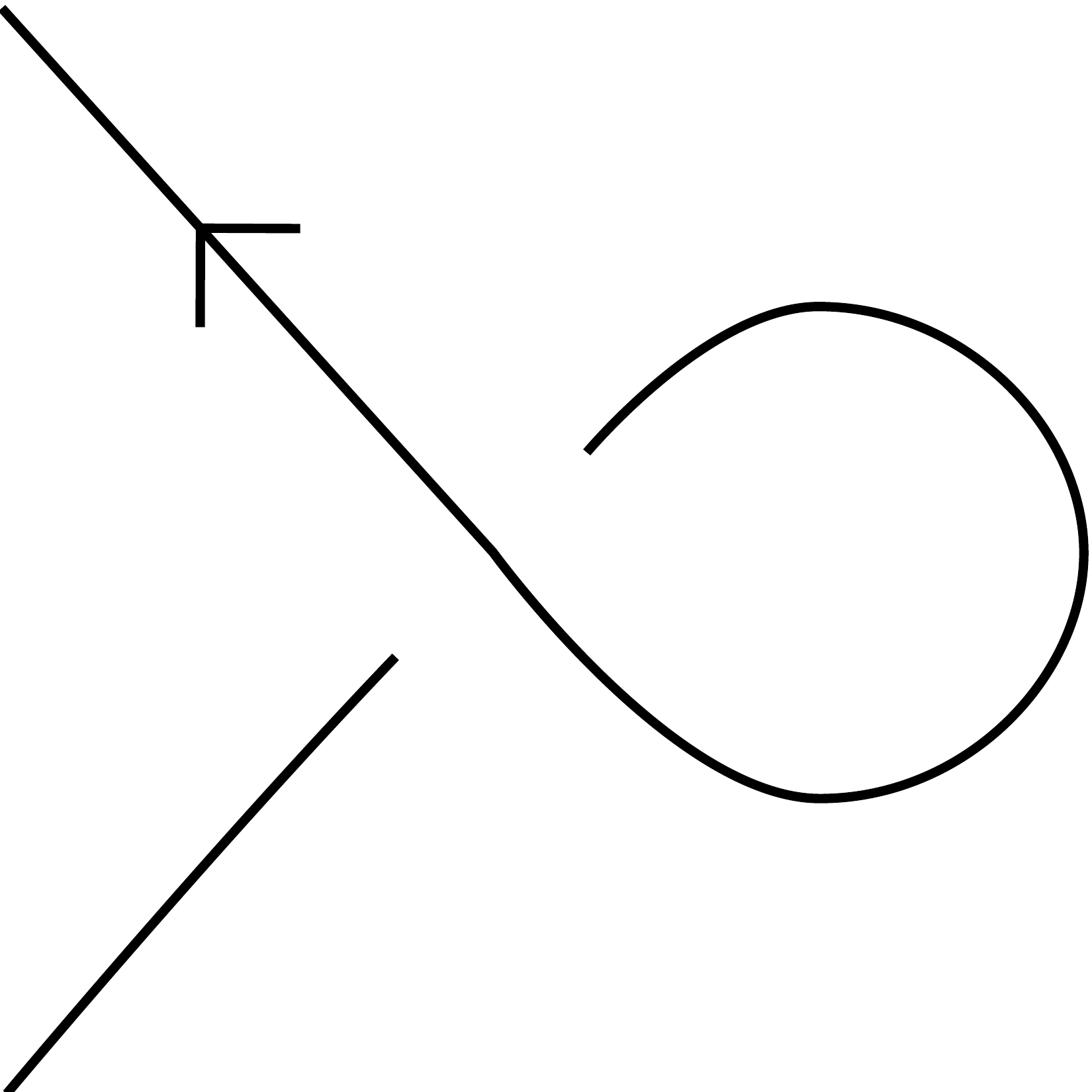}}\stackrel{\Omega 1c}{\longleftrightarrow}\raisebox{-13pt}{\includegraphics[height =0.45in]{Generating-Sets/O1ac-2.pdf}}\hspace{1.5cm}\raisebox{-13pt}{\includegraphics[height=0.45in]{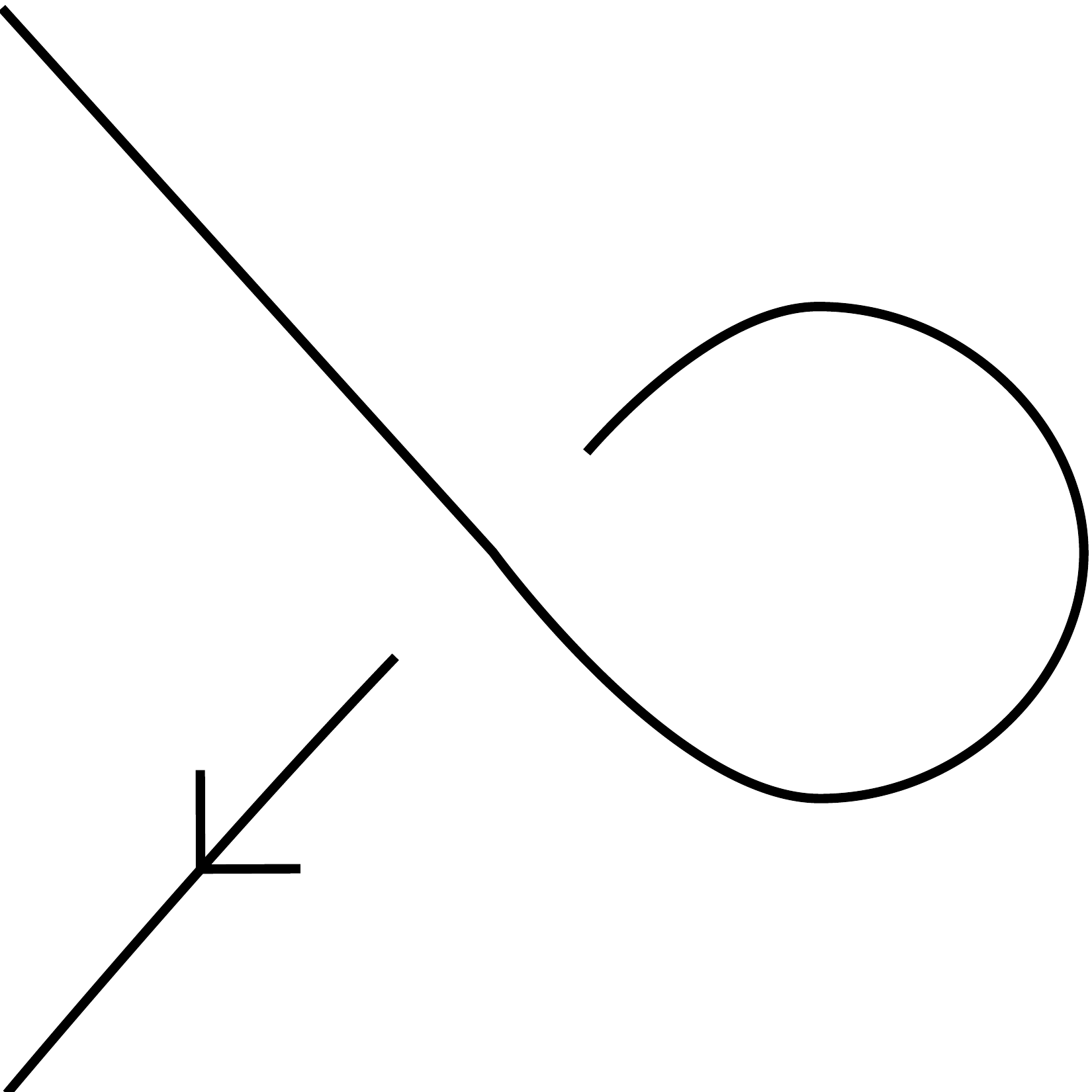}}\stackrel{\Omega 1d}{\longleftrightarrow}\raisebox{-13pt}{\includegraphics[height =0.45in]{Generating-Sets/O1bd-2.pdf}}\]
    \caption{Oriented $\Omega1$ moves}
    \label{fig:Omega1 Moves}
\end{figure}

\begin{figure}[ht]
    \[\raisebox{-13pt}{\includegraphics[height=0.45in]{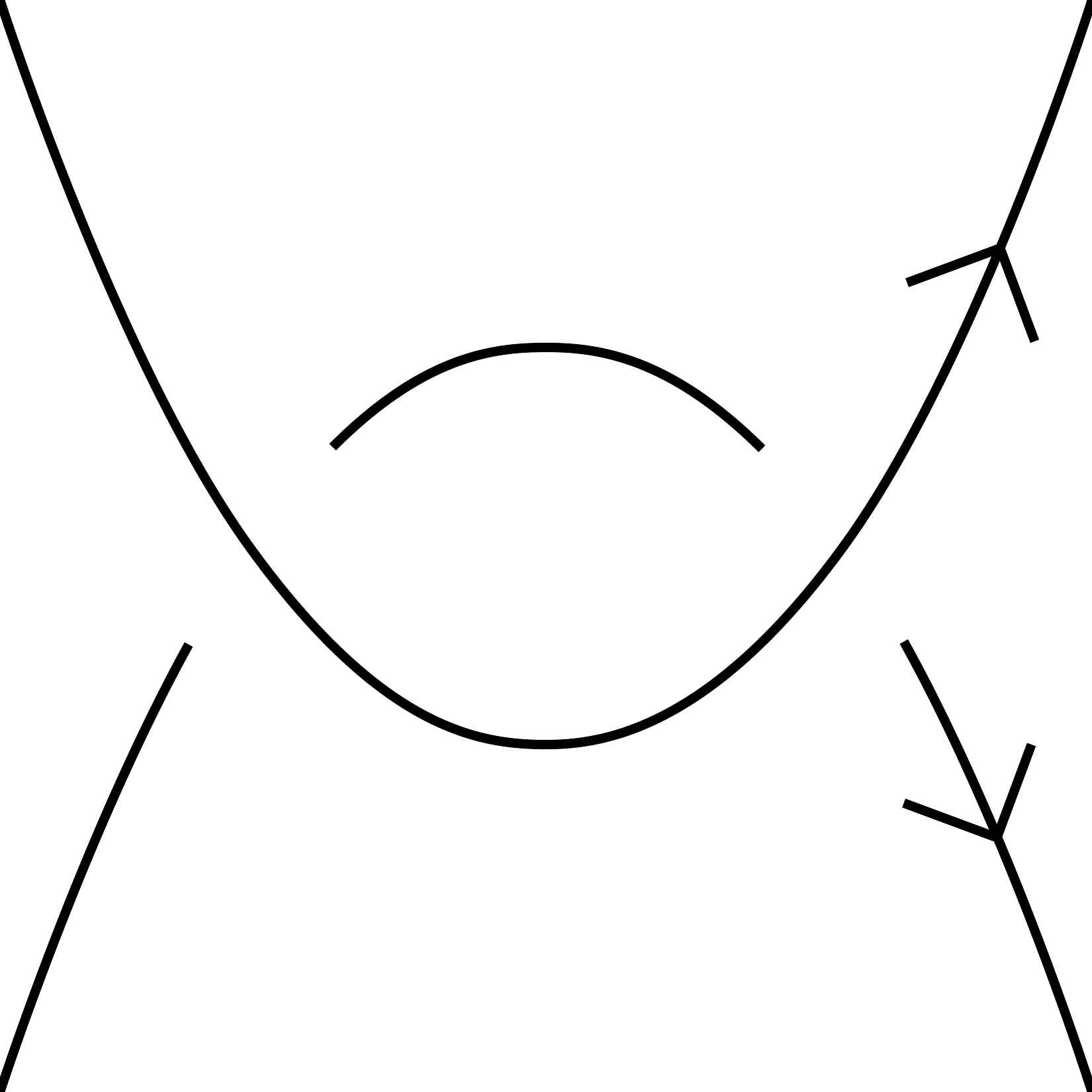}}\stackrel{\Omega 2a}{\longleftrightarrow}\raisebox{-13pt}{\includegraphics[height =0.45in]{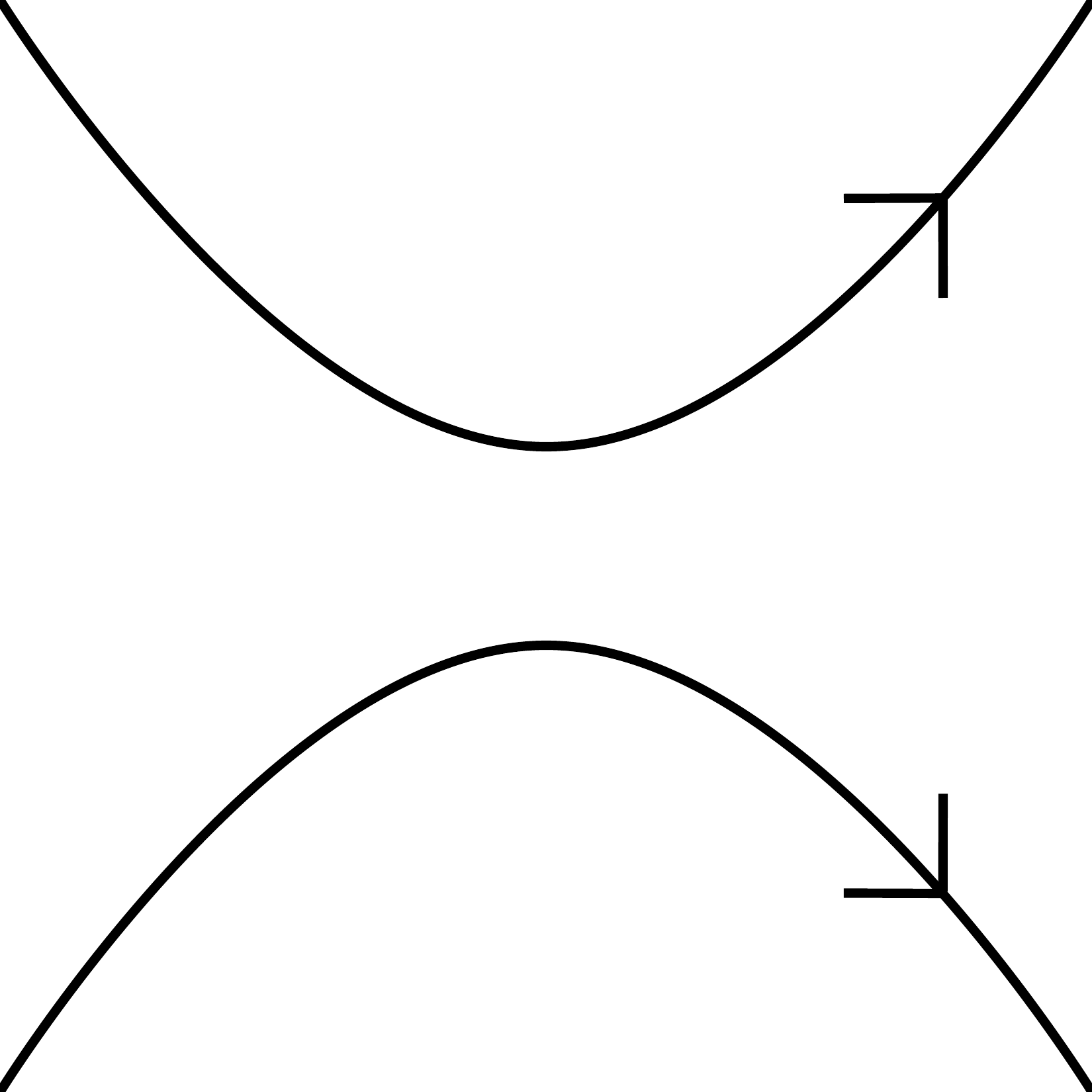}}\hspace{1.5cm}\raisebox{-13pt}{\includegraphics[height=0.45in]{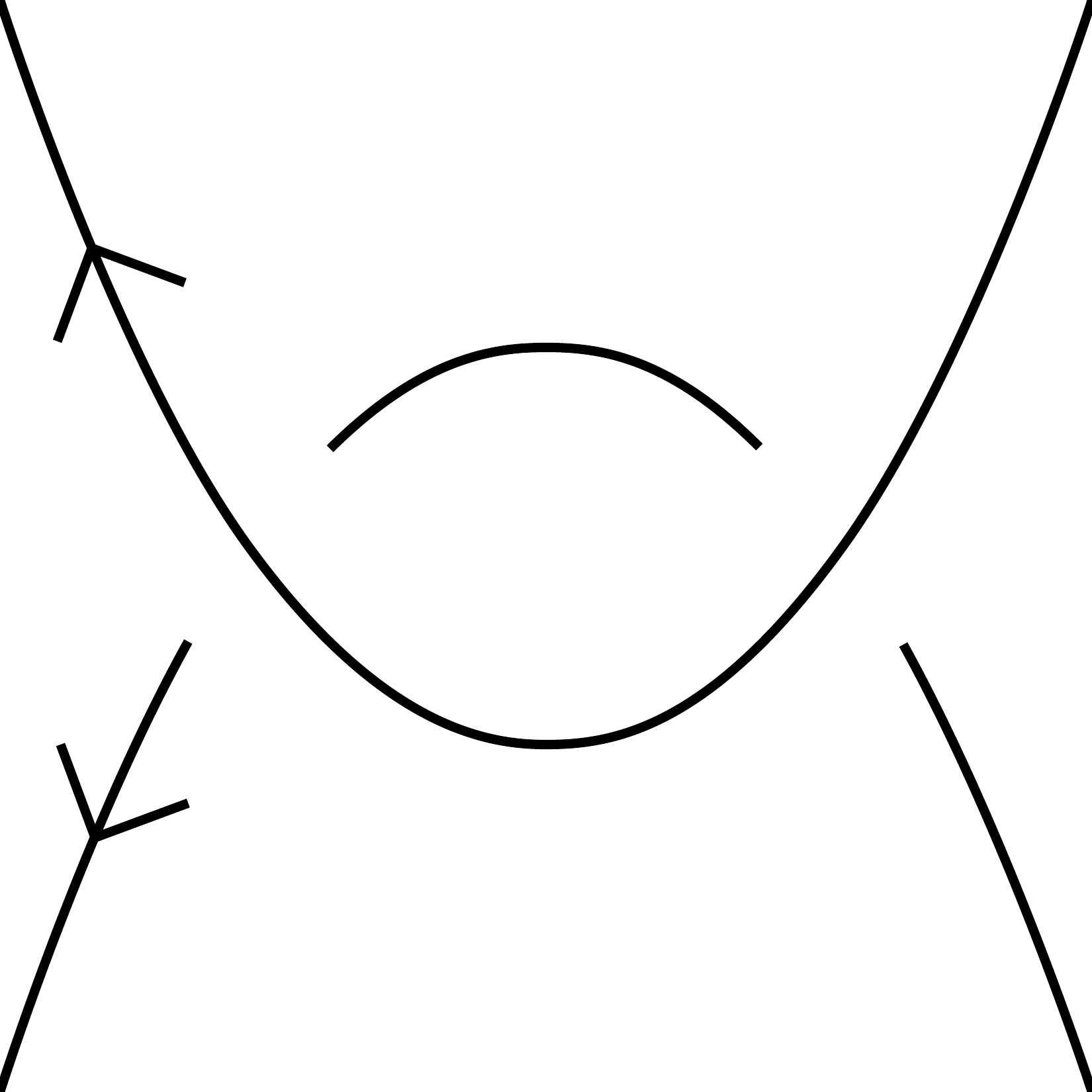}}\stackrel{\Omega 2b}{\longleftrightarrow}\raisebox{-13pt}{\includegraphics[height =0.45in]{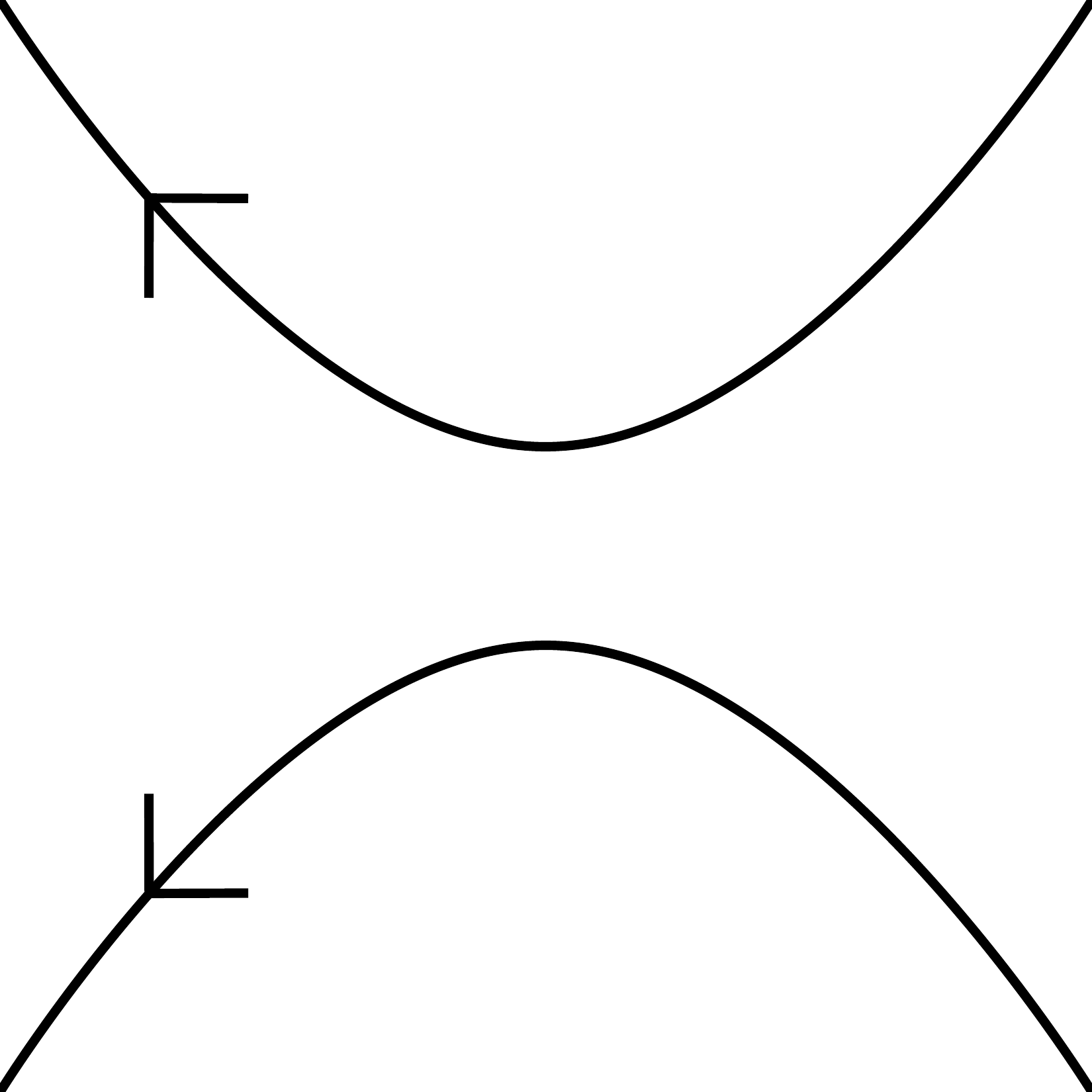}}\]
    \vspace{0.10cm}
    \[\raisebox{-13pt}{\includegraphics[height=0.45in]{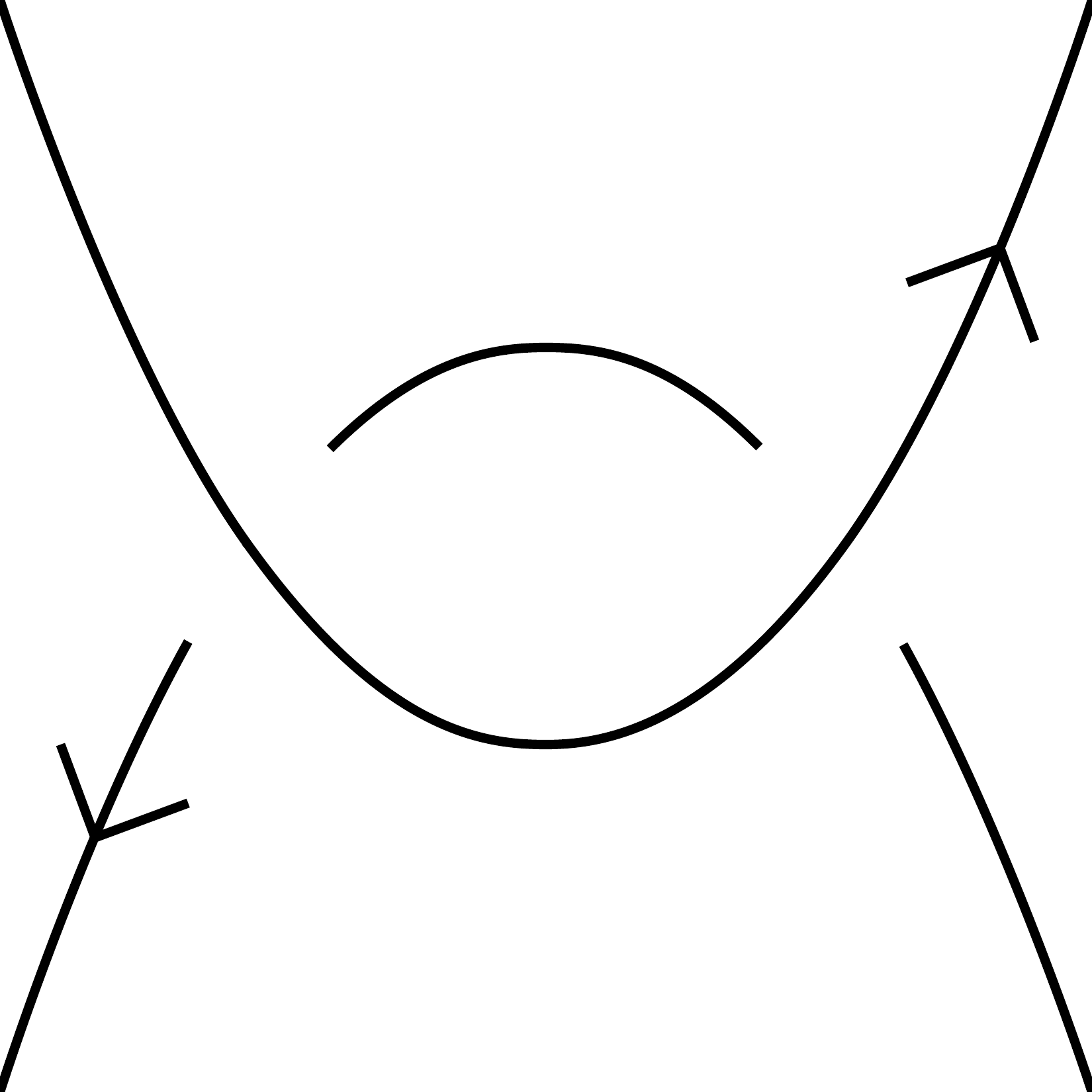}}\stackrel{\Omega 2c}{\longleftrightarrow}\raisebox{-13pt}{\includegraphics[height =0.45in]{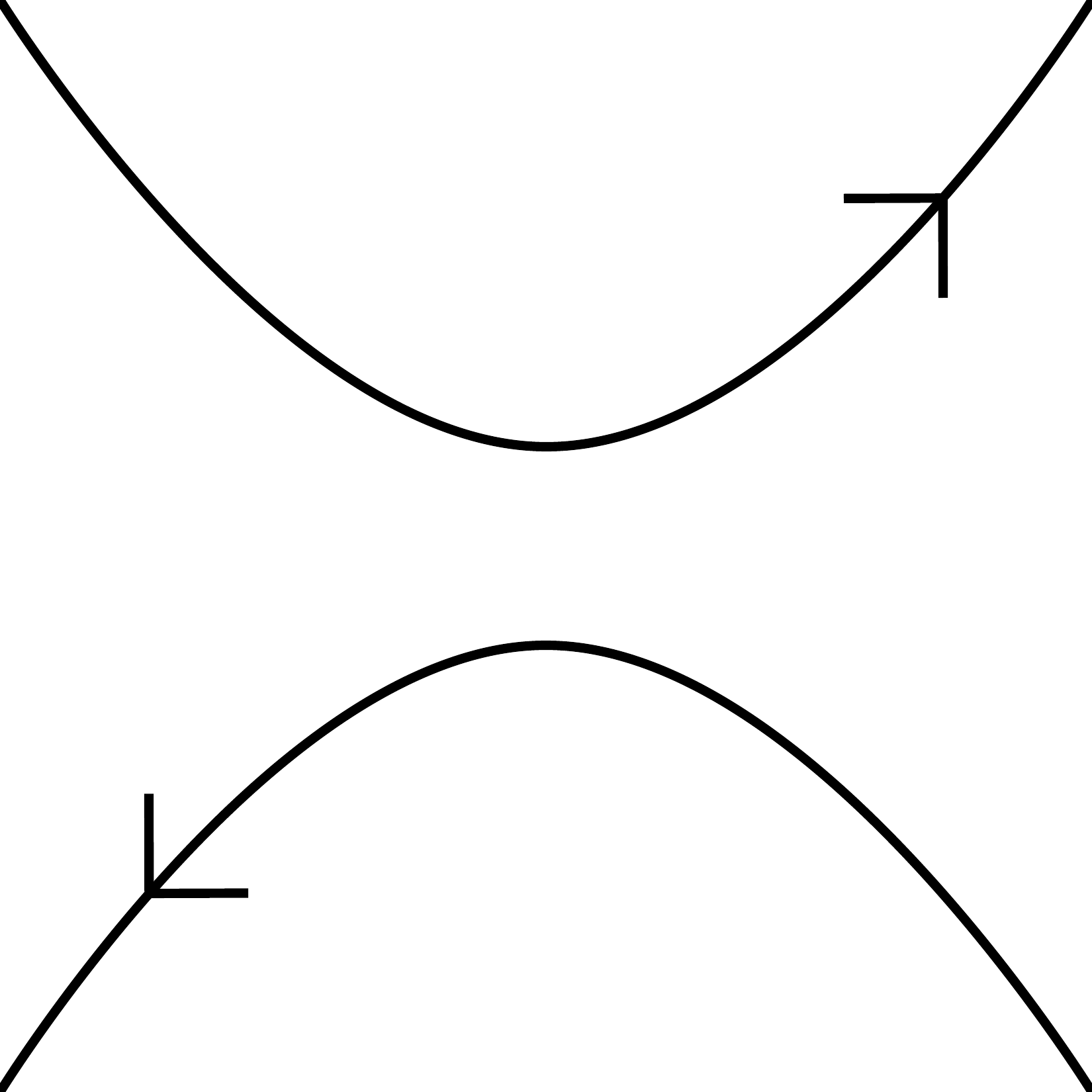}}\hspace{1.5cm}\raisebox{-13pt}{\includegraphics[height=0.45in]{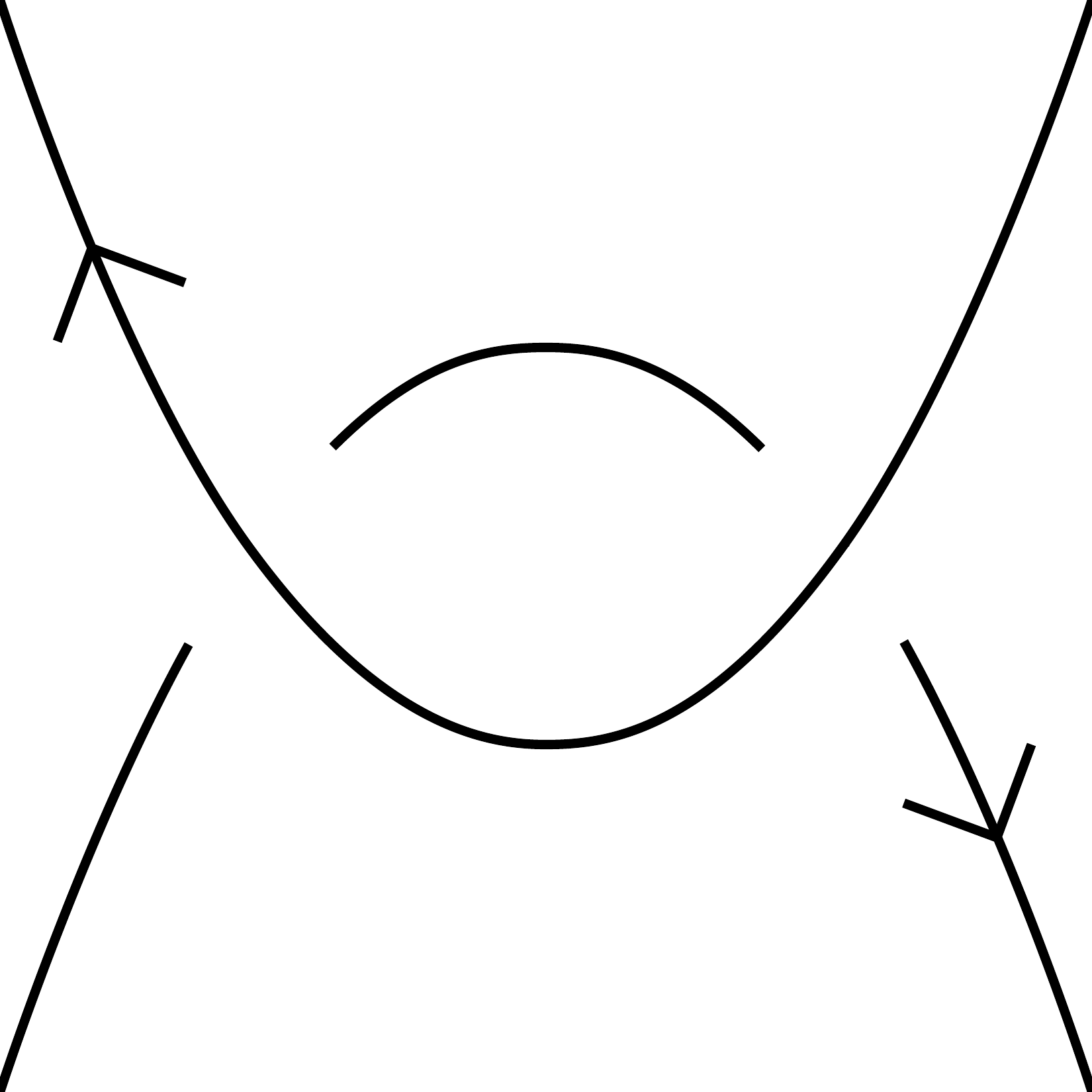}}\stackrel{\Omega 2d}{\longleftrightarrow}\raisebox{-13pt}{\includegraphics[height =0.45in]{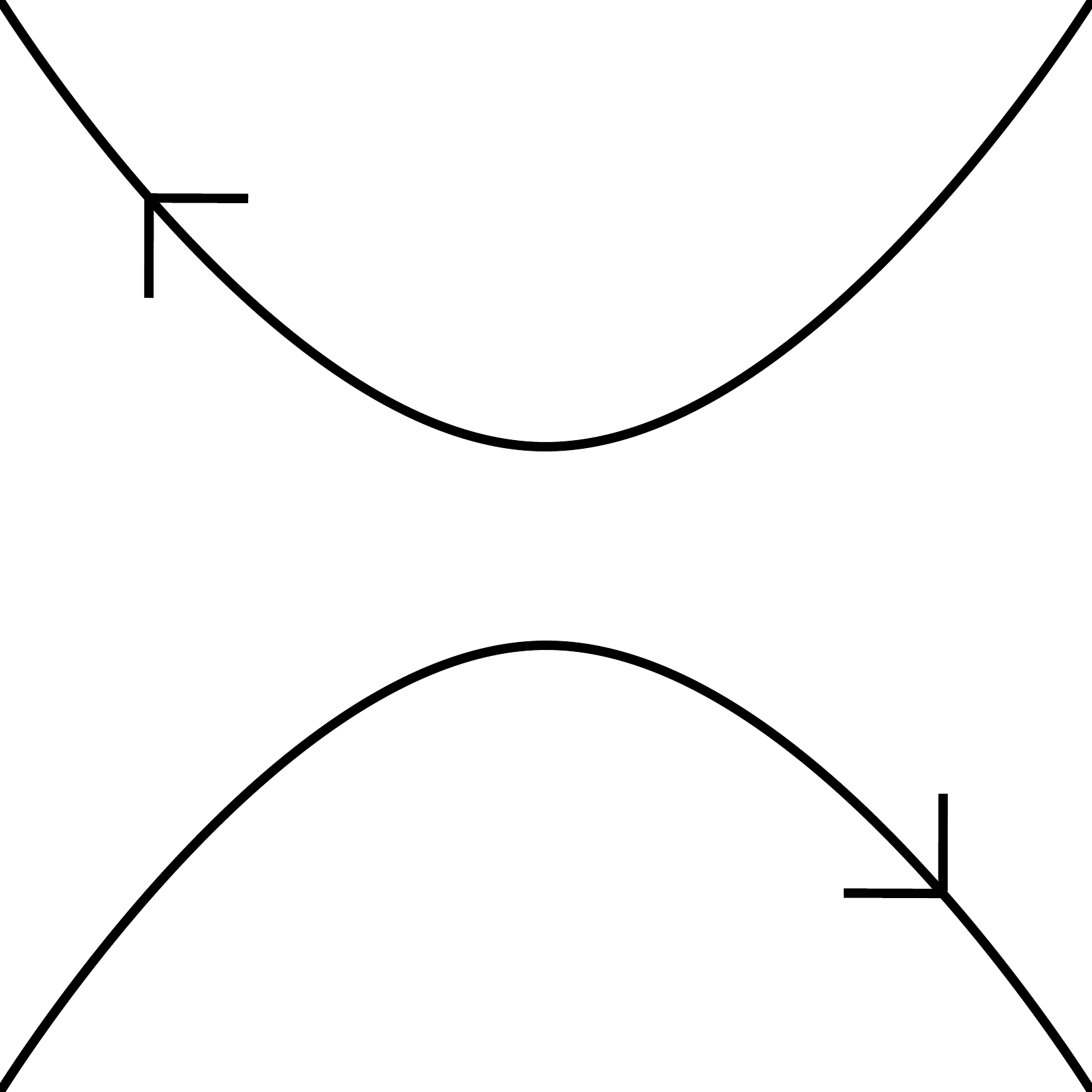}}\]
    \caption{Oriented $\Omega2$ moves}
    \label{fig:Omega2 Moves}
\end{figure}

\begin{figure}[ht]
    \[\raisebox{-13pt}{\includegraphics[height=0.45in]{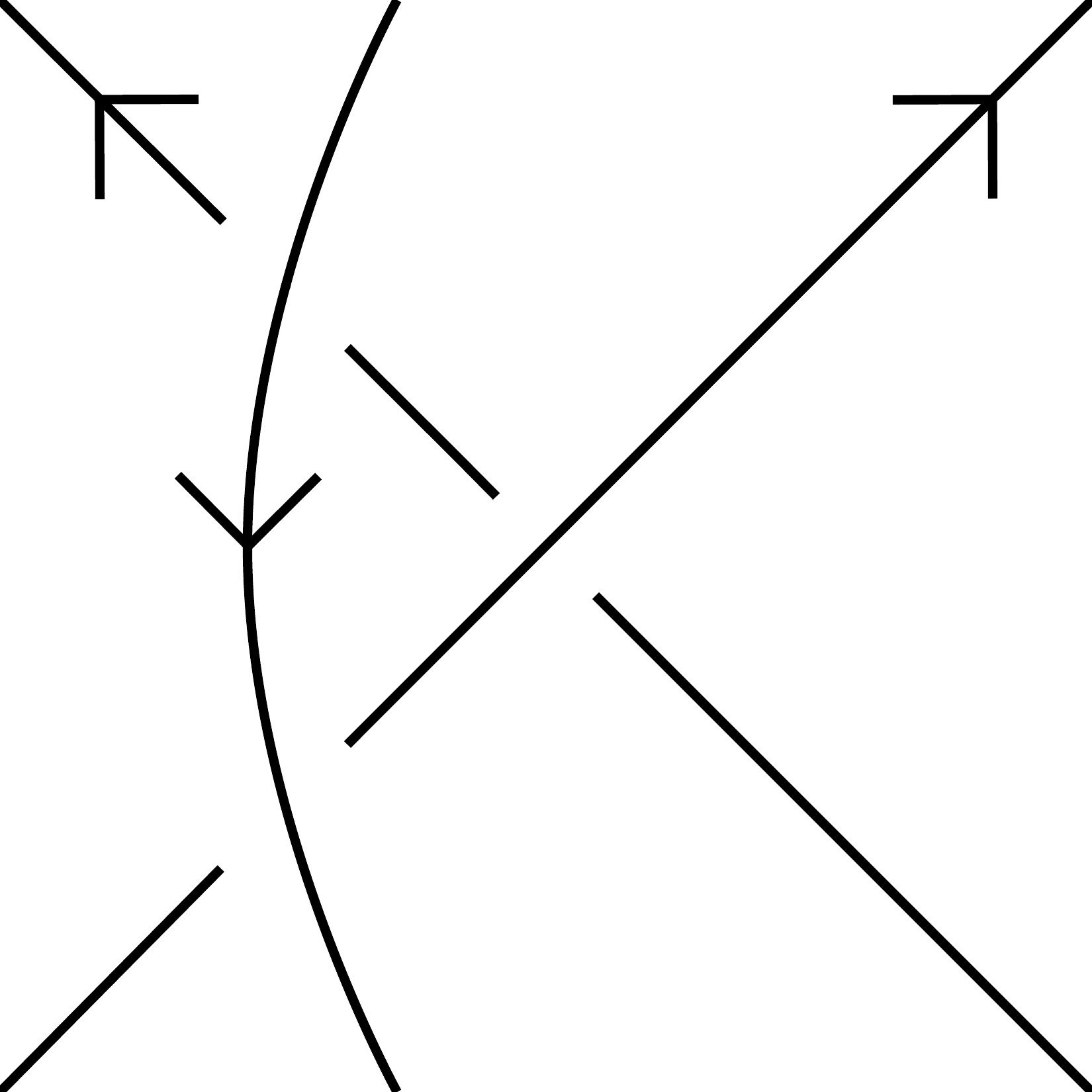}}\stackrel{\Omega 3a}{\longleftrightarrow}\raisebox{-13pt}{\includegraphics[height =0.45in]{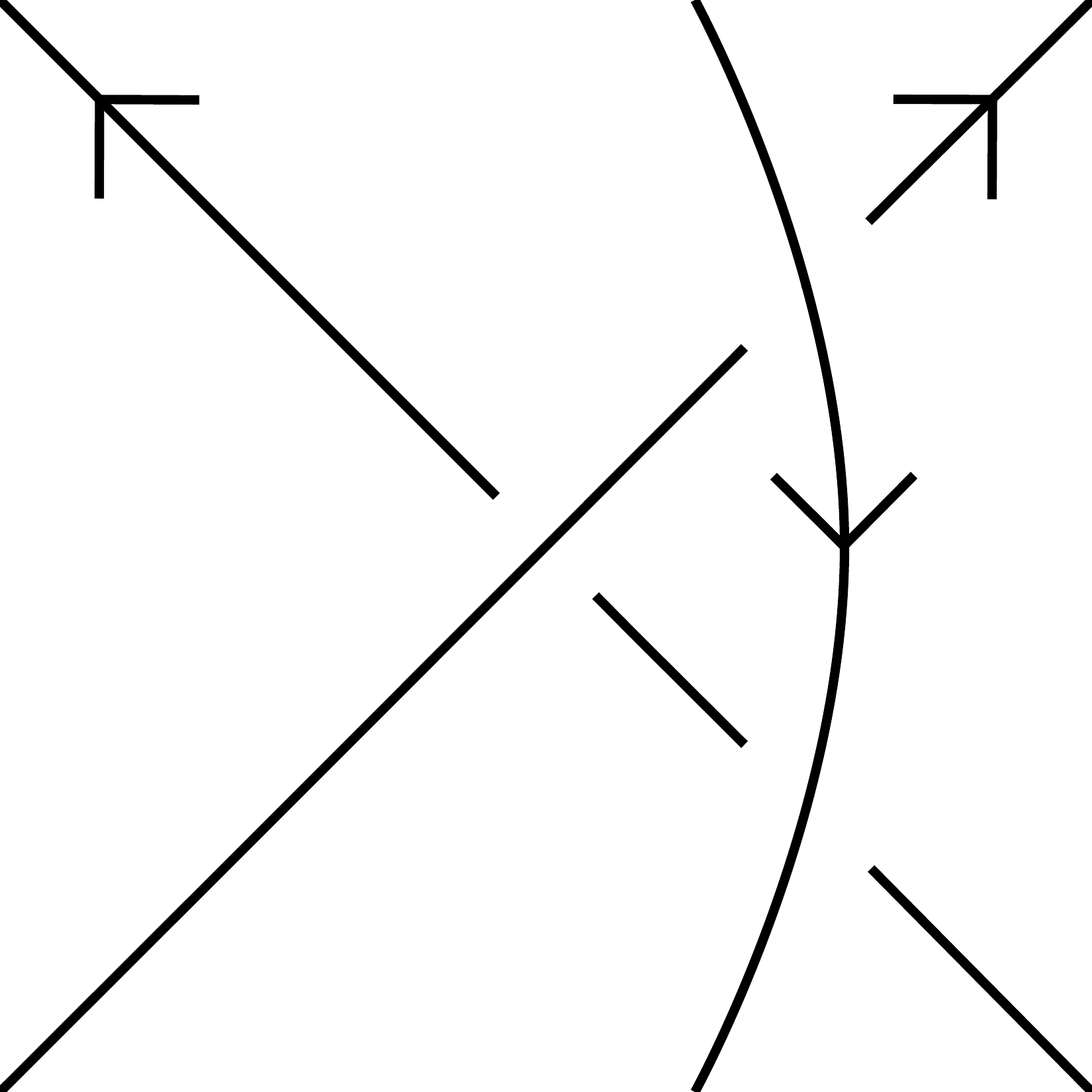}}\hspace{1.5cm}\raisebox{-13pt}{\includegraphics[height=0.45in]{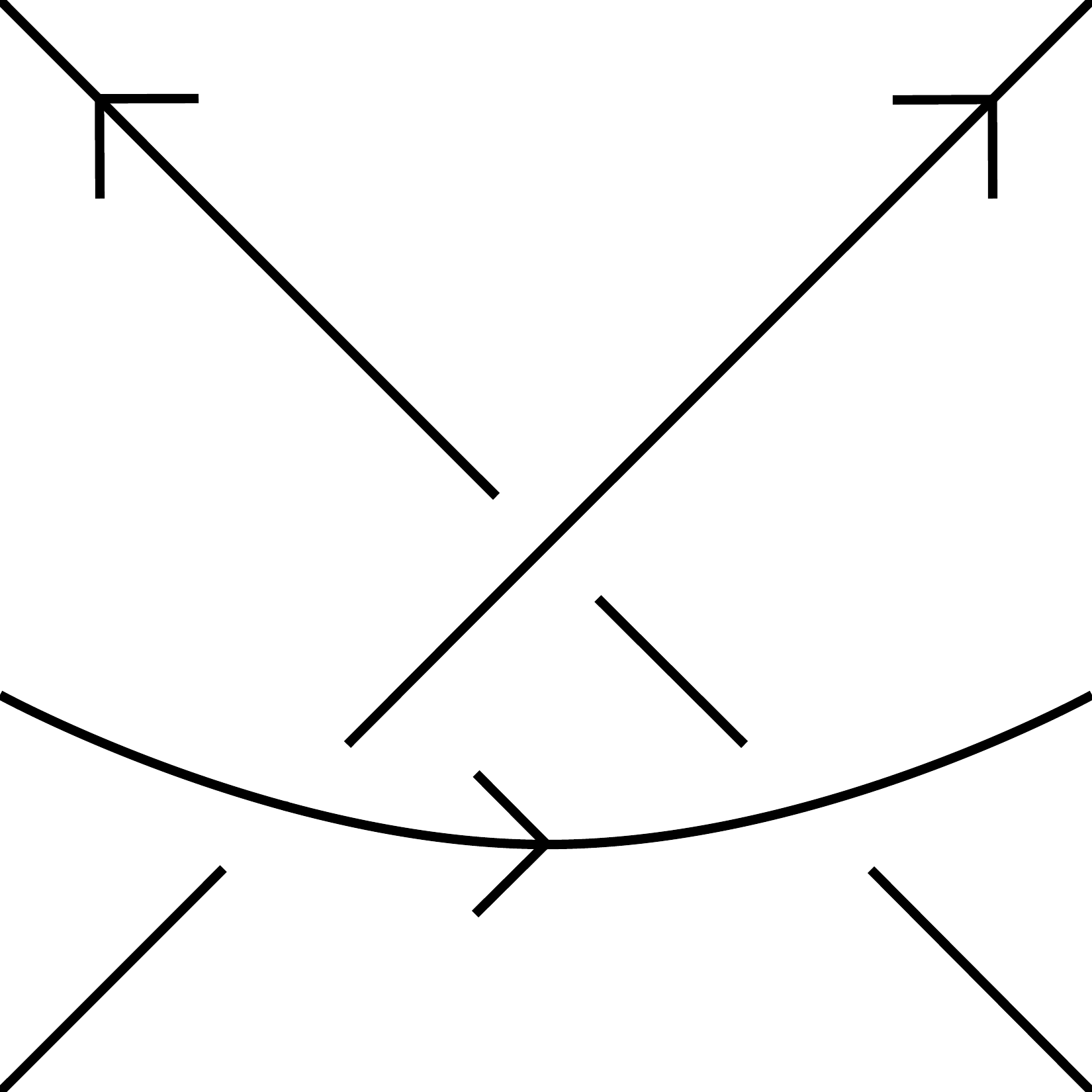}}\stackrel{\Omega 3b}{\longleftrightarrow}\raisebox{-13pt}{\includegraphics[height =0.45in]{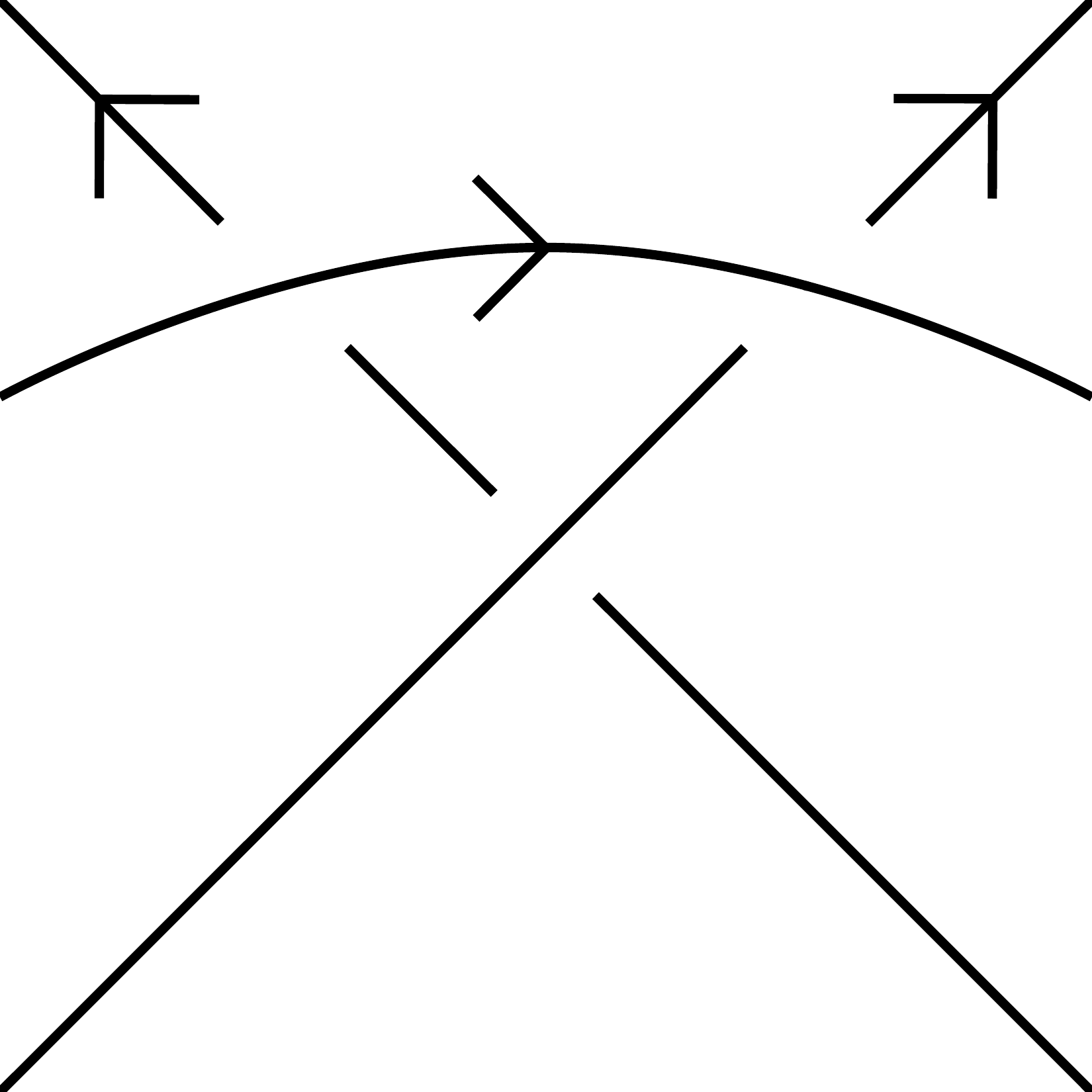}}\]
    \vspace{0.10cm}
    \[\raisebox{-13pt}{\includegraphics[height=0.45in]{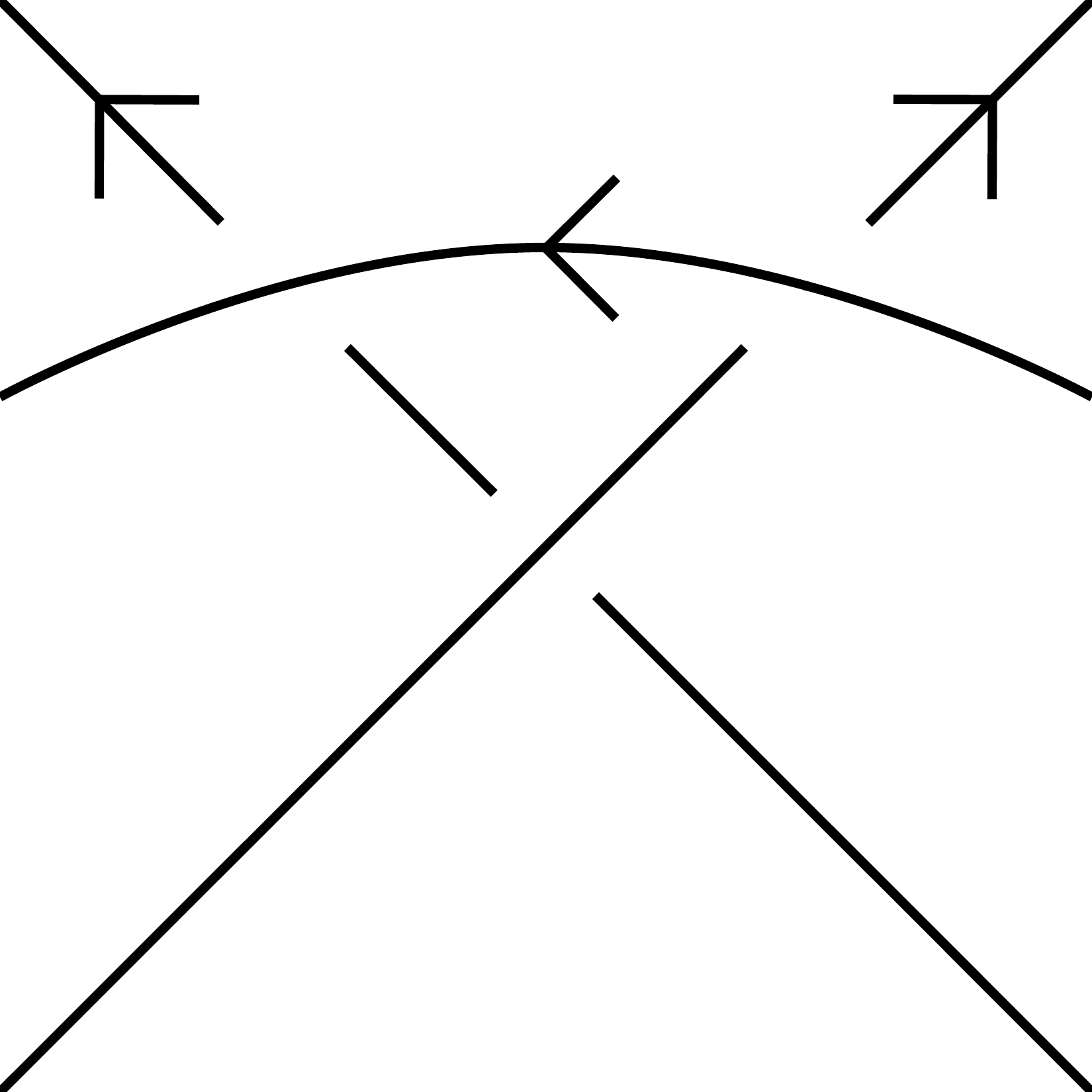}}\stackrel{\Omega 3c}{\longleftrightarrow}\raisebox{-13pt}{\includegraphics[height =0.45in]{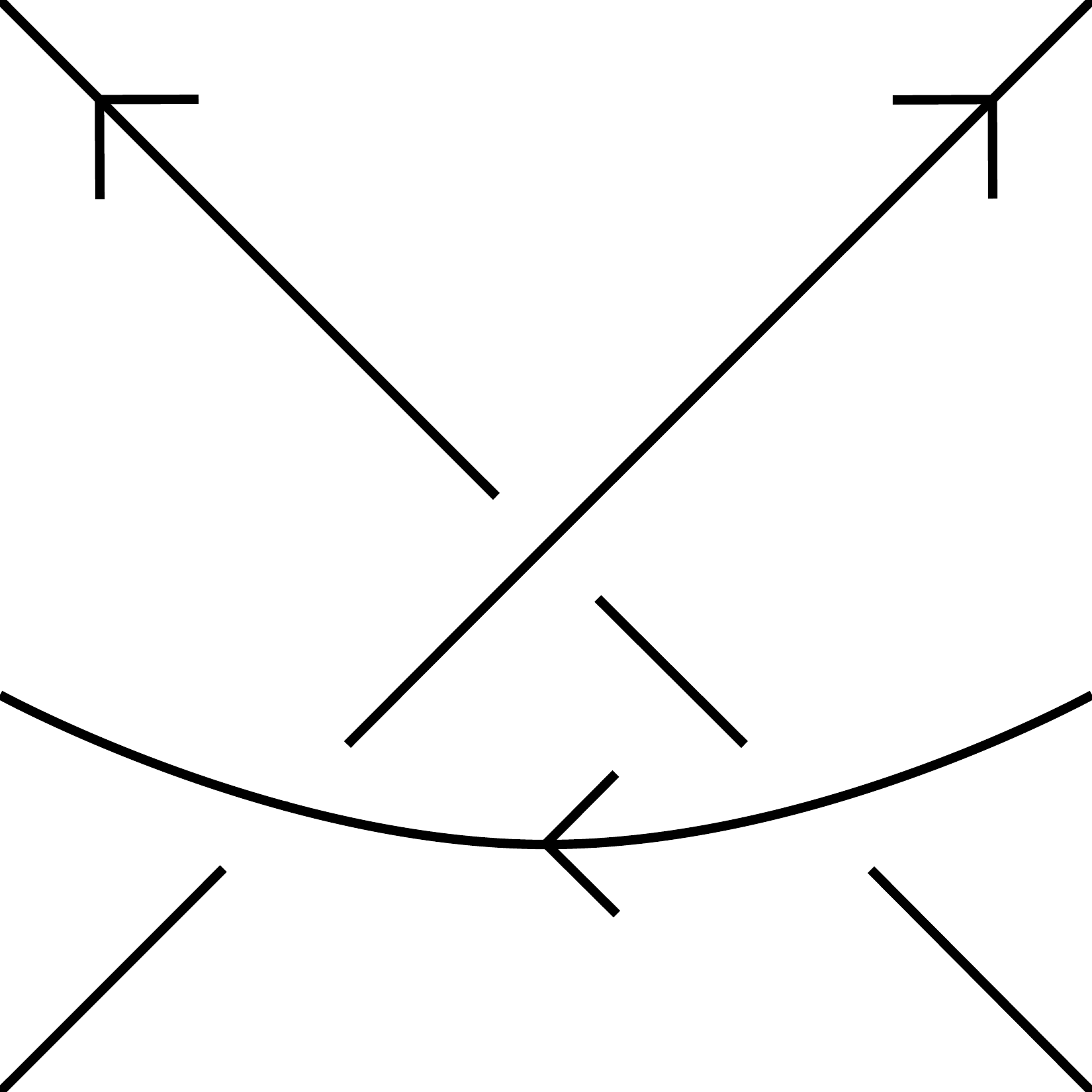}}\hspace{1.5cm}\raisebox{-13pt}{\includegraphics[height=0.45in]{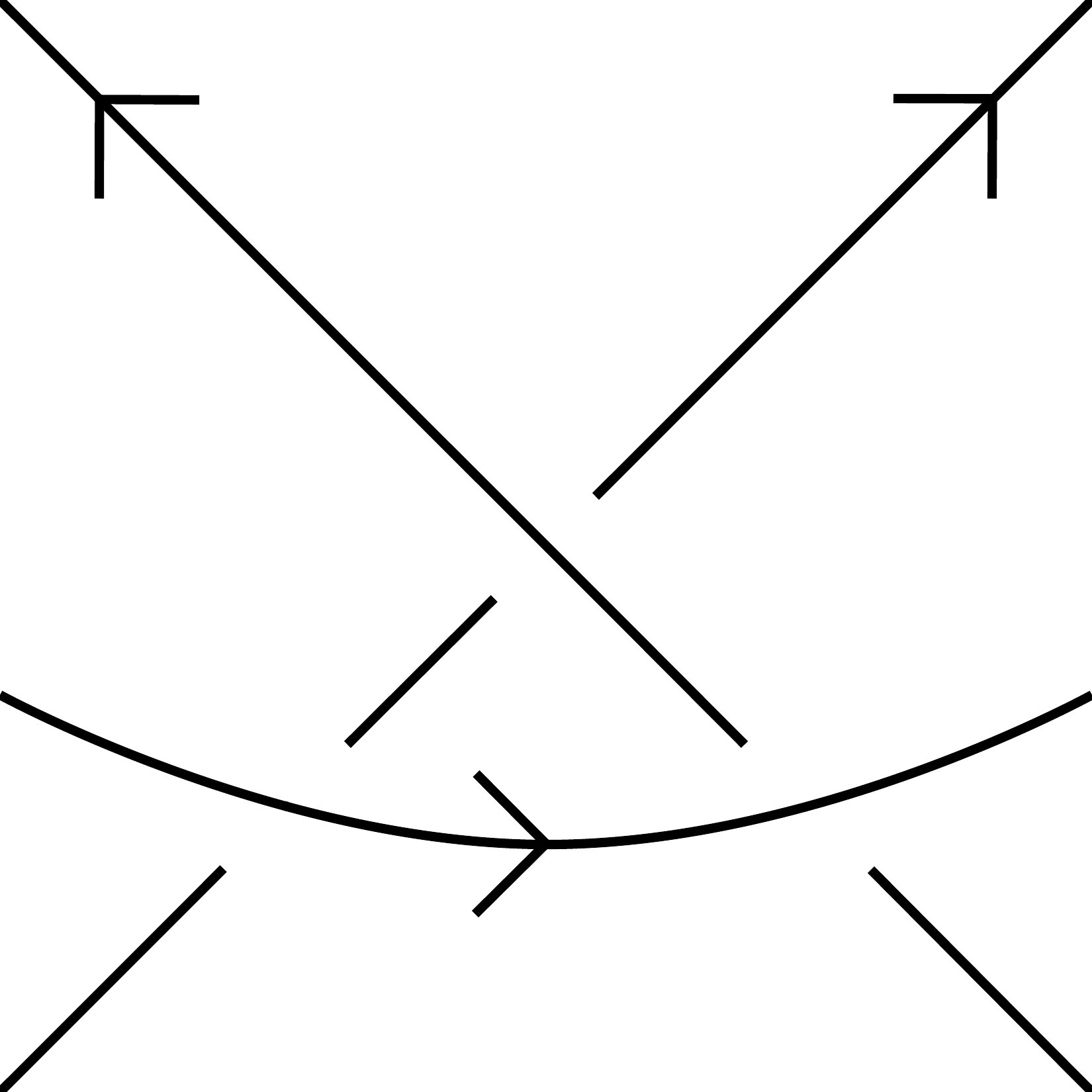}}\stackrel{\Omega 3d}{\longleftrightarrow}\raisebox{-13pt}{\includegraphics[height =0.45in]{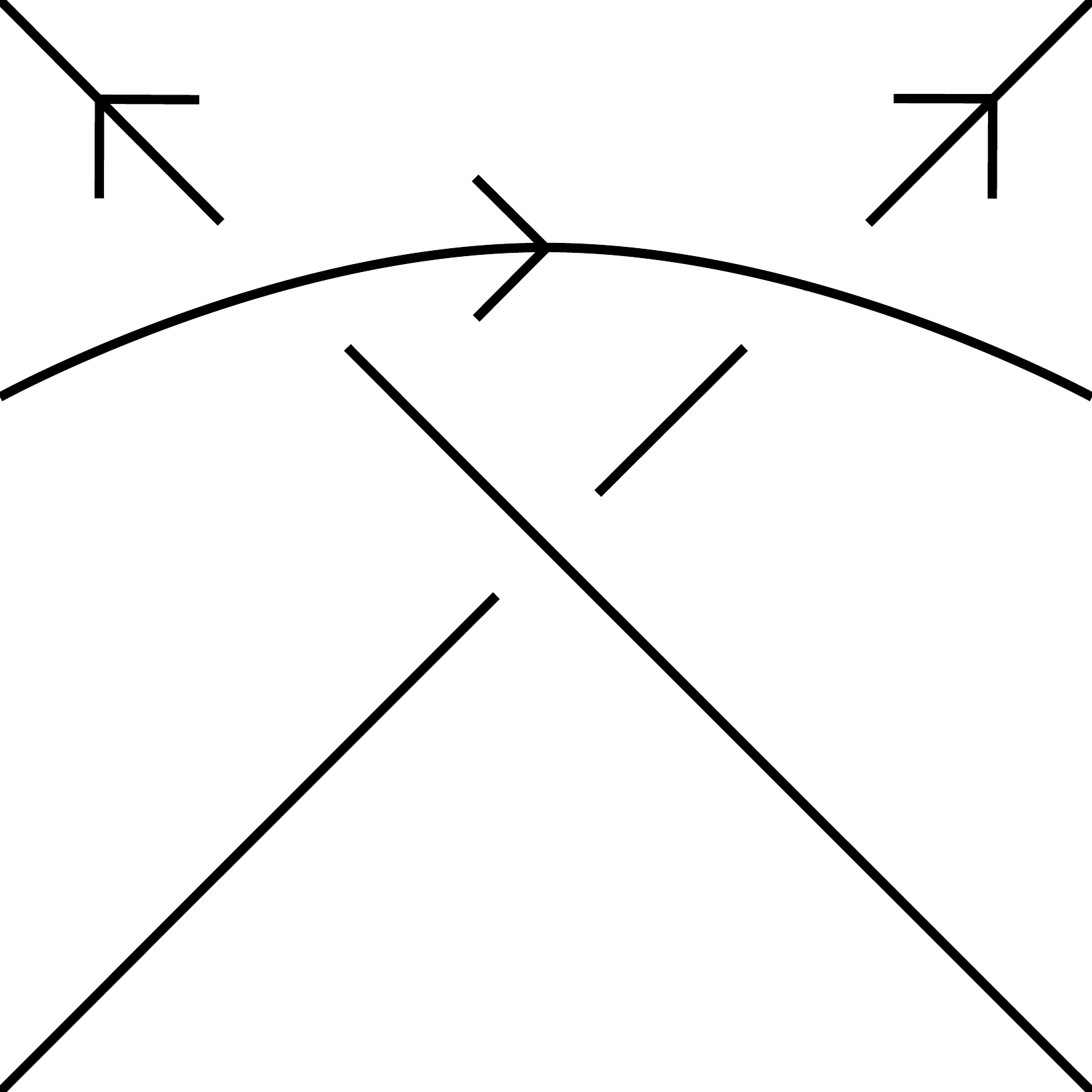}}\]
    \vspace{0.10cm}
    \[\raisebox{-13pt}{\includegraphics[height=0.45in]{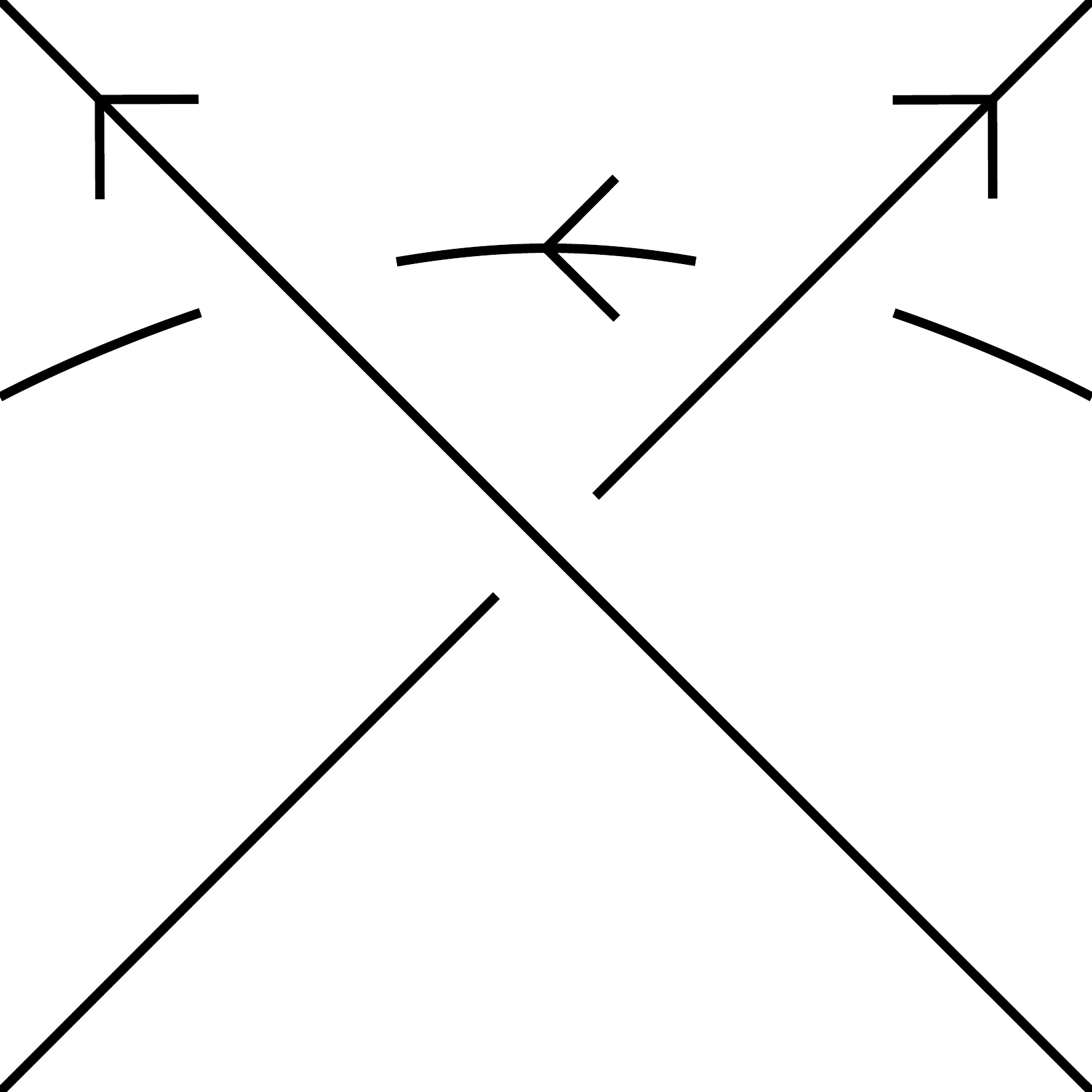}}\stackrel{\Omega 3e}{\longleftrightarrow}\raisebox{-13pt}{\includegraphics[height =0.45in]{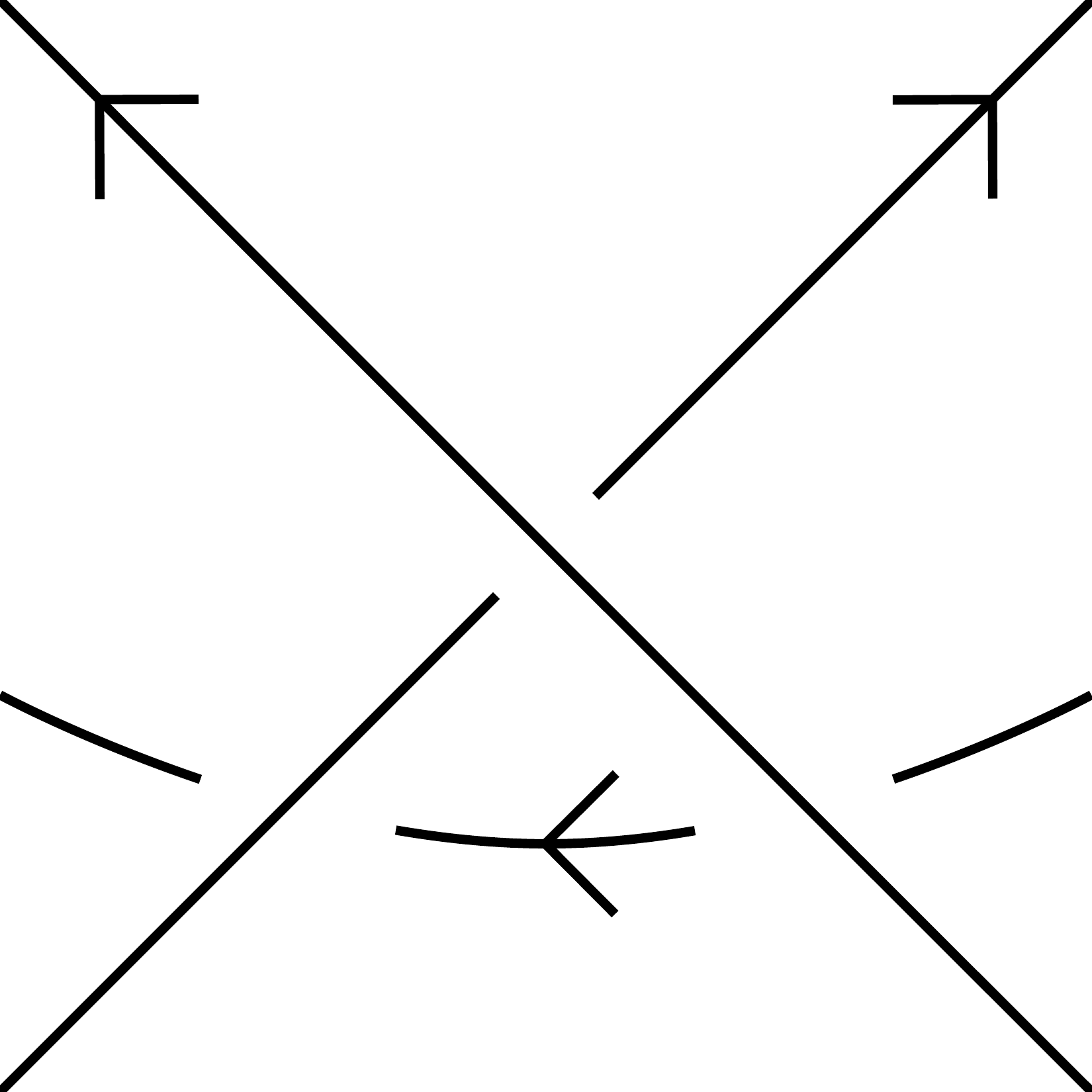}}\hspace{1.5cm}\raisebox{-13pt}{\includegraphics[height=0.45in]{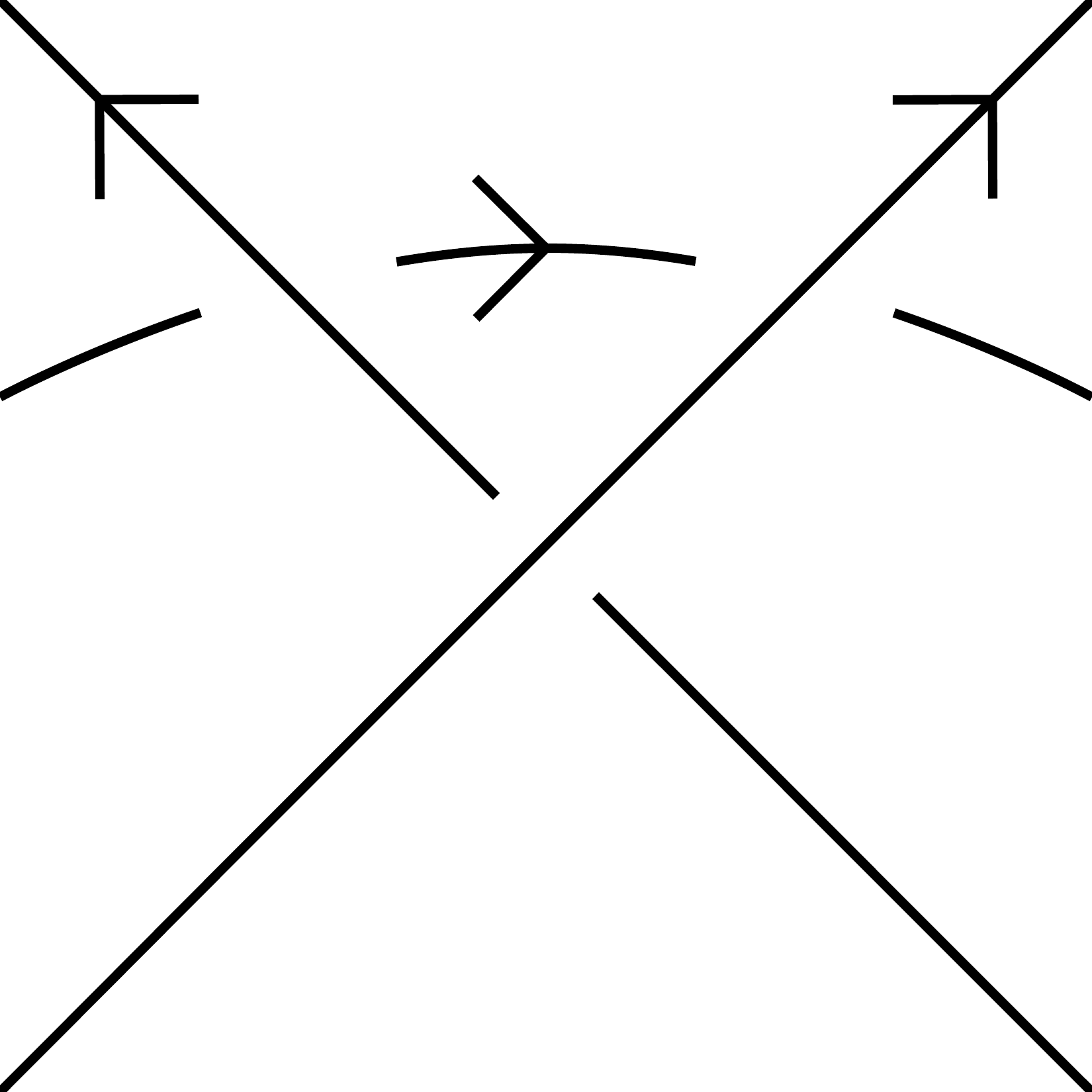}}\stackrel{\Omega 3f}{\longleftrightarrow}\raisebox{-13pt}{\includegraphics[height =0.45in]{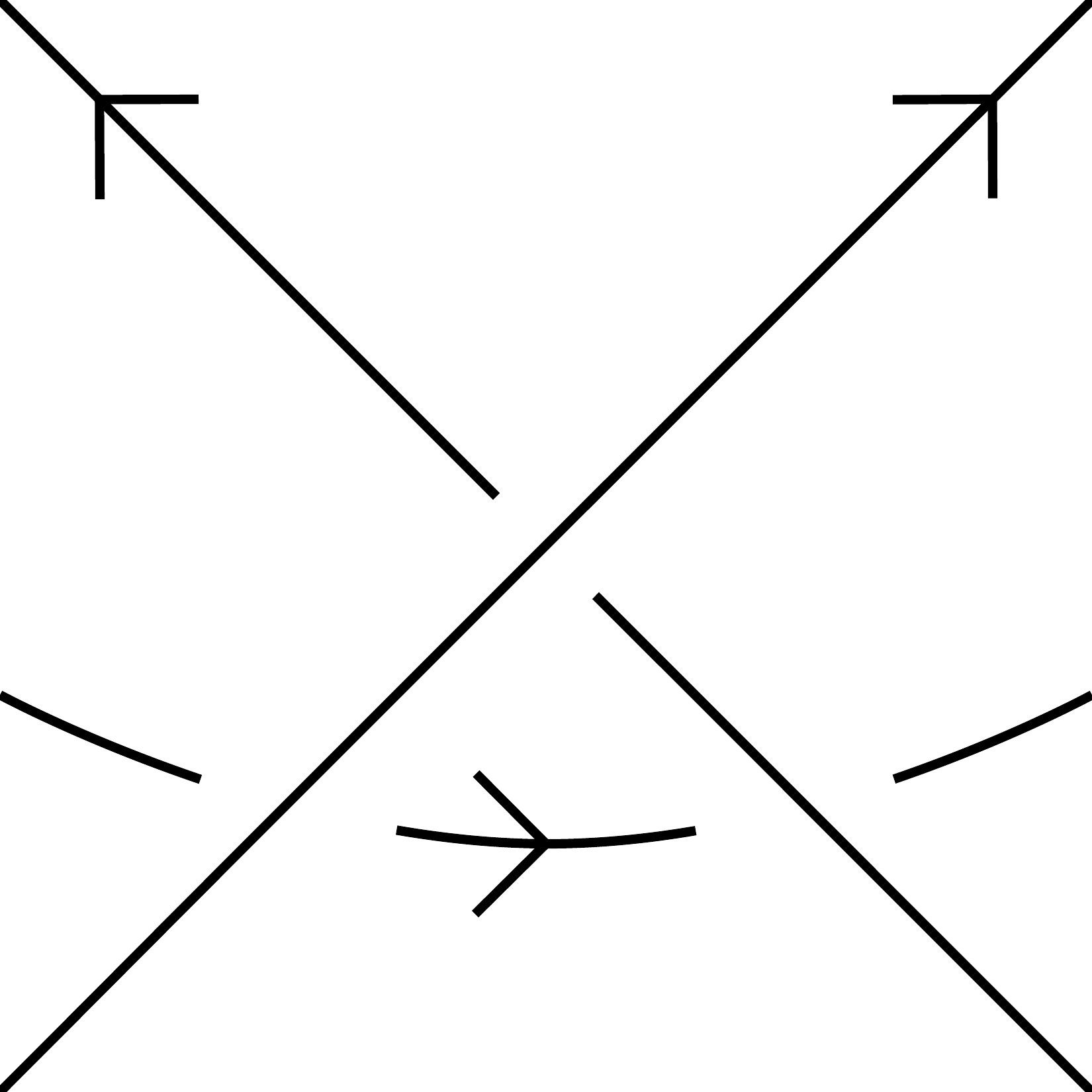}}\]
    \vspace{0.10cm}
    \[\raisebox{-13pt}{\includegraphics[height=0.45in]{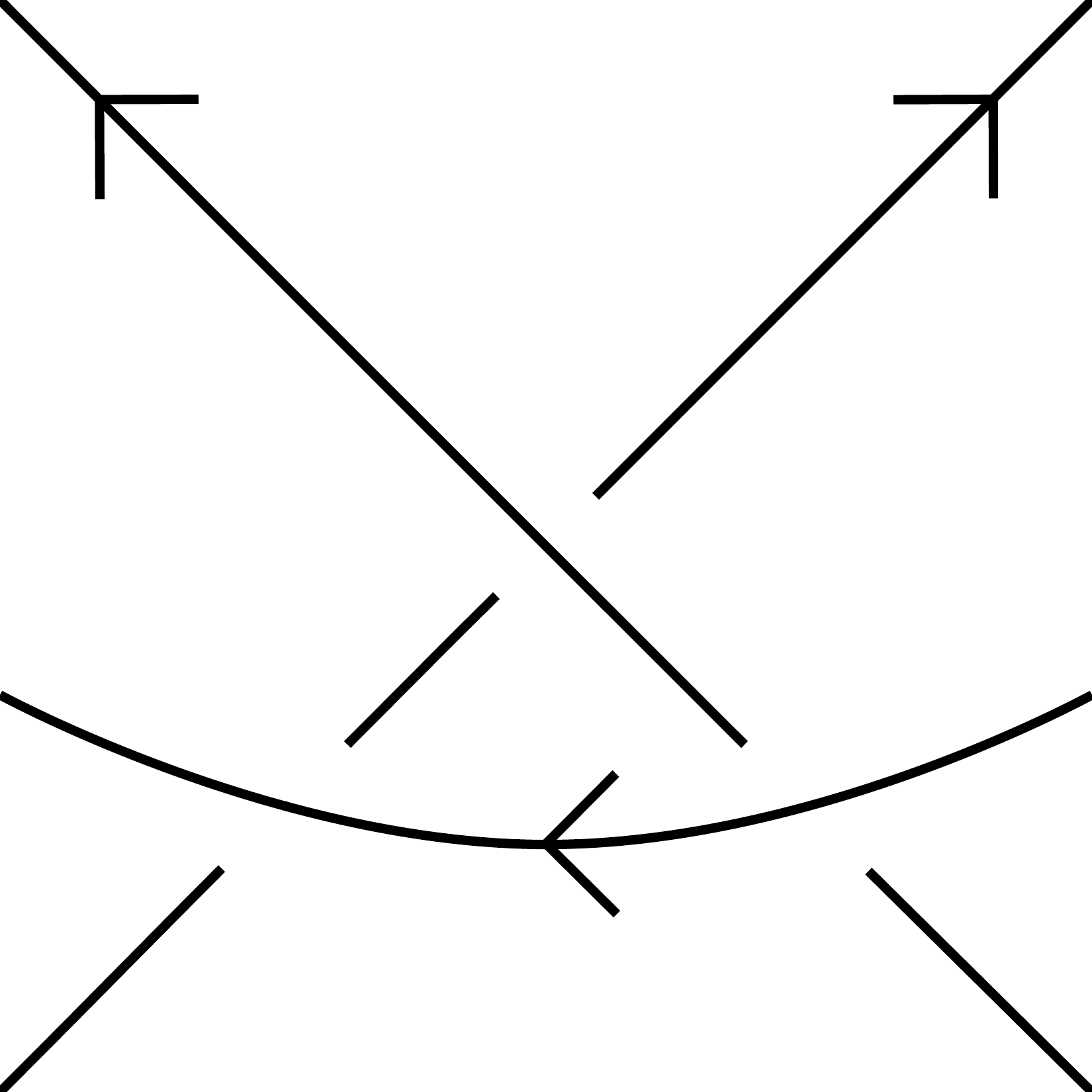}}\stackrel{\Omega 3g}{\longleftrightarrow}\raisebox{-13pt}{\includegraphics[height =0.45in]{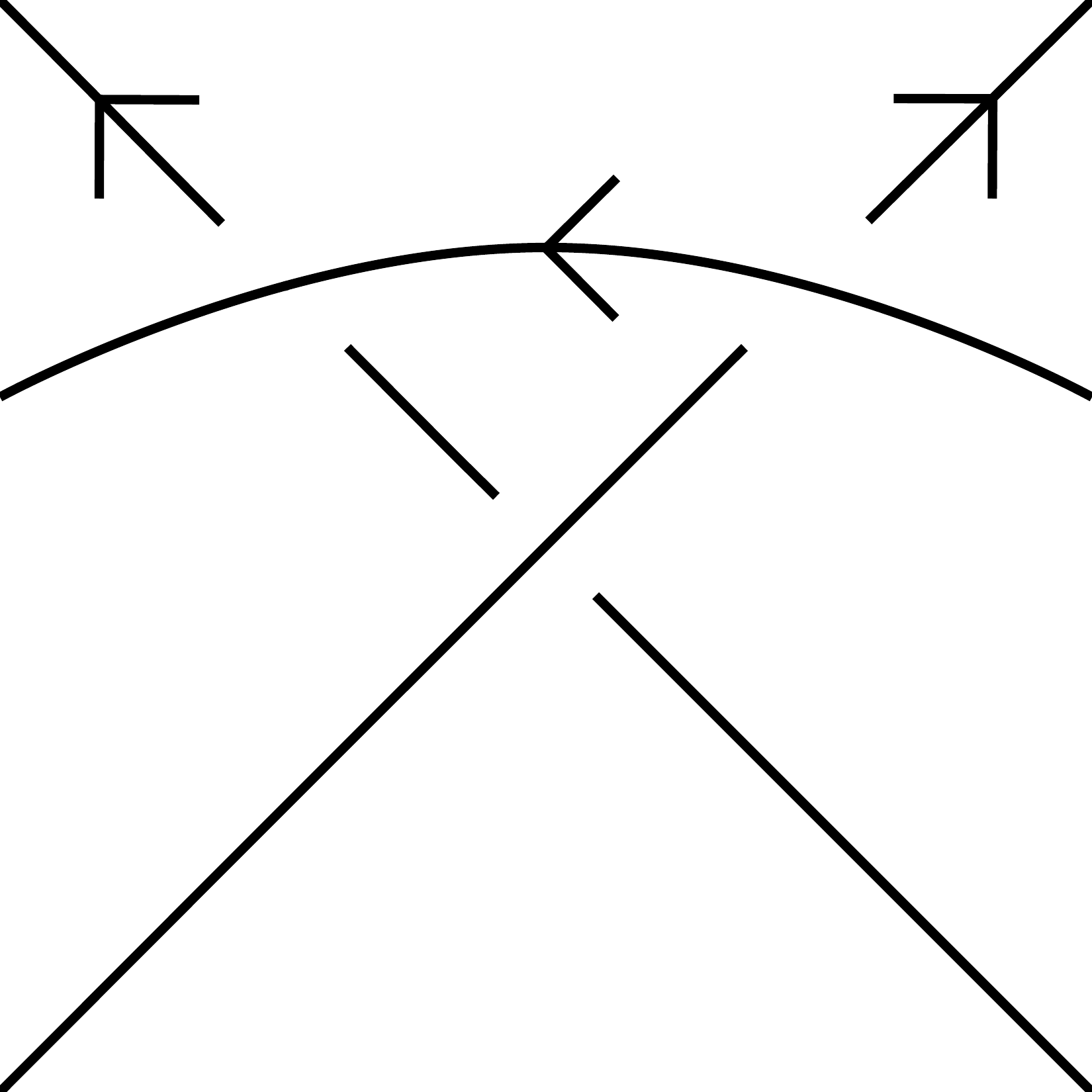}}\hspace{1.5cm}\raisebox{-13pt}{\includegraphics[height=0.45in]{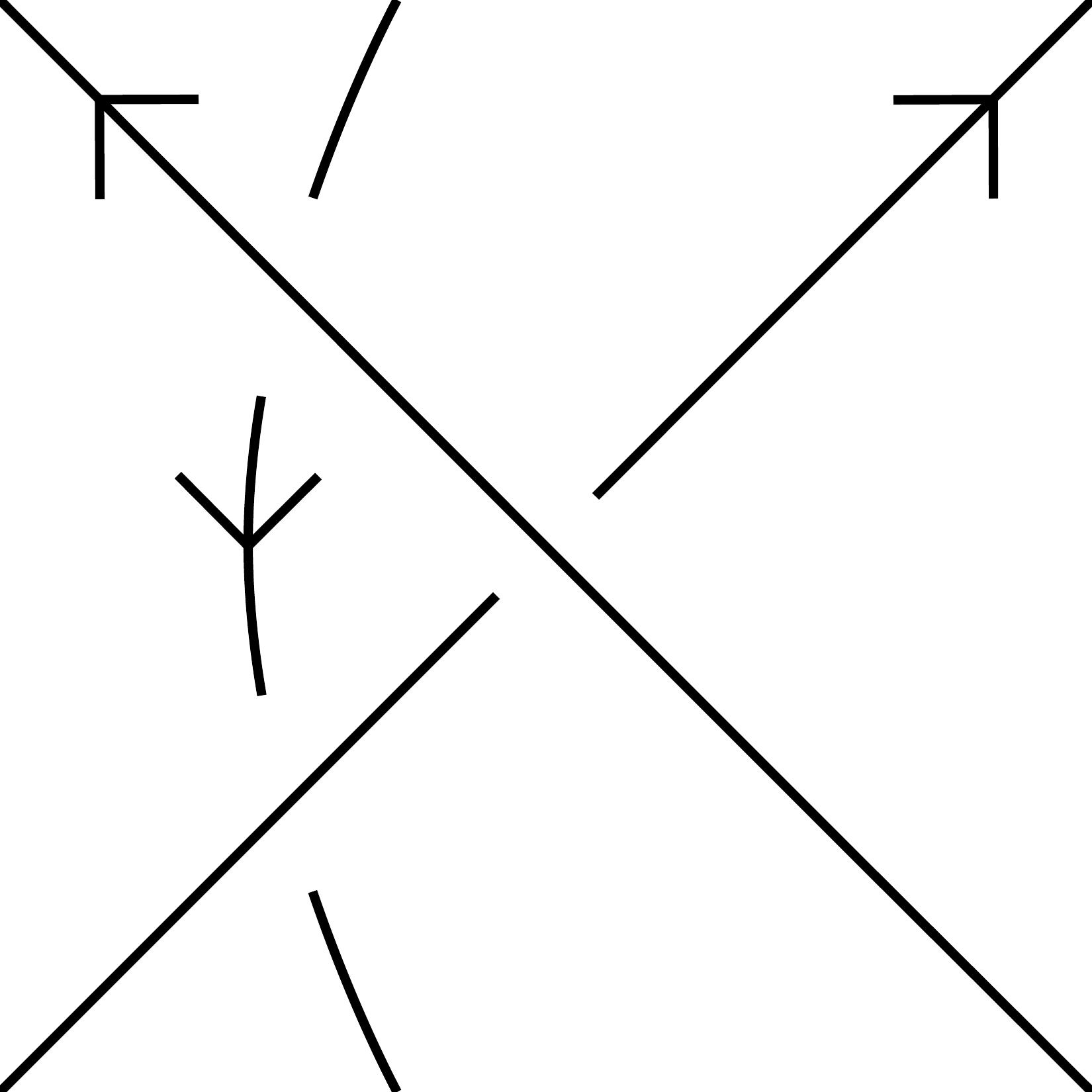}}\stackrel{\Omega 3h}{\longleftrightarrow}\raisebox{-13pt}{\includegraphics[height =0.45in]{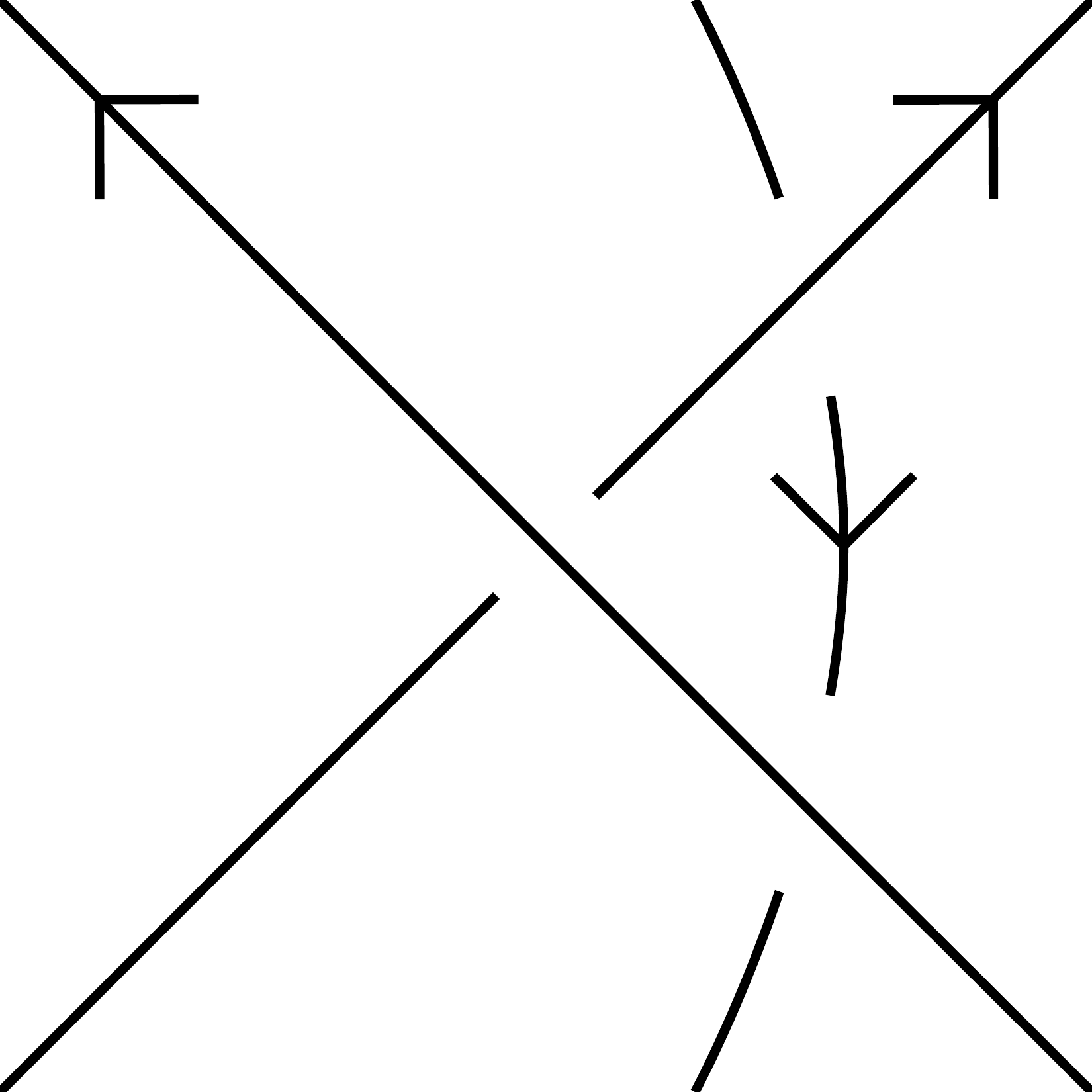}}\]
    \caption{Oriented $\Omega3$ moves}
    \label{fig:Omega3 Moves}
\end{figure}

We split the oriented versions of the moves $R4$ and respectively $R5$ into two sets, to separate the moves applied in a neighborhood of a vertex of type In-In-Out-Out from those applied in a neighborhood of a vertex of type In-Out-In-Out. This will also allow us to apply certain results from the work by Bataineh et al. in~\cite{Bataineh}. We remark that the paper ~\cite{Bataineh} is about singular links and that a diagram of a singular link contains classical crossings and 4-valent vertices of type In-In-Out-Out. Moreover, singular links are equivalent to knotted 4-valent graphs
(regarded up to rigid-vertex isotopy) whose vertices are of type In-In-Out-Out.

In Figure~\ref{fig:IIOO Omega4 Moves}, we present the oriented versions of the moves $R4$ with vertices of type In-In-Out-Out, and we denote these oriented versions of the moves by $\Omega4$, following the conventions used in~\cite{Bataineh}. Similarly, in Figure~\ref{fig:IOIO Omega4 Moves}, we give all oriented versions of the moves $\Omega4$ with vertices of type In-Out-In-Out.

\begin{figure}[ht]
    \[\raisebox{-13pt}{\includegraphics[height=0.45in]{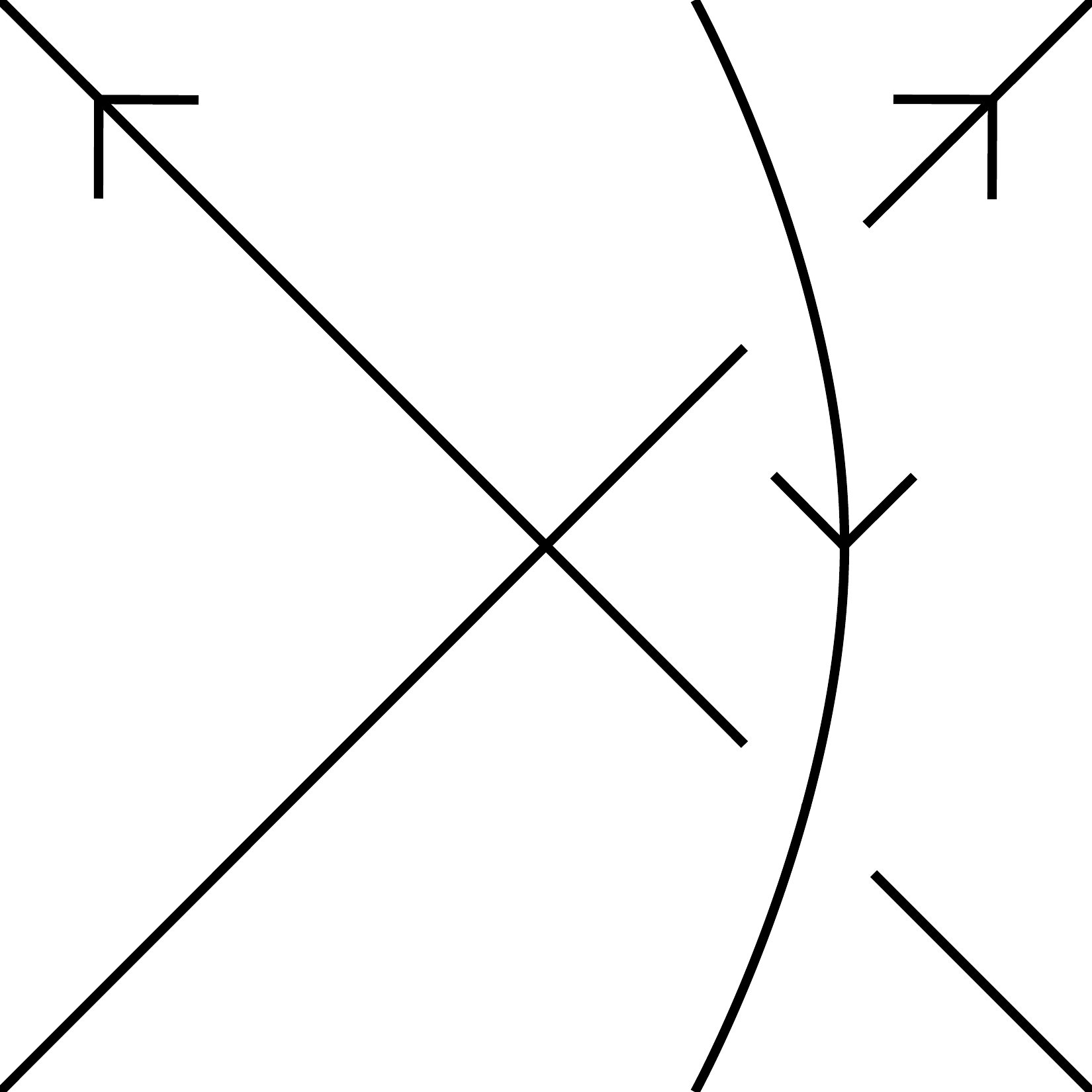}}\stackrel{\Omega 4a}{\longleftrightarrow}\raisebox{-13pt}{\includegraphics[height =0.45in]{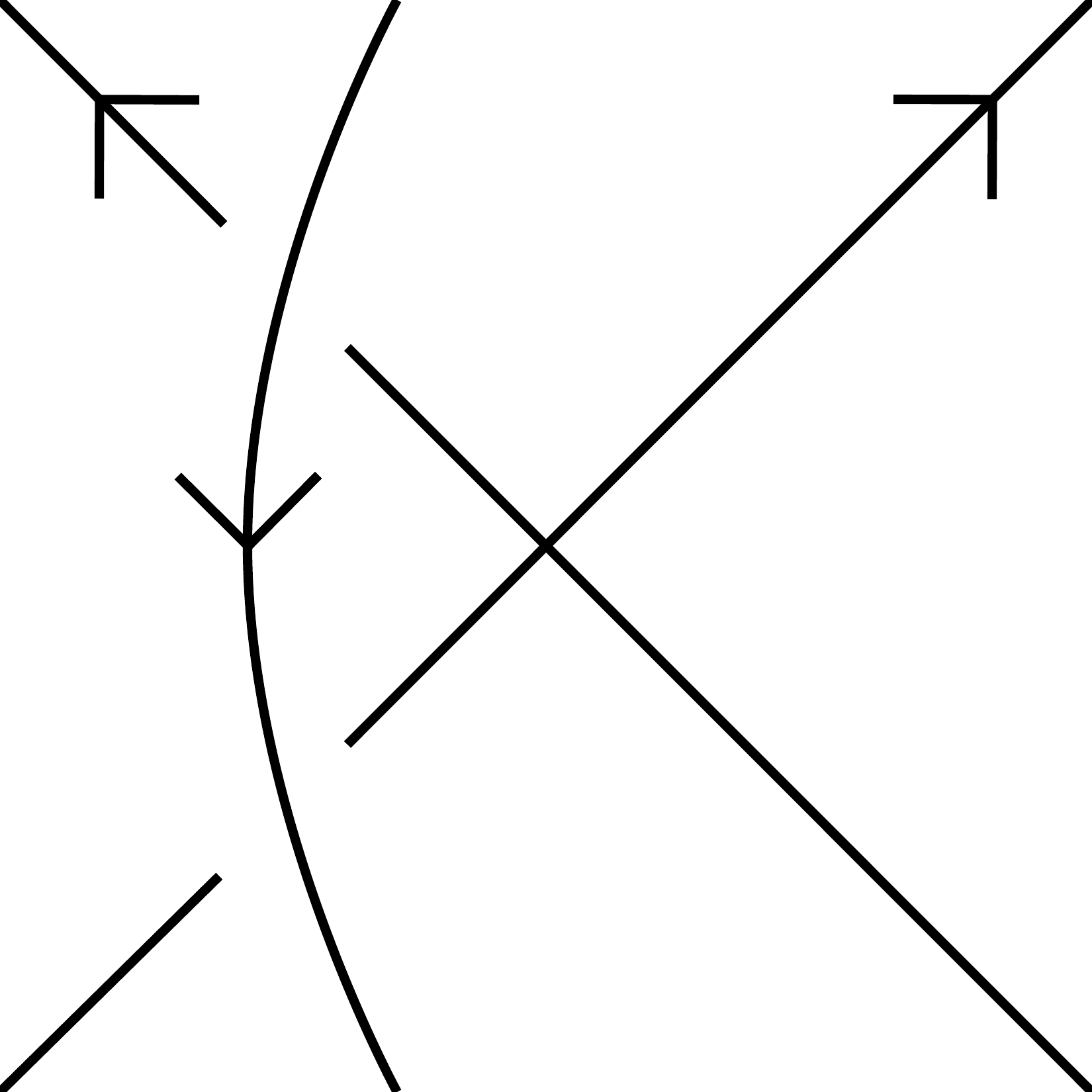}}\hspace{1.5cm}\raisebox{-13pt}{\includegraphics[height=0.45in]{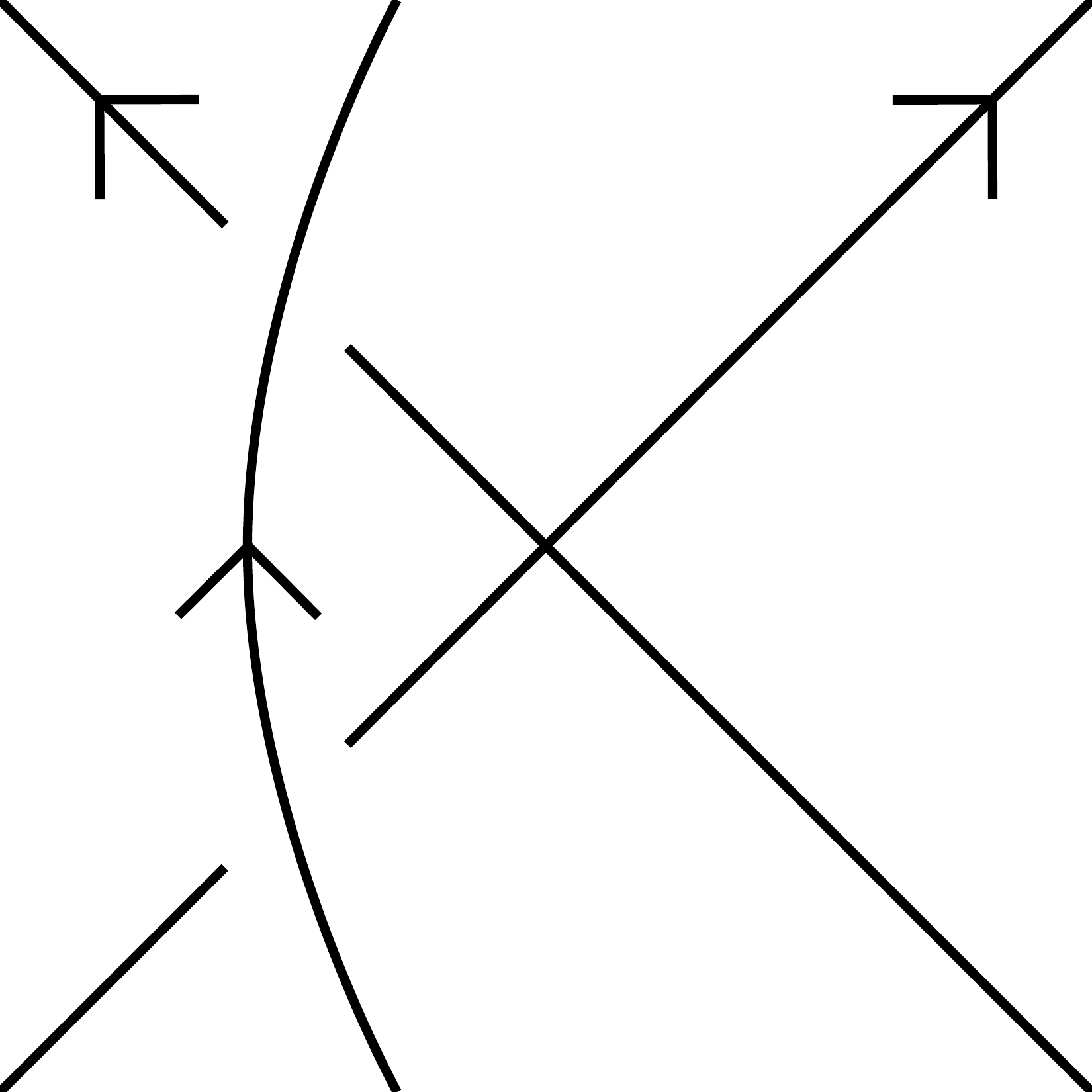}}\stackrel{\Omega 4b}{\longleftrightarrow}\raisebox{-13pt}{\includegraphics[height =0.45in]{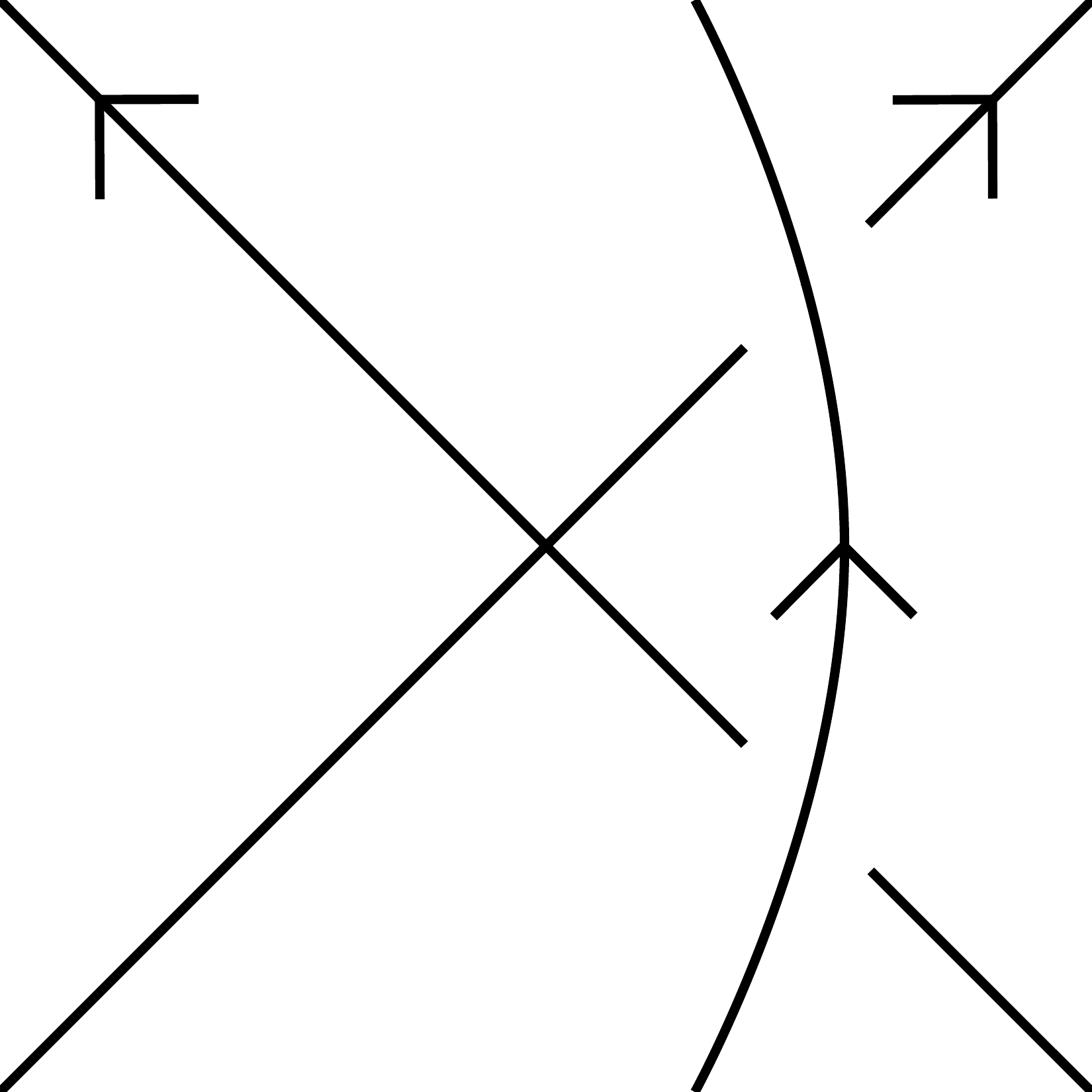}}\]
    \vspace{0.10cm}
    \[\raisebox{-13pt}{\includegraphics[height=0.45in]{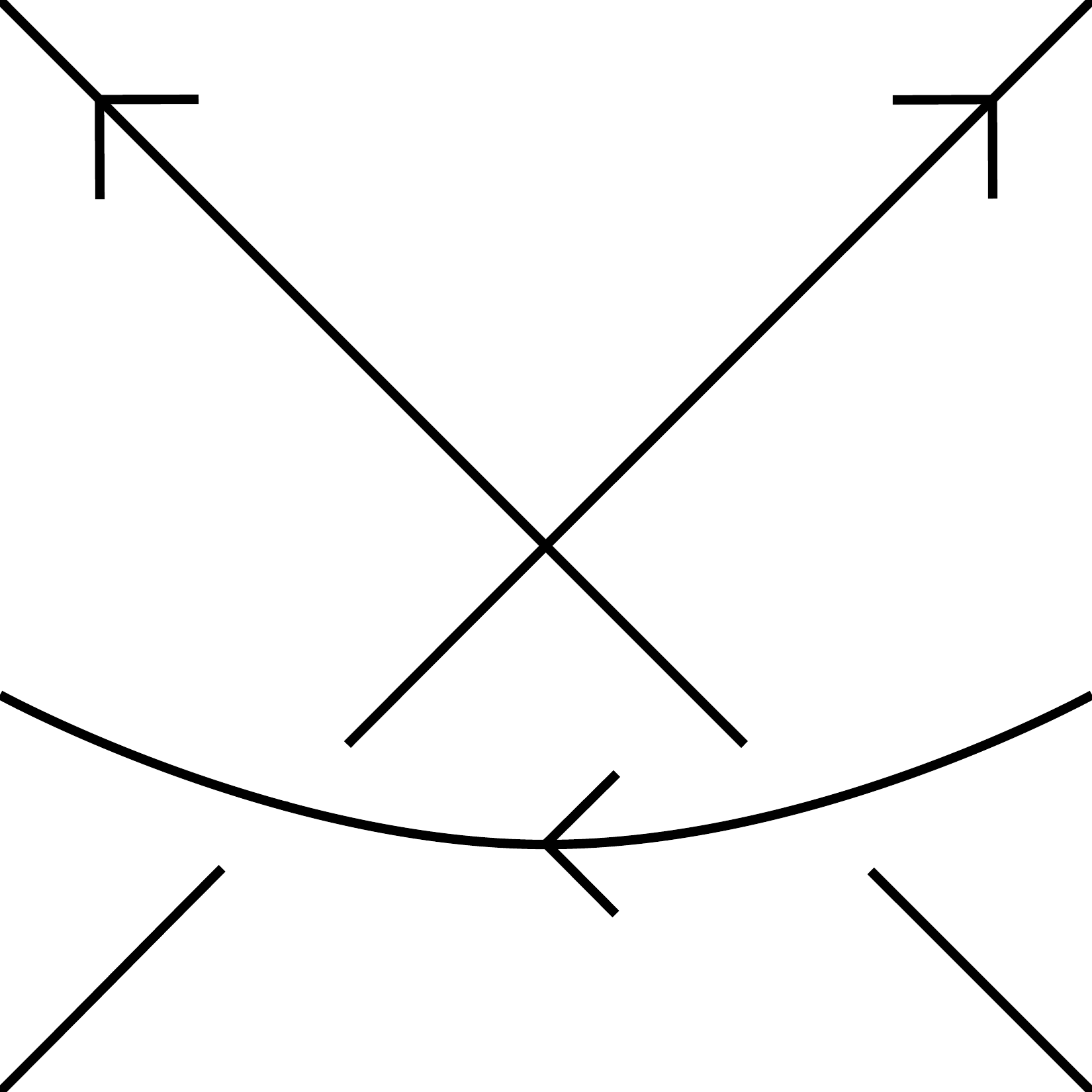}}\stackrel{\Omega 4c}{\longleftrightarrow}\raisebox{-13pt}{\includegraphics[height =0.45in]{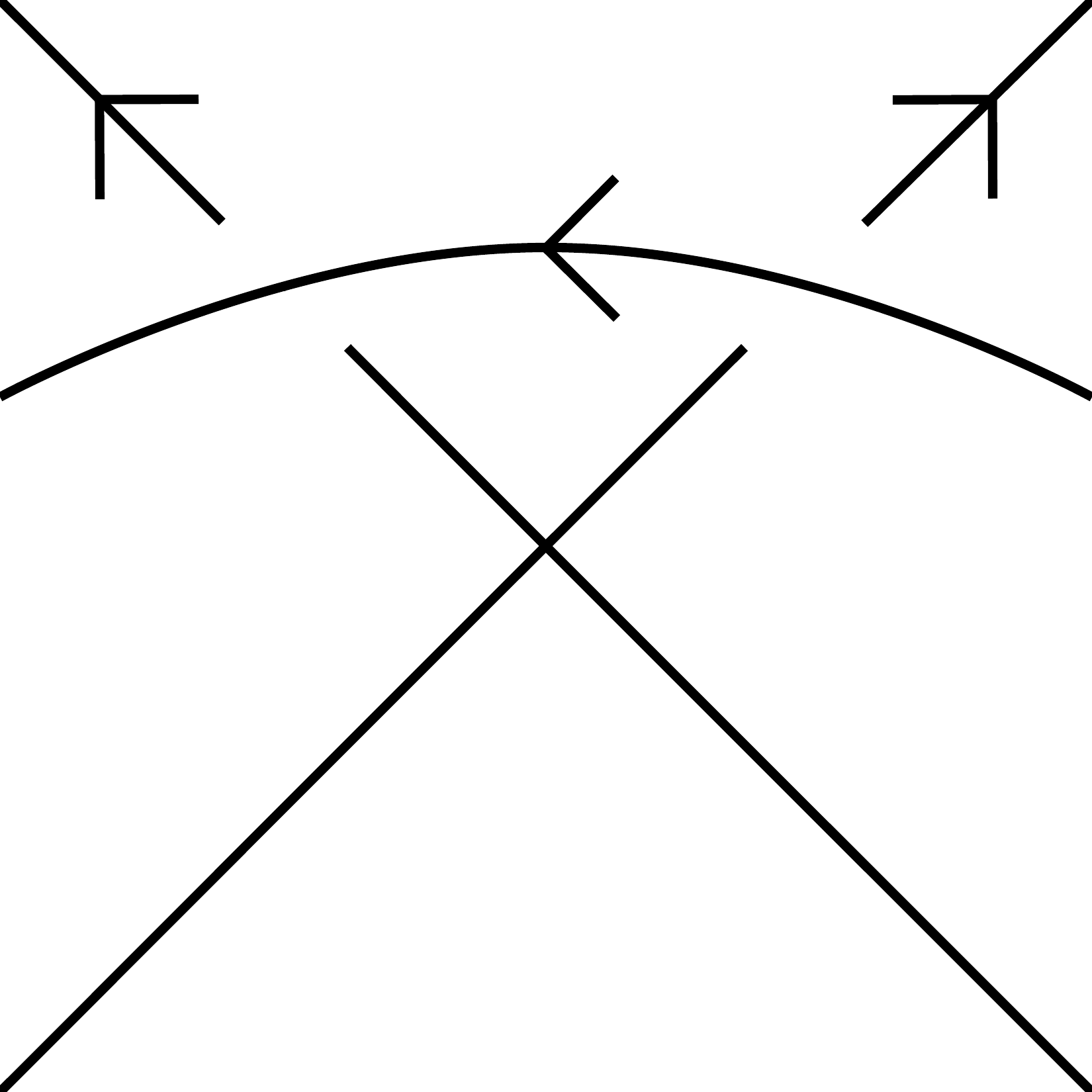}}\hspace{1.5cm}\raisebox{-13pt}{\includegraphics[height=0.45in]{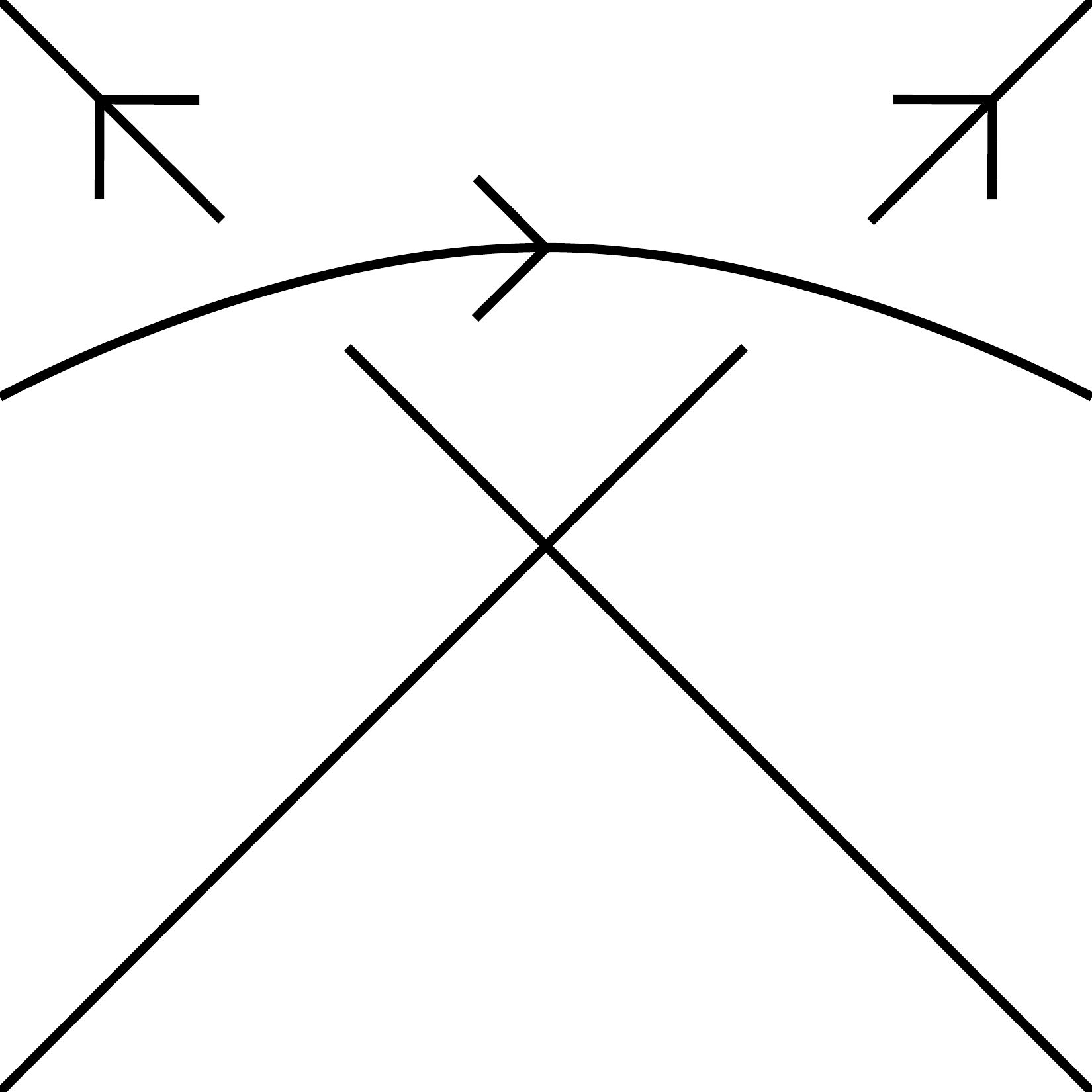}}\stackrel{\Omega 4d}{\longleftrightarrow}\raisebox{-13pt}{\includegraphics[height =0.45in]{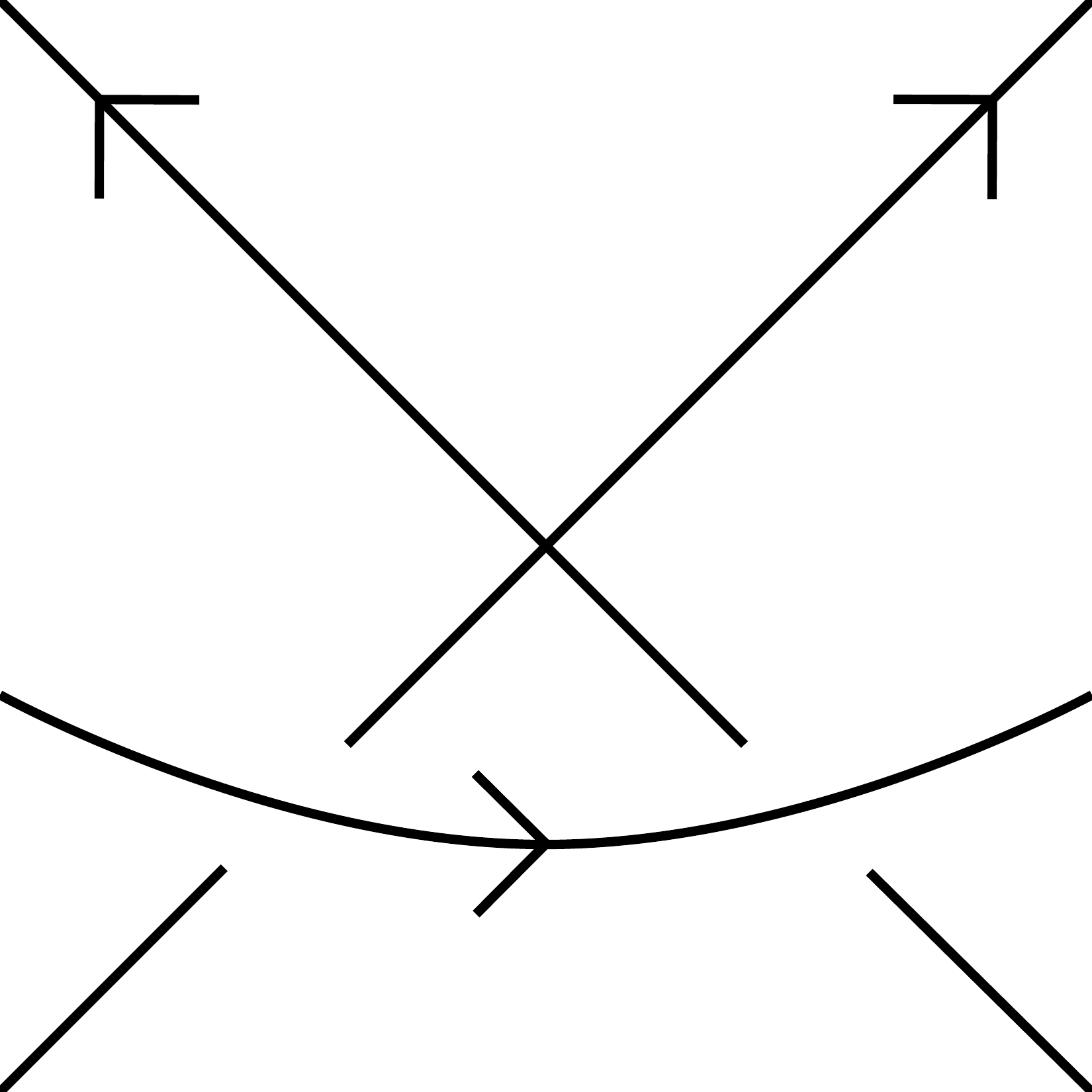}}\]
    \vspace{0.10cm}
    \[\raisebox{-13pt}{\includegraphics[height=0.45in]{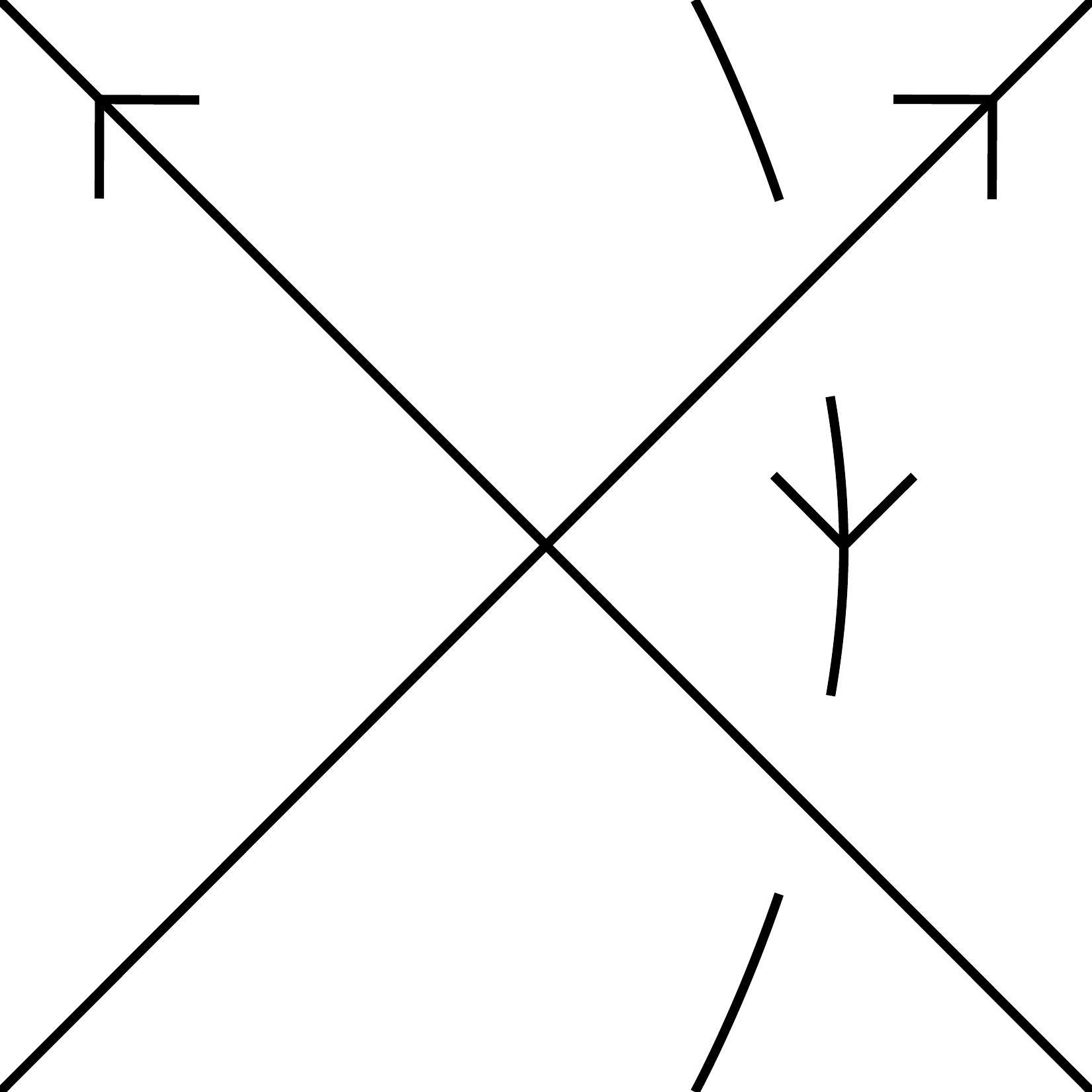}}\stackrel{\Omega 4e}{\longleftrightarrow}\raisebox{-13pt}{\includegraphics[height =0.45in]{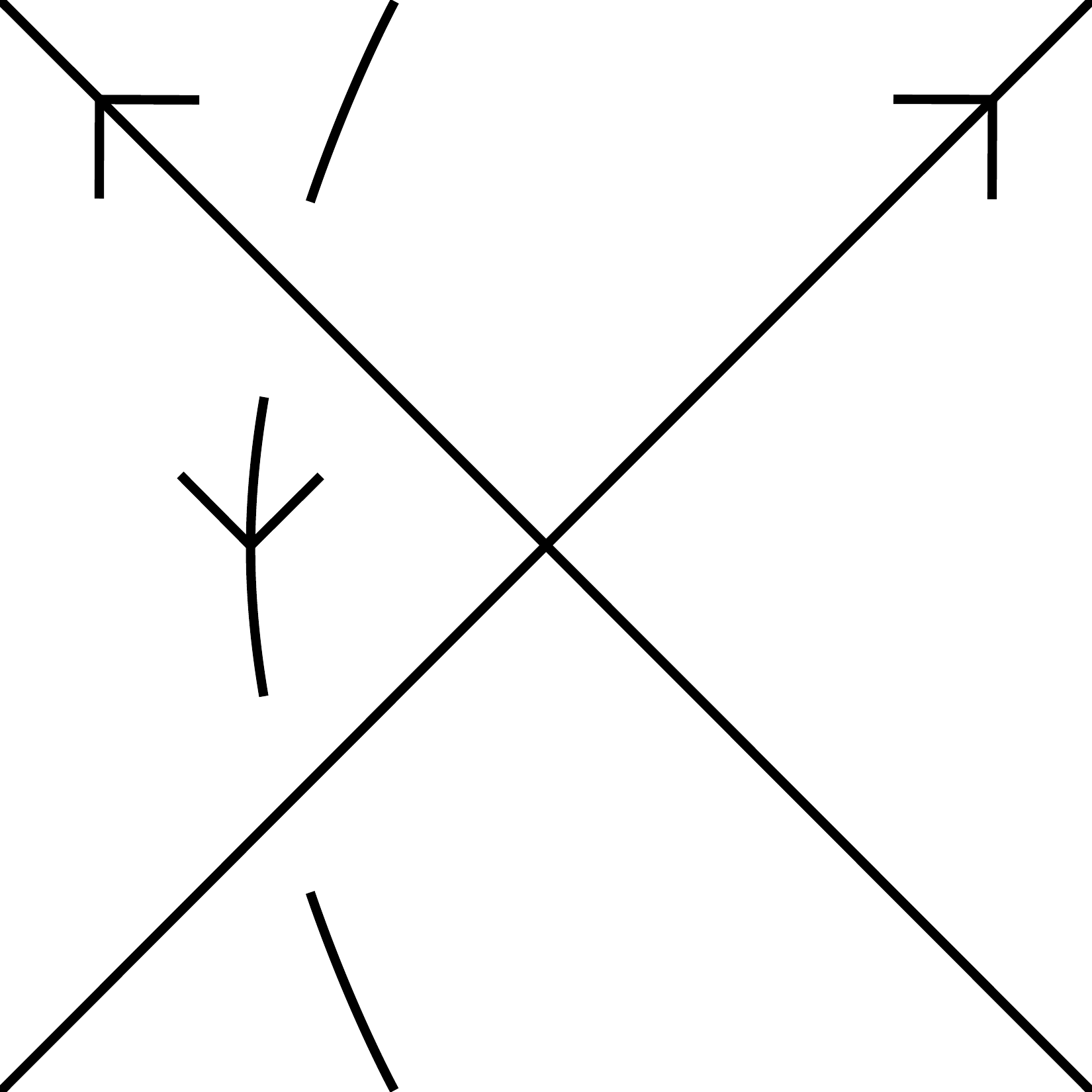}}\hspace{1.5cm}\raisebox{-13pt}{\includegraphics[height=0.45in]{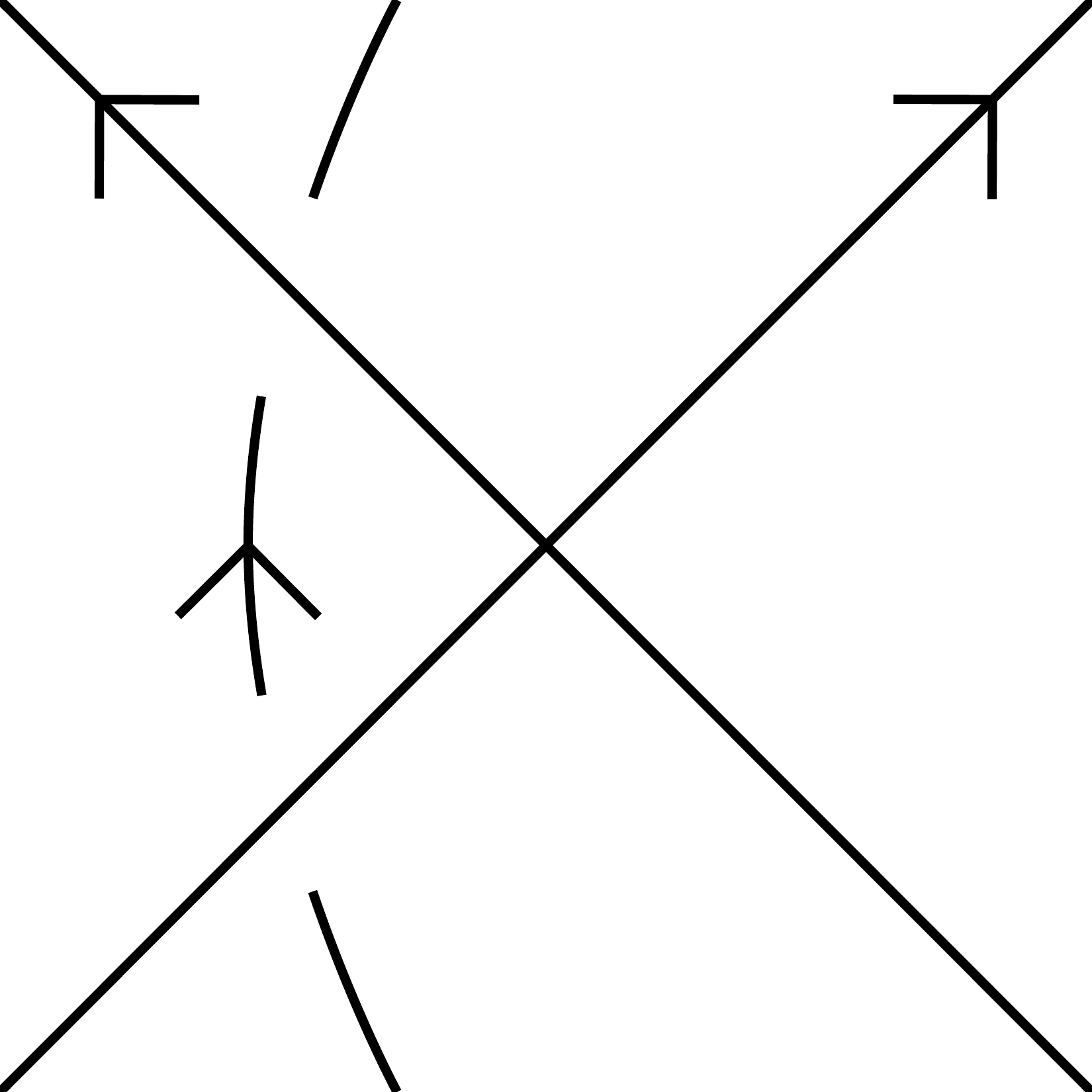}}\stackrel{\Omega 4f}{\longleftrightarrow}\raisebox{-13pt}{\includegraphics[height =0.45in]{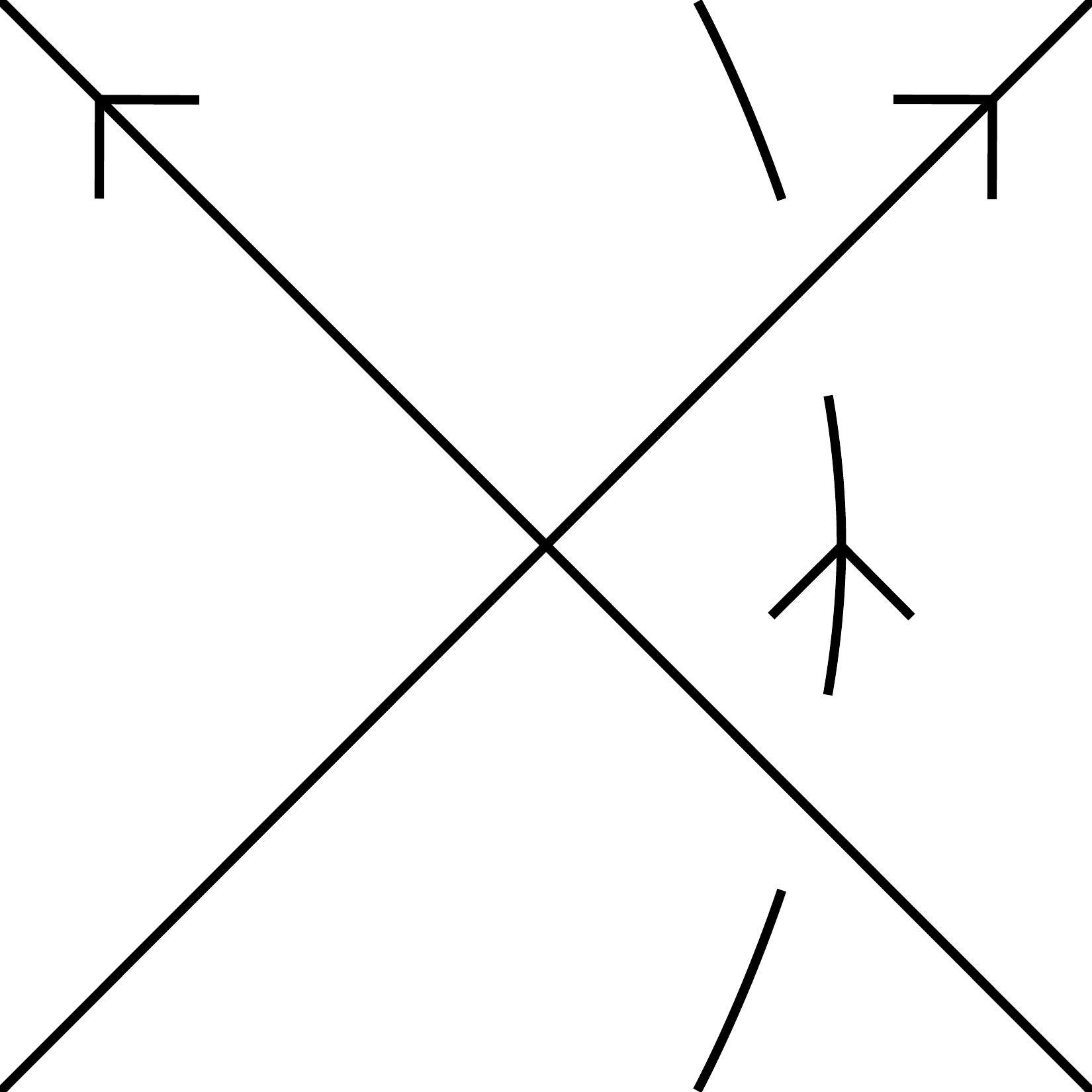}}\]
    \vspace{0.10cm}
    \[\raisebox{-13pt}{\includegraphics[height=0.45in]{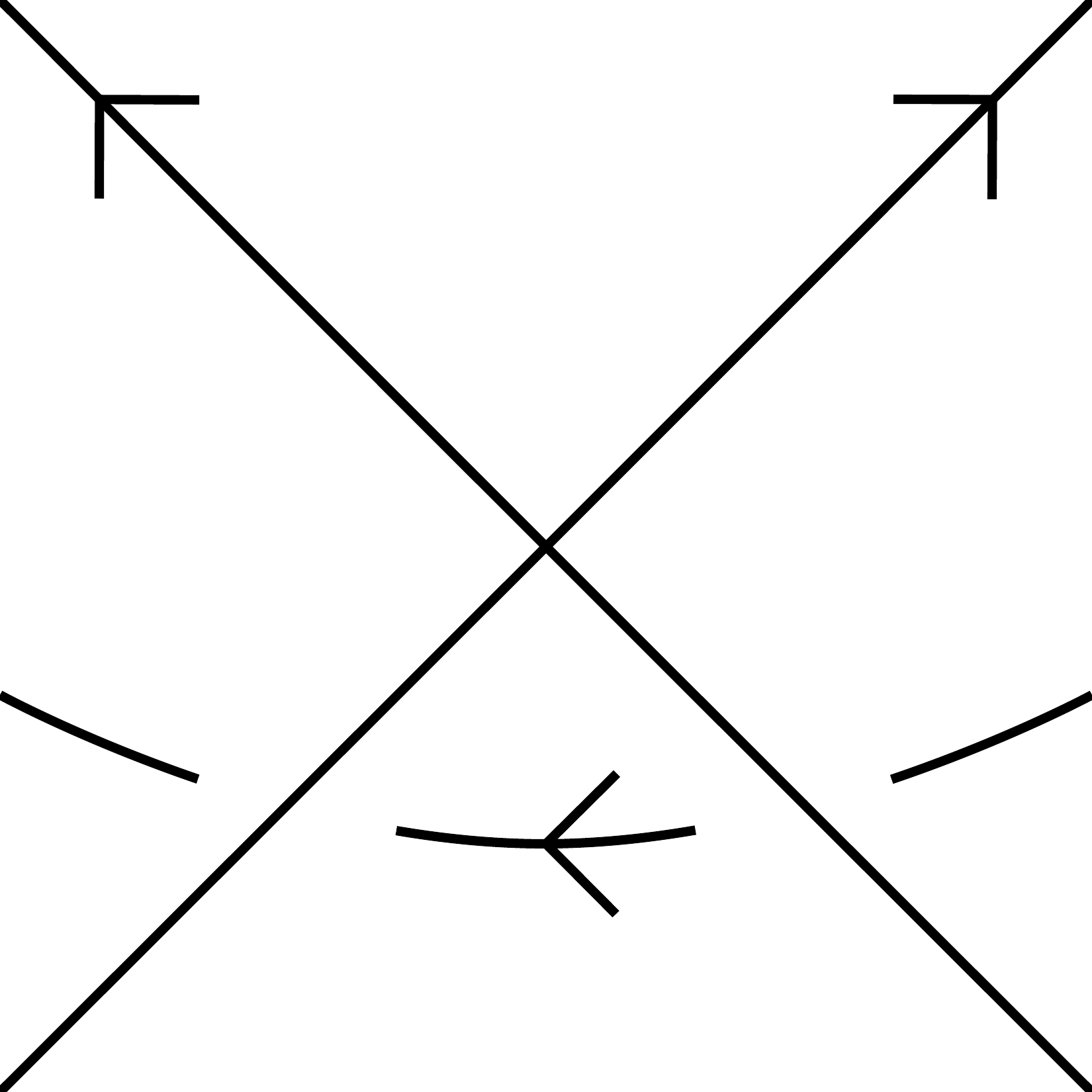}}\stackrel{\Omega 4g}{\longleftrightarrow}\raisebox{-13pt}{\includegraphics[height =0.45in]{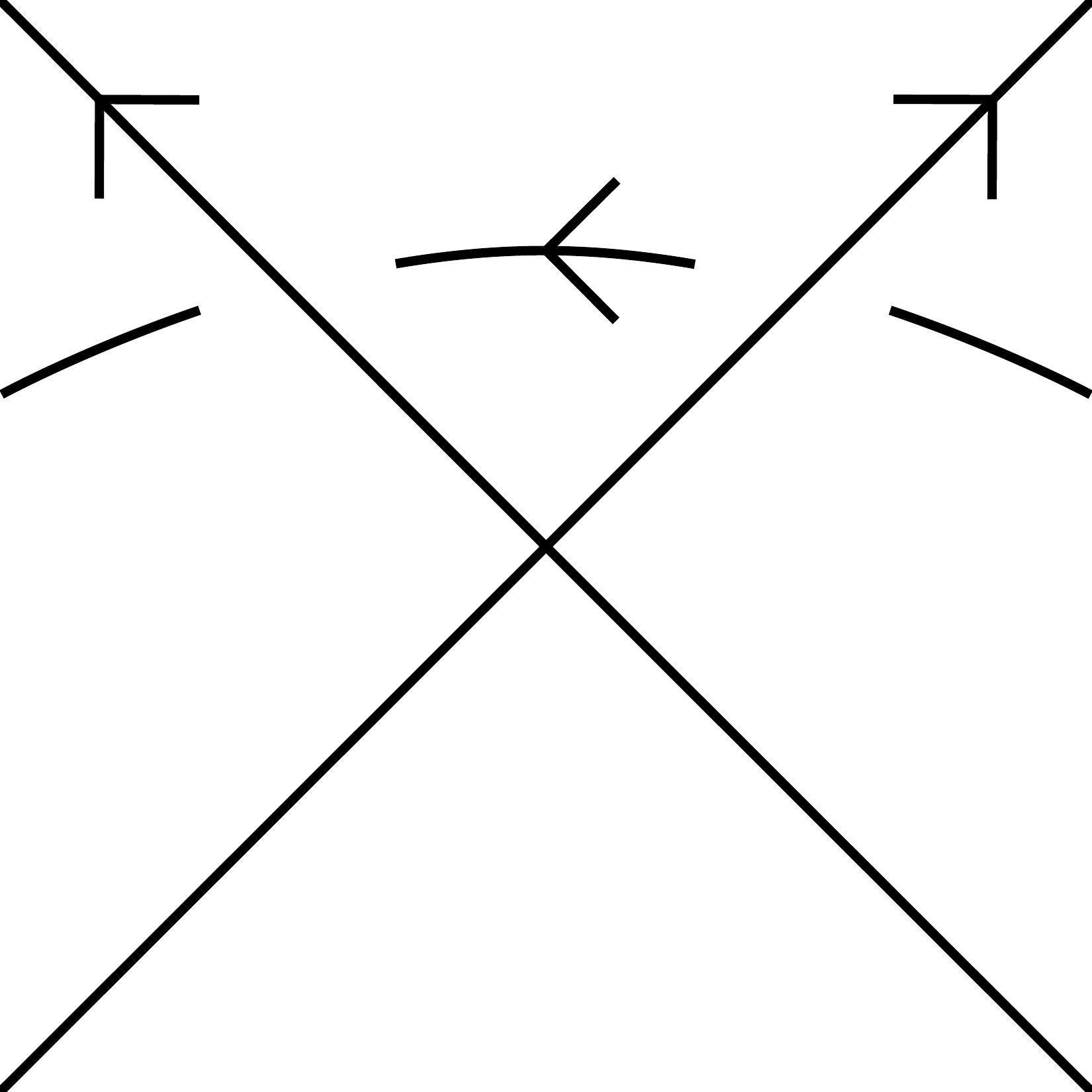}}\hspace{1.5cm}\raisebox{-13pt}{\includegraphics[height=0.45in]{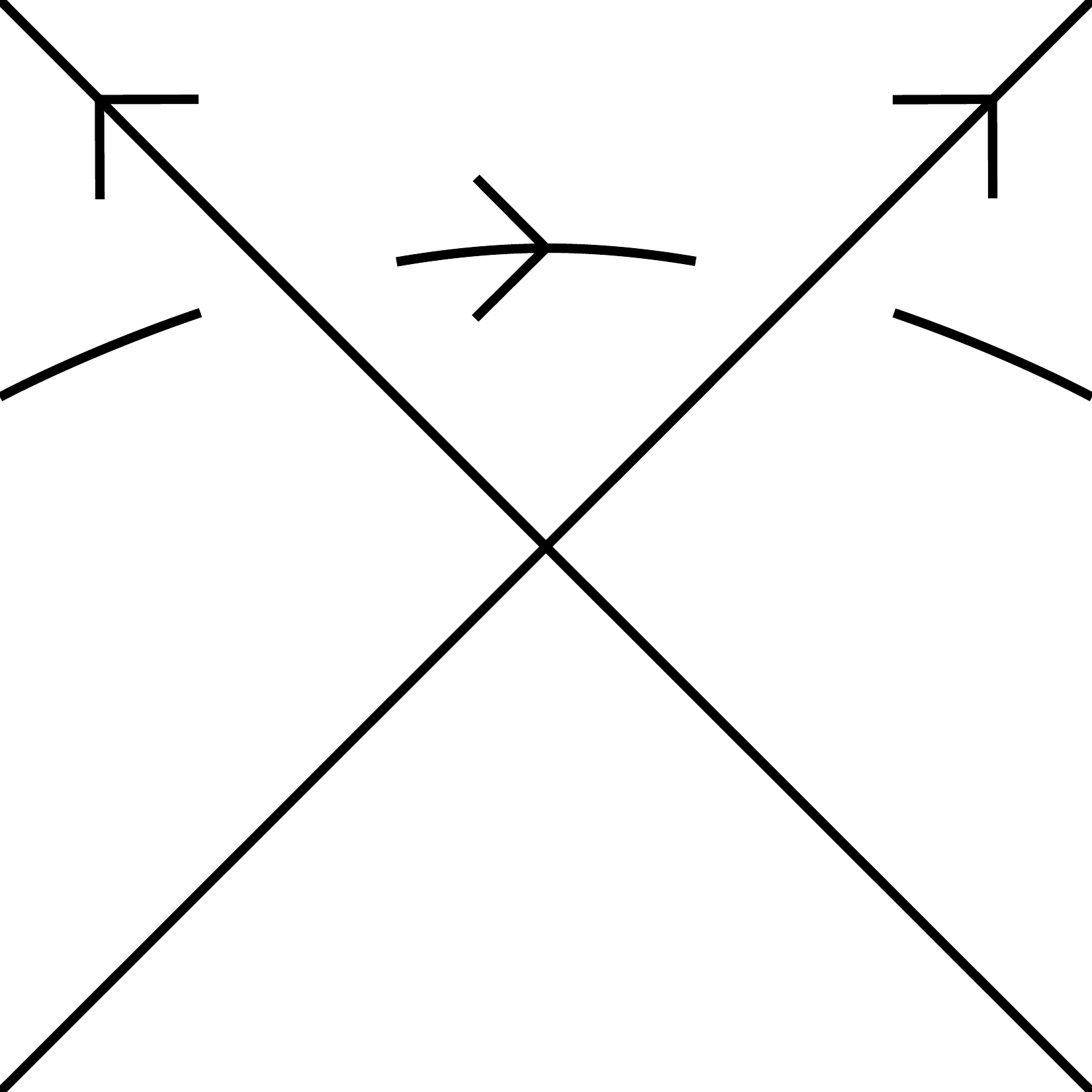}}\stackrel{\Omega 4h}{\longleftrightarrow}\raisebox{-13pt}{\includegraphics[height =0.45in]{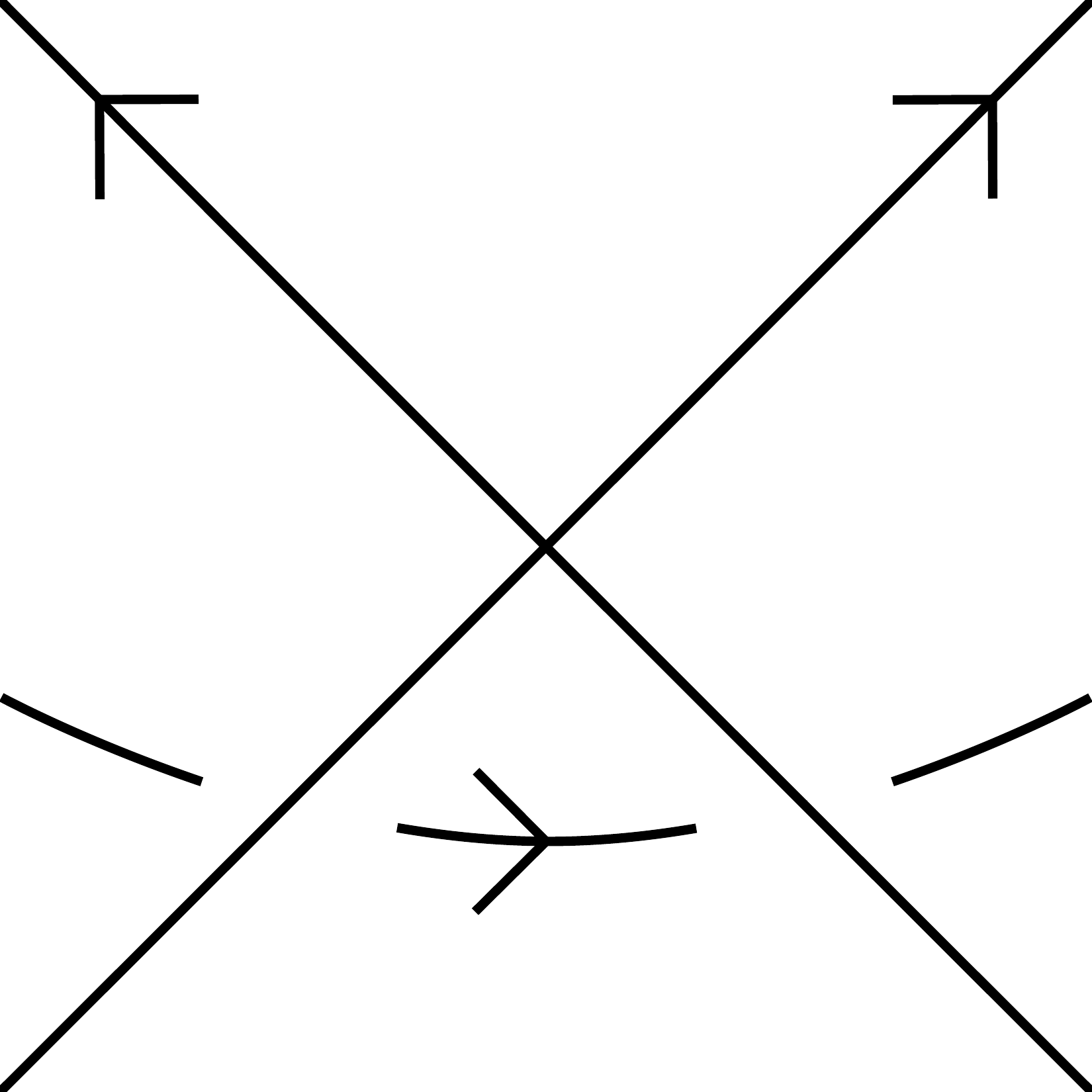}}\]
    \caption{$\Omega4$ moves with vertices of type In-In-Out-Out}
    \label{fig:IIOO Omega4 Moves}
\end{figure}

\begin{figure}[ht]
    \[\raisebox{-13pt}{\includegraphics[height=0.45in]{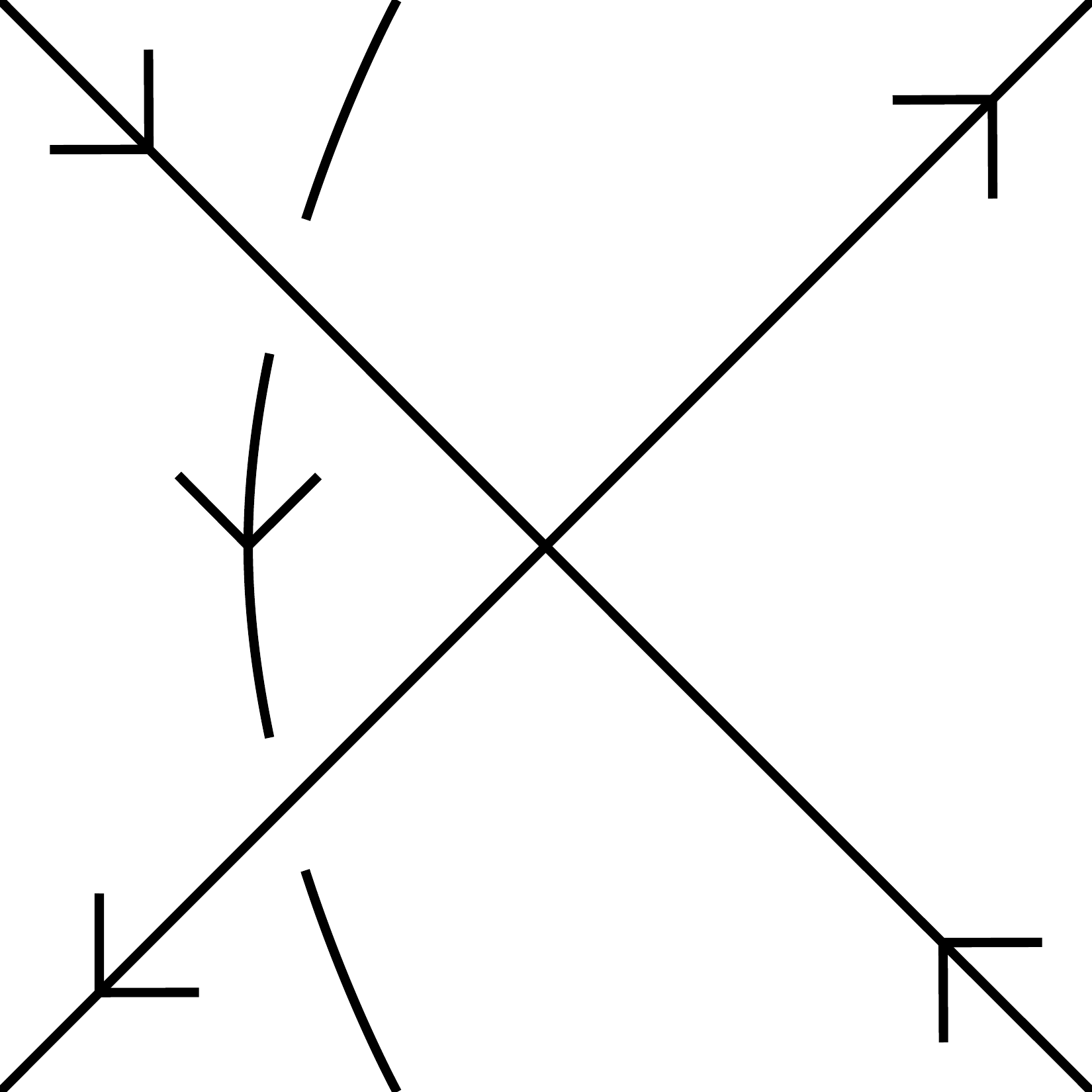}}\stackrel{\Omega 4i}{\longleftrightarrow}\raisebox{-13pt}{\includegraphics[height =0.45in]{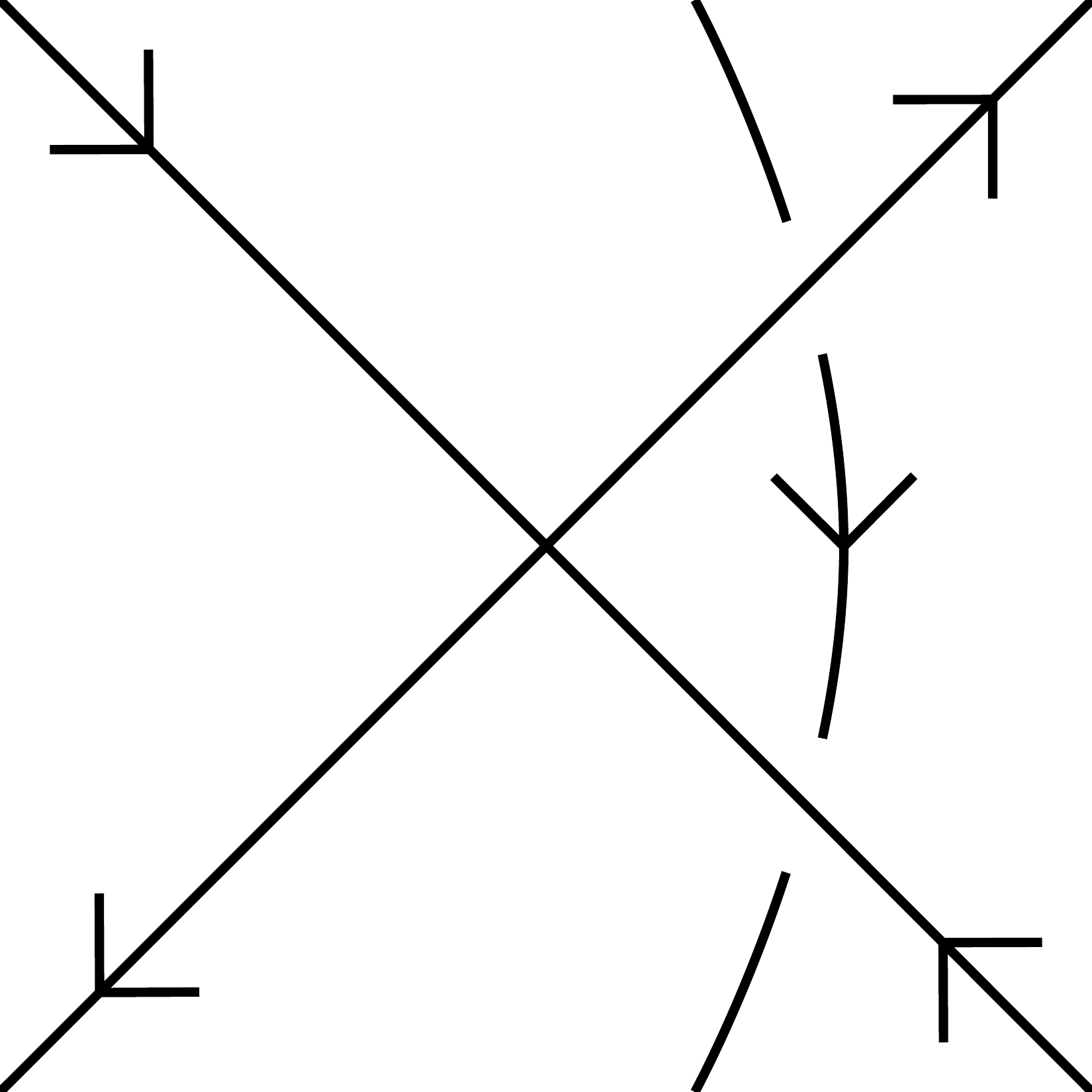}}\hspace{1.5cm}\raisebox{-13pt}{\includegraphics[height=0.45in]{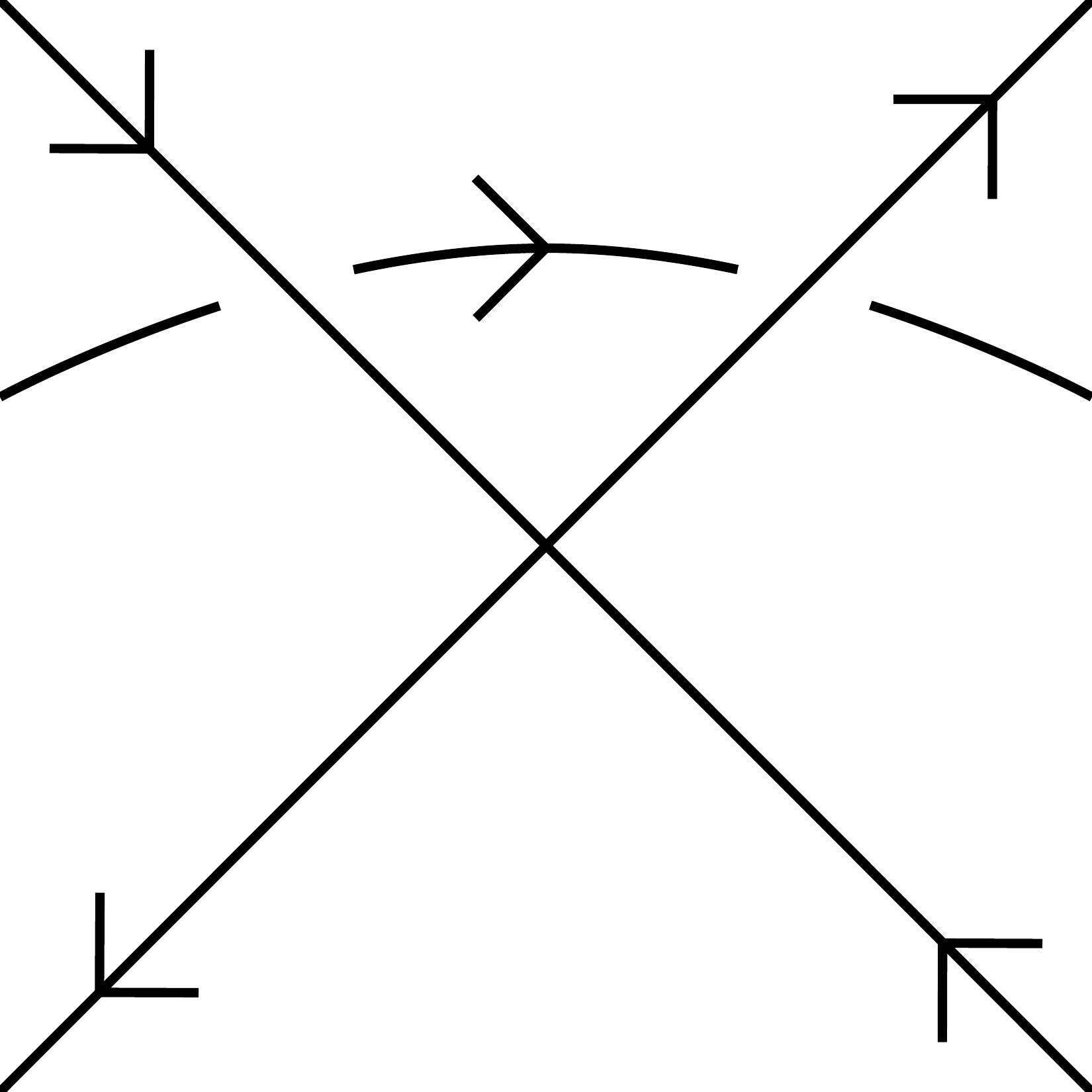}}\stackrel{\Omega 4j}{\longleftrightarrow}\raisebox{-13pt}{\includegraphics[height =0.45in]{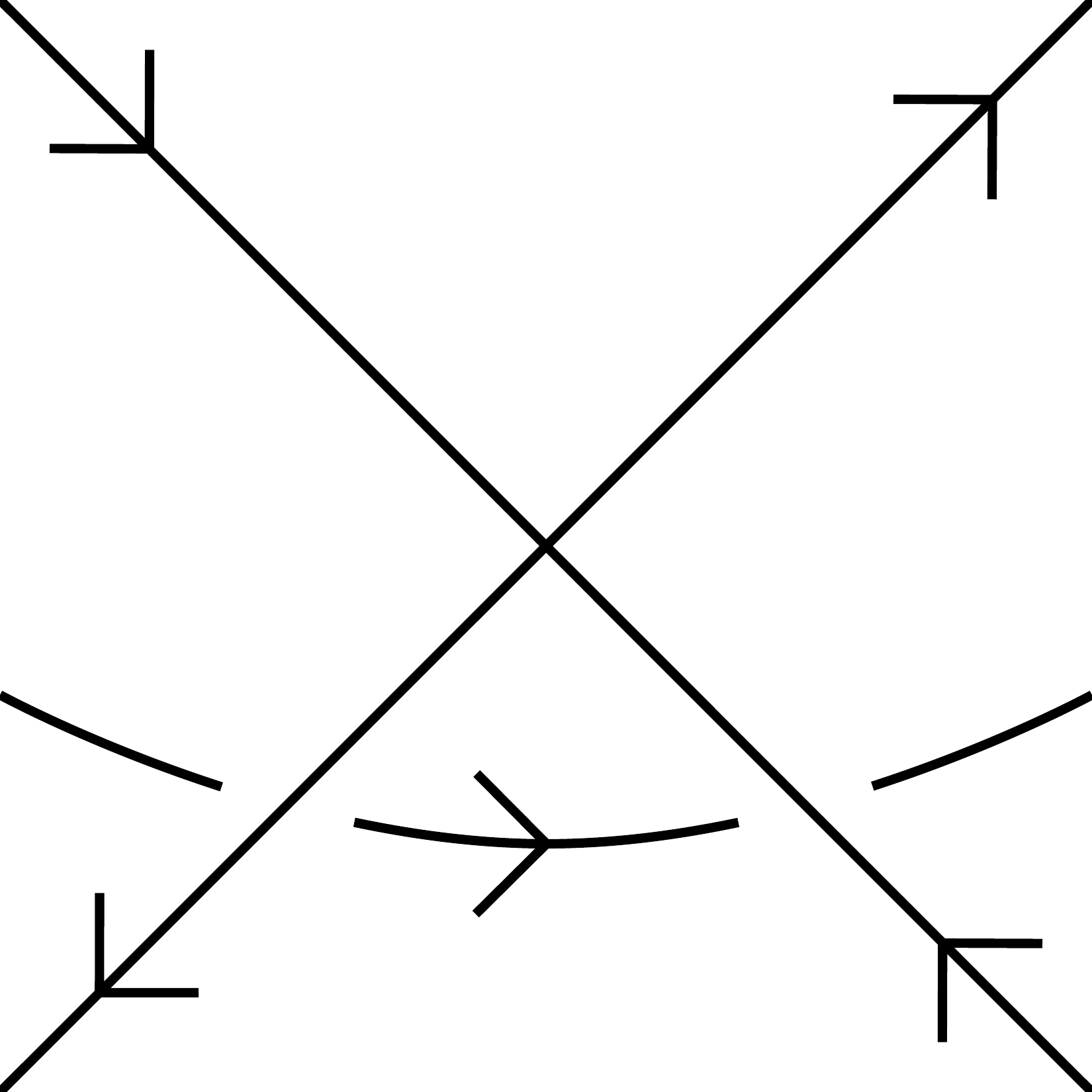}}\]
    \vspace{0.10cm}
    \[\raisebox{-13pt}{\includegraphics[height=0.45in]{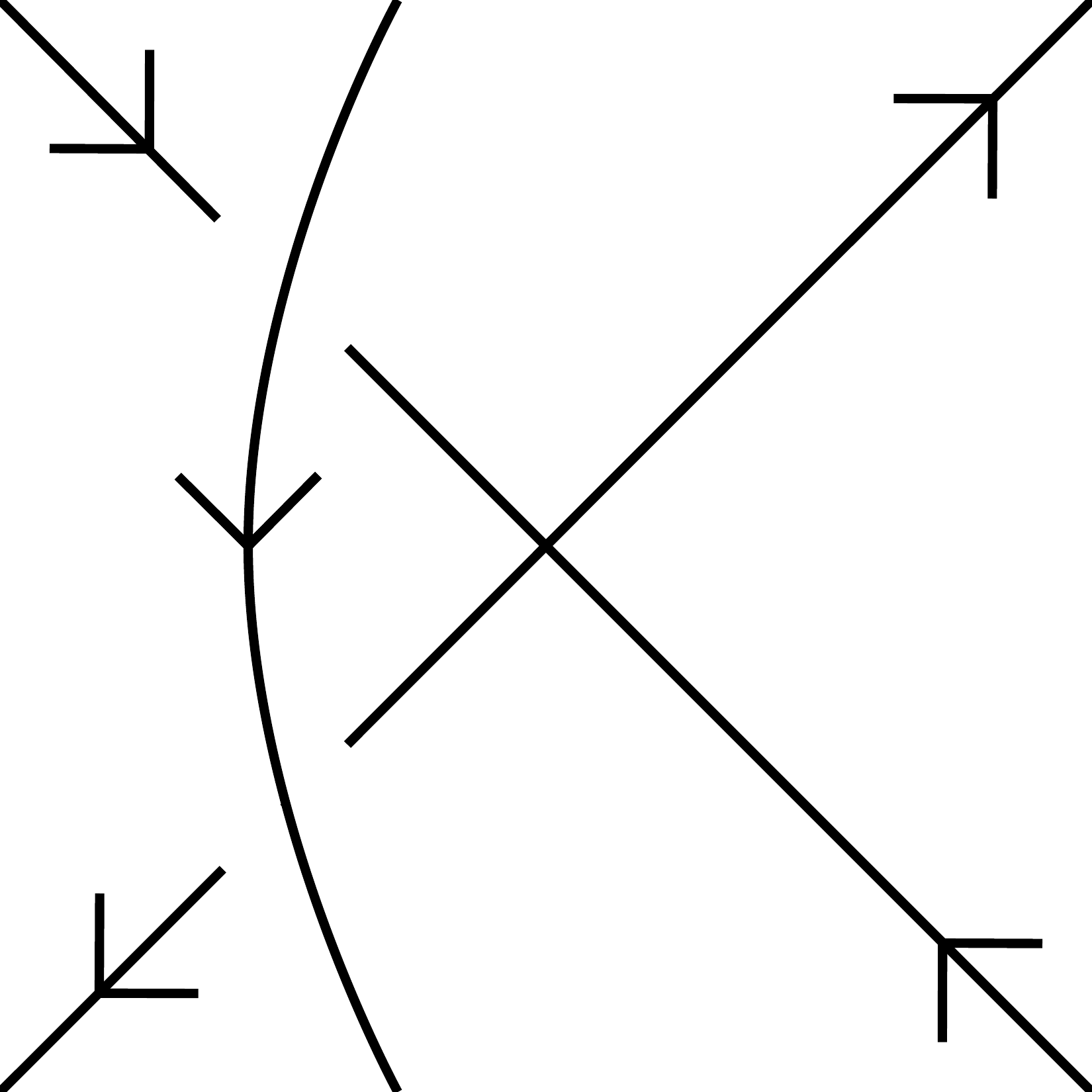}}\stackrel{\Omega 4k}{\longleftrightarrow}\raisebox{-13pt}{\includegraphics[height =0.45in]{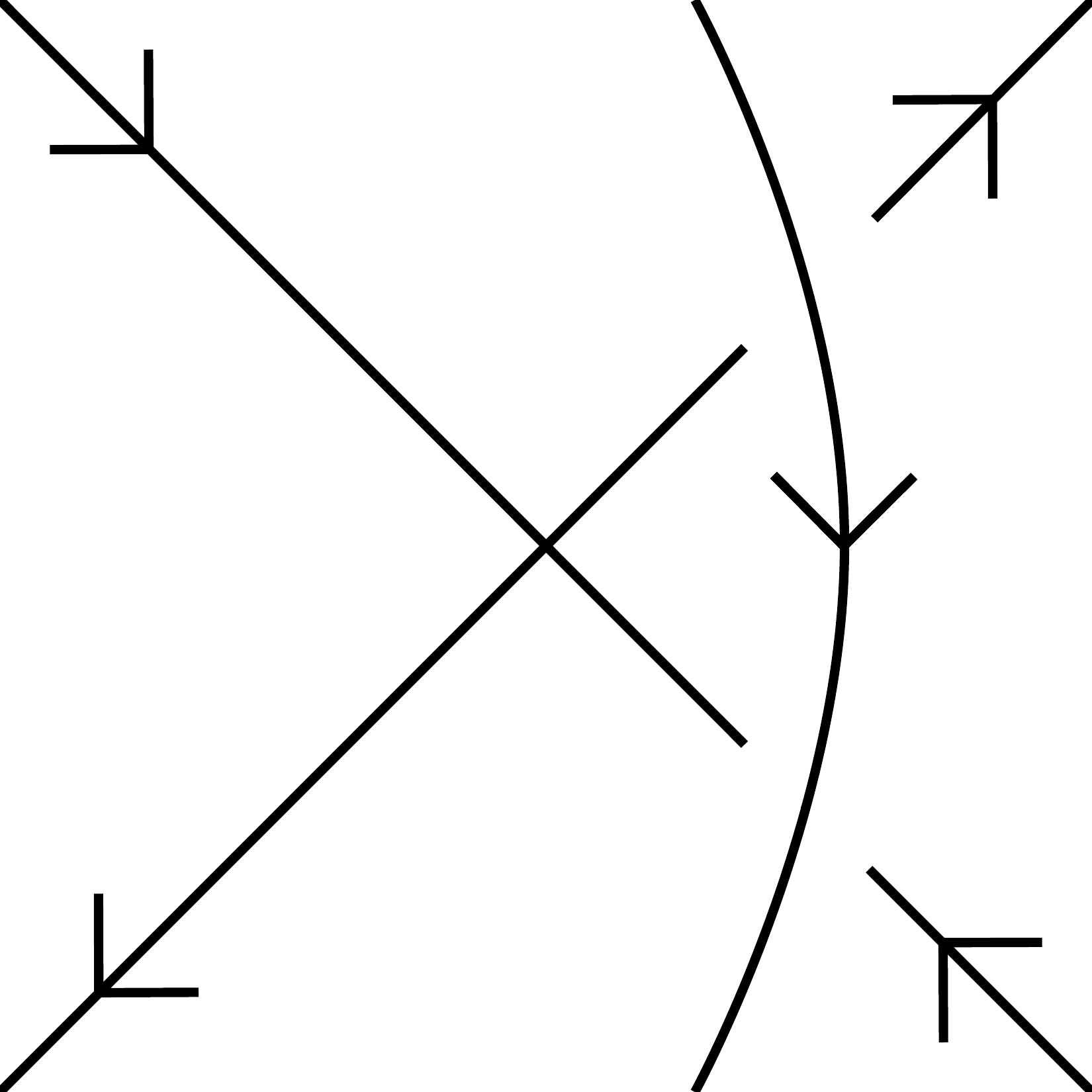}}\hspace{1.5cm}\raisebox{-13pt}{\includegraphics[height=0.45in]{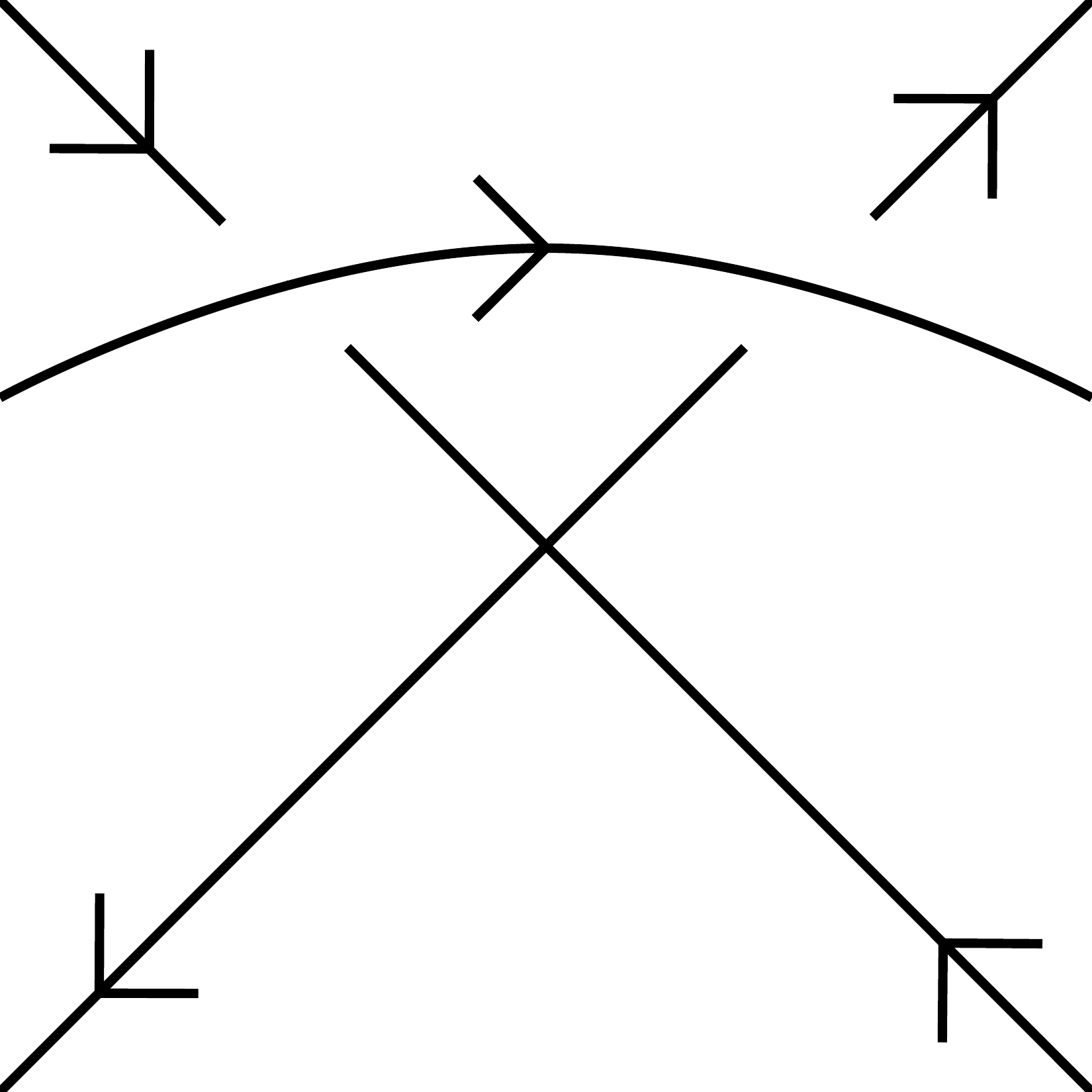}}\stackrel{\Omega 4l}{\longleftrightarrow}\raisebox{-13pt}{\includegraphics[height =0.45in]{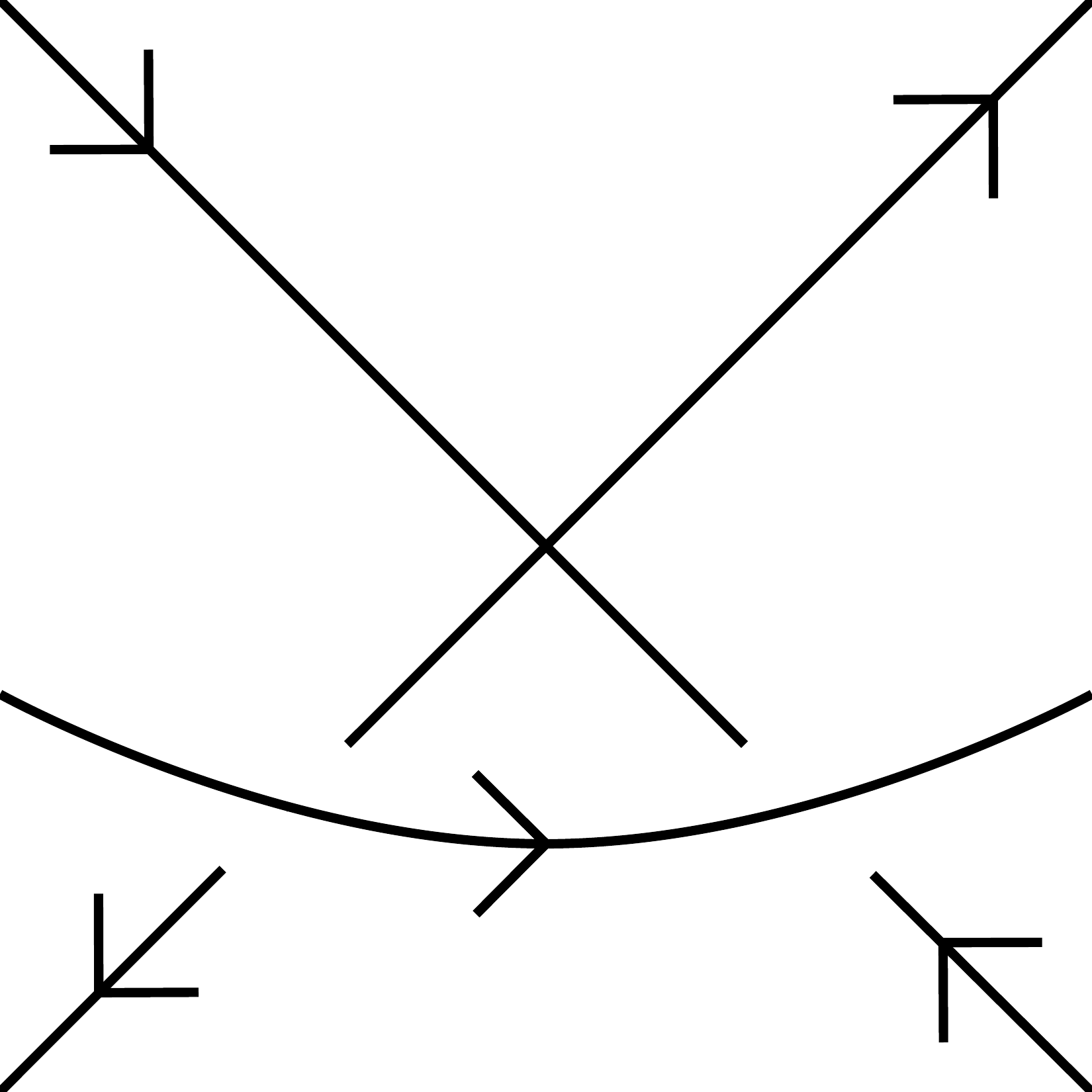}}\]
    \caption{$\Omega4$ moves with vertices of type In-Out-In-Out}
    \label{fig:IOIO Omega4 Moves}
\end{figure}

Figure~\ref{fig:IIOO Omega5 Moves} shows all of the oriented versions of the moves $R5$ with vertices of type In-In-Out-Out, and we denote these oriented versions of the move by $\Omega5$, to be again consistent with the conventions used in~\cite{Bataineh}. Finally, in Figure~\ref{fig:IOIO Omega5 Moves} we present all oriented versions of the moves $\Omega5$ with vertices of type In-Out-In-Out.

\begin{figure}[ht]
    \[\raisebox{-13pt}{\includegraphics[height=0.45in]{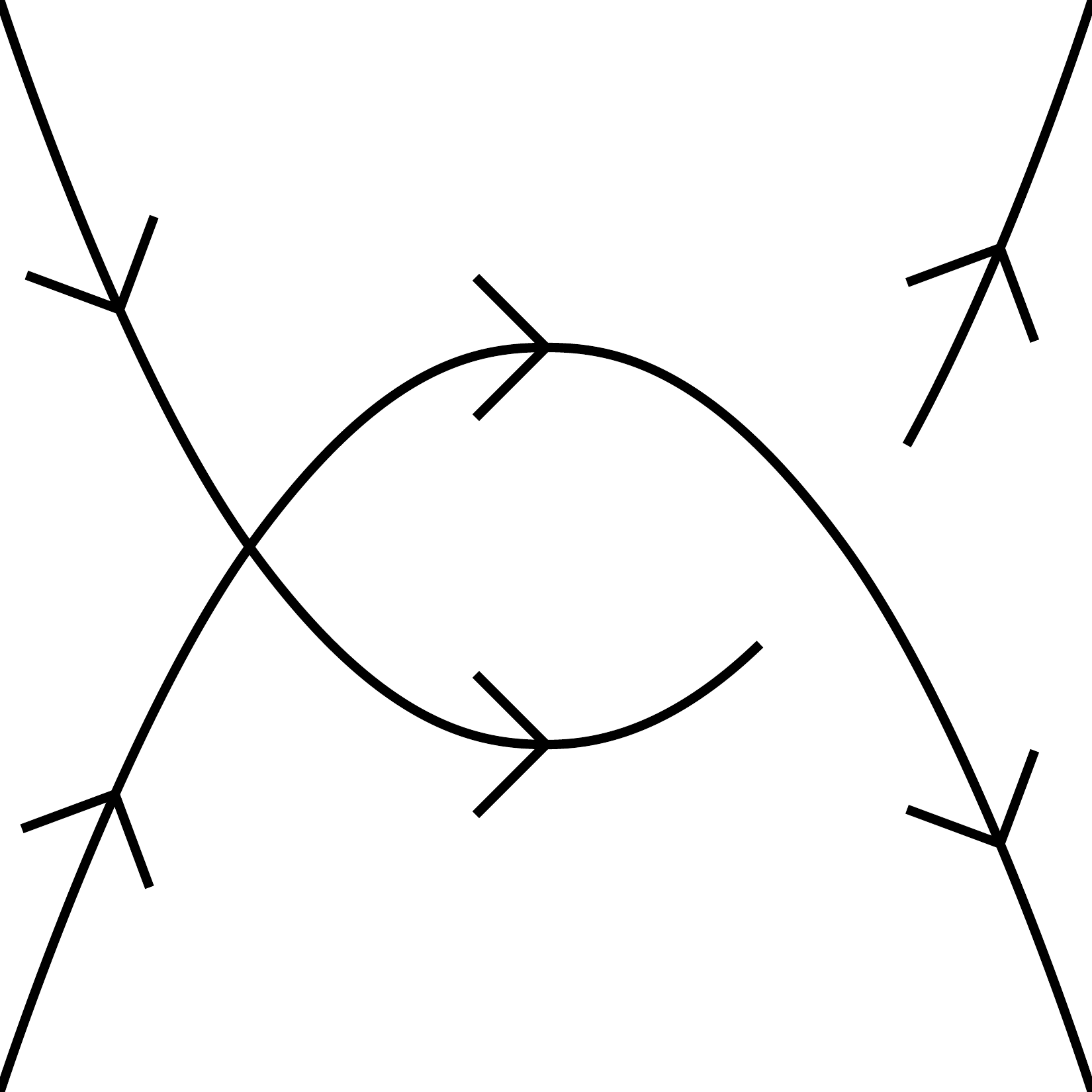}}\stackrel{\Omega 5a}{\longleftrightarrow}\raisebox{-13pt}{\includegraphics[height =0.45in]{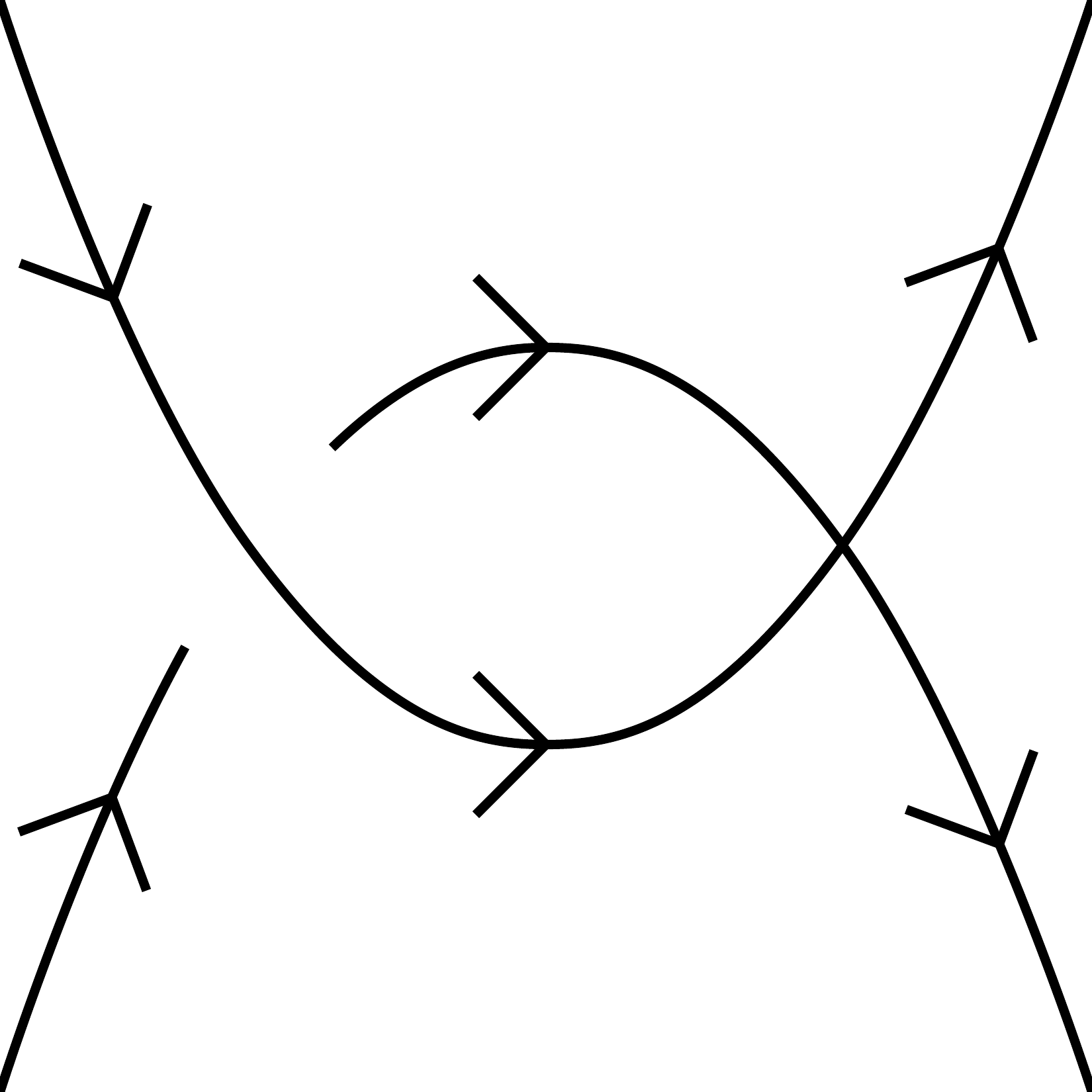}}\hspace{1.5cm}\raisebox{-13pt}{\includegraphics[height=0.45in]{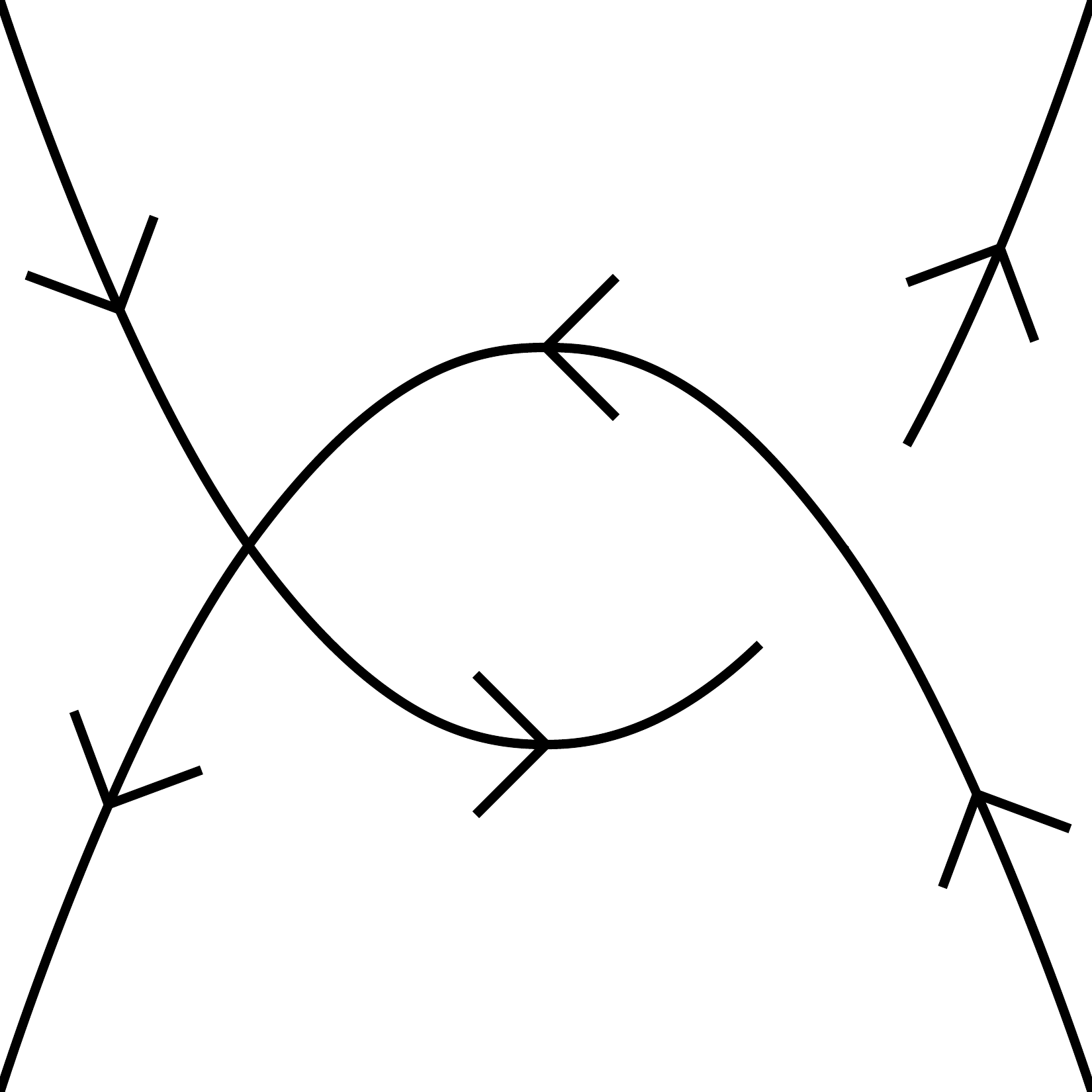}}\stackrel{\Omega 5b}{\longleftrightarrow}\raisebox{-13pt}{\includegraphics[height =0.45in]{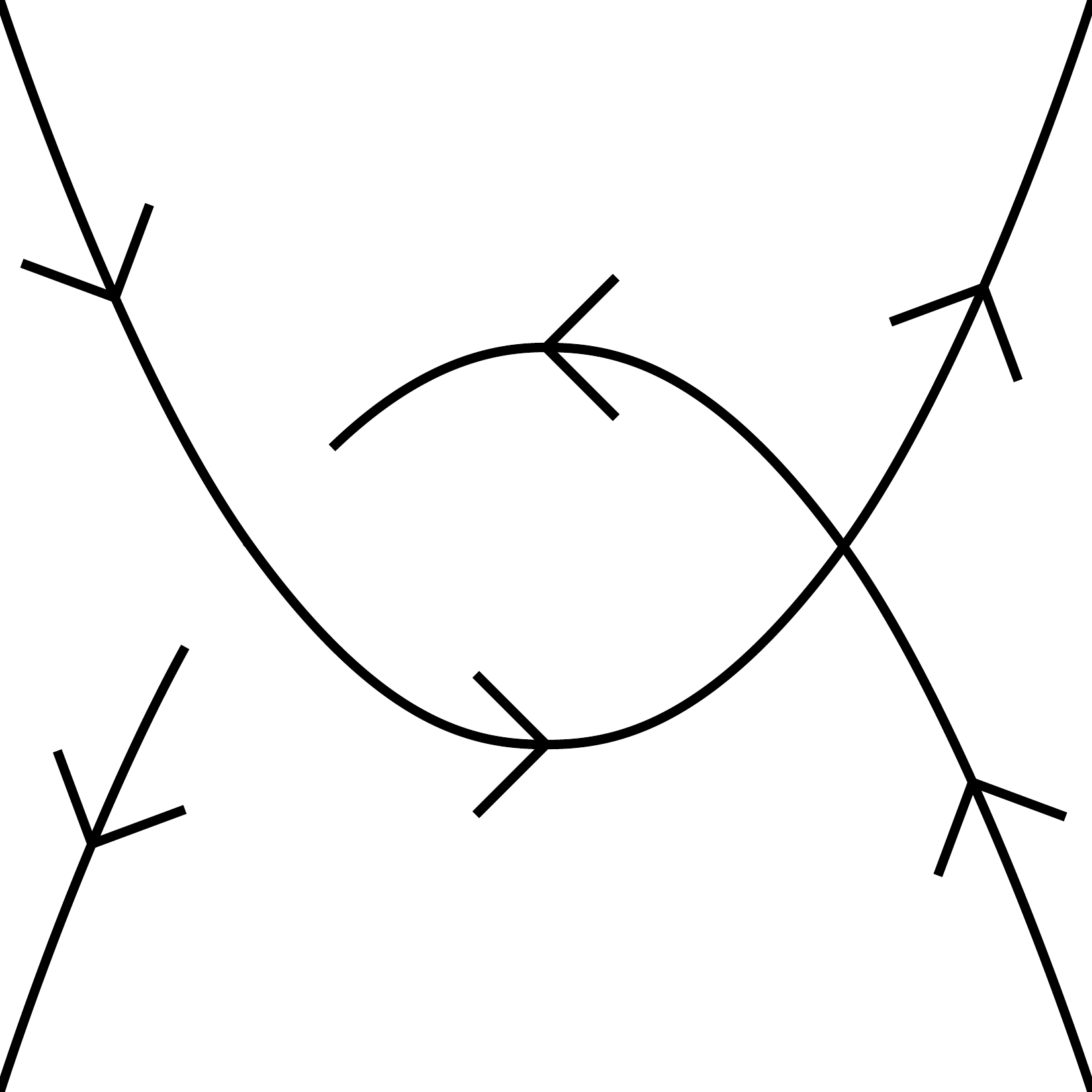}}\]
    \vspace{0.10cm}
    \[\raisebox{-13pt}{\includegraphics[height=0.45in]{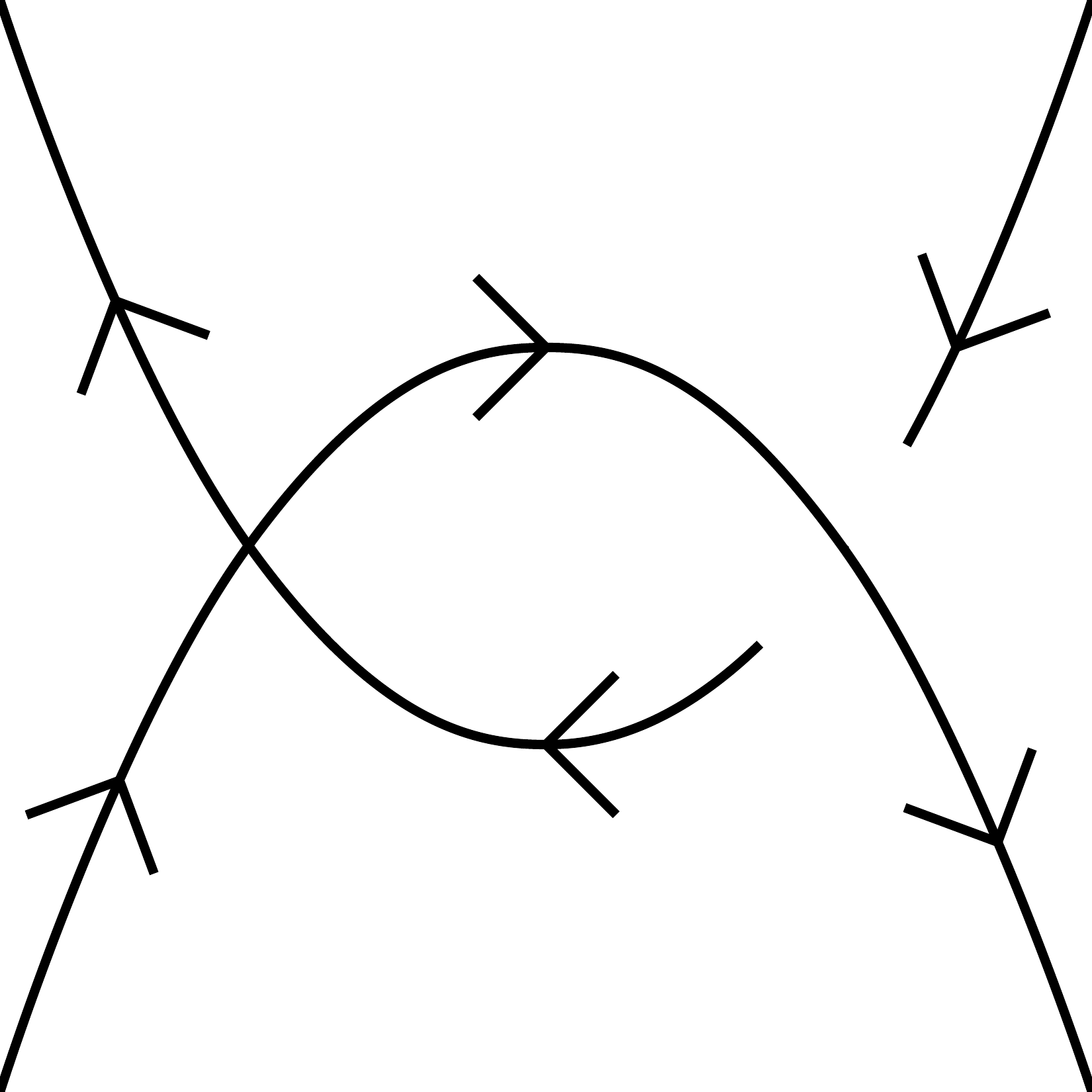}}\stackrel{\Omega 5c}{\longleftrightarrow}\raisebox{-13pt}{\includegraphics[height =0.45in]{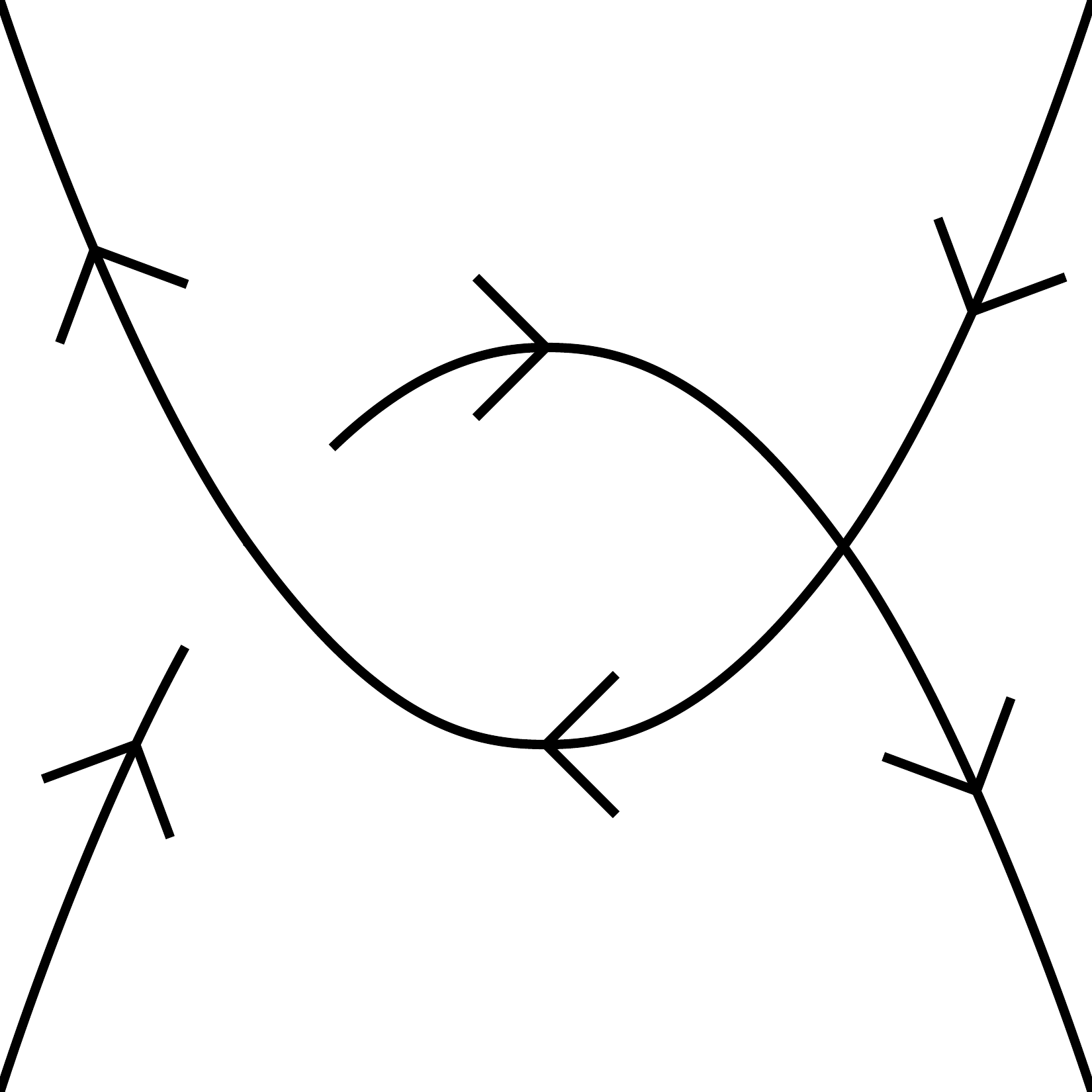}}\hspace{1.5cm}\raisebox{-13pt}{\includegraphics[height=0.45in]{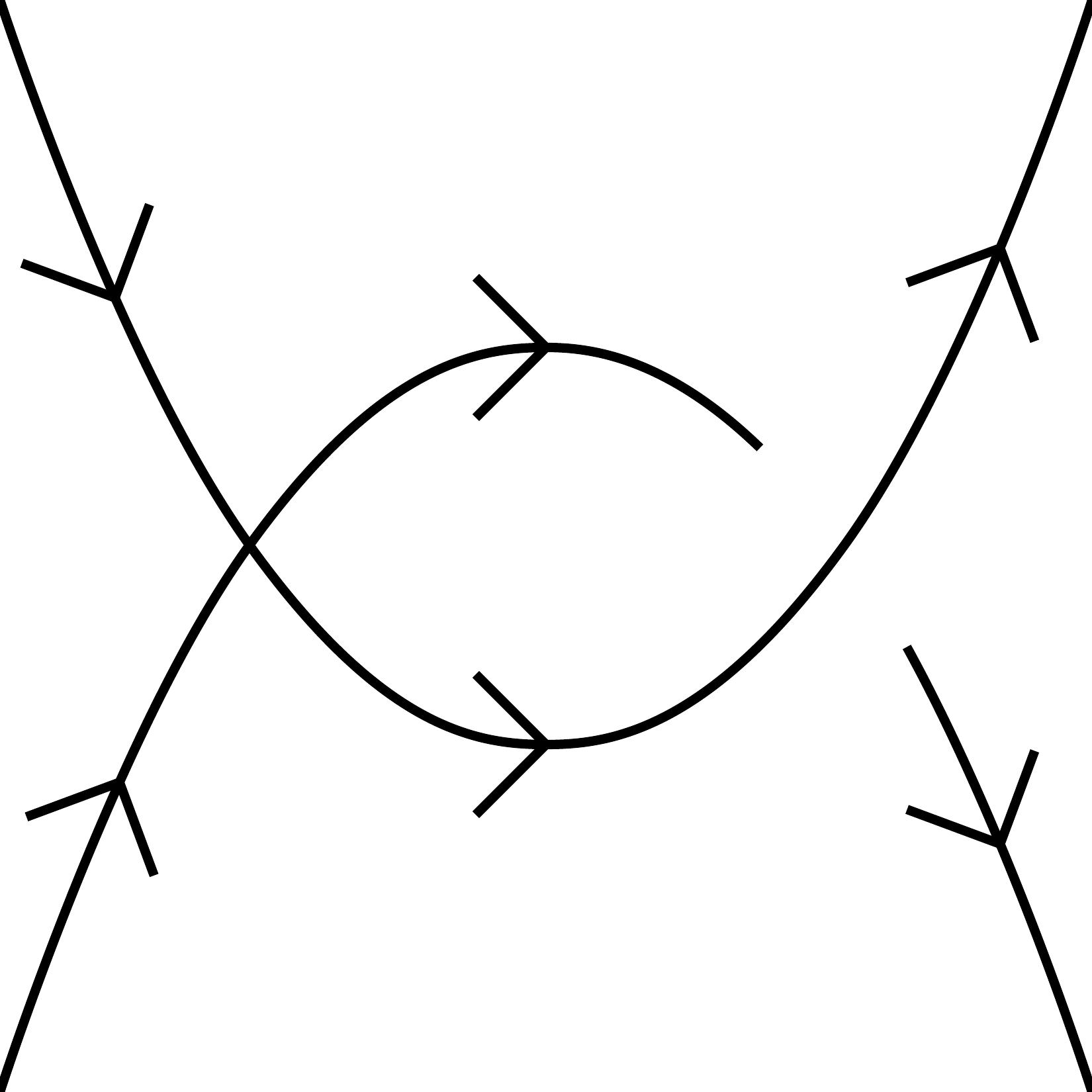}}\stackrel{\Omega 5d}{\longleftrightarrow}\raisebox{-13pt}{\includegraphics[height =0.45in]{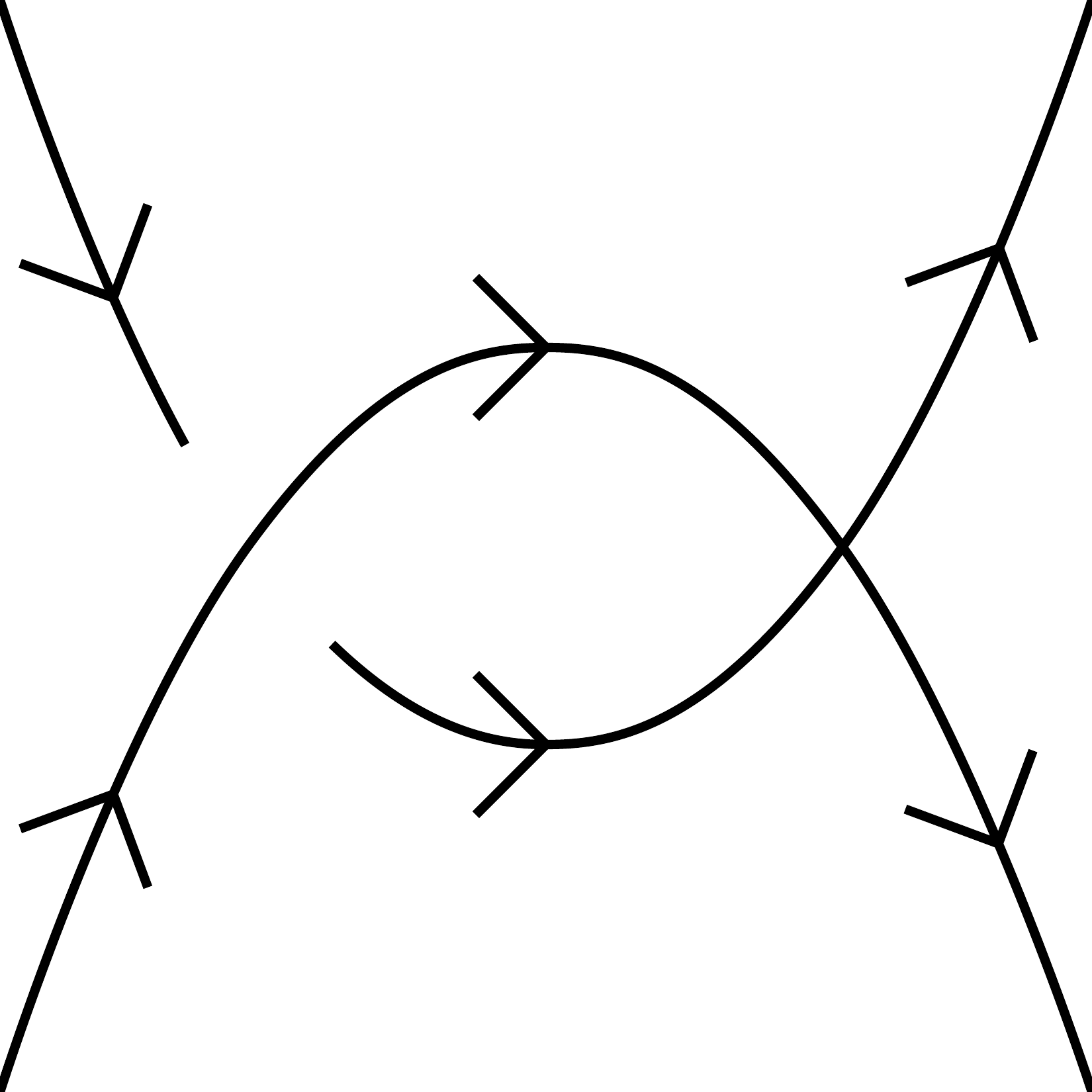}}\]
    \vspace{0.10cm}
    \[\raisebox{-13pt}{\includegraphics[height=0.45in]{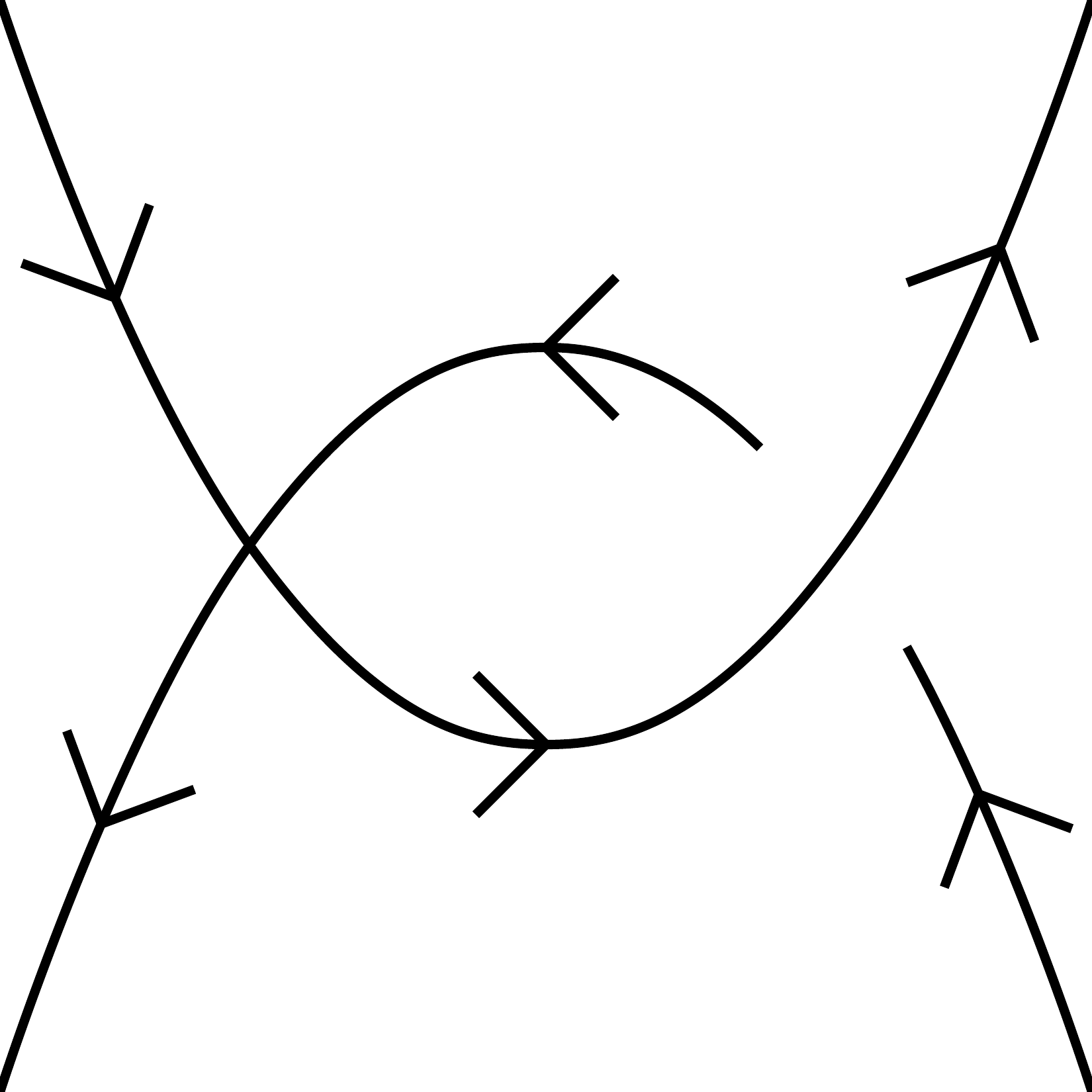}}\stackrel{\Omega 5e}{\longleftrightarrow}\raisebox{-13pt}{\includegraphics[height =0.45in]{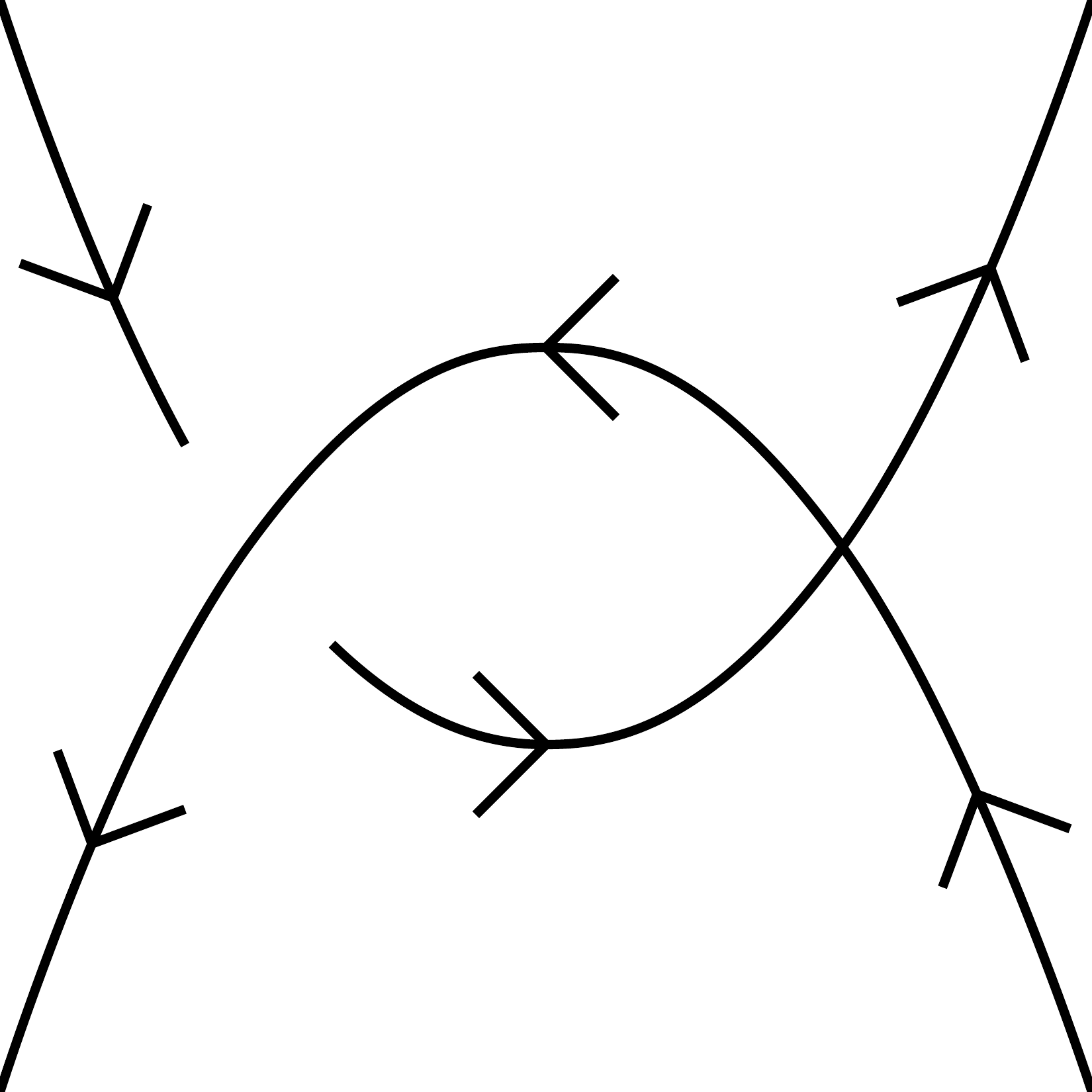}}\hspace{1.5cm}\raisebox{-13pt}{\includegraphics[height=0.45in]{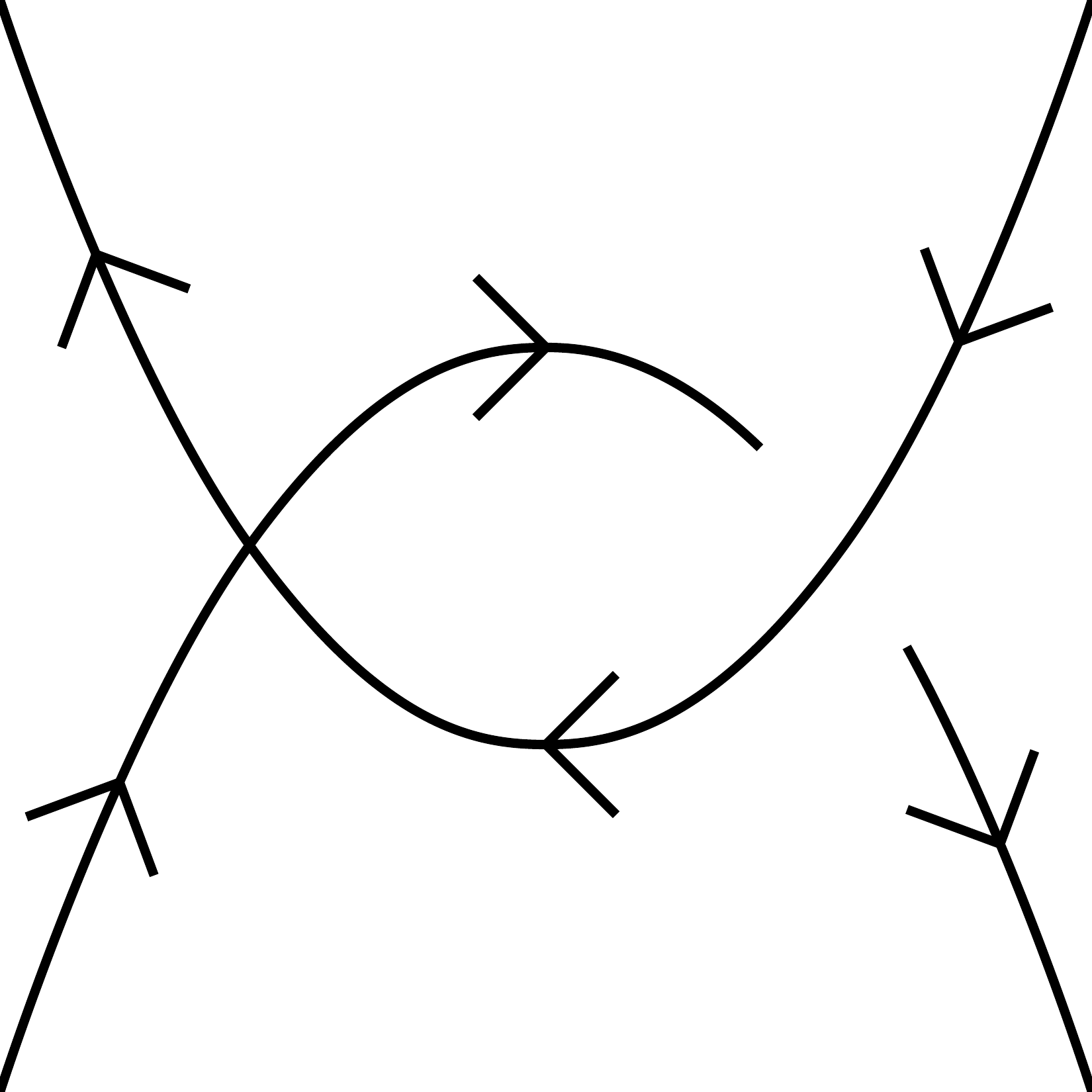}}\stackrel{\Omega 5f}{\longleftrightarrow}\raisebox{-13pt}{\includegraphics[height =0.45in]{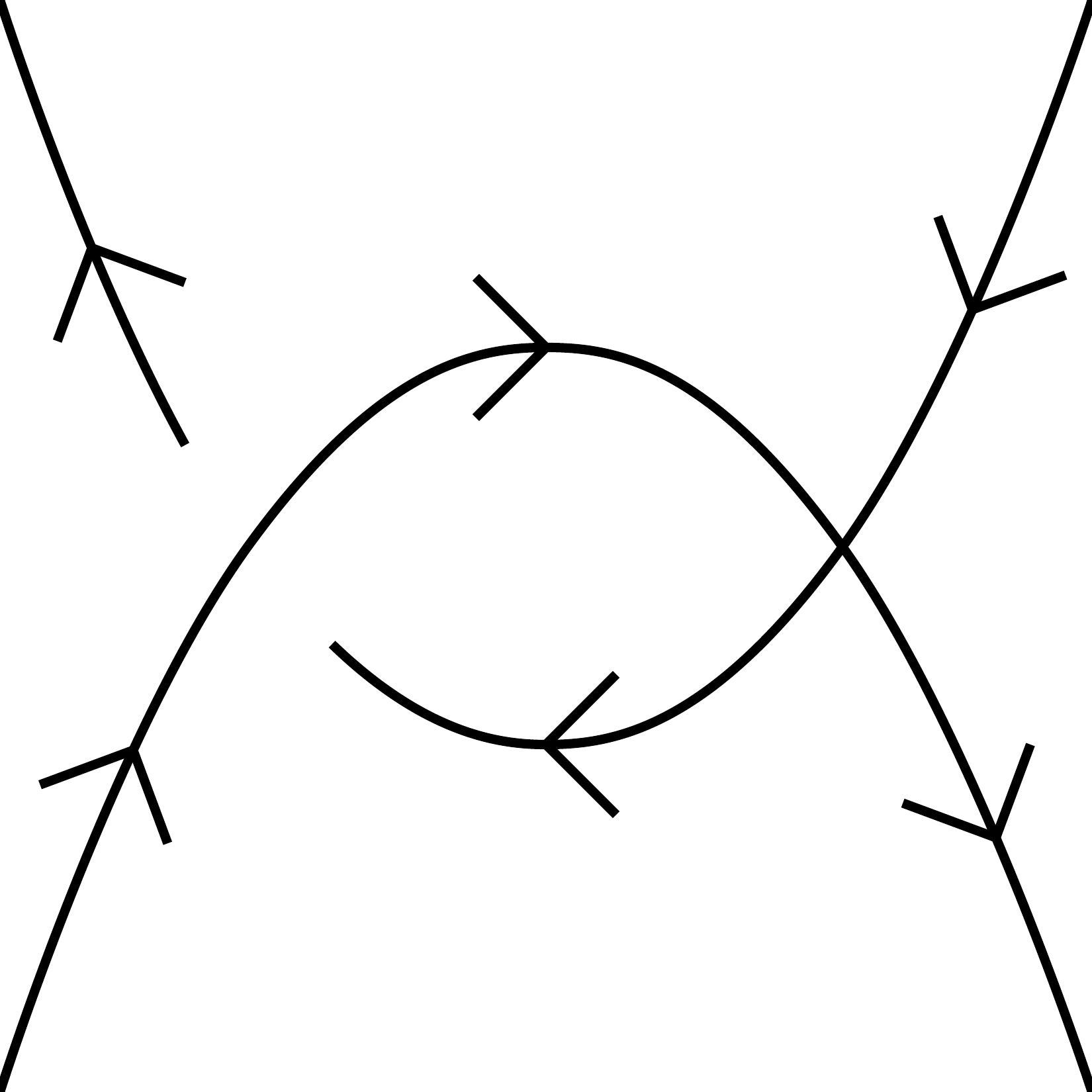}}\]
    \caption{$\Omega5$ Moves with vertices of type In-In-Out-Out}
    \label{fig:IIOO Omega5 Moves}
\end{figure}

\begin{figure}[ht]
    \[\raisebox{-13pt}{\includegraphics[height=0.45in]{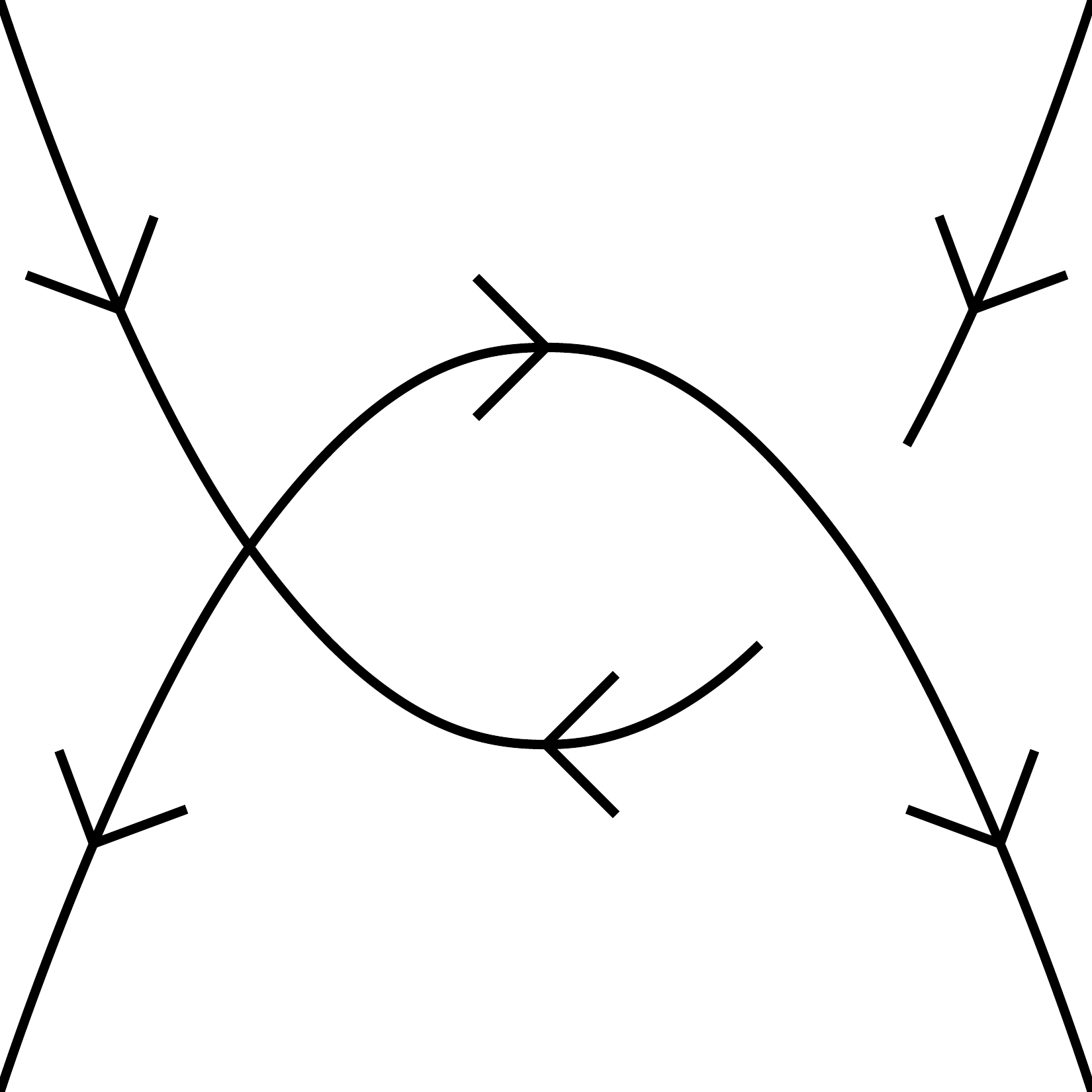}}\stackrel{\Omega 5g}{\longleftrightarrow}\raisebox{-13pt}{\includegraphics[height =0.45in]{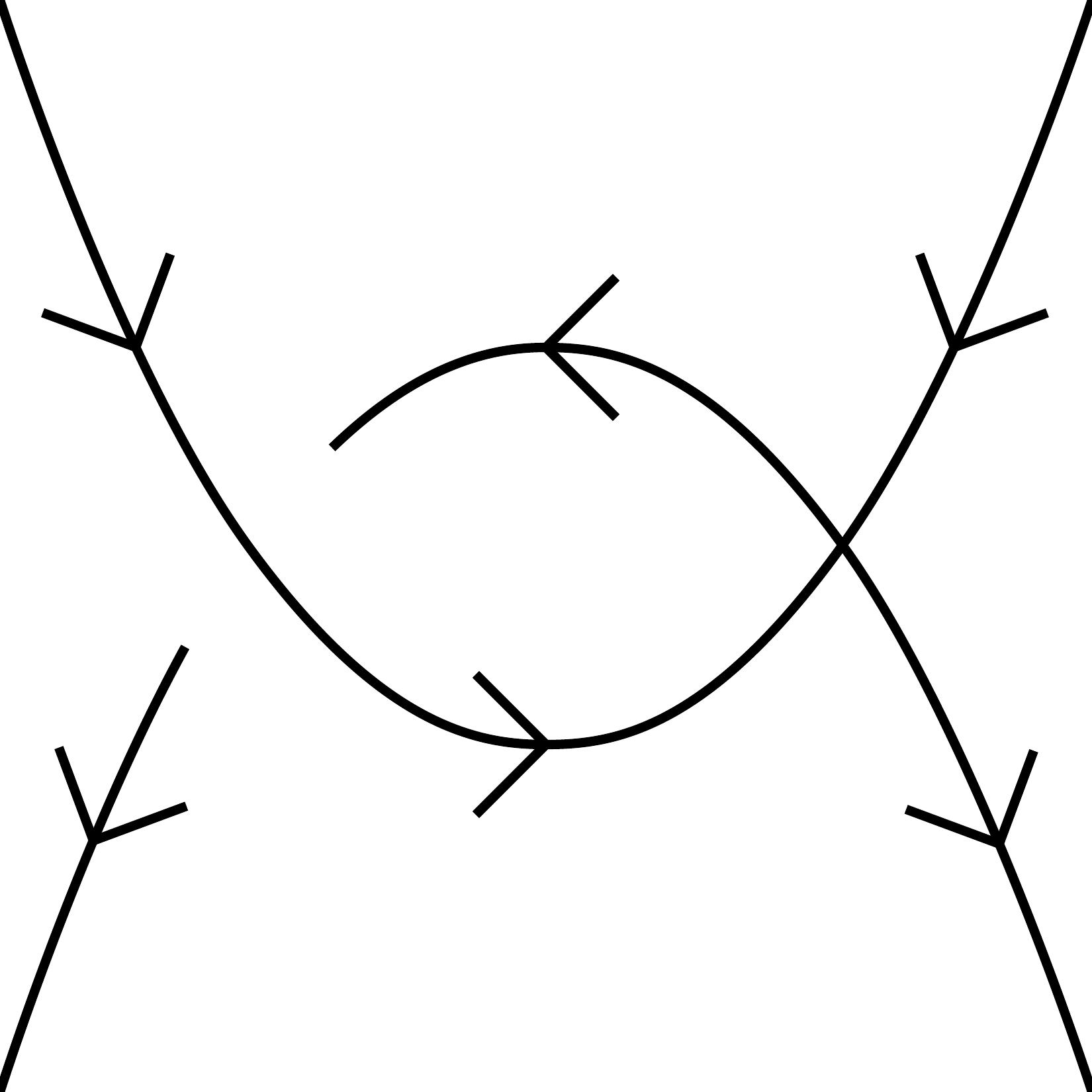}}\hspace{1.5cm}\raisebox{-13pt}{\includegraphics[height=0.45in]{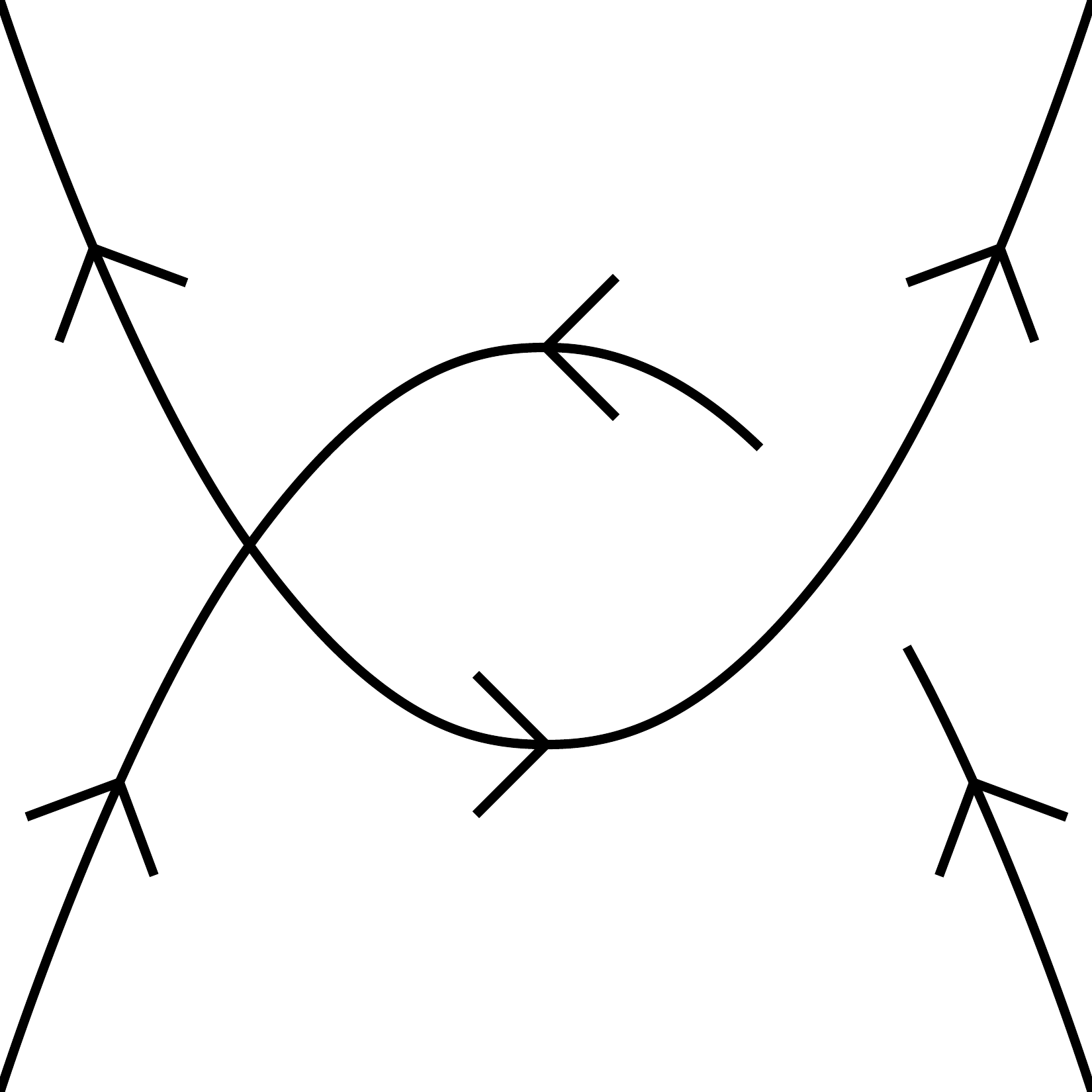}}\stackrel{\Omega 5h}{\longleftrightarrow}\raisebox{-13pt}{\includegraphics[height =0.45in]{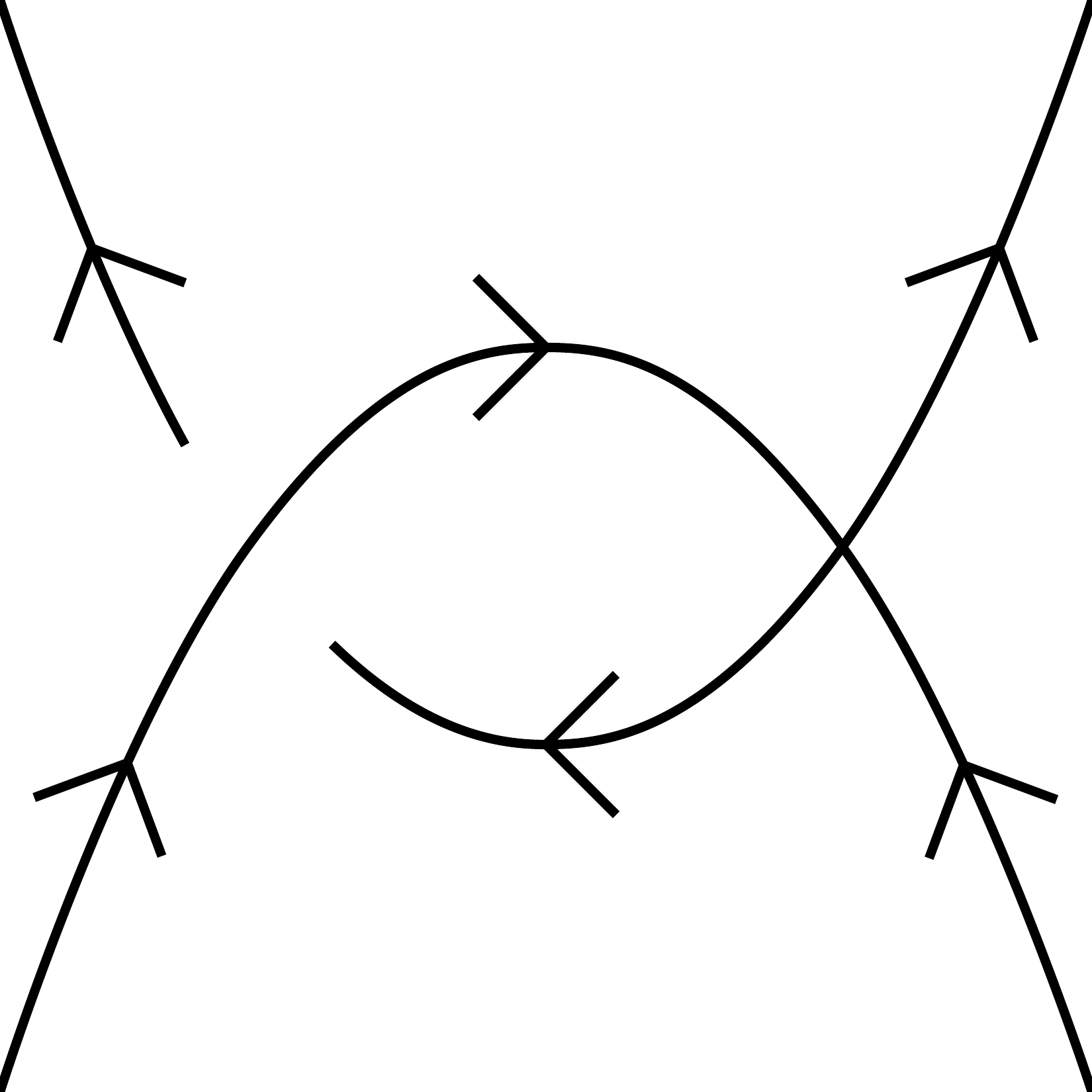}}\]
    \caption{$\Omega5$ Moves with vertices of type In-Out-In-Out}
    \label{fig:IOIO Omega5 Moves}
\end{figure}

\subsection{A Minimal Generating Set of Reidemeister-type Moves for Knotted 4-Valent Graphs}

In order to prove that $P$ is an invariant for balanced-oriented, knotted 4-valent graphs with rigid vertices, it is necessary to prove that $P$ is invariant under all oriented versions of every move for diagrams describing the rigid-vertex isotopy of balanced-oriented, knotted 4-valent graphs, which is computationally taxing. A collection of oriented Reidemeister-type moves is a generating set of all the oriented Reidemeister-type moves if every move can be obtained through a finite sequence of applications of moves in the generating set. By showing that $P$ is invariant under each move in the generating set, we know that $P$ is invariant under all of the oriented versions of the Reidemeister-type moves.

Polyak~\cite{Polyak_2010} showed that $\{\Omega1a,\Omega1b,\Omega2a,\Omega3a\}$ is a minimal generating set for all the oriented versions of the Reidemeister moves $\Omega1,\Omega2,$ and $\Omega3$ for oriented knot diagrams.
Building on Polyak's work, Caprau and Scott~\cite{CaprauScott} found all minimal generating sets for the moves $\Omega1,\Omega2,$ and $\Omega3$. For our purposes, we will use Polyak's minimal generating set of moves for oriented knot diagrams.

Furthermore, Bataineh et al.~\cite{Bataineh} showed that the moves $\Omega4a,\Omega4e$, and $\Omega5a$ generate all of the oriented $\Omega4$ and $\Omega5$ moves involving In-In-Out-Out vertex types listed above. 

Similar to Bataineh et al.~\cite{Bataineh}, we show that only two $\Omega4$ moves and one $\Omega5$ move, namely $\Omega4j$, $\Omega4l$, and $\Omega5g$, are needed to generate all oriented $\Omega4$ and $\Omega5$ moves involving vertices of type In-Out-In-Out. 

\begin{lemma}
    The move $\Omega4i$ can be realized by a sequence of the moves $\Omega2d,\Omega4j,$ and $\Omega2b$.      
\end{lemma}

\begin{proof}
    \[\adjustbox{scale=0.95}{\begin{tikzcd}
    	\raisebox{-17pt}{\rotatebox{90}{\includegraphics[width=0.5in]{Generating-Sets/O4i-1.pdf}}} & \raisebox{-17pt}{\includegraphics[width=0.5in]{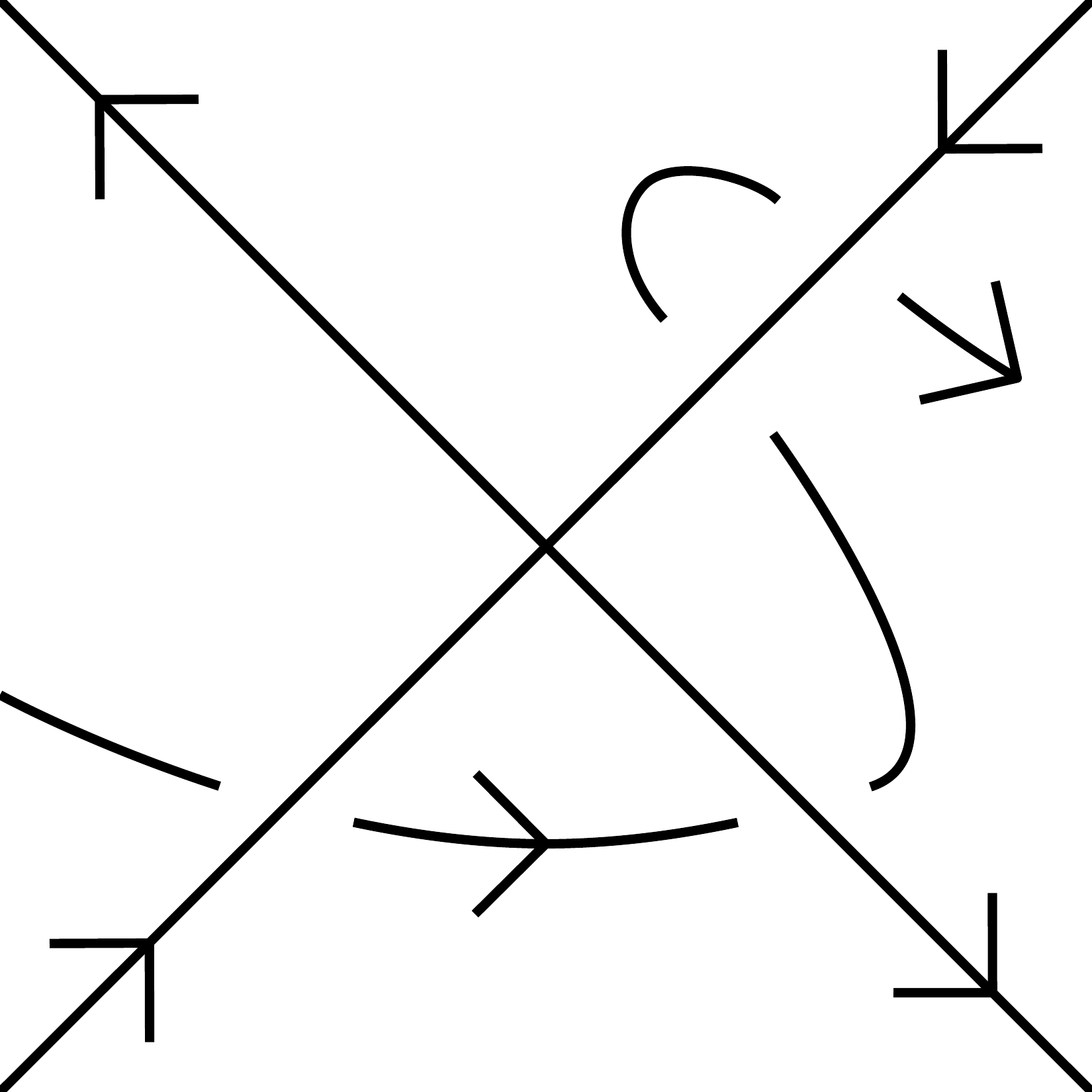}} & \raisebox{-17pt}{\includegraphics[width=0.5in]{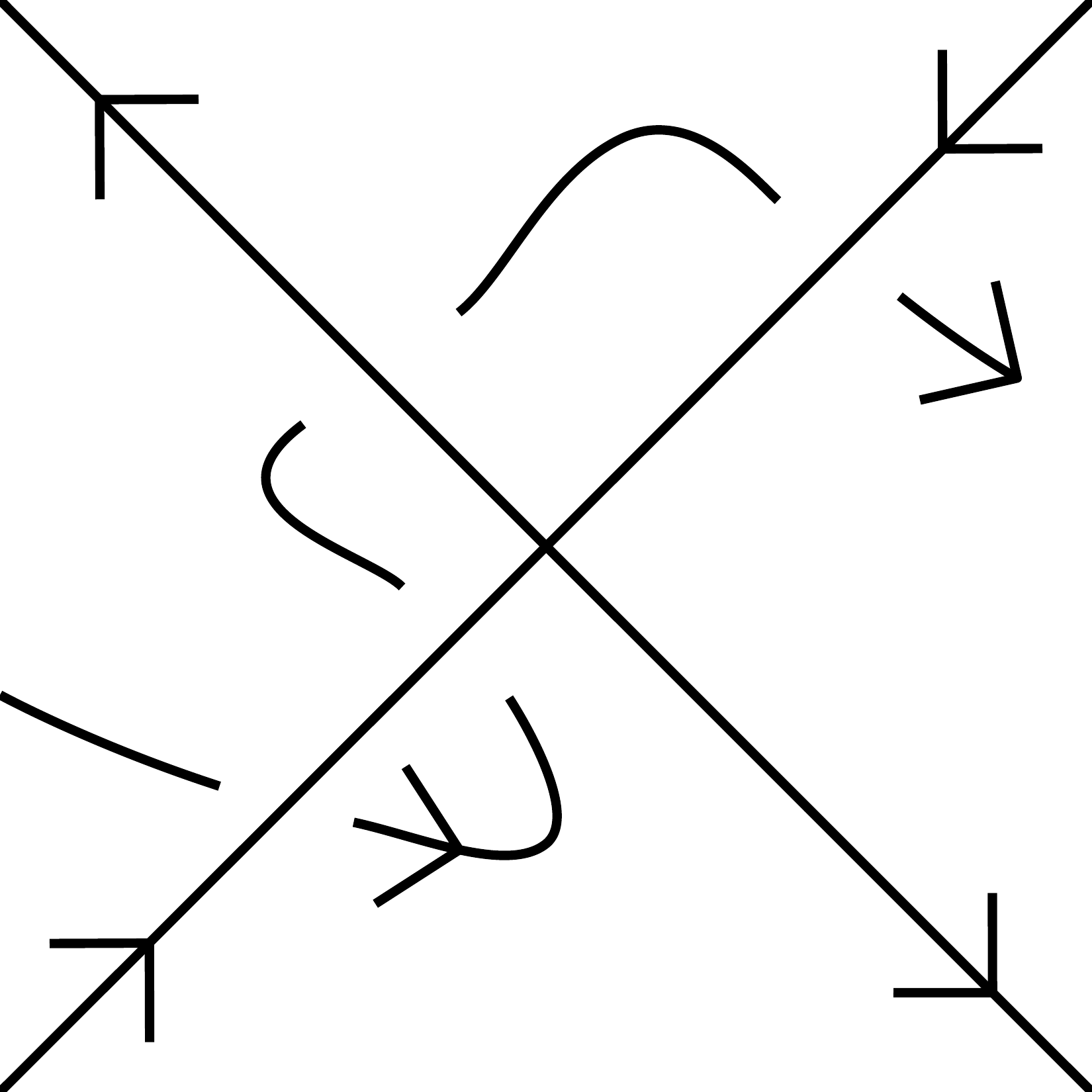}} & \raisebox{-17pt}{\rotatebox{90}{\includegraphics[width=0.5in]{Generating-Sets/O4i-2.pdf}}}
    	\arrow["\Omega2d", from=1-1, to=1-2]
    	\arrow["\Omega4j", from=1-2, to=1-3]
    	\arrow["\Omega2b", from=1-3, to=1-4]
    \end{tikzcd}}\]
\end{proof}

\begin{lemma}
    The move $\Omega4k$ can be realized by a sequence of the moves $\Omega2d,\Omega4l,$ and $\Omega2a$.      
\end{lemma}

\begin{proof}
    \[\adjustbox{scale=0.95}{\begin{tikzcd}
    	\raisebox{-17pt}{\rotatebox{90}{\includegraphics[width=0.5in]{Generating-Sets/O4k-1.pdf}}} & \raisebox{-17pt}{\includegraphics[width=0.5in]{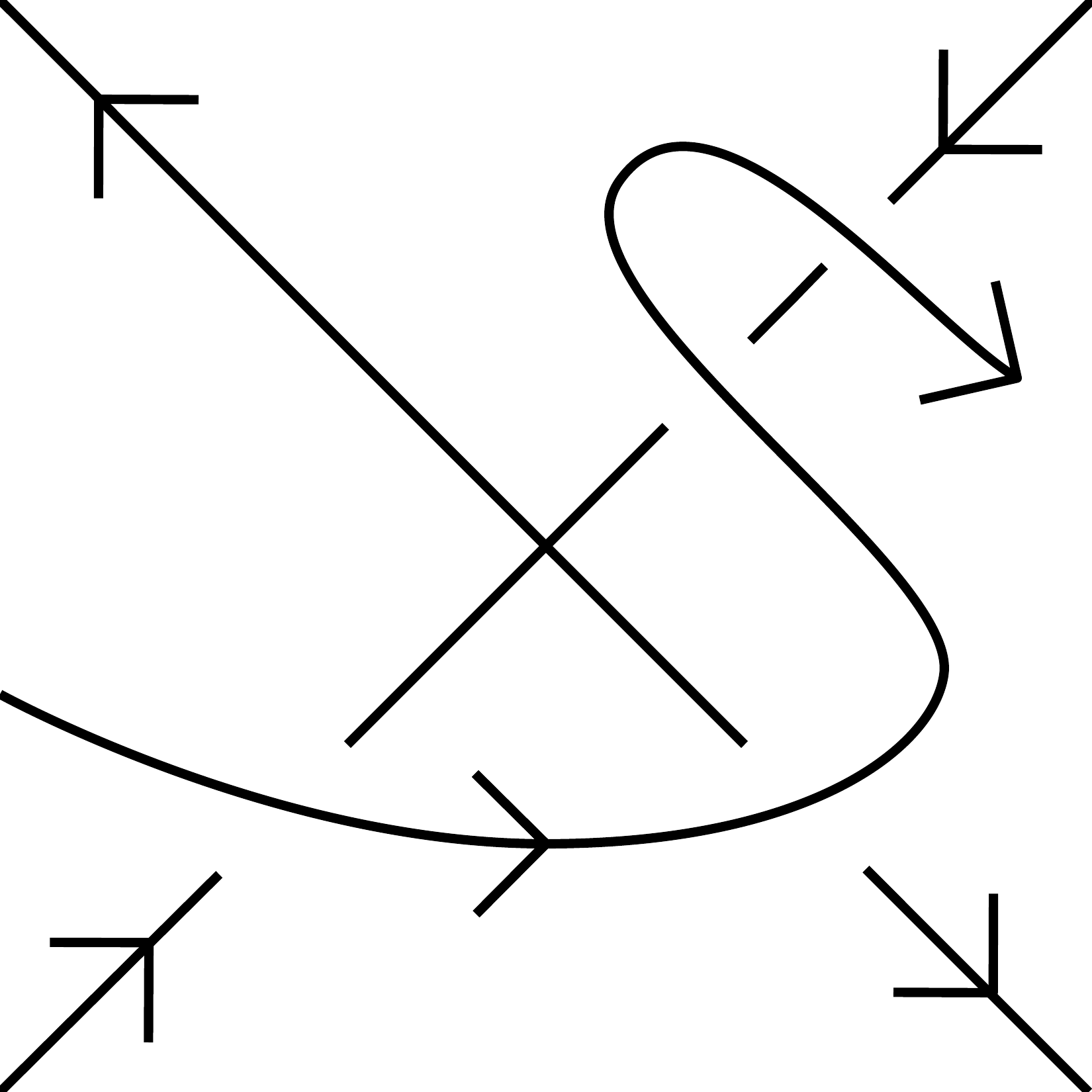}} & \raisebox{-17pt}{\includegraphics[width=0.5in]{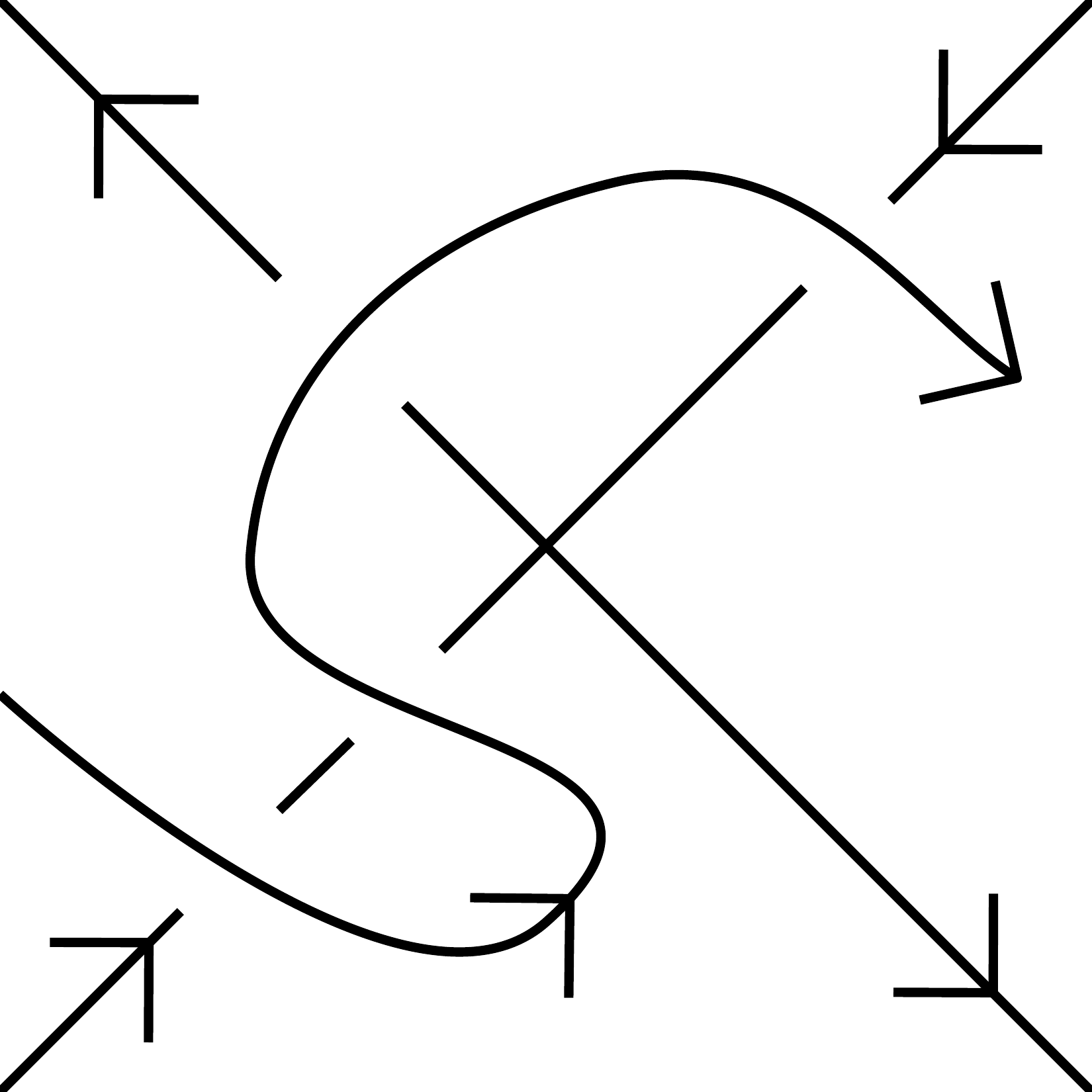}} & \raisebox{-17pt}{\rotatebox{90}{\includegraphics[width=0.5in]{Generating-Sets/O4k-2.pdf}}}
    	\arrow["\Omega2d", from=1-1, to=1-2]
    	\arrow["\Omega4l", from=1-2, to=1-3]
    	\arrow["\Omega2a", from=1-3, to=1-4]
    \end{tikzcd}}\]
\end{proof}

\begin{lemma}
    The move $\Omega5h$ can be realized by a sequence of $\Omega1c,\Omega4j,\Omega5g,\Omega4k,$ and $\Omega1d$ moves.      
\end{lemma}

\begin{proof}
    \[\adjustbox{scale=0.95}{\begin{tikzcd}
    	\raisebox{-17pt}{\includegraphics[width=0.5in]{Generating-Sets/O5h-1.pdf}} & \raisebox{-17pt}{\includegraphics[width=0.5in]{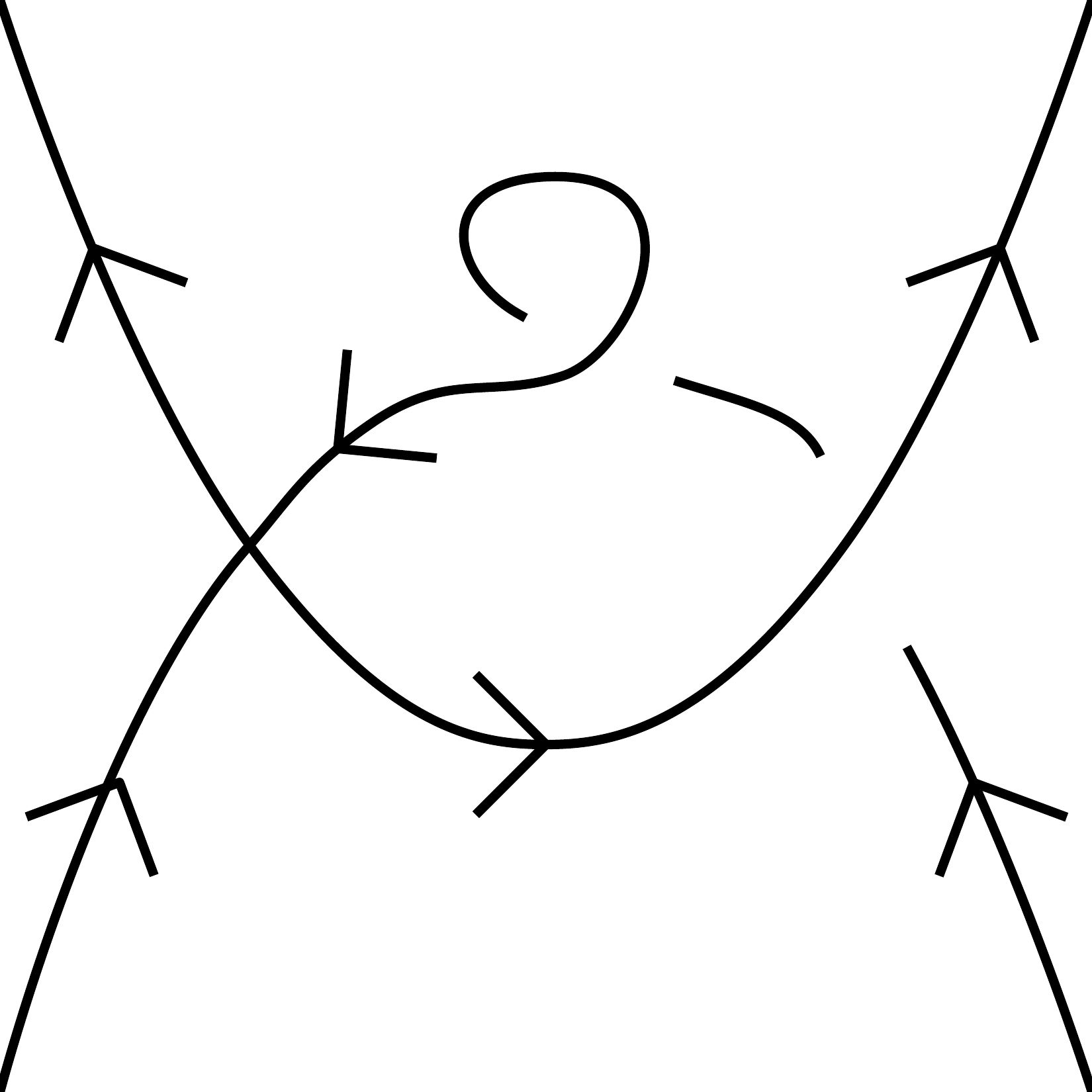}} & \raisebox{-17pt}{\includegraphics[width=0.5in]{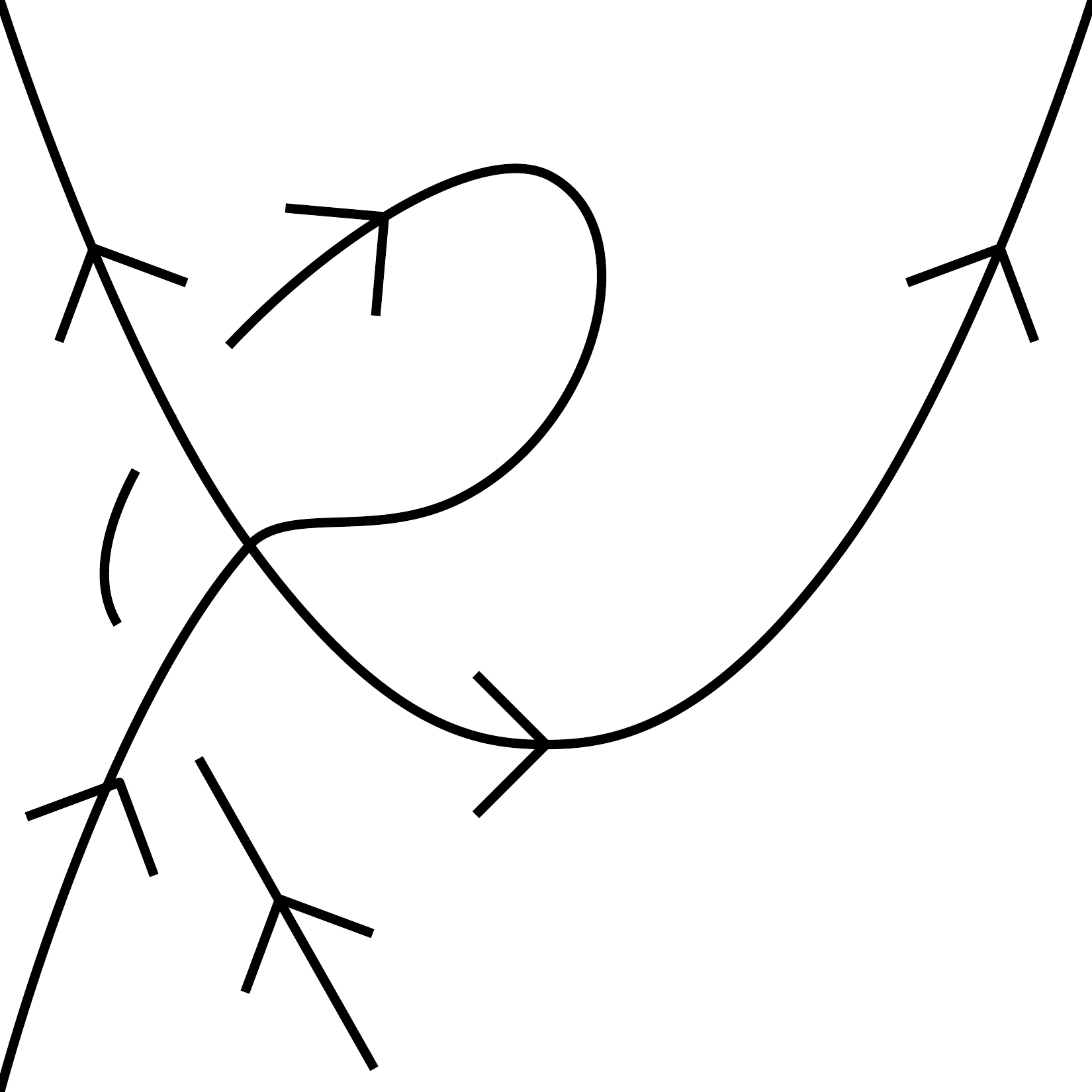}} & \raisebox{-17pt}{\includegraphics[width=0.5in]{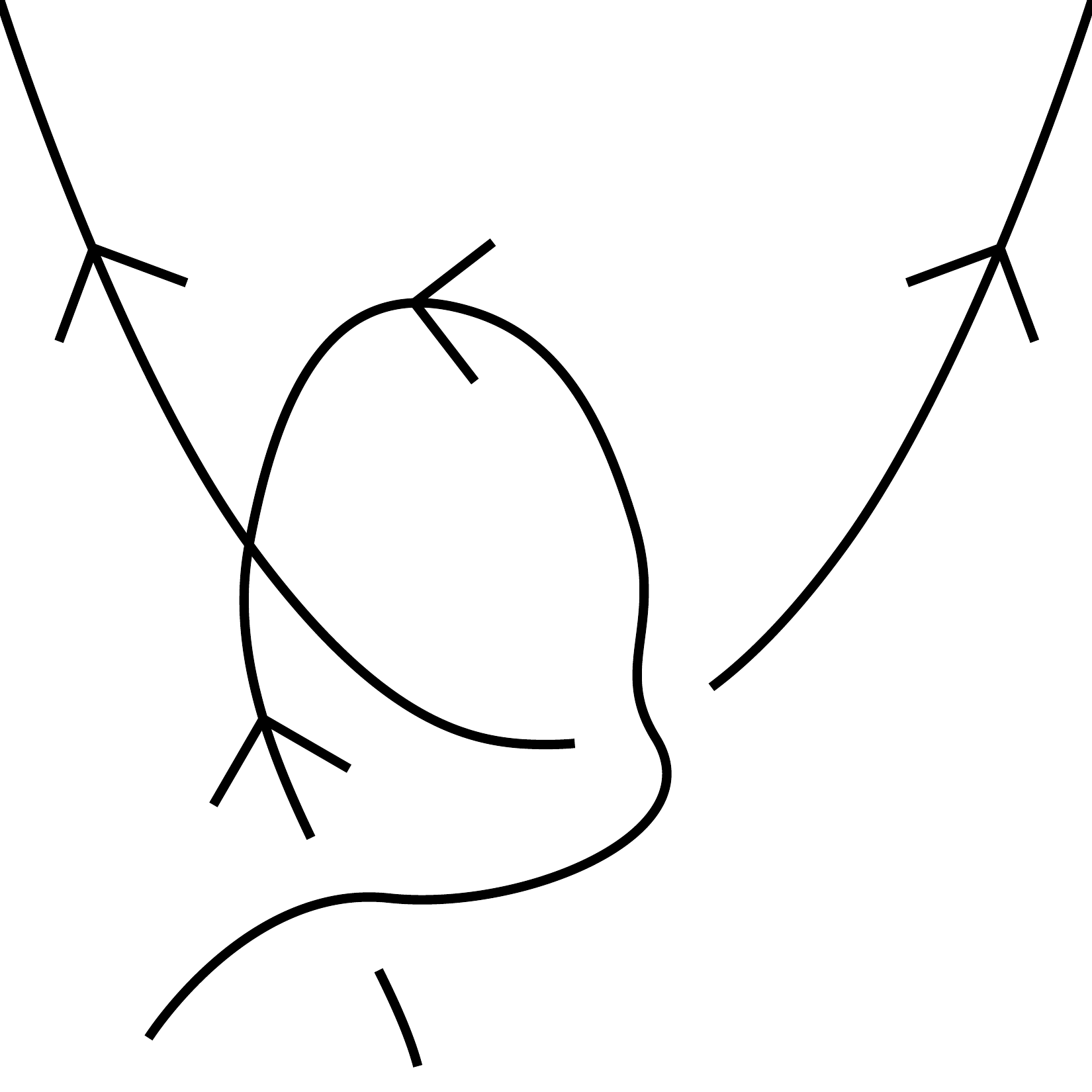}} \\
    	&& \raisebox{-17pt}{\includegraphics[width=0.5in]{Generating-Sets/O5h-2.pdf}} &
        \raisebox{-17pt}{\includegraphics[width=0.5in]{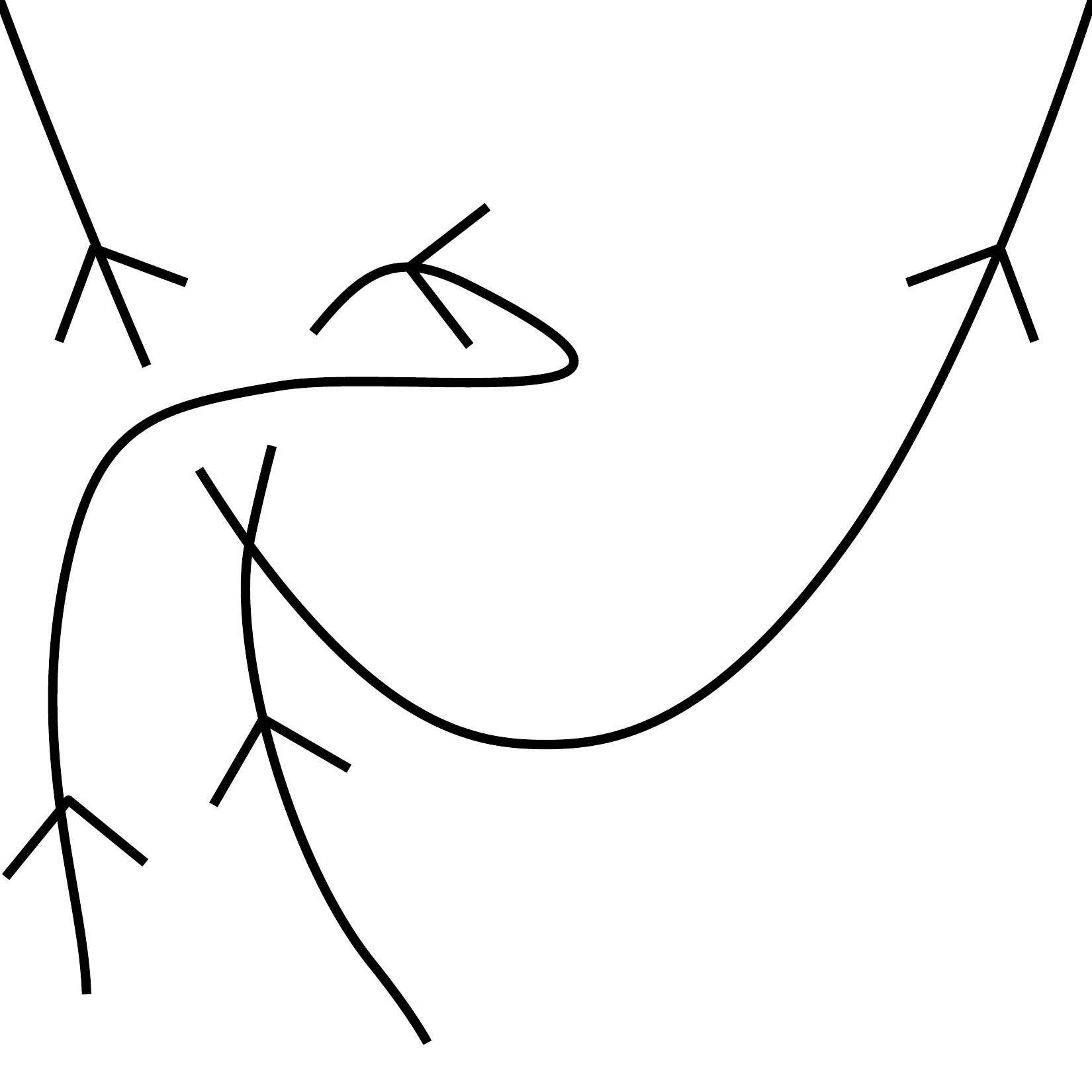}}
    	\arrow["\Omega1c", from=1-1, to=1-2]
    	\arrow["\Omega4j", from=1-2, to=1-3]
    	\arrow["\Omega5g", from=1-3, to=1-4]
    	\arrow["\Omega4k", from=1-4, to=2-4]
    	\arrow["\Omega1d", from=2-4, to=2-3]
    \end{tikzcd}}\]
\end{proof}

 Thus, the set $\{\Omega1a,\Omega1b,\Omega2a,\Omega3a,\Omega4a,\Omega4e,\Omega5a,\Omega4j,\Omega4l,\Omega5g\}$ is a generating set of Reidemeister-type moves for diagrams of knotted, balanced-oriented, 4-valent graphs with rigid vertices.

 \begin{remark}
    Note that since the move $\Omega4j$ involves sliding a strand under a vertex while the move $\Omega4l$ involves sliding a strand over a vertex, there is no way to derive one move from the other. Additionally, we know it is necessary to have at least one $\Omega5$ move for the In-Out-In-Out vertices. Hence, we conclude that our generating set is minimal for In-Out-In-Out vertices.   
 \end{remark}


\begin{theorem}\label{min-gen-set}
    The set $\{\Omega1a,\Omega1b,\Omega2a,\Omega3a,\Omega4a,\Omega4e,\Omega5a,\Omega4j,\Omega4l,\Omega5g\}$ is a minimal generating set of Reidemeister-type moves for diagrams of balanced-oriented, knotted 4-valent graphs with rigid vertices.
\end{theorem}

\begin{proof}
The proof follows from results in~\cite{Bataineh, Polyak_2010} together with our Lemmas 3.1--3.3 and Remark 3.4.
\end{proof}

\section{The Proof of Invariance of $P$} \label{sec:proofinv}

We are ready to prove that the polynomial $P$ is a rigid-vertex isotopy invariant for balanced-oriented, 4-valent knotted graphs. For this, we need to prove that $P$ is independent of the diagram $D$ representing a  knotted graph with balanced-oriented 4-valent right vertices.
Using Theorem~\ref{min-gen-set}, it suffices to show that $P$ is invariant under the moves in our minimal generating set of Reidemeister-type moves for 4-valent knotted graphs with balanced-oriented vertices.

We show first that $P$ is invariant under the move $\Omega 1a$:
\begin{align*}
    P\Biggl(\raisebox{-10pt}{\includegraphics[height = .35in]{Generating-Sets/O1a-1.pdf}}\Biggr)&= q^{n-1}P\Biggl(\raisebox{-10pt}{\includegraphics[height = .35in]{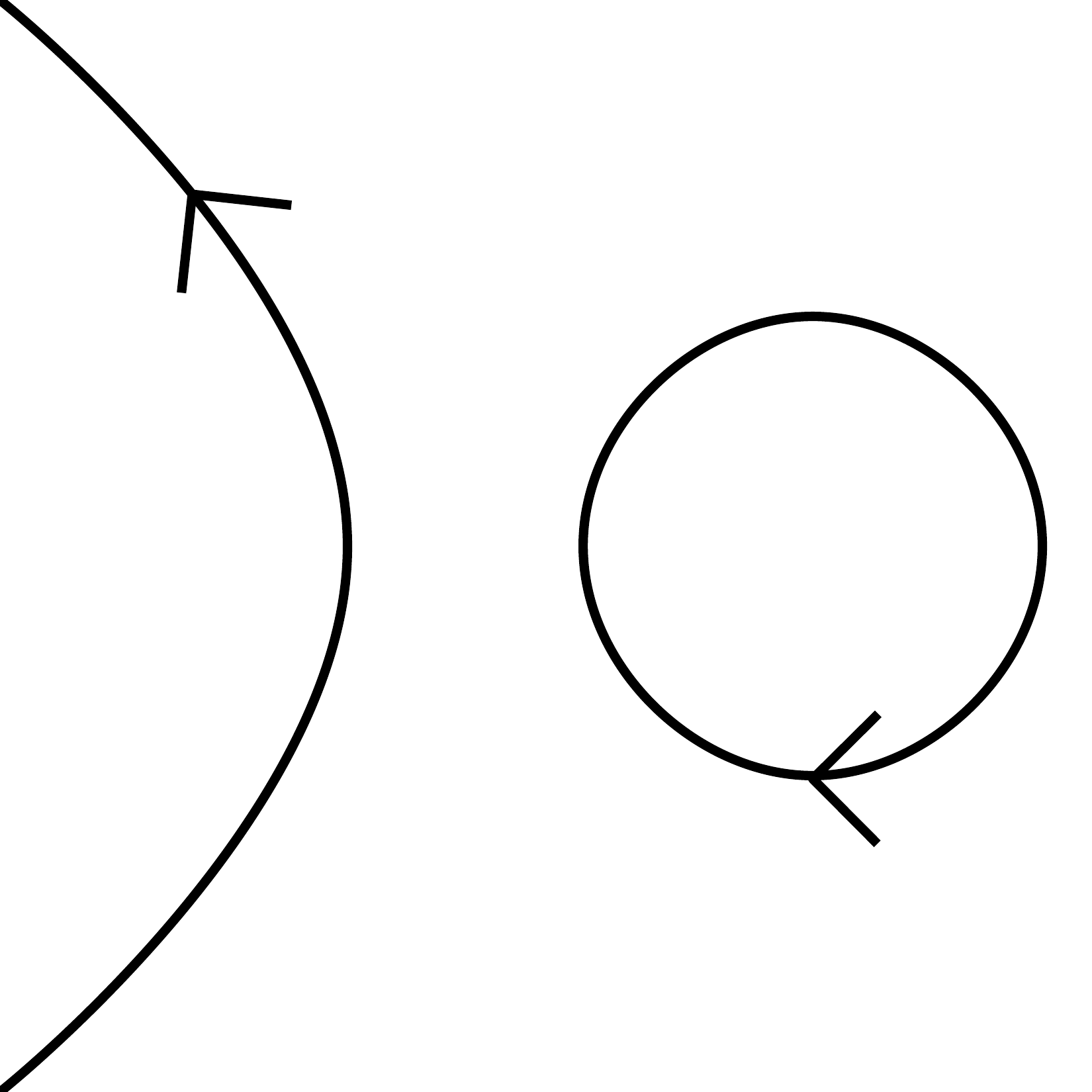}}\Biggr)-q^{n}P\Biggl(\raisebox{-10pt}{\includegraphics[height = .35in]{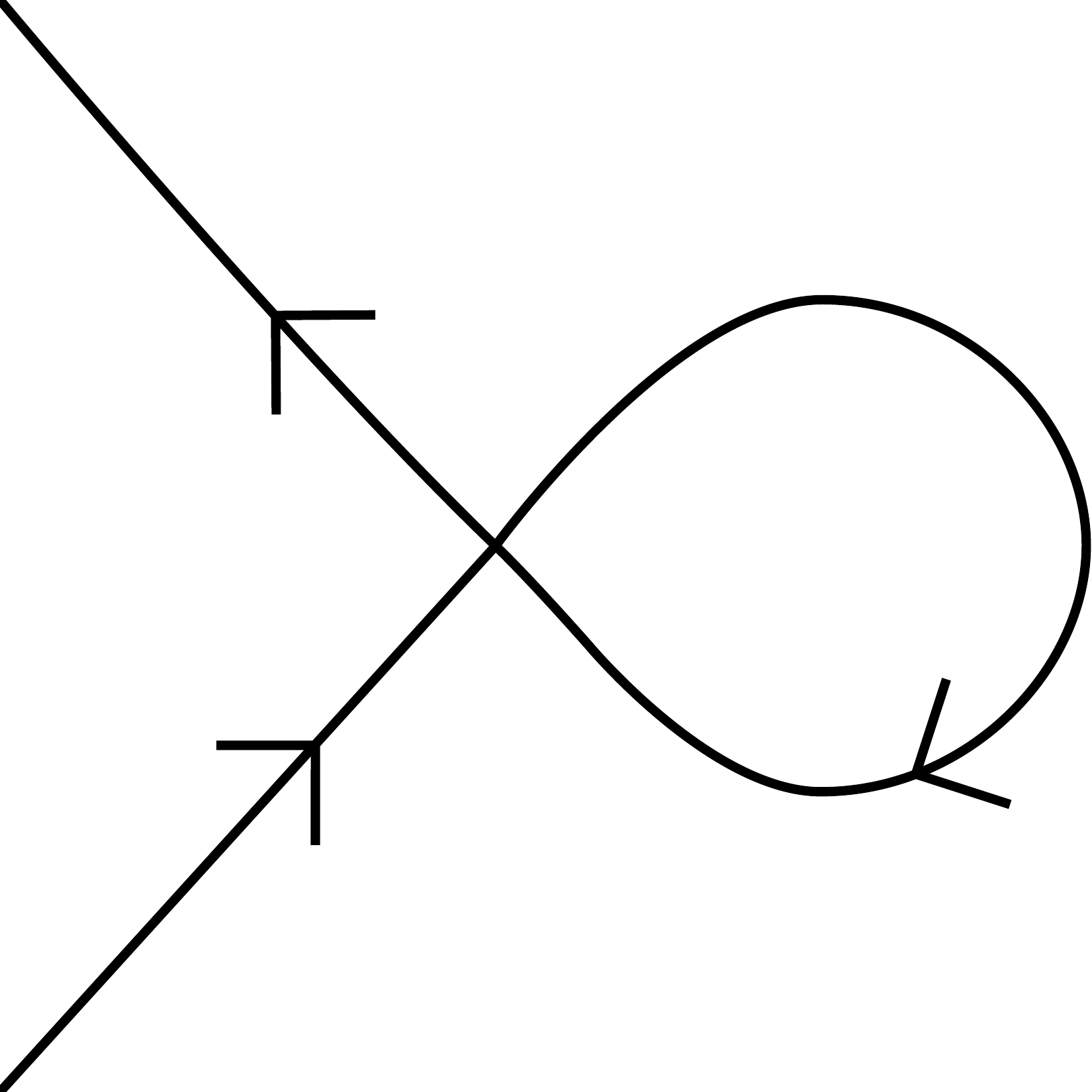}}\Biggr)\\
    &=q^{n-1}[n]P\Biggl(\raisebox{-10pt}{\includegraphics[height = .35in]{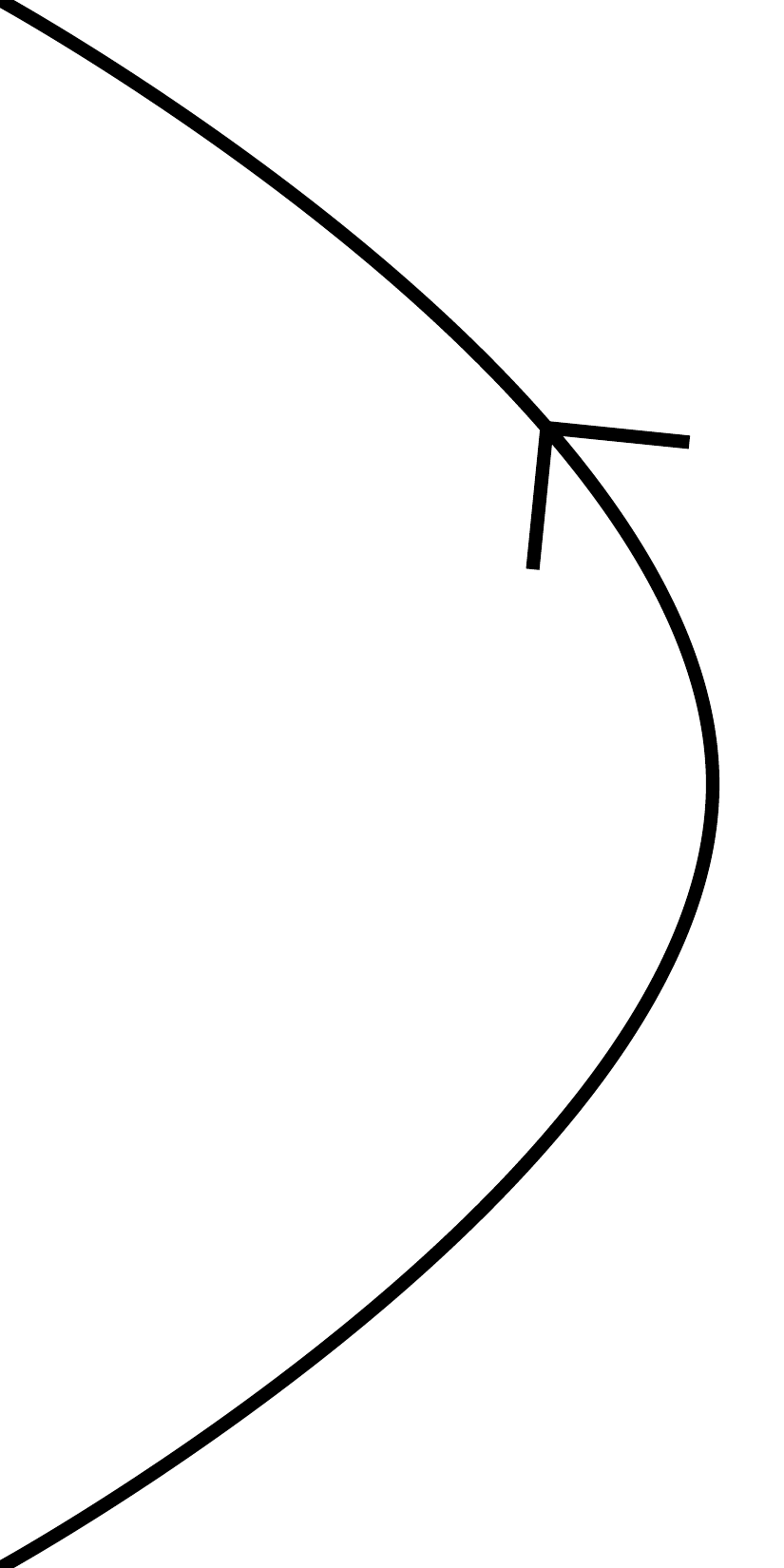}}\Biggr)-q^{n}[n-1]P\Biggl(\raisebox{-10pt}{\includegraphics[height = .35in]{O1ab-Proof/oriented-strand-up-half.pdf}}\Biggr)\\
    &=(q^{n-1}[n] - q^n[n-1])P\Biggl(\raisebox{-10pt}{\includegraphics[height = .35in]{O1ab-Proof/oriented-strand-up-half.pdf}}\Biggr)\\ 
    &=P\Biggl(\raisebox{-10pt}{\includegraphics[height = .35in]{O1ab-Proof/oriented-strand-up-half.pdf}}\Biggr).
    \end{align*}
Above, we applied the skein relation to resolve the positive crossing, followed by the application of the graphical relation involving a loop. We also used that
\[
q^{n-1}[n] - q^n[n-1] 
= \frac{q^{n-1}(q^n-q^{-n})}{q-q^{-1}} - \frac{q^{n}(q^{n-1}-q^{1-n})}{q-q^{-1}}=1.
\]  
We now show that $P$ is invariant under the move $\Omega 1b$. By applying the skein relation for the crossing followed by the graphical relation to resolve the loop, we arrive at the following:
\begin{align*}
    P\Biggl(\raisebox{-10pt}{\includegraphics[height = .35in]{Generating-Sets/O1b-1.pdf}}\Biggr)&= q^{n-1}P\Biggl(\raisebox{-10pt}{\includegraphics[height = .35in]{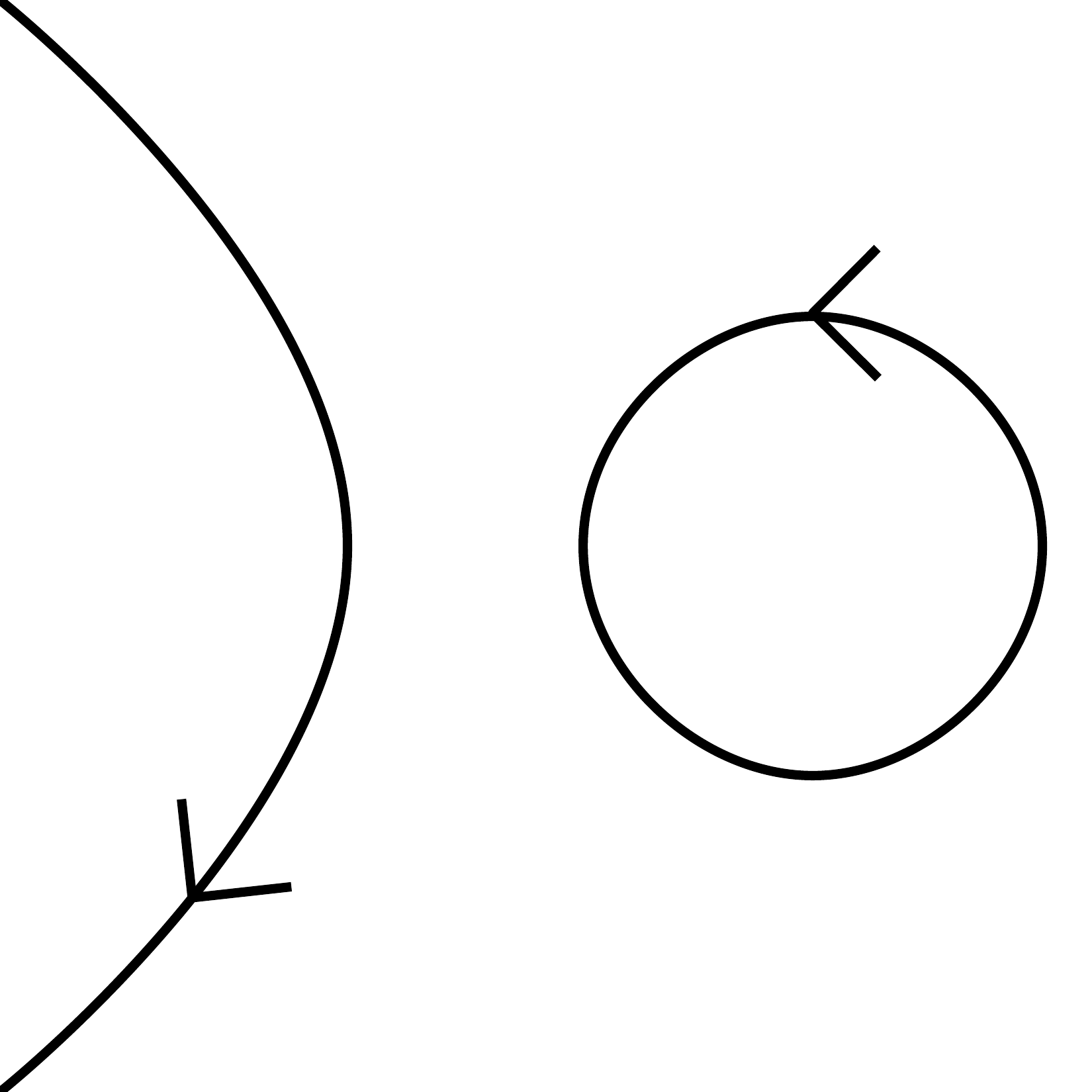}}\Biggr)-q^nP\Biggl(\raisebox{-10pt}{\includegraphics[height = .35in]{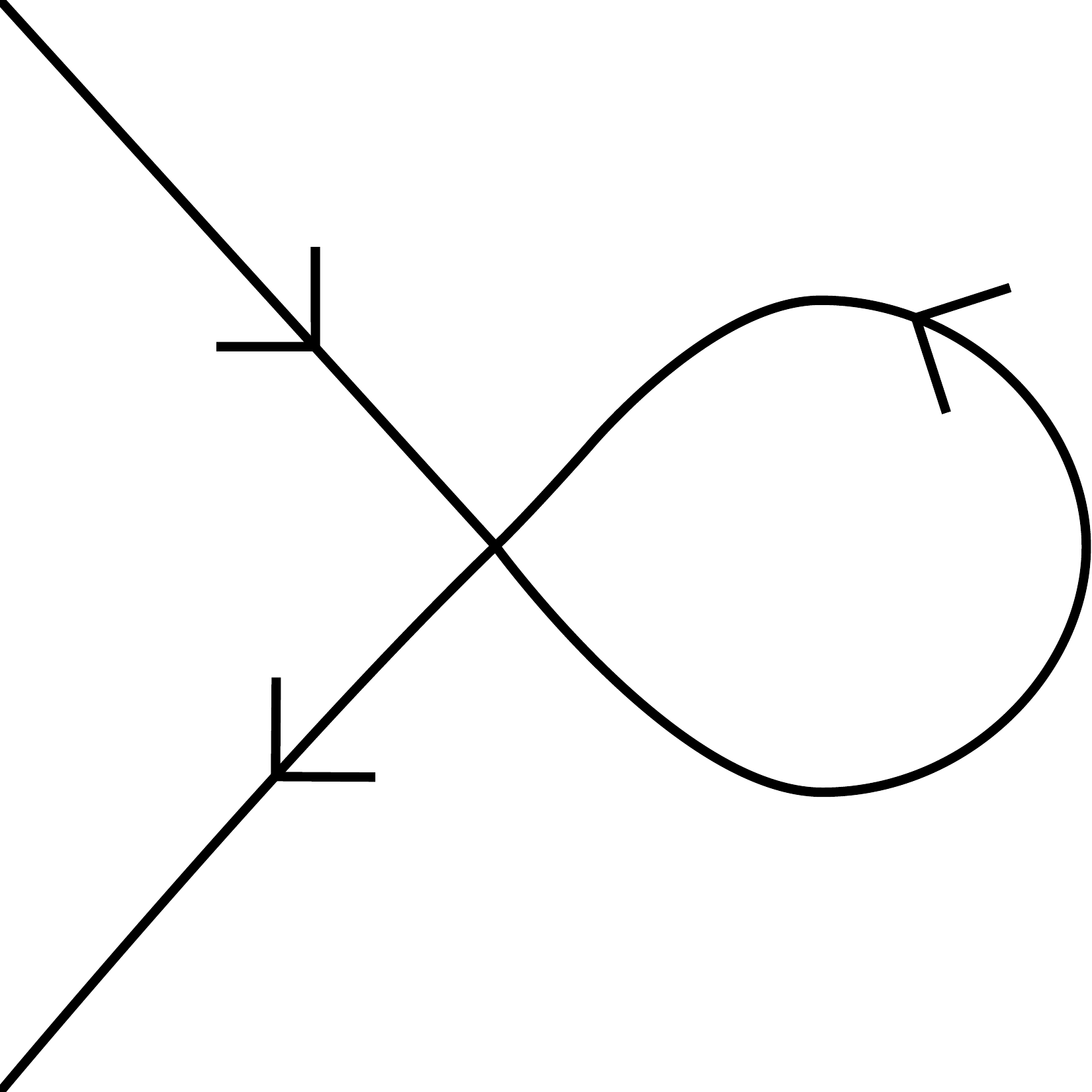}}\Biggr)\\
    &=q^{n-1}[n]P\Biggl(\raisebox{-10pt}{\includegraphics[height = .35in]{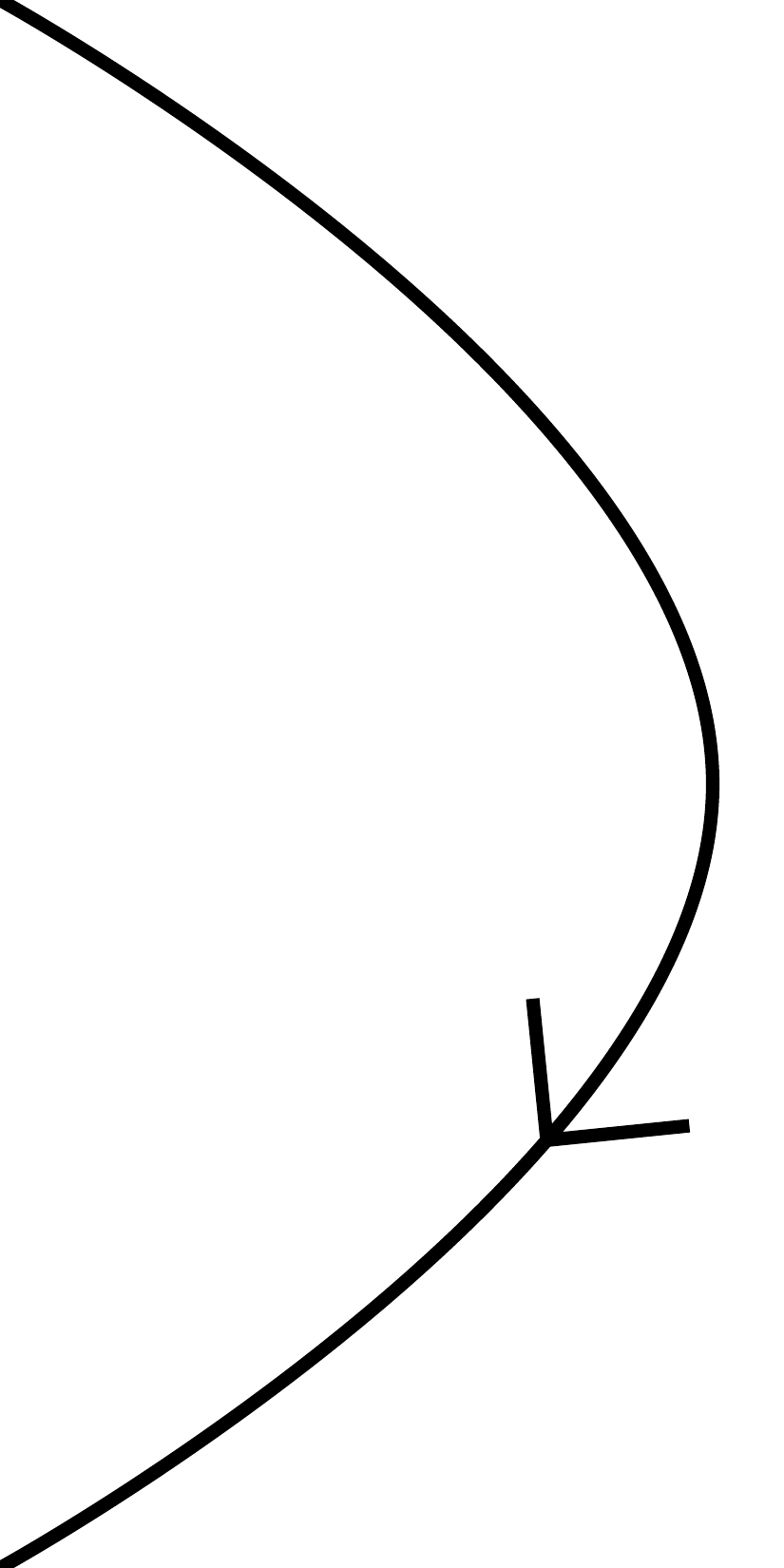}}\Biggr)-q^n[n-1]P\Biggl(\raisebox{-10pt}{\includegraphics[height = .35in]{O1ab-Proof/oriented-strand-down-half.pdf}}\Biggr)\\
    &=P\Biggl(\raisebox{-10pt}{\includegraphics[height = .35in]{O1ab-Proof/oriented-strand-down-half.pdf}}\Biggr).
\end{align*}

We consider next the move $\Omega 2a$. In the first step below, we apply the skein relation for the negative crossing on the right of the diagram and in the second step we apply the skein relation for the positive crossing. In the third step, we use the first graphical relation involving a bigon, and in the last steps we merely simplify combining like terms and use that $[2] = q+q^{-1}$.
\begin{align*}
    P\Biggl(\raisebox{-10pt}{\includegraphics[height = .35in]{Generating-Sets/O2a-1.pdf}}\Biggr)&= q^{1-n}P\Biggl(\raisebox{-10pt}{\includegraphics[height = .35in]{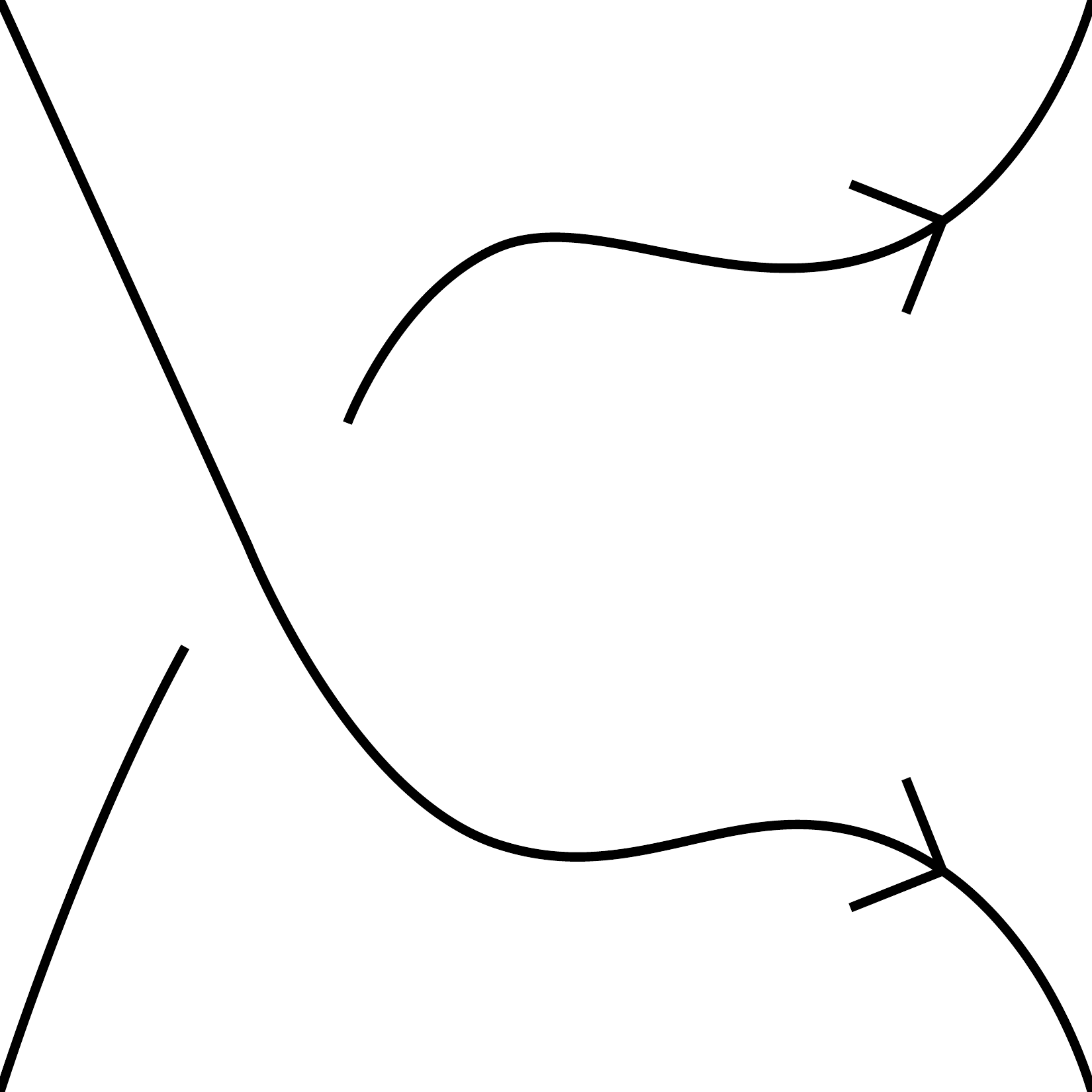}}\Biggr)-q^{-n}P\Biggl(\raisebox{-10pt}{\includegraphics[height = .35in]{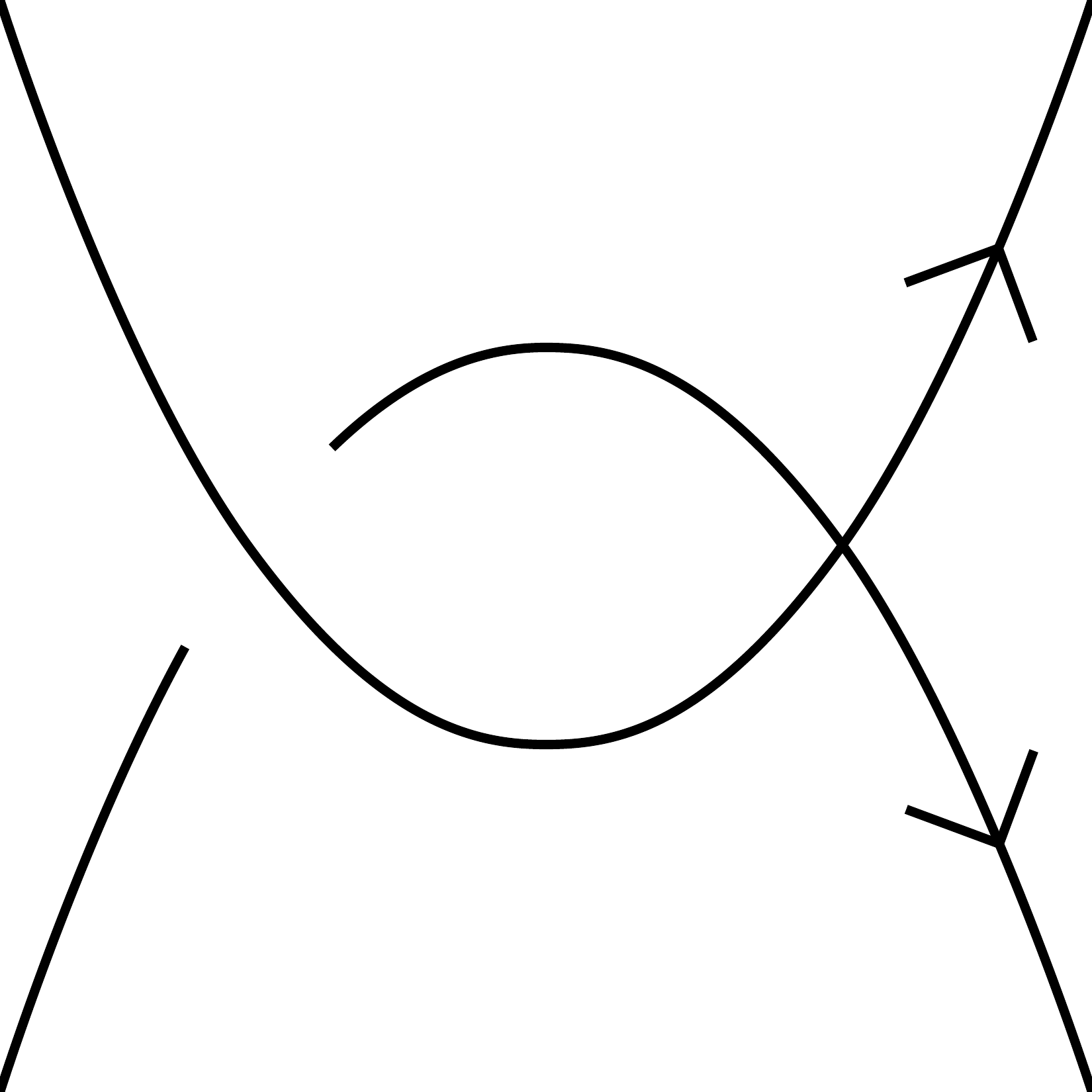}}\Biggr)\\
    &=q^{1-n}\left[q^{n-1}P\Biggl(\raisebox{-10pt}{\includegraphics[height = .35in]{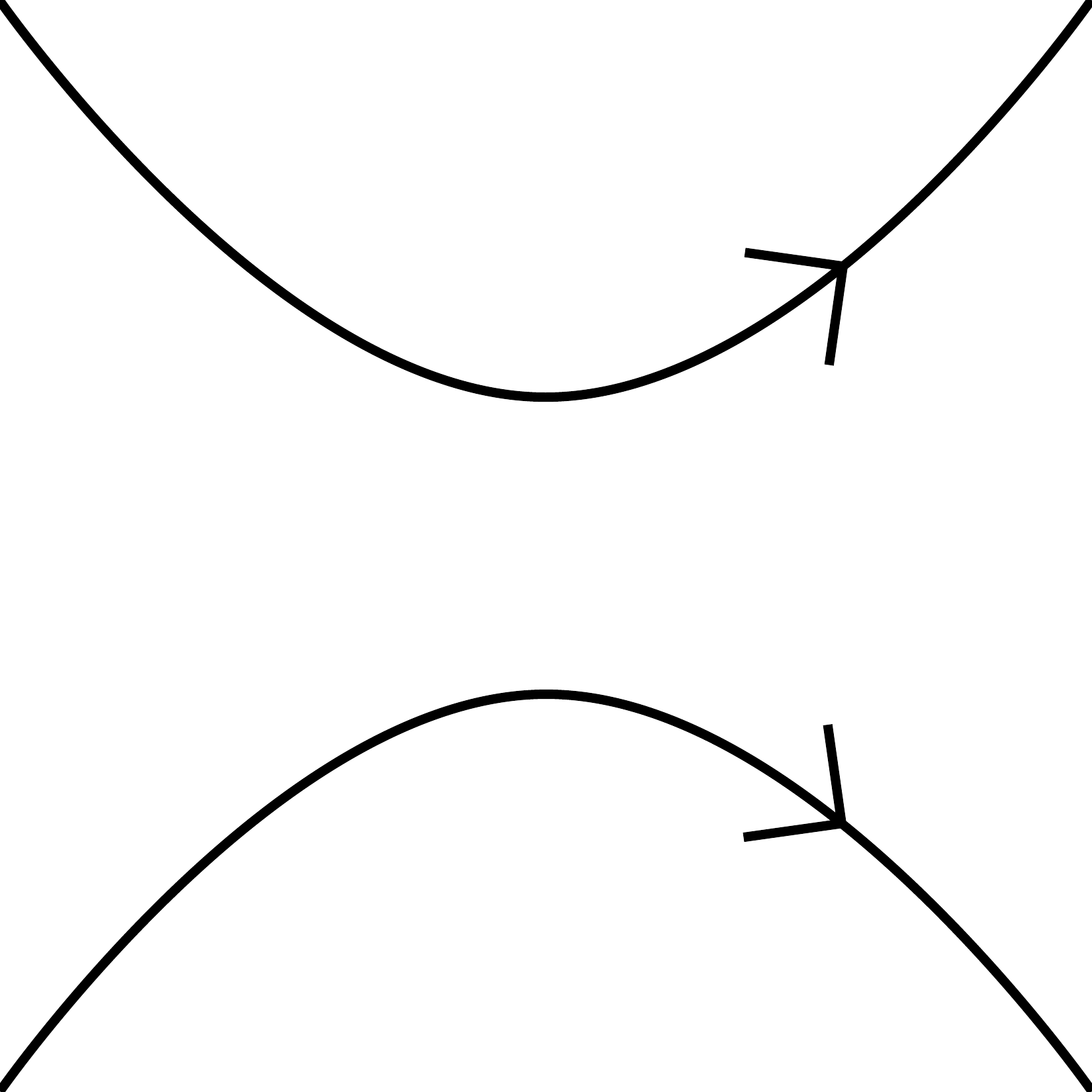}}\Biggr)-q^{n}P\Biggl(\raisebox{-10pt}{\includegraphics[height = .35in]{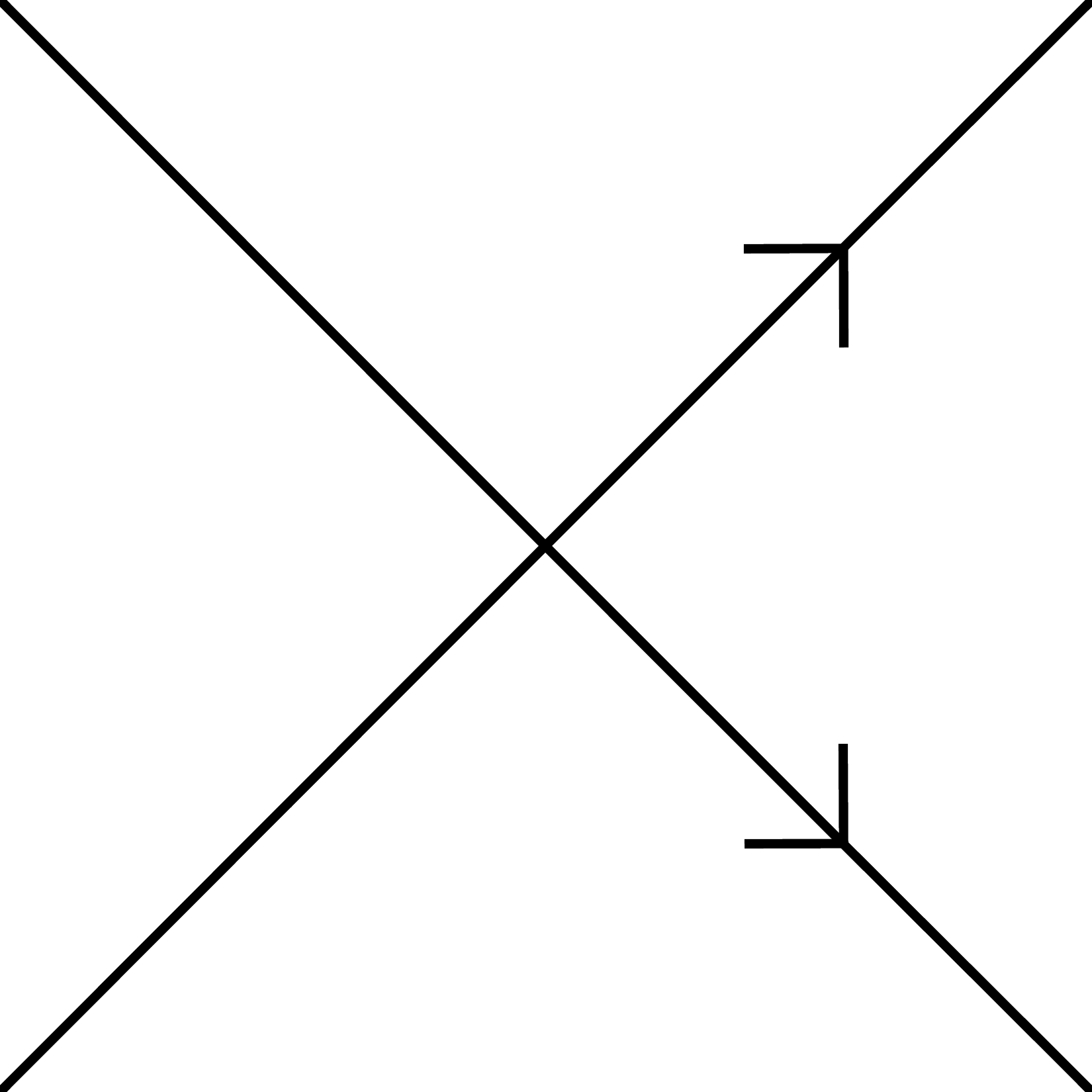}}\Biggr)\right]\\
    & \,\, \,\,  -q^{-n}\left[q^{n-1}P\Biggl(\raisebox{-10pt}{\includegraphics[height = .35in]{Skein-relations/cross-right.pdf}}\Biggr)-q^{n}P\Biggl(\raisebox{-10pt}{\includegraphics[height = .35in]{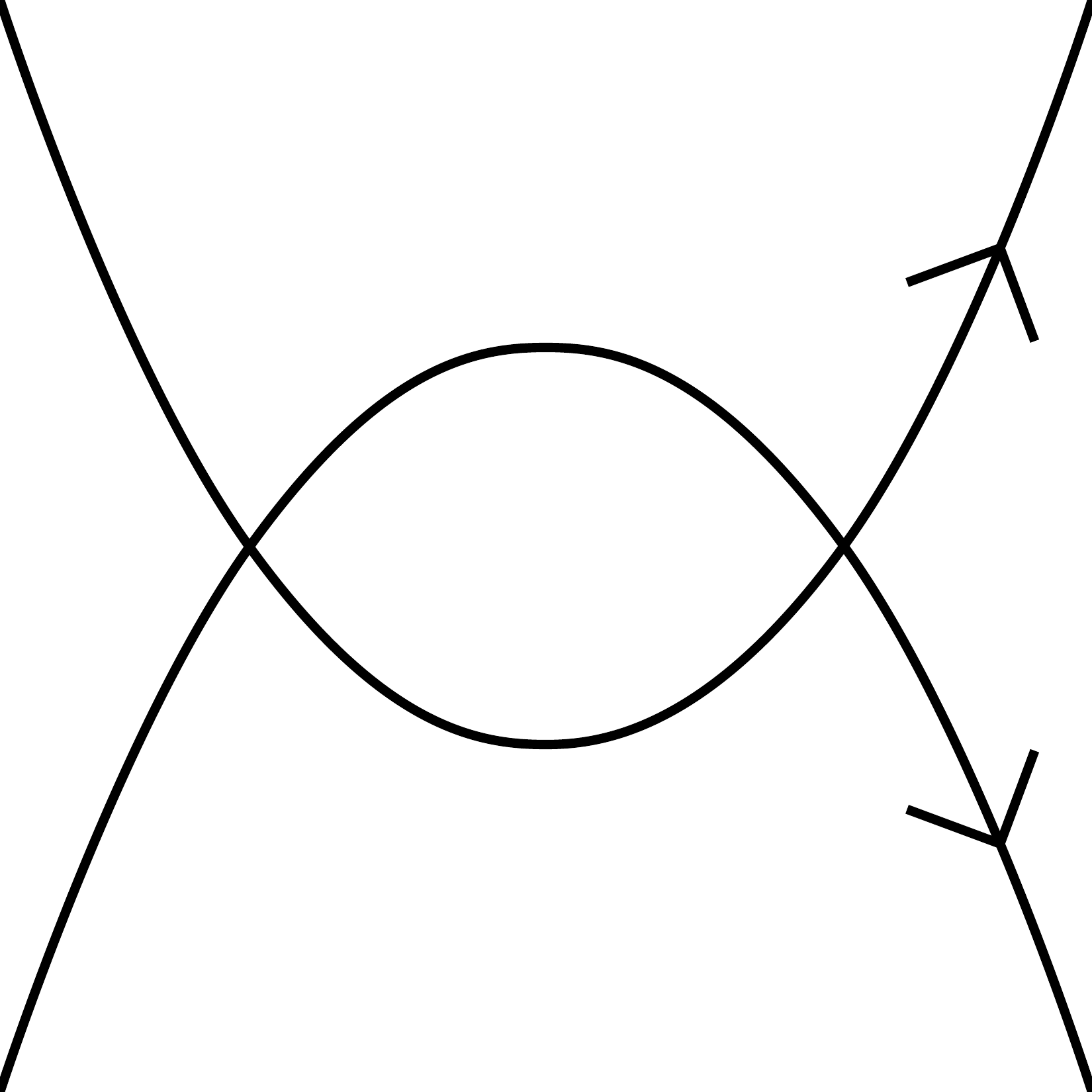}}\Biggr)\right]\\
    &=P\Biggl(\raisebox{-10pt}{\includegraphics[height = .35in]{Skein-relations/oriented-strands-right.pdf}}\Biggr)- (q + q^{-1}) P\Biggl(\raisebox{-10pt}{\includegraphics[height = .35in]{Skein-relations/cross-right.pdf}}\Biggr)+[2]P\Biggl(\raisebox{-10pt}{\includegraphics[height = .35in]{Skein-relations/cross-right.pdf}}\Biggr)\\
    &=P\Biggl(\raisebox{-10pt}{\includegraphics[height = .35in]{Skein-relations/oriented-strands-right.pdf}}\Biggr).
\end{align*}

Before showing that $P$ is invariant under the move $\Omega3a$, we verify that $P$ is invariant under the move $\Omega4a$, to reduce our computations. For the right hand-side of the move, we have the following:

\begin{align*}
    P\Biggl(\raisebox{-10pt}{\includegraphics[height = .35in]{Generating-Sets/O4a-2.pdf}}\Biggr)&=q^{n-1}P\Biggl(\raisebox{-10pt}{\includegraphics[height = .35in]{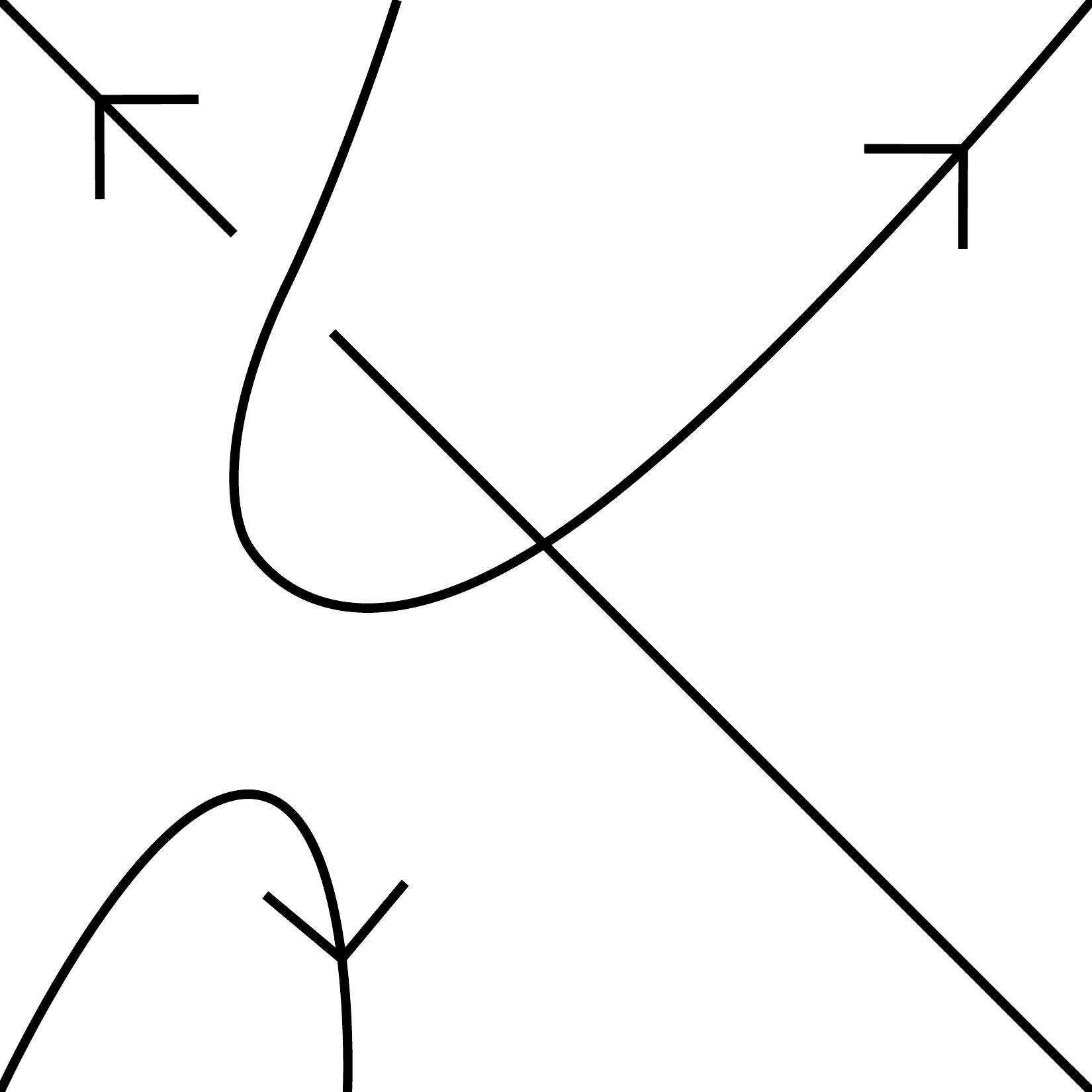}}\Biggr)-q^nP\Biggl(\raisebox{-10pt}{\includegraphics[height = .35in]{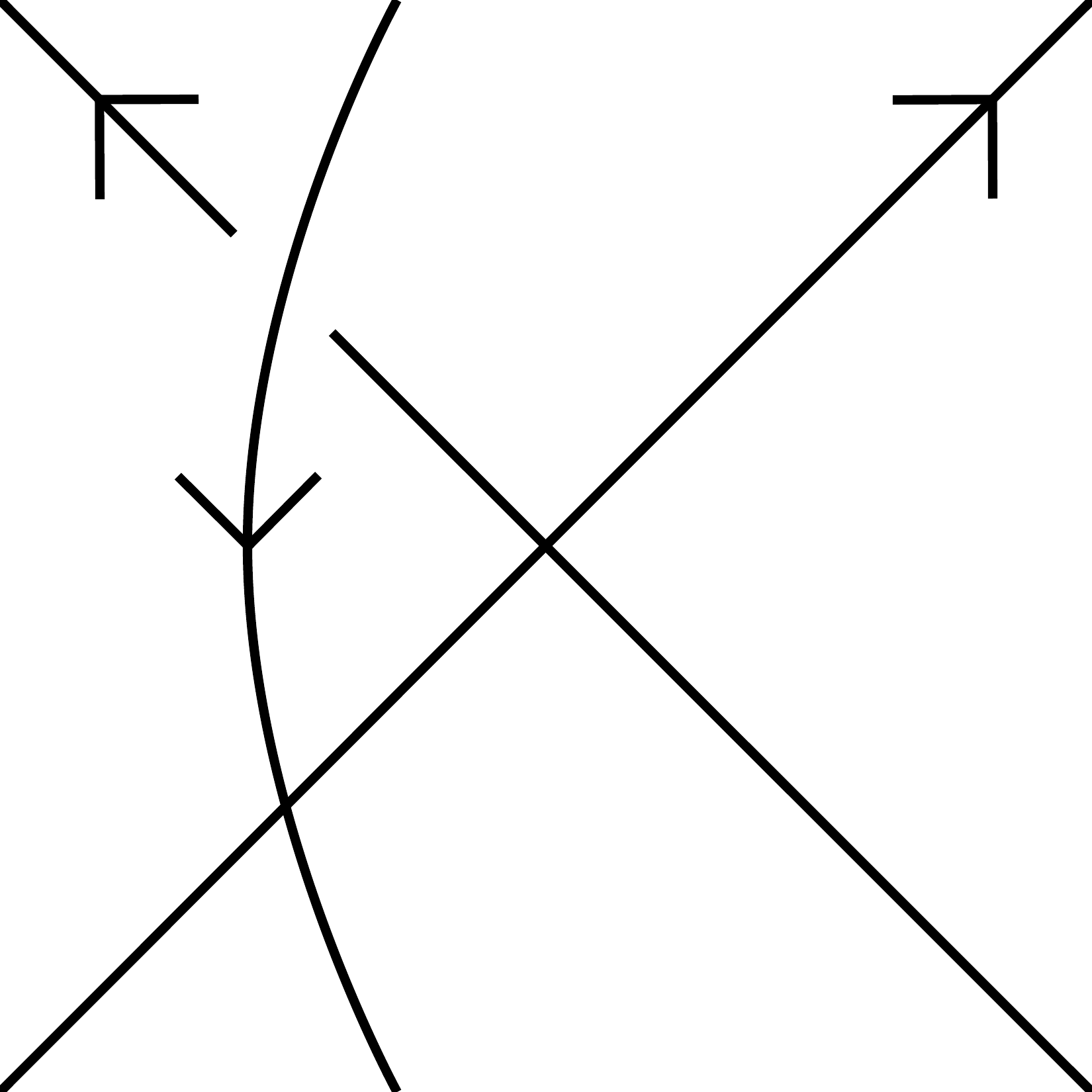}}\Biggr)\\
    &=q^{n-1}\left[q^{1-n}P\Biggl(\raisebox{-10pt}{\includegraphics[height = .35in]{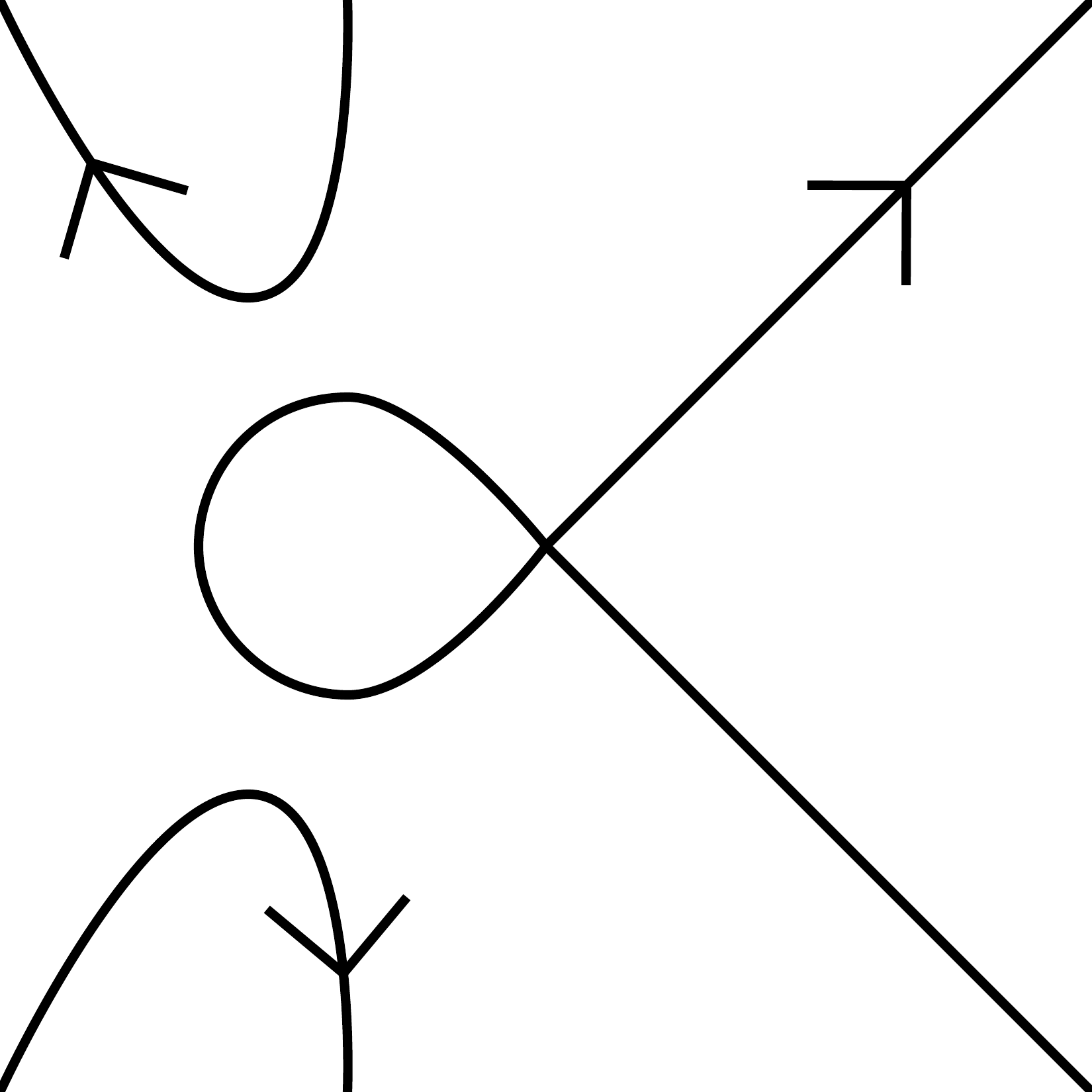}}\Biggr)-q^{-n}P\Biggl(\raisebox{-10pt}{\includegraphics[height = .35in]{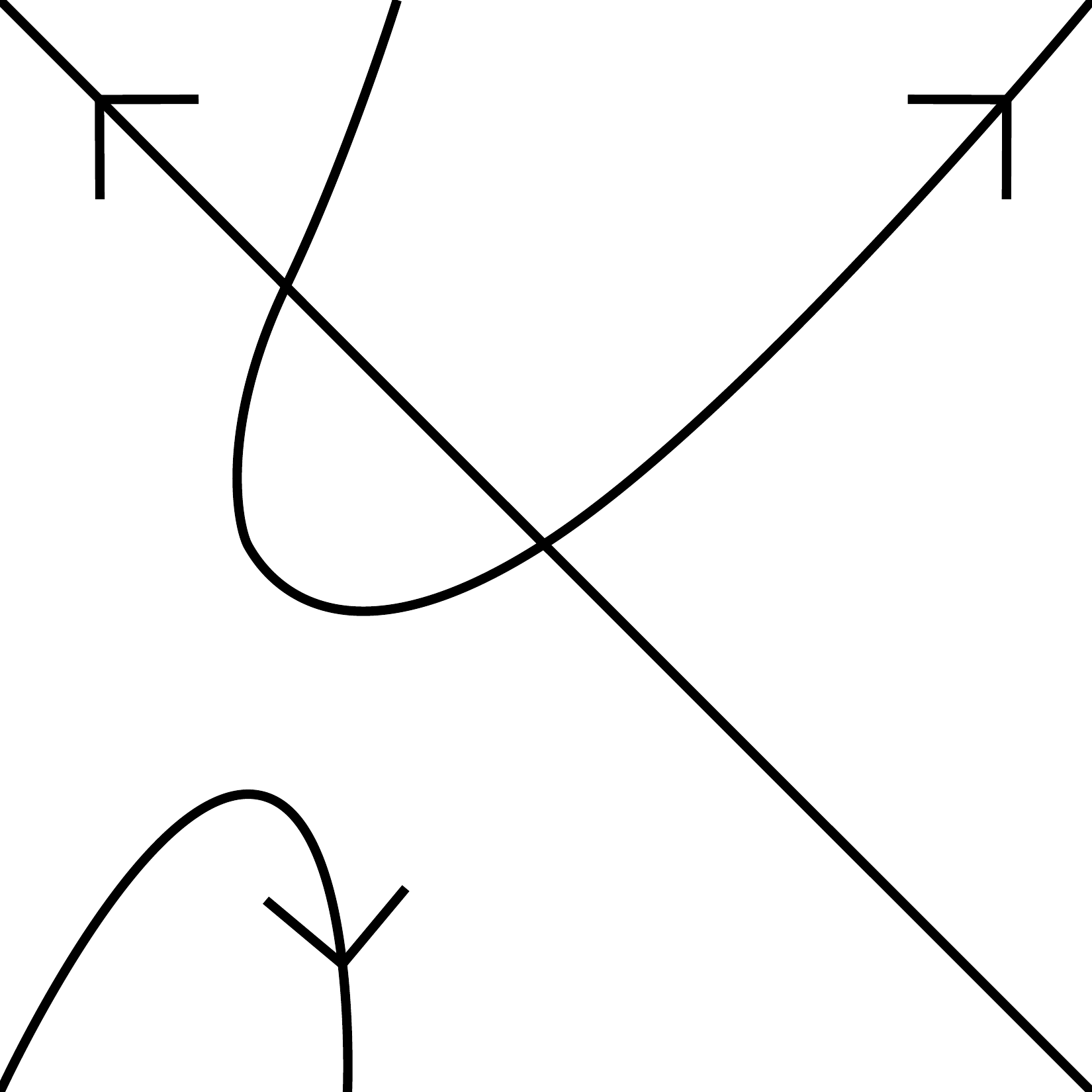}}\Biggr)\right]\\&-q^{n}\left[q^{1-n}P\Biggl(\raisebox{-10pt}{\includegraphics[height = .35in]{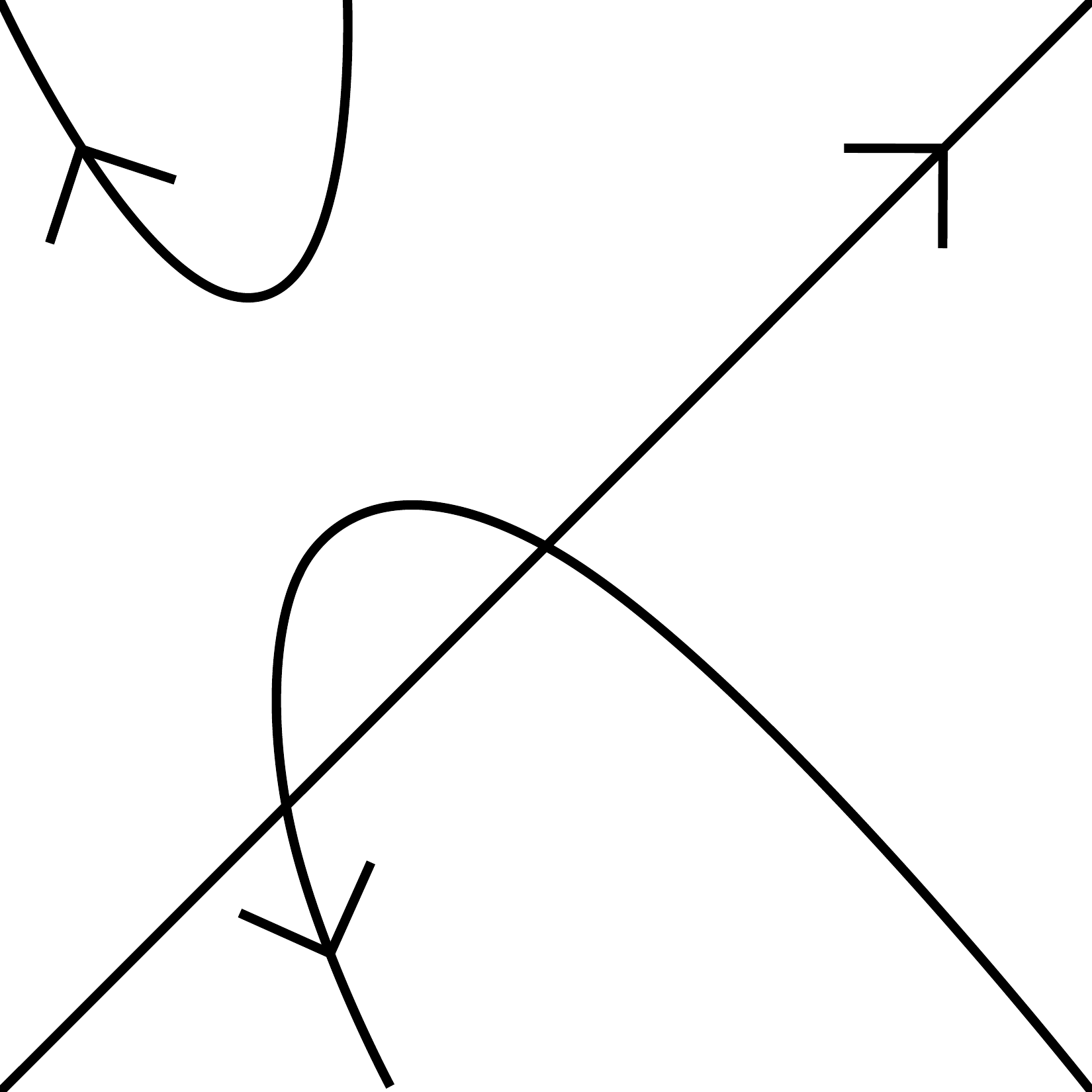}}\Biggr)-q^{-n}P\Biggl(\raisebox{-10pt}{\includegraphics[height = .35in]{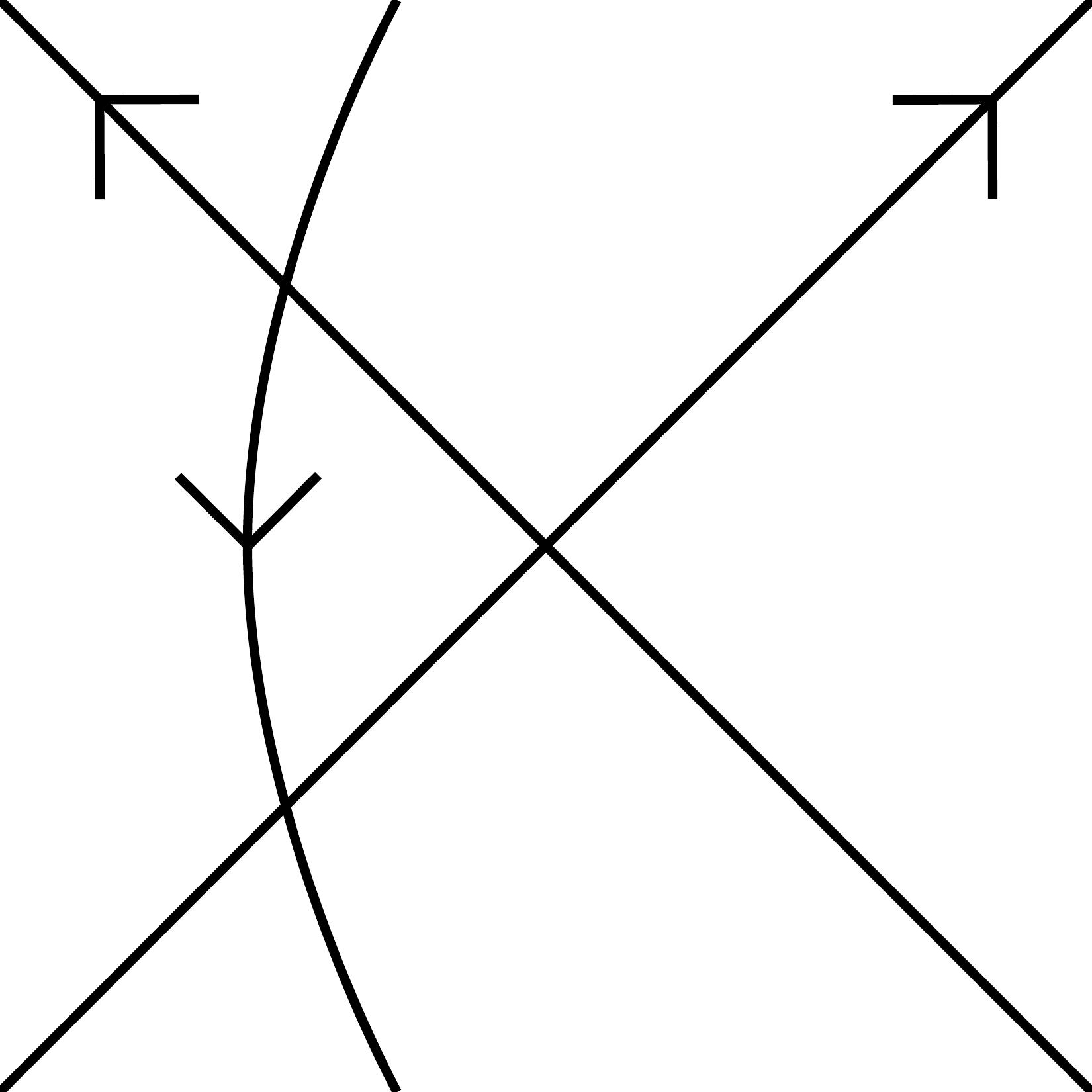}}\Biggr)\right]\\
    &=([n-1]-q^{-1}[n-2]-q[n-2])P\Biggl(\raisebox{-10pt}{\includegraphics[height = .35in]{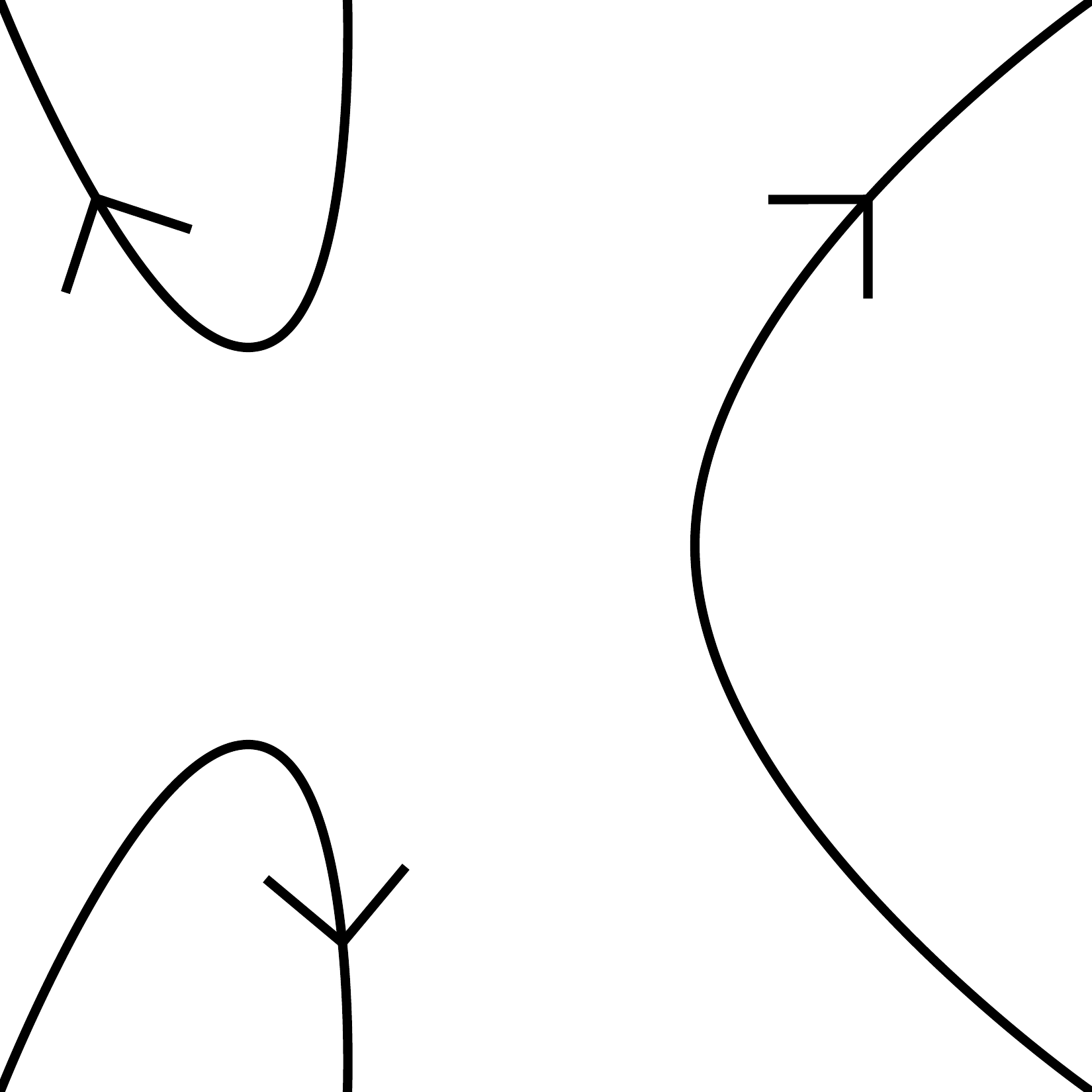}}\Biggr)-q^{-1}P\Biggl(\raisebox{-10pt}{\includegraphics[height = .35in]{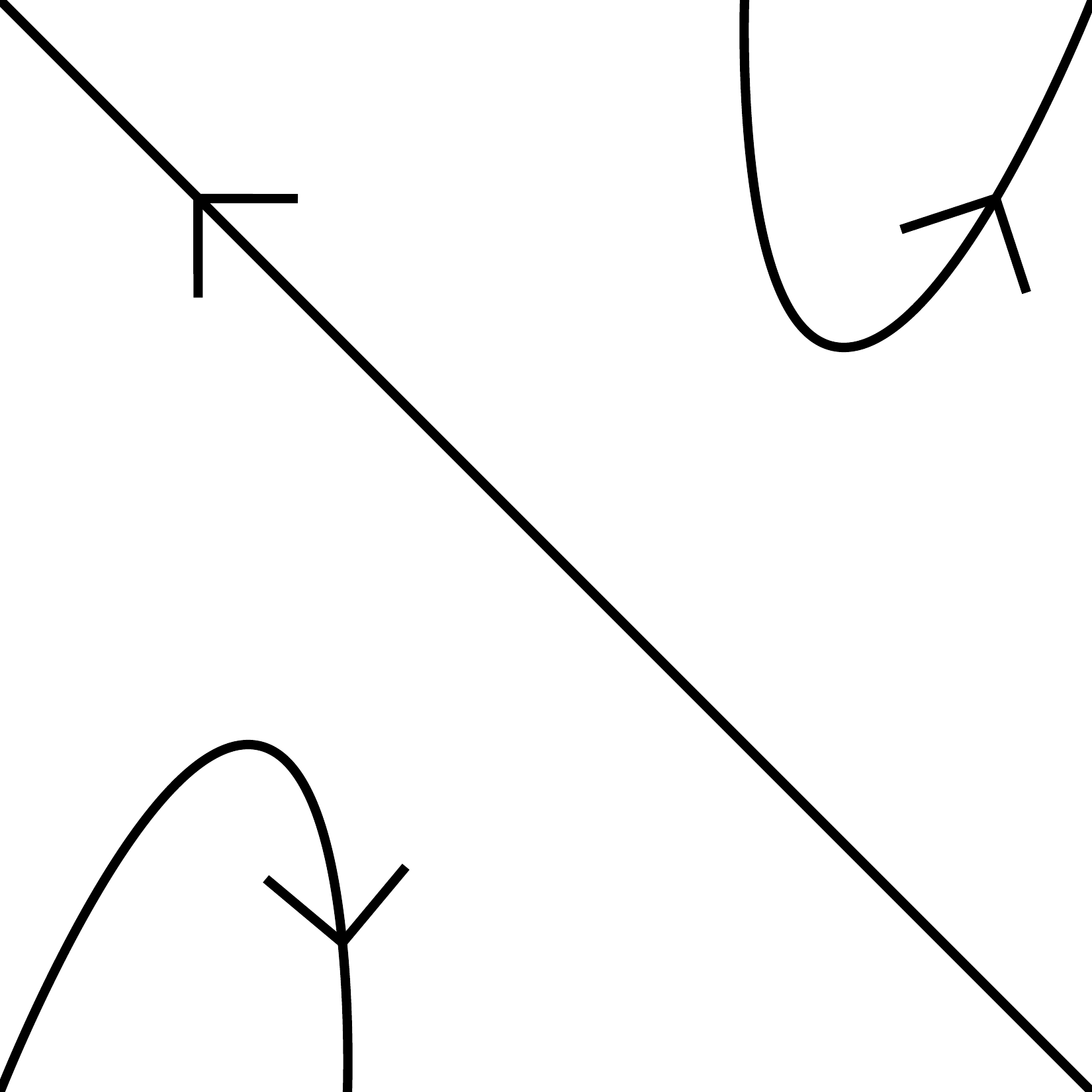}}\Biggr)\\&-qP\Biggl(\raisebox{-10pt}{\includegraphics[height = .35in]{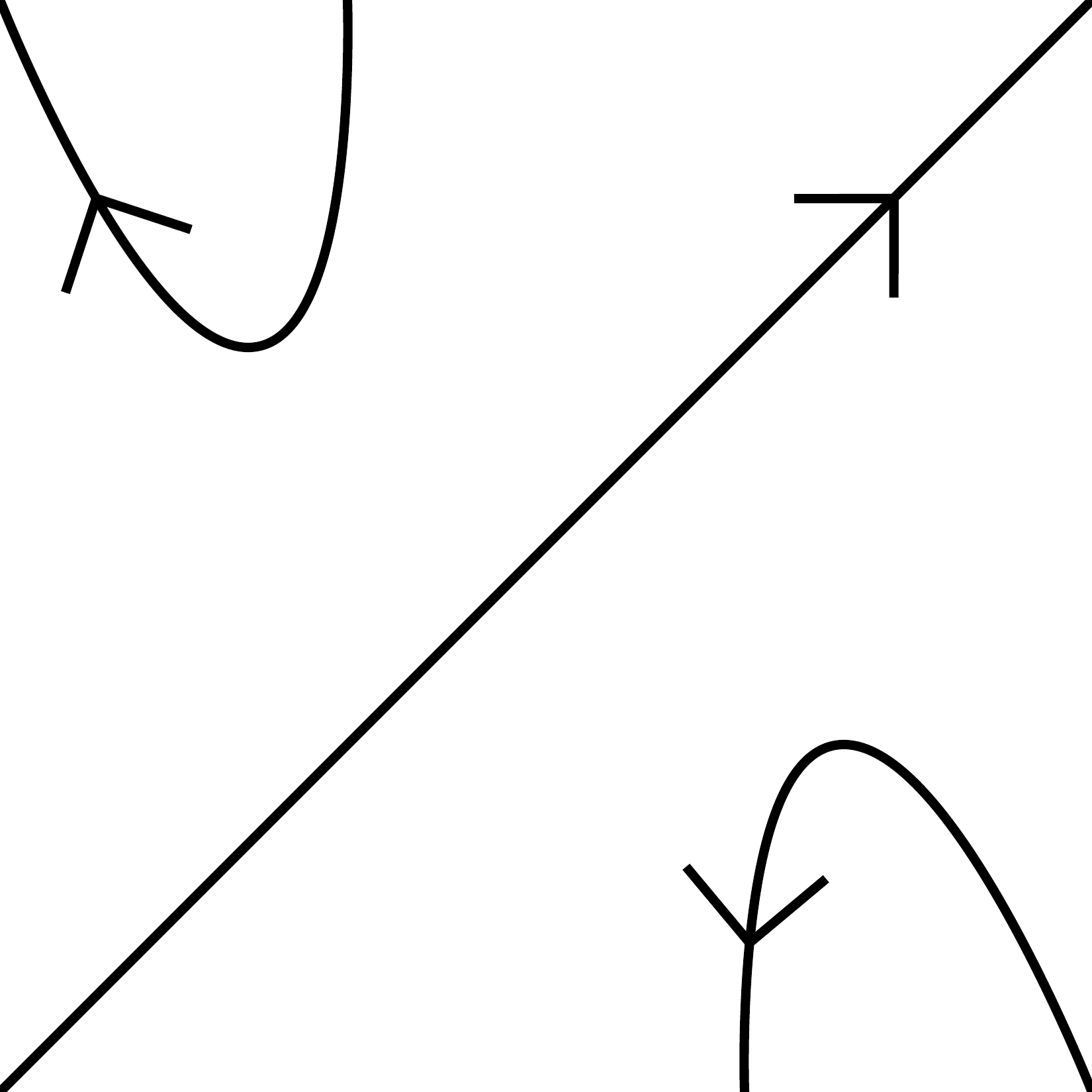}}\Biggr)+P\Biggl(\raisebox{-10pt}{\includegraphics[height = .35in]{O4-Proof/O4a-SR6.pdf}}\Biggr)\\
    &=- [n-3]P\Biggl(\raisebox{-10pt}{\includegraphics[height = .35in]{O4-Proof/O4a-SR7.pdf}}\Biggr)+P\Biggl(\raisebox{-10pt}{\includegraphics[height = .35in]{O4-Proof/O4a-SR6.pdf}}\Biggr)-q^{-1}P\Biggl(\raisebox{-10pt}{\includegraphics[height = .35in]{O4-Proof/O4a-SR8.pdf}}\Biggr)-qP\Biggl(\raisebox{-10pt}{\includegraphics[height = .35in]{O4-Proof/O4a-SR9.pdf}}\Biggr).
\end{align*}
In the third equality above, we employed the graphical relation involving a loop and the graphical relation with an oriented bigon.
Using a similar approach, we obtain the following steps for the evaluation of the left hand-side of the move $\Omega 4a$:

\begin{align*}
    P\Biggl(\raisebox{-10pt}{\includegraphics[height = .35in]{Generating-Sets/O4a-1.pdf}}\Biggr)&=q^{1-n}P\Biggl(\raisebox{-10pt}{\includegraphics[height = .35in]{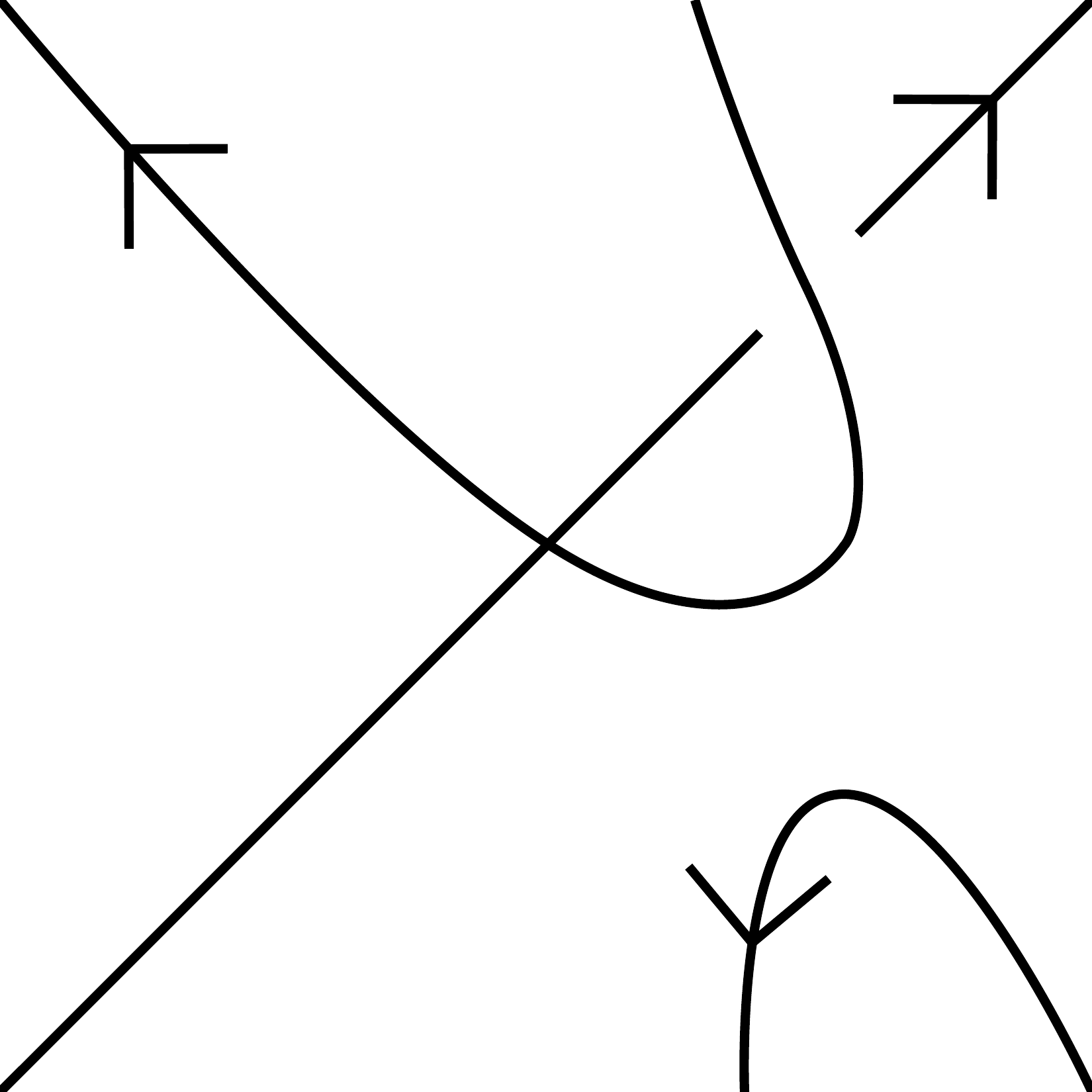}}\Biggr)-q^{-n}P\Biggl(\raisebox{-10pt}{\includegraphics[height = .35in]{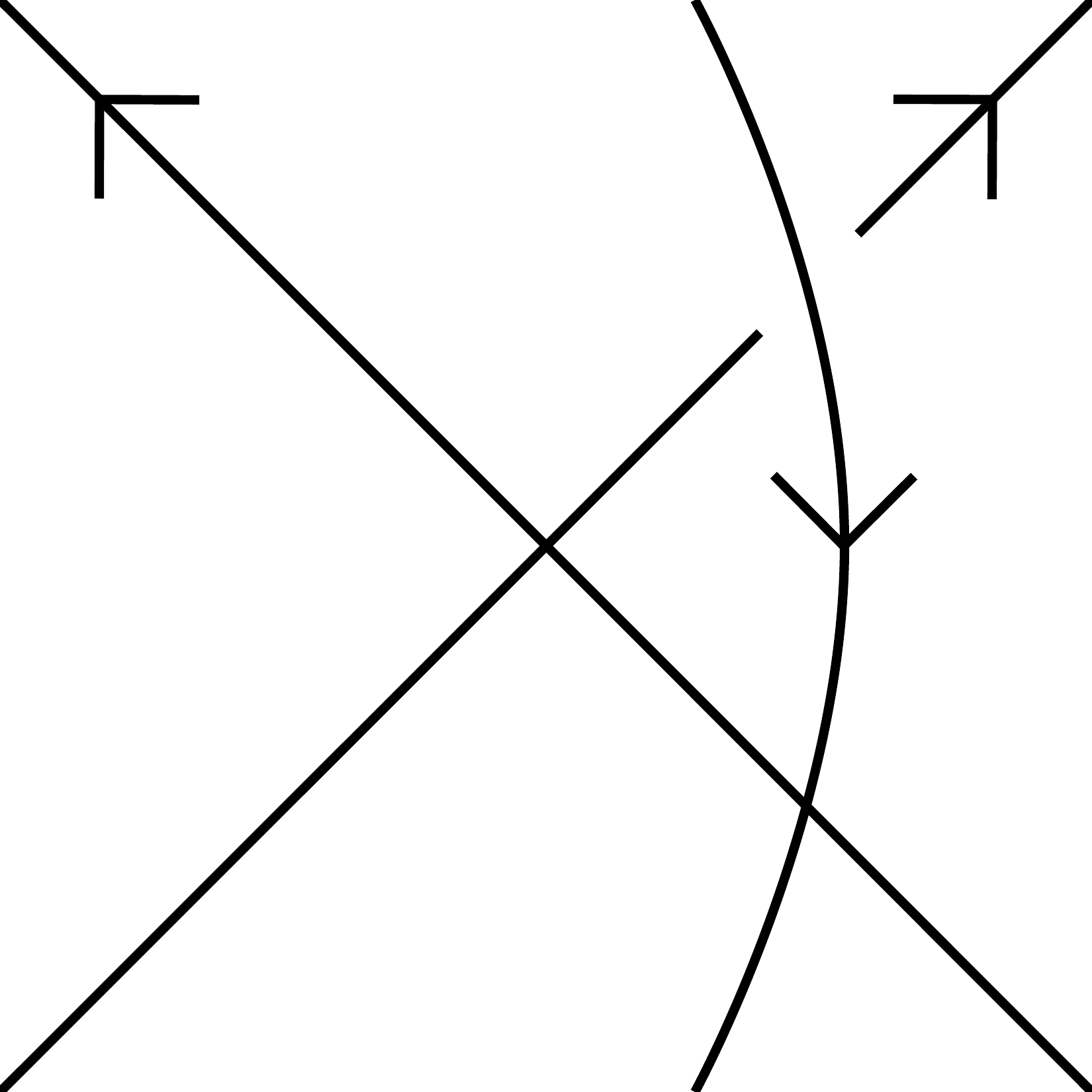}}\Biggr)\\
    &=q^{1-n}\left[q^{n-1}P\Biggl(\raisebox{-10pt}{\includegraphics[height = .35in]{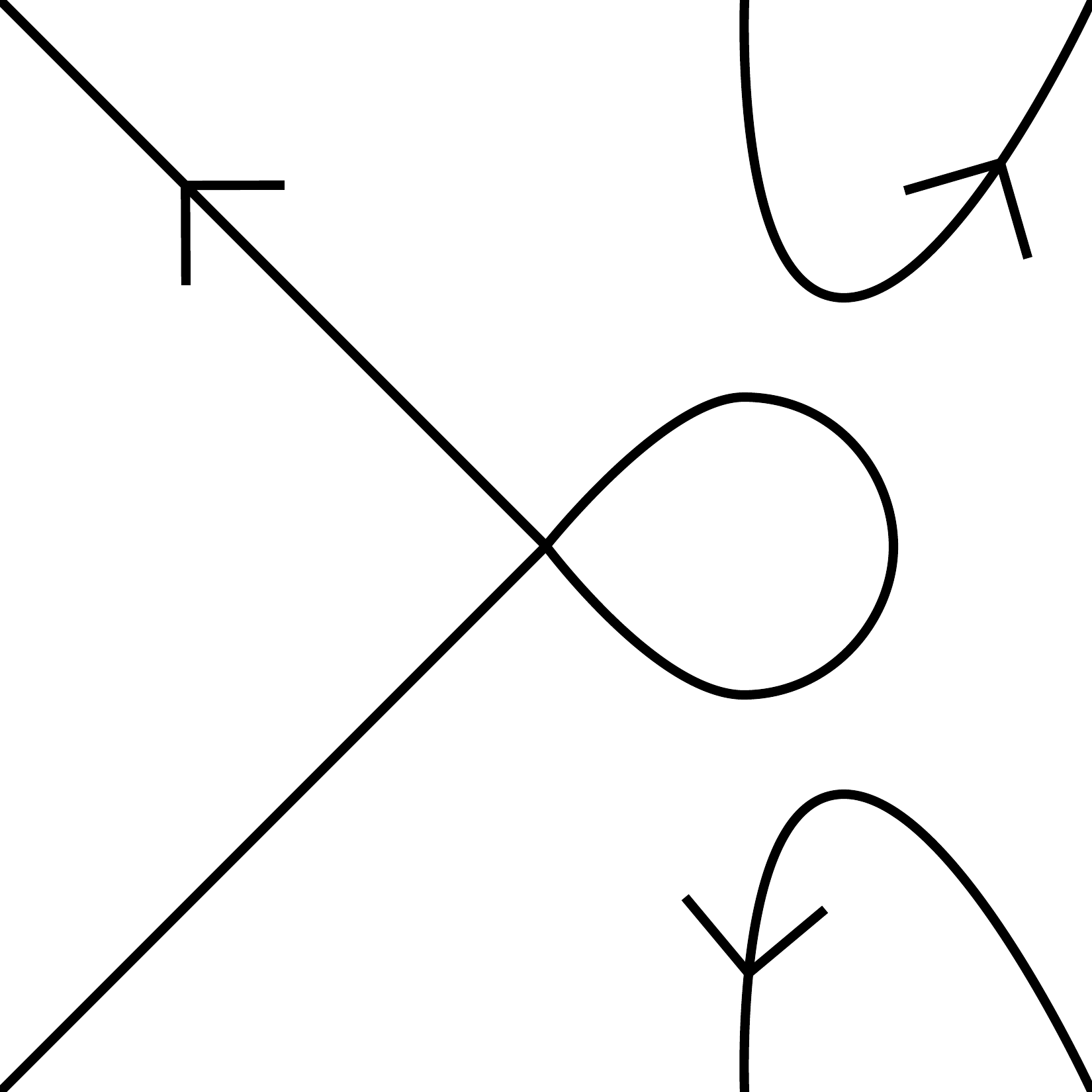}}\Biggr)-q^{n}P\Biggl(\raisebox{-10pt}{\includegraphics[height = .35in]{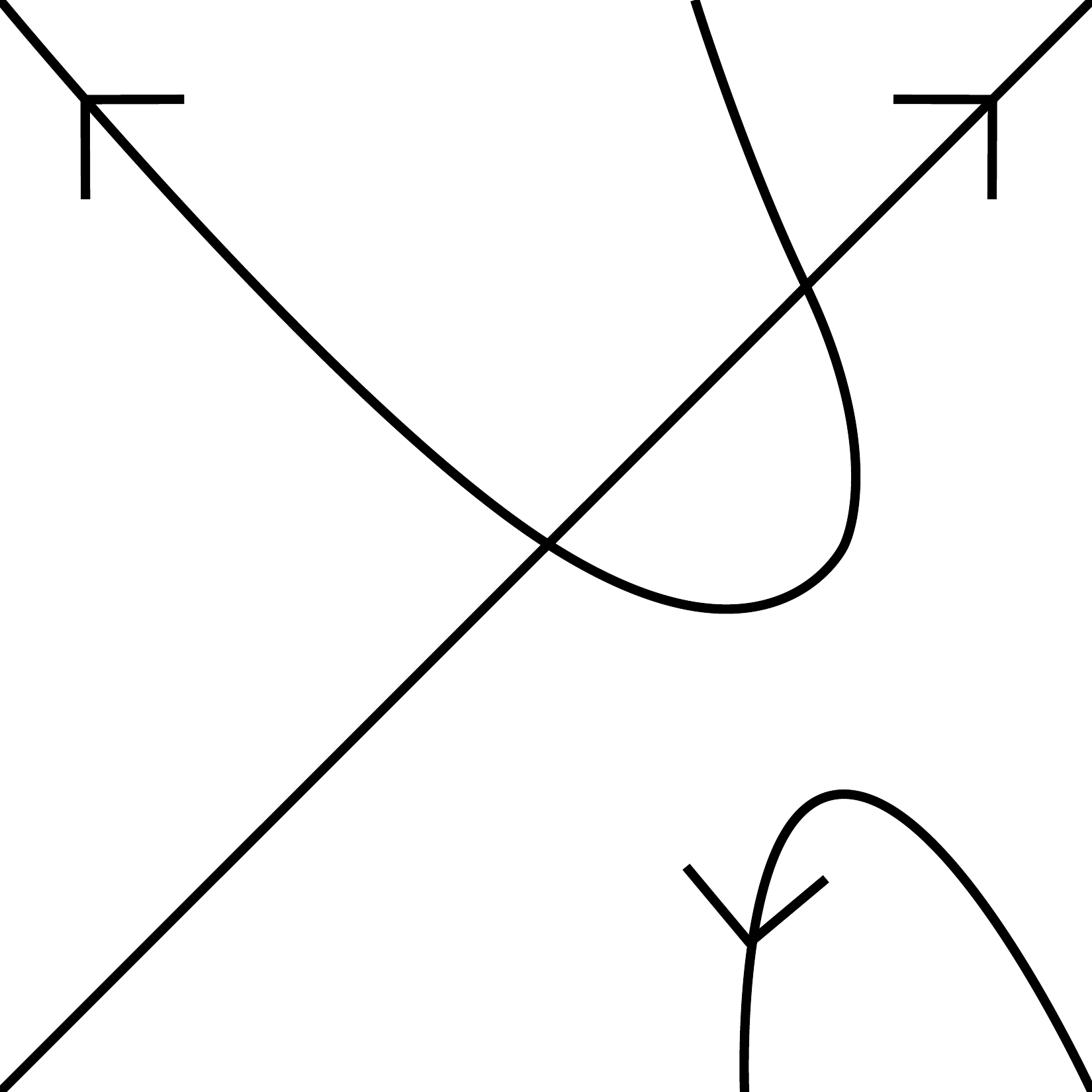}}\Biggr)\right]\\&-q^{-n}\left[q^{n-1}P\Biggl(\raisebox{-10pt}{\includegraphics[height = .35in]{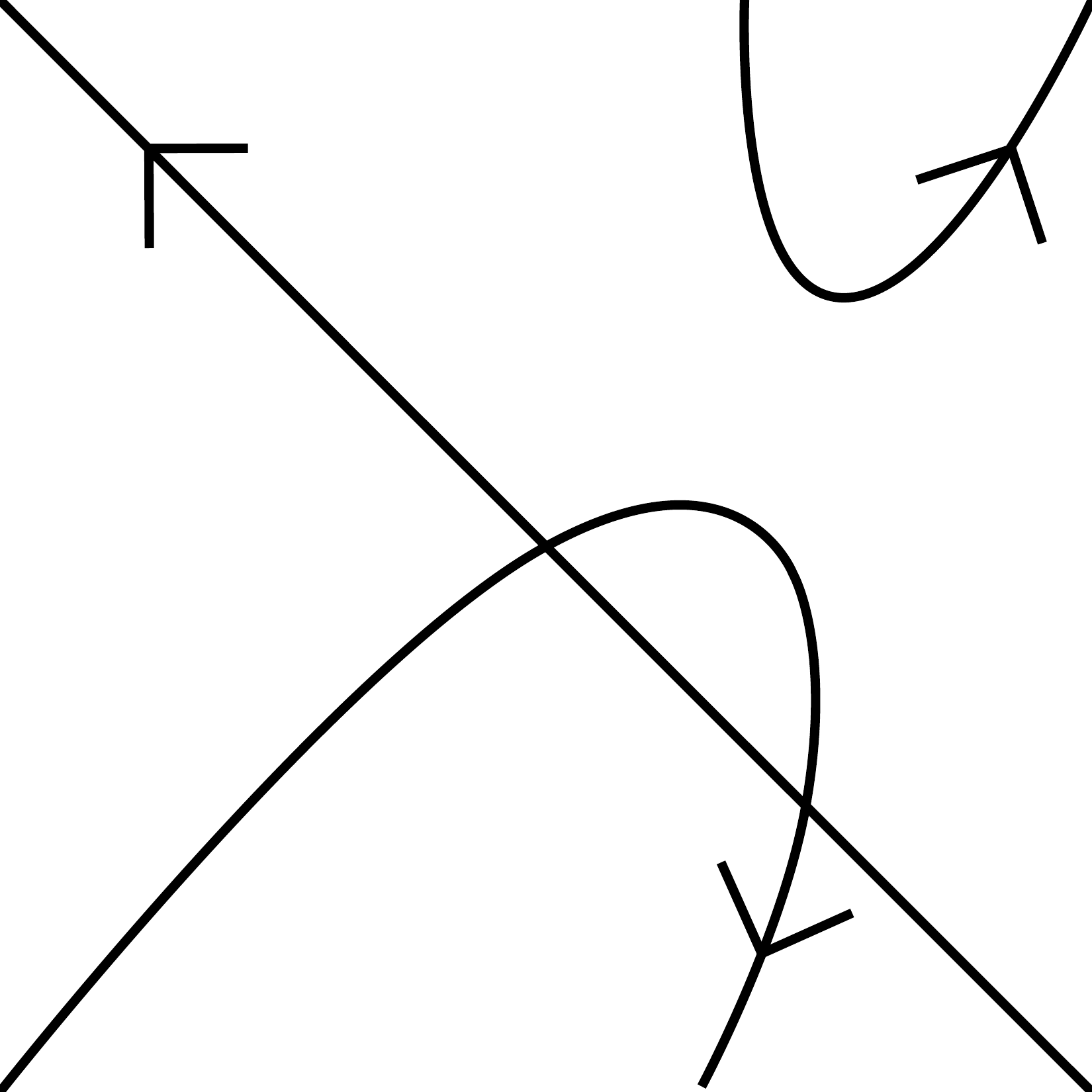}}\Biggr)-q^{n}P\Biggl(\raisebox{-10pt}{\includegraphics[height = .35in]{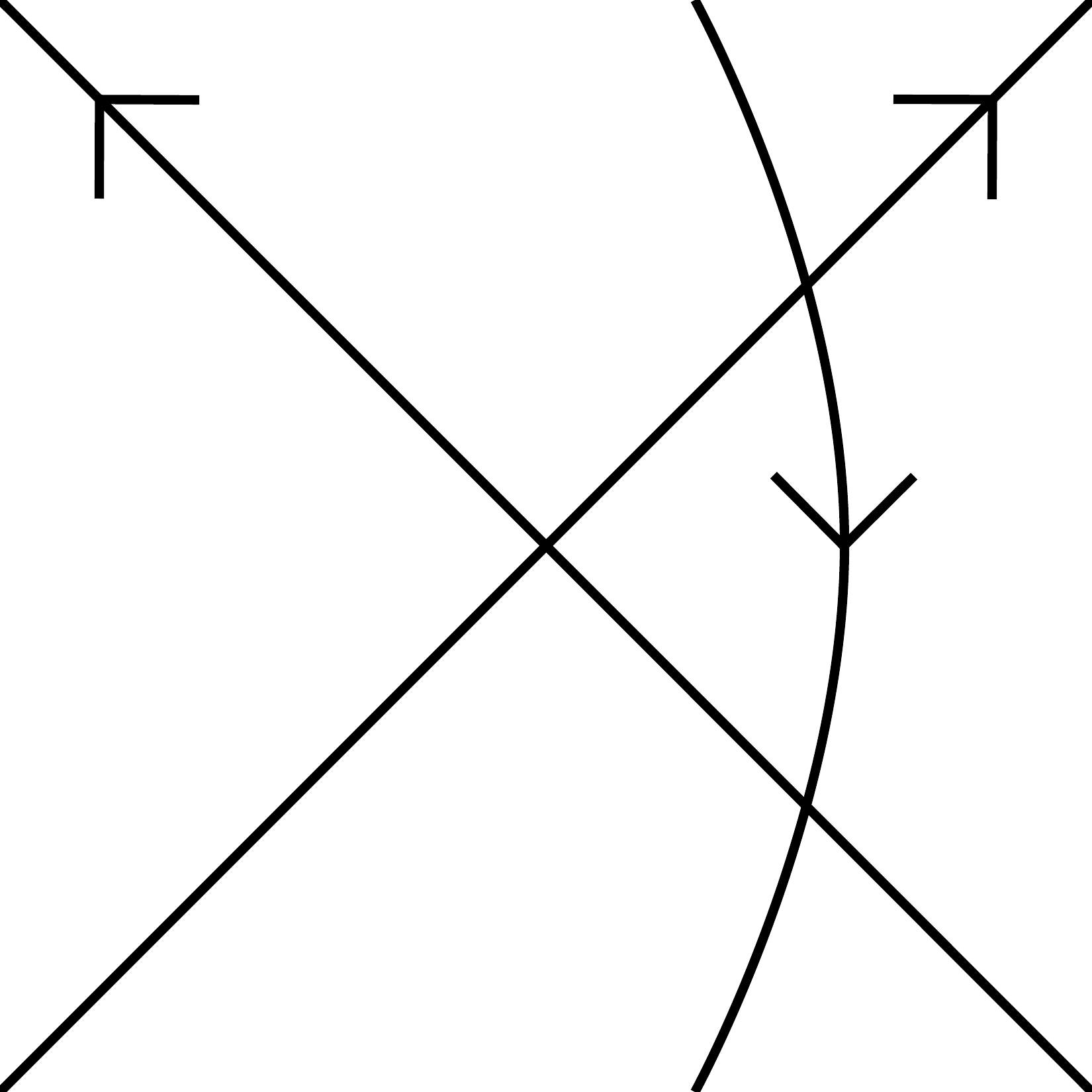}}\Biggr)\right]\\
    &=([n-1]-q[n-2]-q^{-1}[n-2])P\Biggl(\raisebox{-10pt}{\includegraphics[height = .35in]{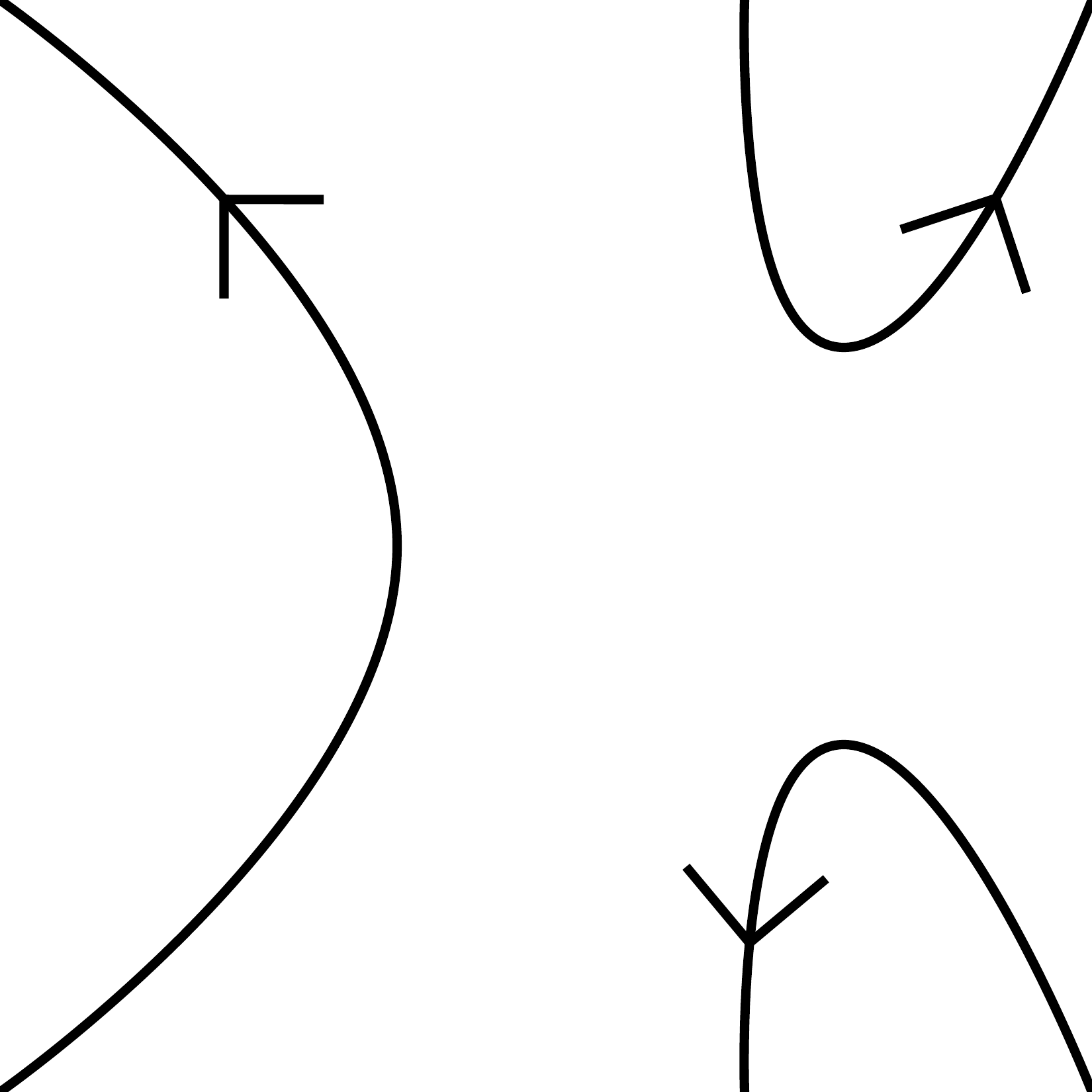}}\Biggr)-qP\Biggl(\raisebox{-10pt}{\includegraphics[height = .35in]{O4-Proof/O4a-SR9.pdf}}\Biggr)\\&-q^{-1}P\Biggl(\raisebox{-10pt}{\includegraphics[height = .35in]{O4-Proof/O4a-SR8.pdf}}\Biggr)+P\Biggl(\raisebox{-10pt}{\includegraphics[height = .35in]{O4-Proof/O4a-SR15.pdf}}\Biggr)\\
     &= -[n-3]P\Biggl(\raisebox{-10pt}{\includegraphics[height = .35in]{O4-Proof/O4a-SR16.pdf}}\Biggr)+P\Biggl(\raisebox{-10pt}{\includegraphics[height = .35in]{O4-Proof/O4a-SR15.pdf}}\Biggr)-q^{-1}P\Biggl(\raisebox{-10pt}{\includegraphics[height = .35in]{O4-Proof/O4a-SR8.pdf}}\Biggr)-qP\Biggl(\raisebox{-10pt}{\includegraphics[height = .35in]{O4-Proof/O4a-SR9.pdf}}\Biggr).
\end{align*}
Since by the last graphical relation in Figure \ref{fig:GSR} we have that 
\[- [n-3]P\Biggl(\raisebox{-10pt}{\includegraphics[height = .35in]{O4-Proof/O4a-SR7.pdf}}\Biggr)+P\Biggl(\raisebox{-10pt}{\includegraphics[height = .35in]{O4-Proof/O4a-SR6.pdf}}\Biggr) = -[n-3]P\Biggl(\raisebox{-10pt}{\includegraphics[height = .35in]{O4-Proof/O4a-SR16.pdf}}\Biggr)+P\Biggl(\raisebox{-10pt}{\includegraphics[height = .35in]{O4-Proof/O4a-SR15.pdf}}\Biggr),\] 
we obtain the following equality:
\[P\Biggl(\raisebox{-10pt}{\includegraphics[height = .35in]{Generating-Sets/O4a-2.pdf}}\Biggr) = P\Biggl(\raisebox{-10pt}{\includegraphics[height = .35in]{Generating-Sets/O4a-1.pdf}}\Biggr),\]
and therefore $P$ is invariant under the move $\Omega 4a$.

We will now use this result to prove that $P$ is invariant under the move $\Omega 3a$. Starting with the left hand-side of the move and applying the skein relation to resolve the middle positive crossing, we obtain:

\begin{align*}
    P\Biggl(\raisebox{-10pt}{\includegraphics[height = .35in]{Generating-Sets/O3a-1.pdf}}\Biggr)&=q^{n-1}P\Biggl(\raisebox{-10pt}{\includegraphics[height = .35in]{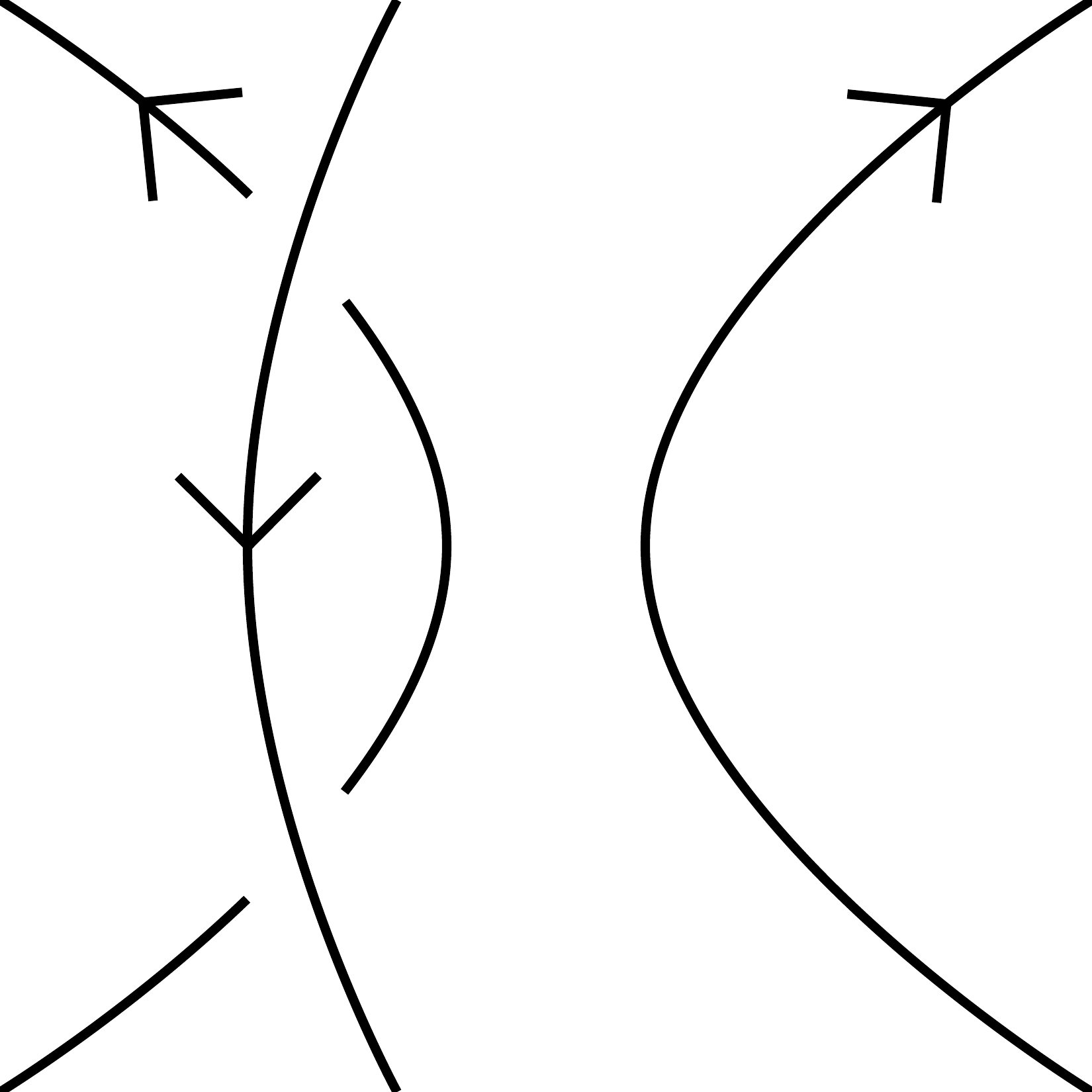}}\Biggr)-q^nP\Biggl(\raisebox{-10pt}{\includegraphics[height = .35in]{Generating-Sets/O4a-2.pdf}}\Biggr)\\
    &=q^{n-1}\left[q^{1-n}P\Biggl(\raisebox{-10pt}{\includegraphics[height = .35in]{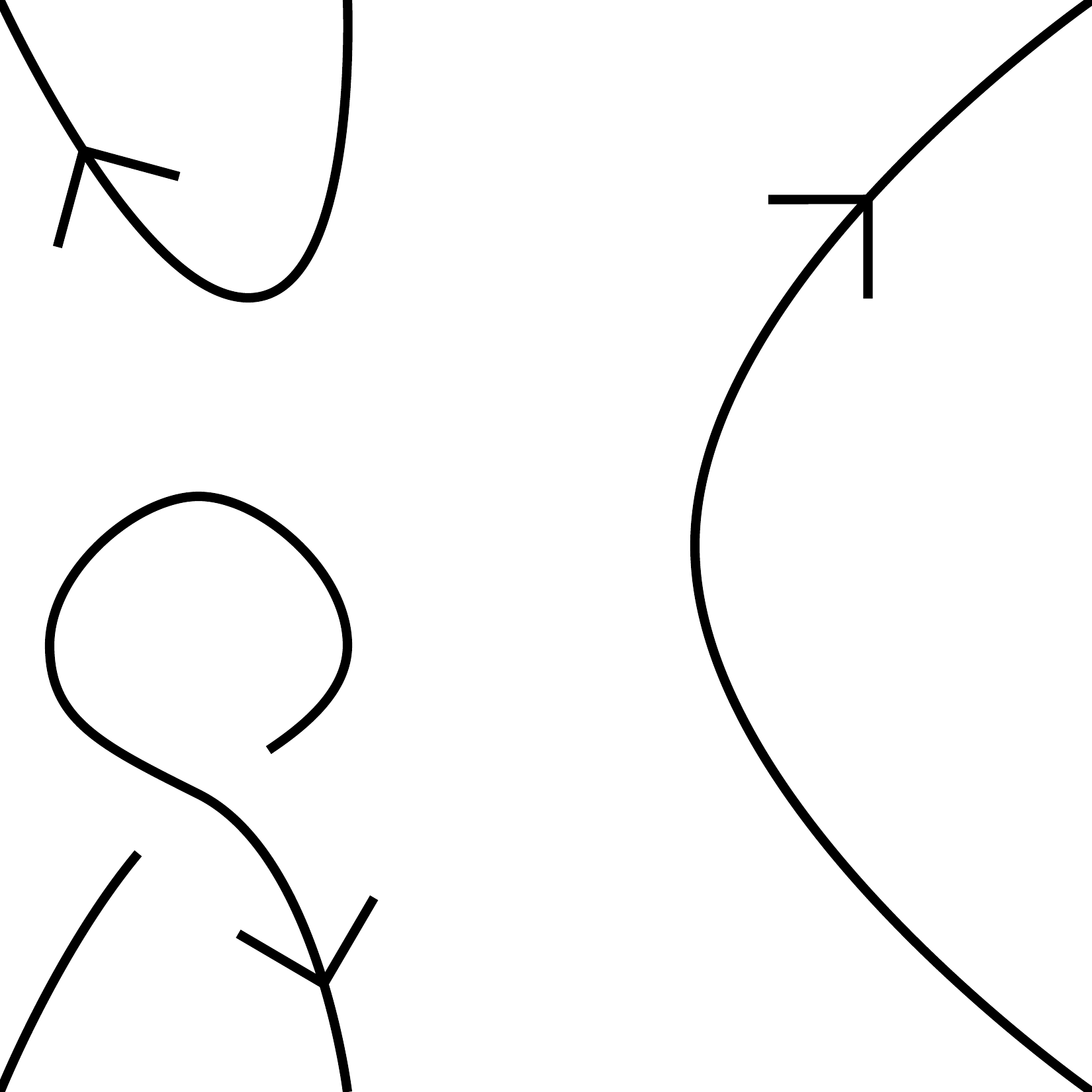}}\Biggr)-q^{-n}P\Biggl(\raisebox{-10pt}{\includegraphics[height = .35in]{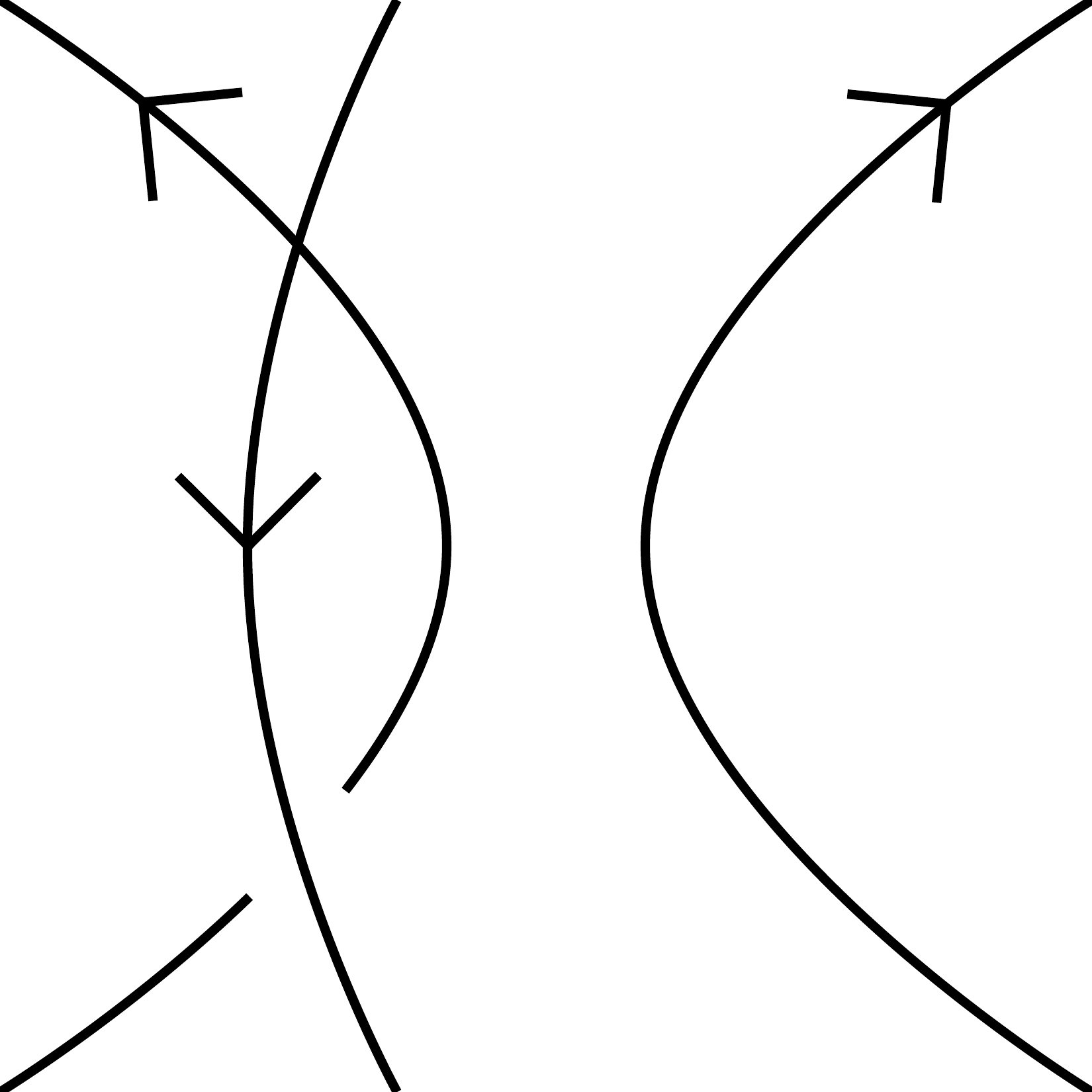}}\Biggr)\right]-q^nP\Biggl(\raisebox{-10pt}{\includegraphics[height = .35in]{Generating-Sets/O4a-2.pdf}}\Biggr)\\
    &=P\Biggl(\raisebox{-10pt}{\includegraphics[height = .35in]{O4-Proof/O4a-SR7.pdf}}\Biggr)-q^{-1}\left[q^{n-1}P\Biggl(\raisebox{-10pt}{\includegraphics[height = .35in]{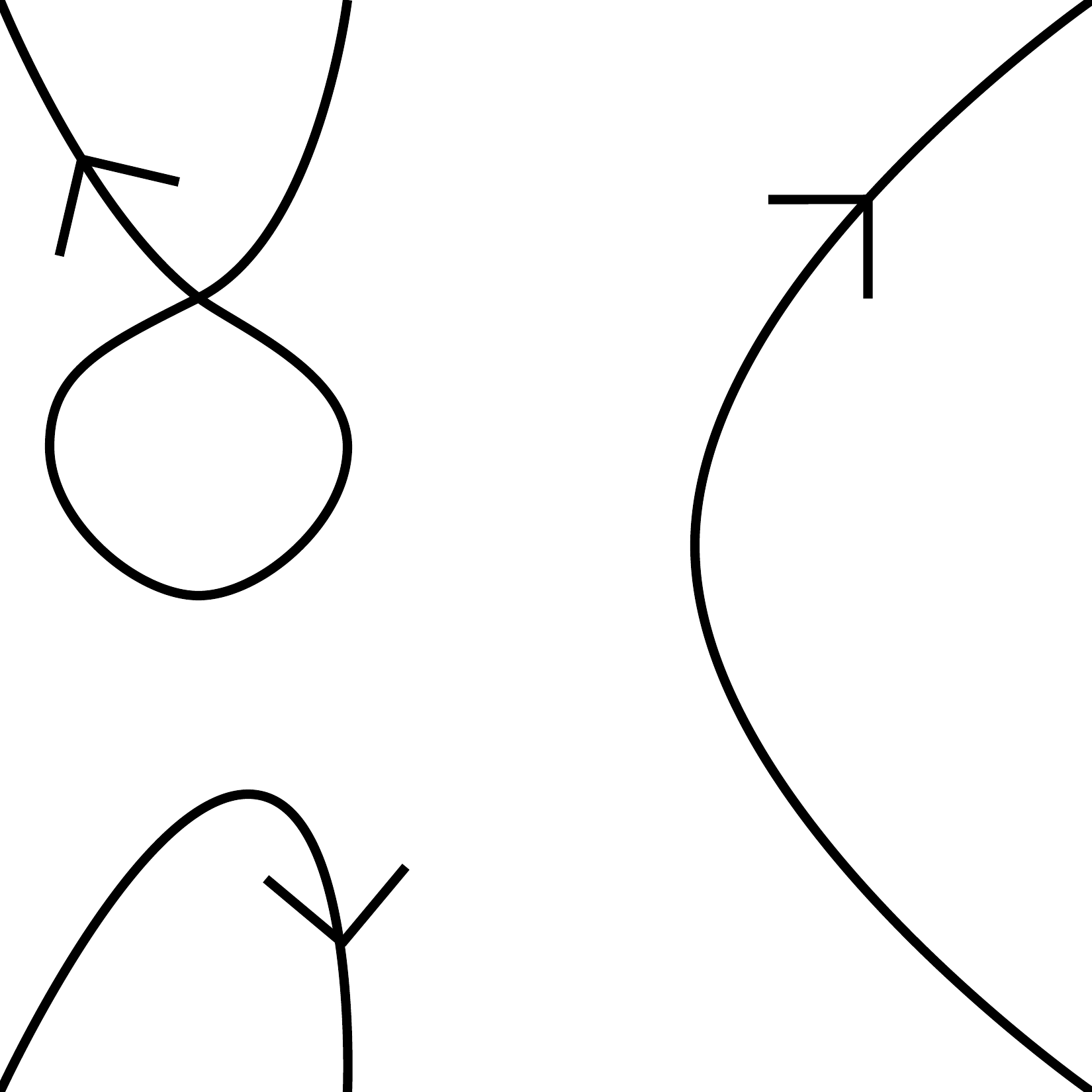}}\Biggr)-q^{n}P\Biggl(\raisebox{-10pt}{\includegraphics[height = .35in]{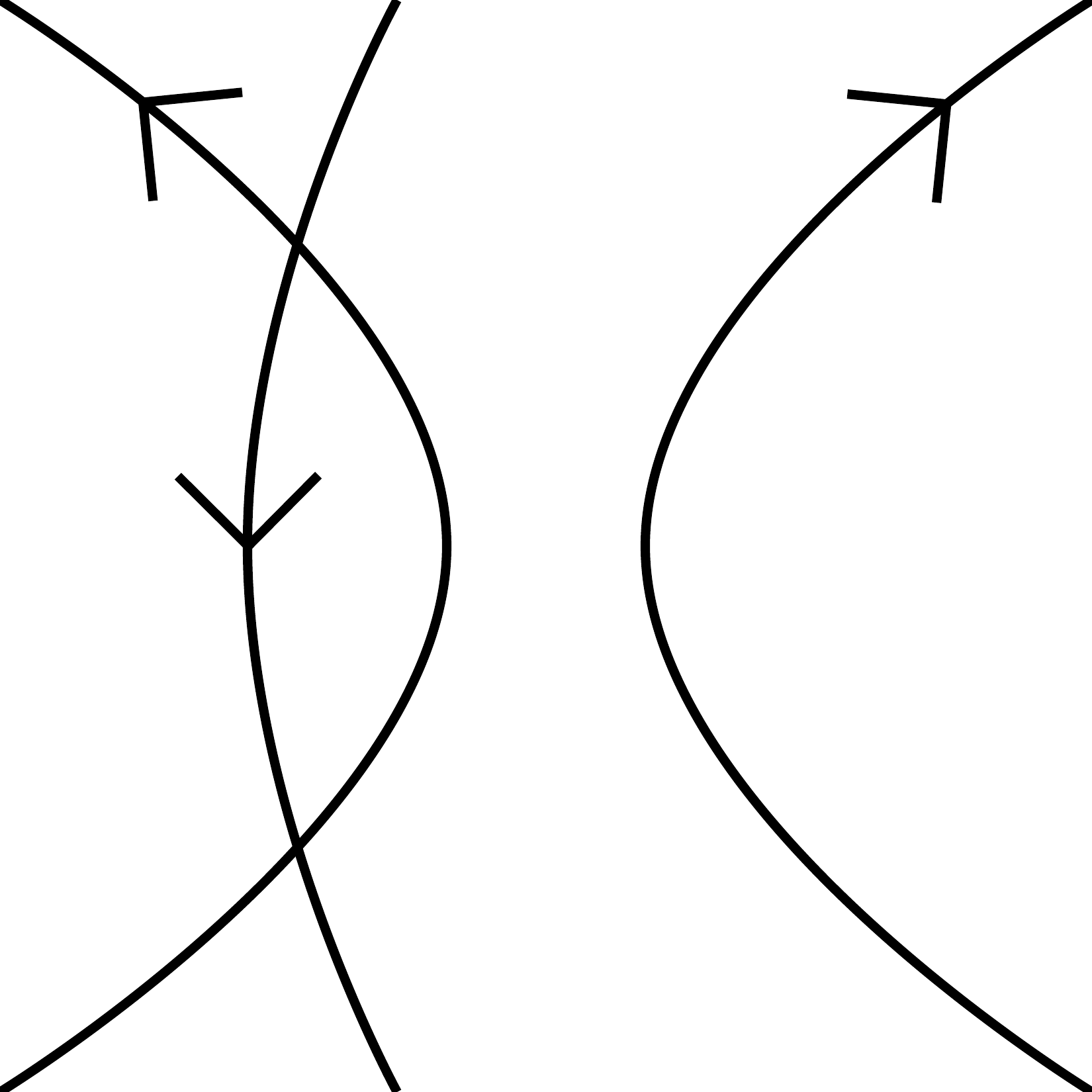}}\Biggr)\right]-q^nP\Biggl(\raisebox{-10pt}{\includegraphics[height = .35in]{Generating-Sets/O4a-2.pdf}}\Biggr)\\
    &=P\Biggl(\raisebox{-10pt}{\includegraphics[height = .35in]{O4-Proof/O4a-SR7.pdf}}\Biggr)-q^{n-2}[n-1]P\Biggl(\raisebox{-10pt}{\includegraphics[height = .35in]{O4-Proof/O4a-SR7.pdf}}\Biggr)+q^{n-1}P\Biggl(\raisebox{-10pt}{\includegraphics[height = .35in]{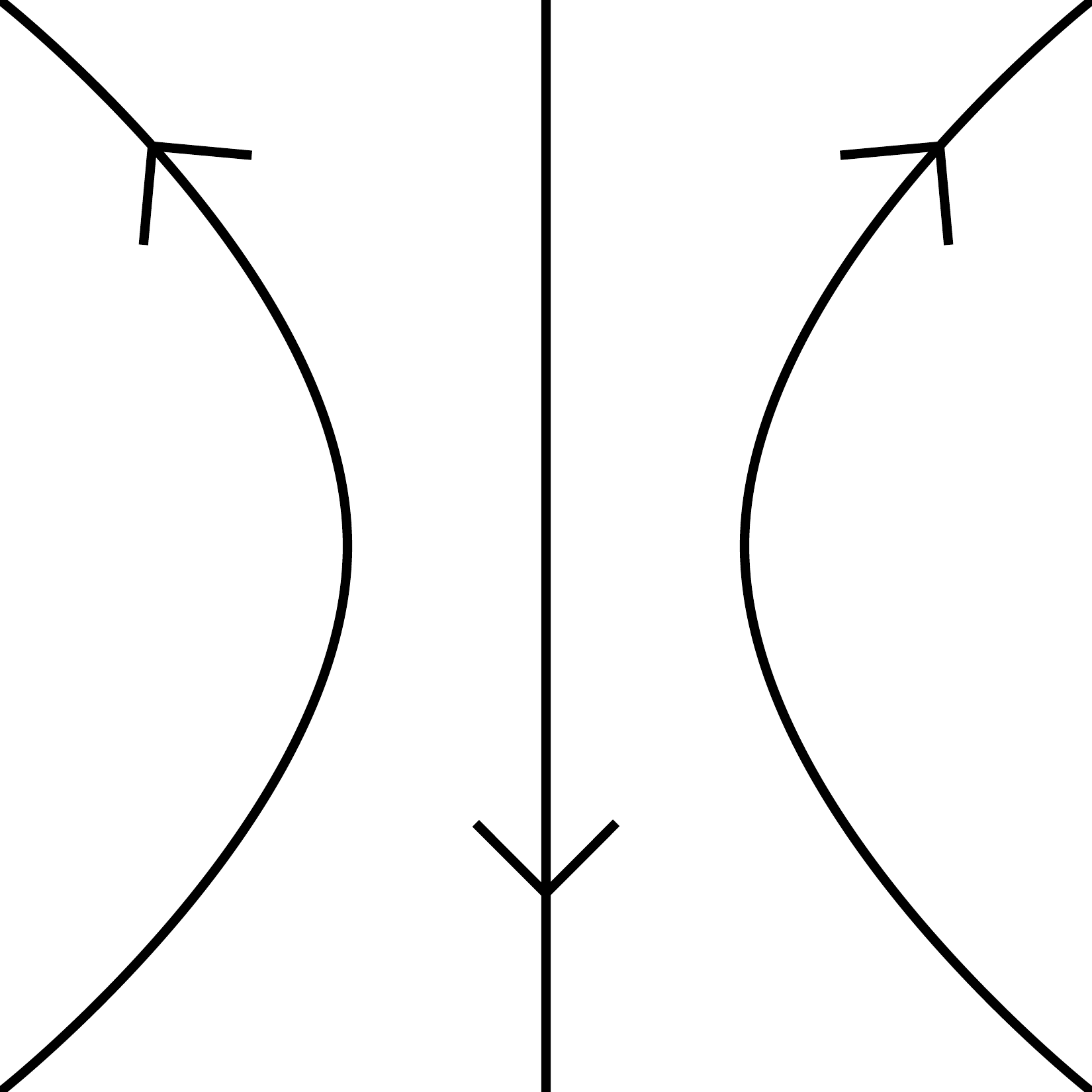}}\Biggr)\\&+q^{n-1}[n-2]P\Biggl(\raisebox{-10pt}{\includegraphics[height = .35in]{O4-Proof/O4a-SR7.pdf}}\Biggr)-q^nP\Biggl(\raisebox{-10pt}{\includegraphics[height = .35in]{Generating-Sets/O4a-2.pdf}}\Biggr)\\
    &=q^{n-1}P\Biggl(\raisebox{-10pt}{\includegraphics[height = .35in]{O3a-Proof/O3a-SR16.pdf}}\Biggr)-q^nP\Biggl(\raisebox{-10pt}{\includegraphics[height = .35in]{Generating-Sets/O4a-2.pdf}}\Biggr).
\end{align*}
In the third equality above, we used that $P$ is invariant under the move $\Omega1b$, while in the fourth equality we used the graphical relation involving a loop and that with an oriented bigon. In the last equality we used that $1-q^{n-2}[n-1]+q^{n-1}[n-2] = 0$.


We apply similar computations for the right hand-side of the move $\Omega 3a$, as shown below:

\begin{align*}
    P\Biggl(\raisebox{-10pt}{\includegraphics[height = .35in]{Generating-Sets/O3a-2.pdf}}\Biggr)&=q^{n-1}P\Biggl(\raisebox{-10pt}{\includegraphics[height = .35in]{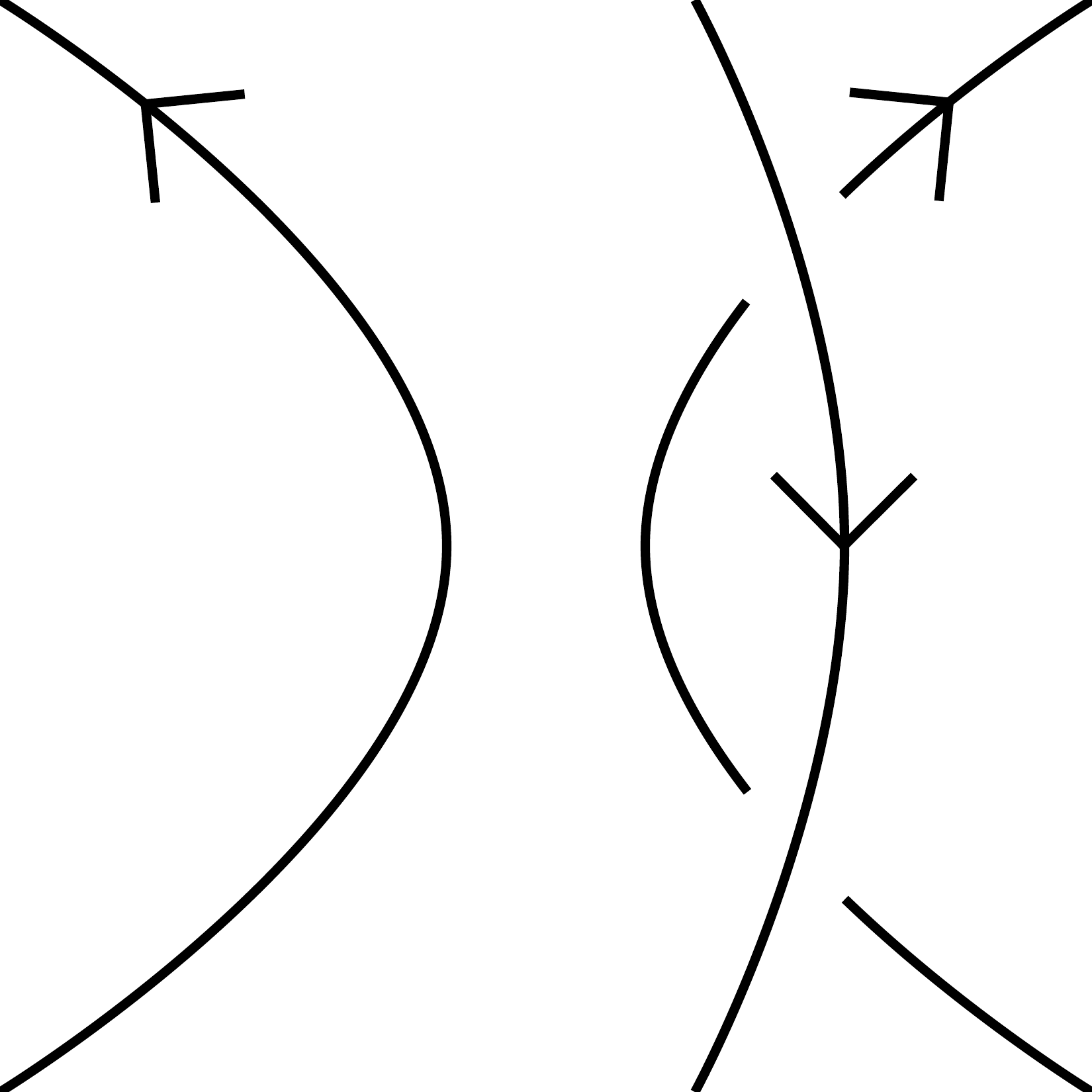}}\Biggr)-q^nP\Biggl(\raisebox{-10pt}{\includegraphics[height = .35in]{Generating-Sets/O4a-1.pdf}}\Biggr)\\
    &=q^{n-1}\left[q^{1-n}P\Biggl(\raisebox{-10pt}{\includegraphics[height = .35in]{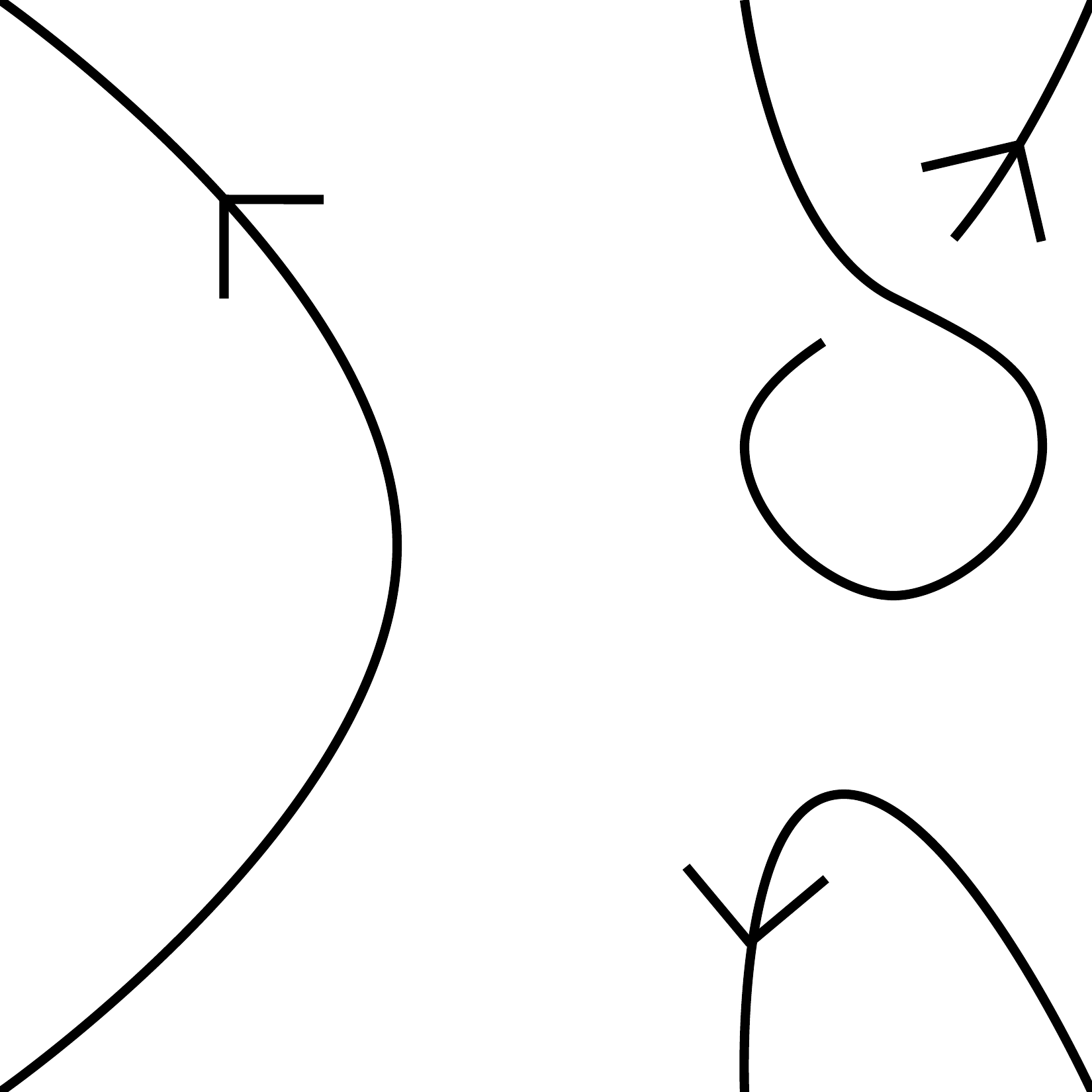}}\Biggr)-q^{-n}P\Biggl(\raisebox{-10pt}{\includegraphics[height = .35in]{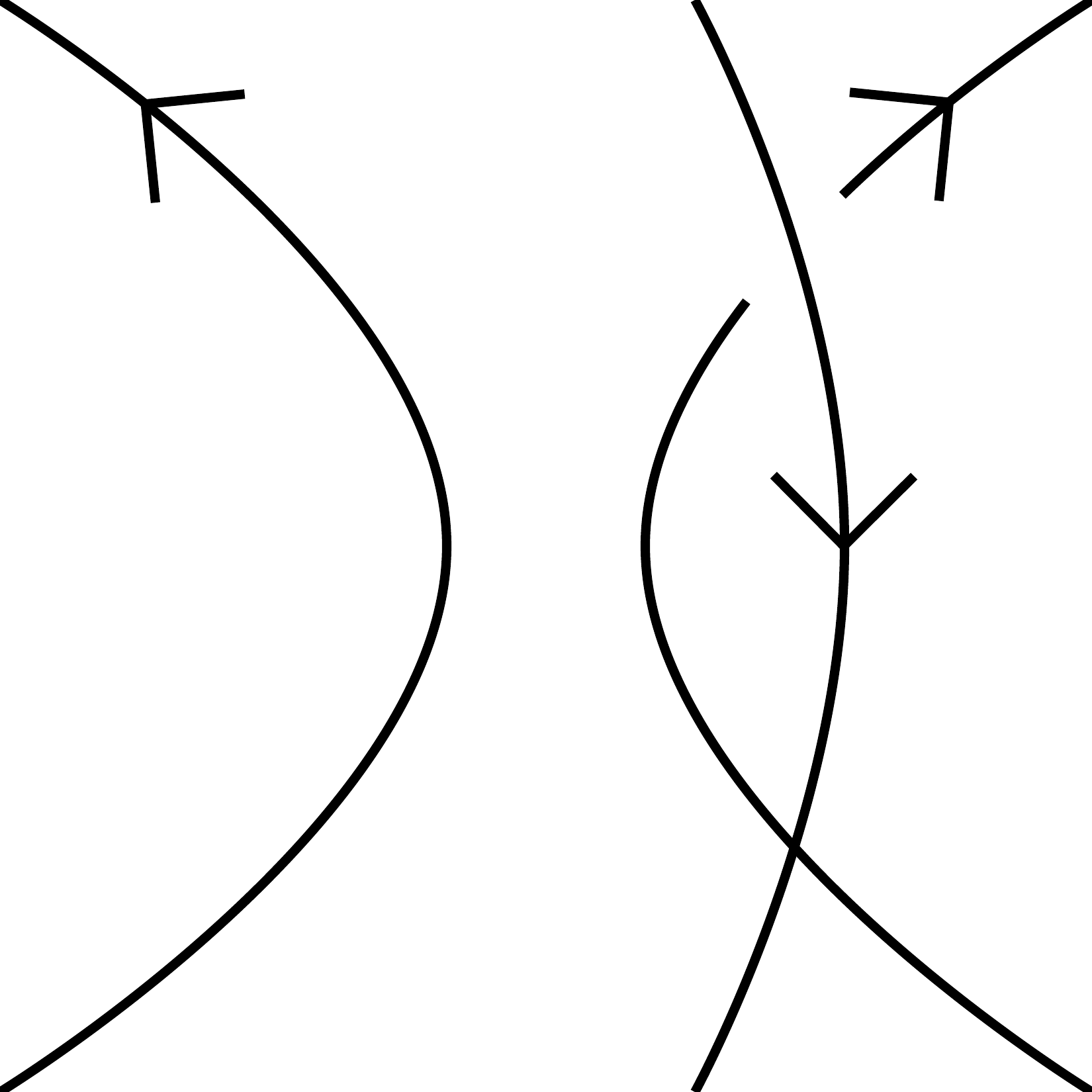}}\Biggr)\right]-q^nP\Biggl(\raisebox{-10pt}{\includegraphics[height = .35in]{Generating-Sets/O4a-1.pdf}}\Biggr)\\
    &=P\Biggl(\raisebox{-10pt}{\includegraphics[height = .35in]{O4-Proof/O4a-SR16.pdf}}\Biggr)-q^{-1}\left[q^{n-1}P\Biggl(\raisebox{-10pt}{\includegraphics[height = .35in]{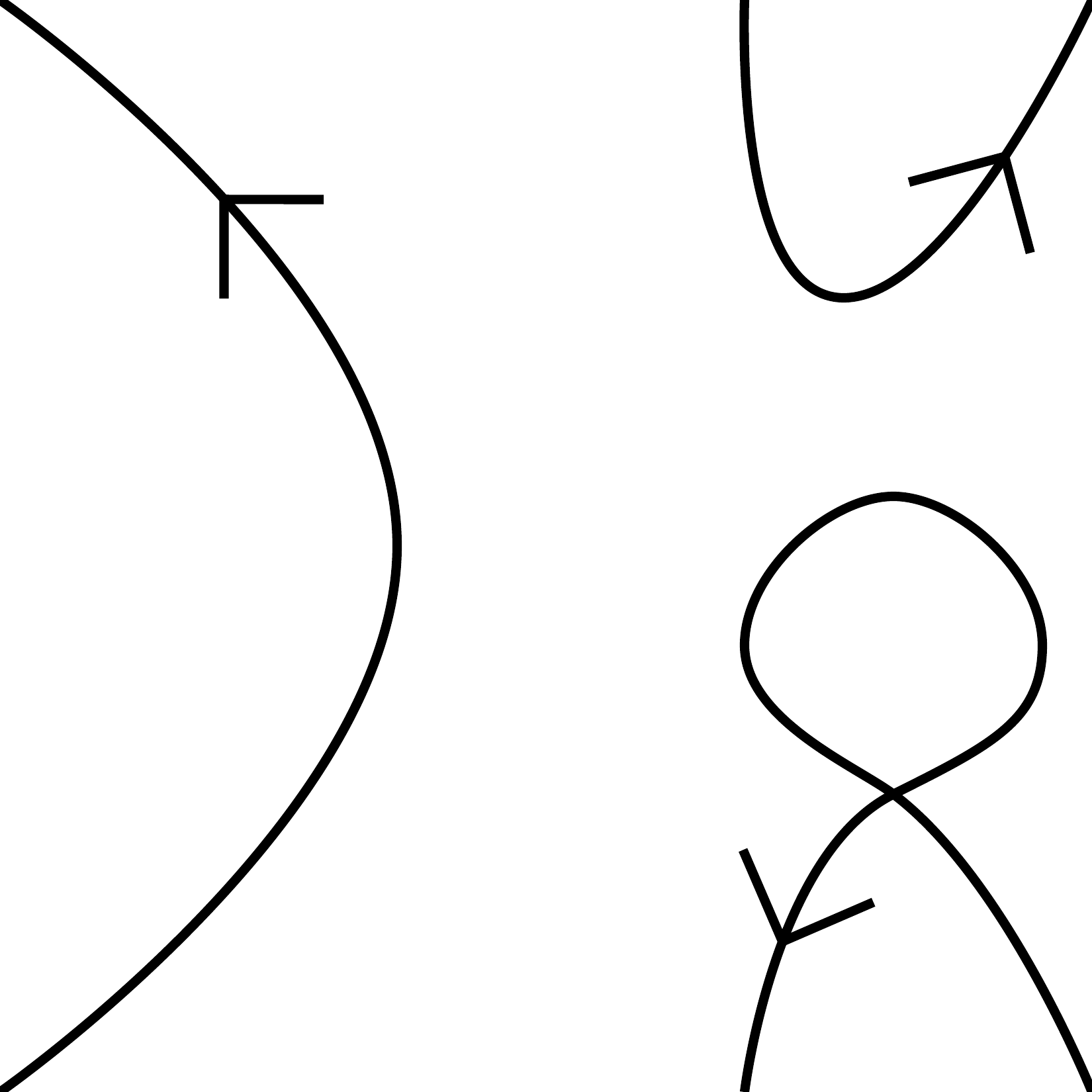}}\Biggr)-q^{n}P\Biggl(\raisebox{-10pt}{\includegraphics[height = .35in]{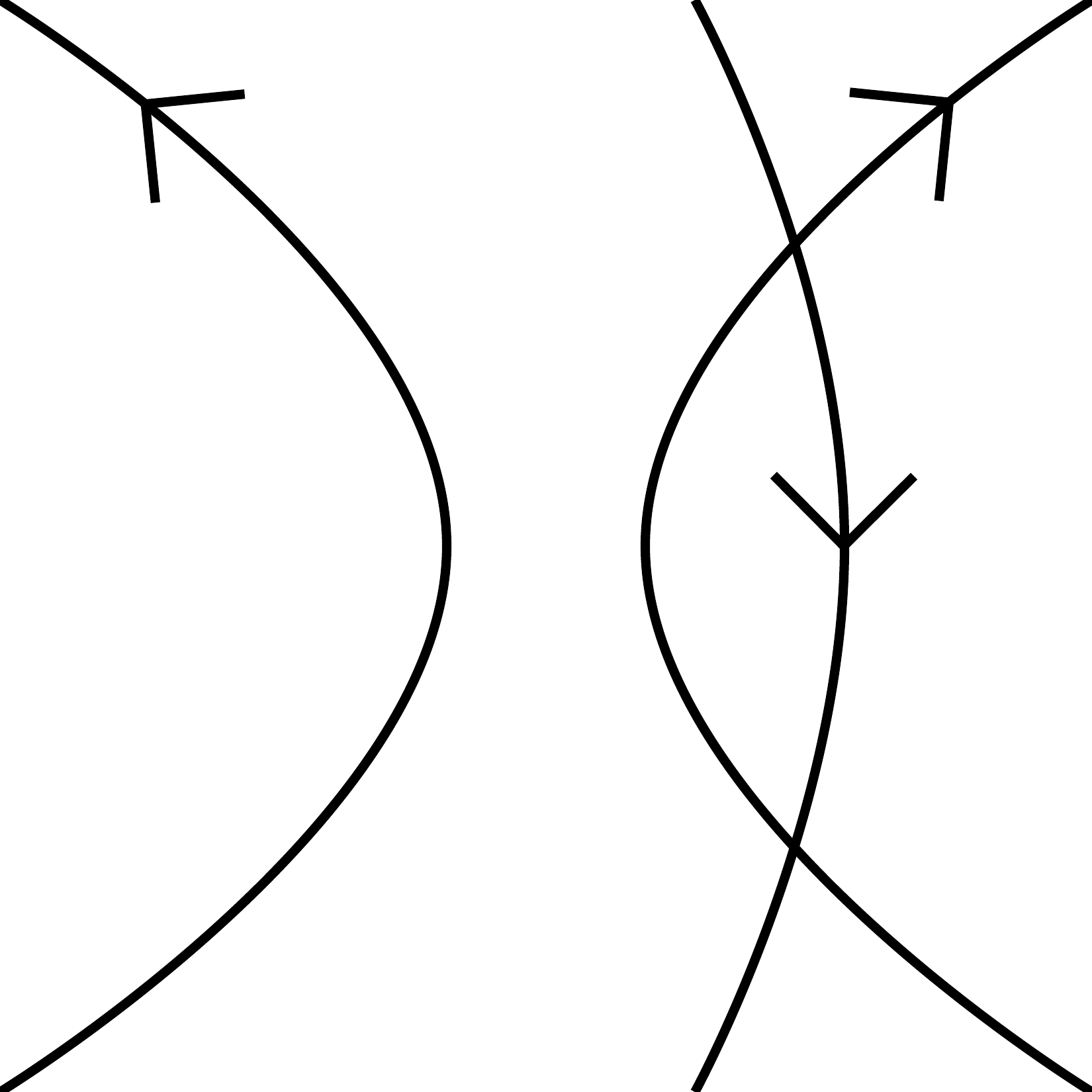}}\Biggr)\right]-q^nP\Biggl(\raisebox{-10pt}{\includegraphics[height = .35in]{Generating-Sets/O4a-1.pdf}}\Biggr).
\end{align*}
Using the graphical relations involving the loop and oriented bigon, we have the following:
\begin{align*}
    P\Biggl(\raisebox{-10pt}{\includegraphics[height = .35in]{Generating-Sets/O3a-2.pdf}}\Biggr)&=P\Biggl(\raisebox{-10pt}{\includegraphics[height = .35in]{O4-Proof/O4a-SR16.pdf}}\Biggr)-q^{n-2}[n-1]P\Biggl(\raisebox{-10pt}{\includegraphics[height = .35in]{O4-Proof/O4a-SR16.pdf}}\Biggr)+q^{n-1}P\Biggl(\raisebox{-10pt}{\includegraphics[height = .35in]{O3a-Proof/O3a-SR16.pdf}}\Biggr)\\&+q^{n-1}[n-2]P\Biggl(\raisebox{-10pt}{\includegraphics[height = .35in]{O4-Proof/O4a-SR16.pdf}}\Biggr)-q^nP\Biggl(\raisebox{-10pt}{\includegraphics[height = .35in]{Generating-Sets/O4a-1.pdf}}\Biggr)\\
    &=q^{n-1}P\Biggl(\raisebox{-10pt}{\includegraphics[height = .35in]{O3a-Proof/O3a-SR16.pdf}}\Biggr)-q^nP\Biggl(\raisebox{-10pt}{\includegraphics[height = .35in]{Generating-Sets/O4a-1.pdf}}\Biggr)\\
    &=q^{n-1}P\Biggl(\raisebox{-10pt}{\includegraphics[height = .35in]{O3a-Proof/O3a-SR16.pdf}}\Biggr)-q^nP\Biggl(\raisebox{-10pt}{\includegraphics[height = .35in]{Generating-Sets/O4a-2.pdf}}\Biggr).
\end{align*}
The last equality above holds, since we have proved that $P$ is invariant under the move $\Omega 4a$. Hence, $P$ is invariant under the move $\Omega3a$.
Next we prove the invariance of $P$ under the move $\Omega 4e$:
\begin{align*}
    P\Biggl(\raisebox{-10pt}{\includegraphics[height = .35in]{Generating-Sets/O4e-2.pdf}}\Biggr)&=q^{1-n}P\Biggl(\raisebox{-10pt}{\includegraphics[height = .35in]{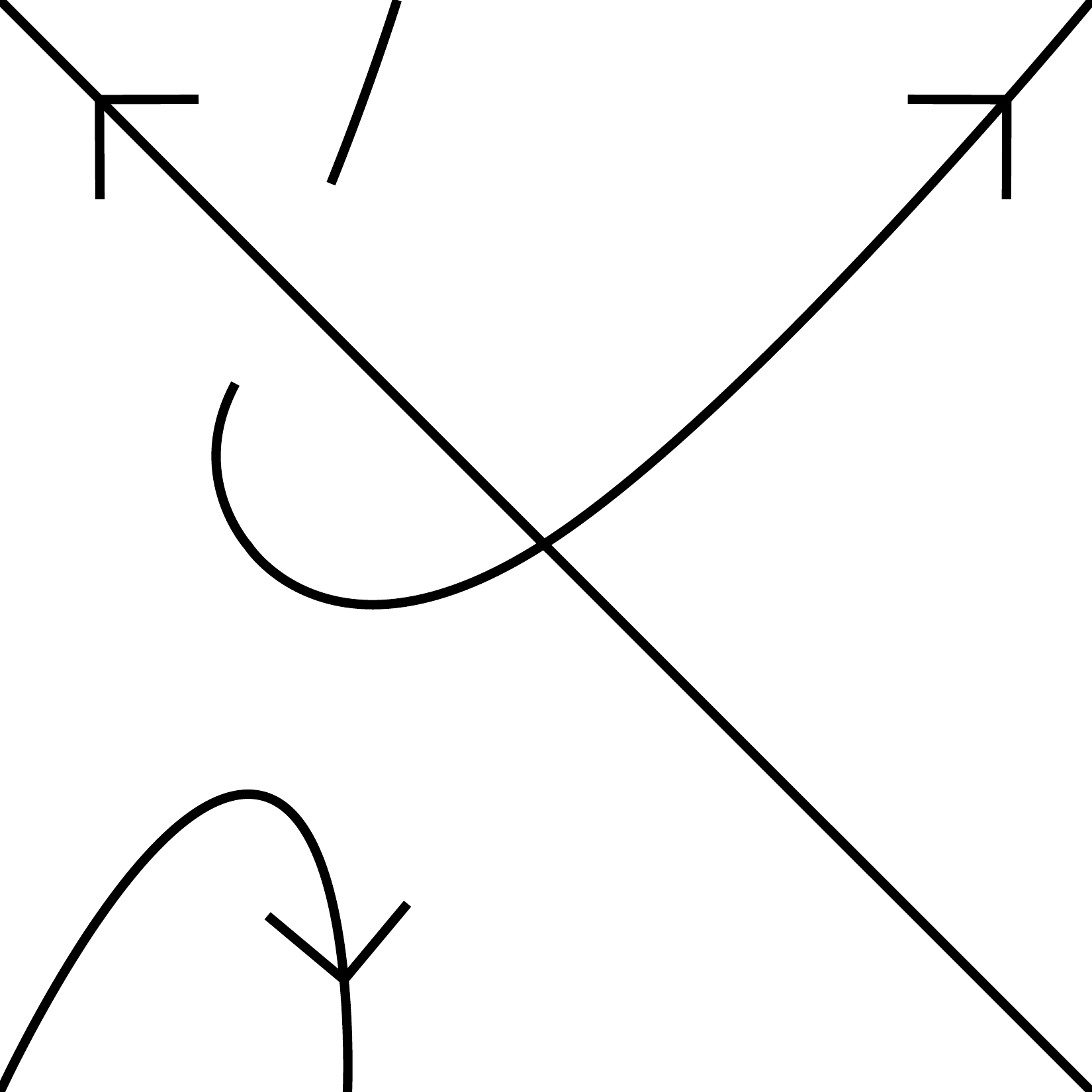}}\Biggr)-q^{-n}P\Biggl(\raisebox{-10pt}{\includegraphics[height = .35in]{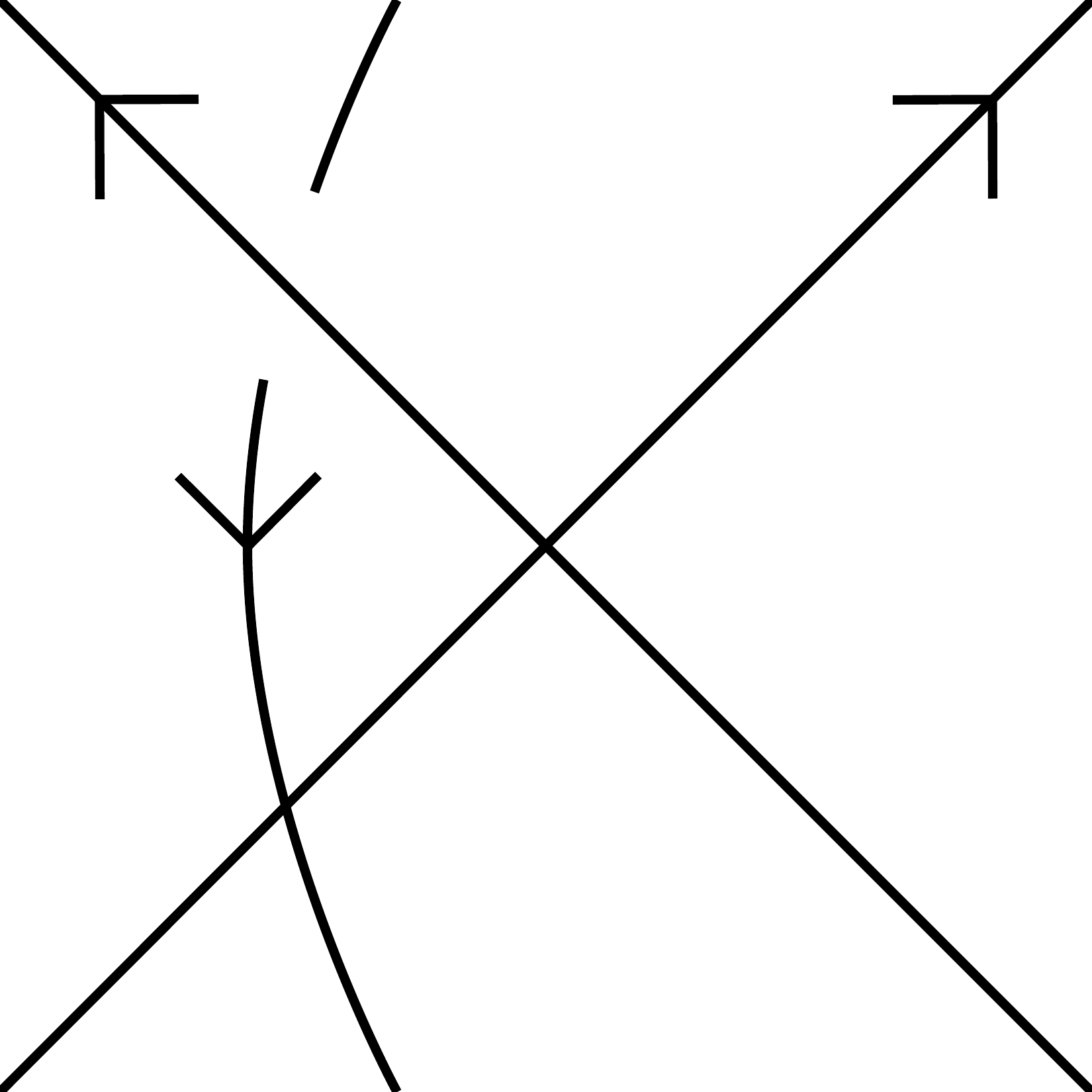}}\Biggr)\\
    &=q^{1-n}\left[q^{n-1}P\Biggl(\raisebox{-10pt}{\includegraphics[height = .35in]{O4-Proof/O4a-SR3-1.pdf}}\Biggr)-q^{n}P\Biggl(\raisebox{-10pt}{\includegraphics[height = .35in]{O4-Proof/O4a-SR4.pdf}}\Biggr)\right]\\&-q^{-n}\left[q^{n-1}P\Biggl(\raisebox{-10pt}{\includegraphics[height = .35in]{O4-Proof/O4a-SR5.pdf}}\Biggr)-q^{n}P\Biggl(\raisebox{-10pt}{\includegraphics[height = .35in]{O4-Proof/O4a-SR6.pdf}}\Biggr)\right]\\
    &=([n-1]-q[n-2]-q^{-1}[n-2])P\Biggl(\raisebox{-10pt}{\includegraphics[height = .35in]{O4-Proof/O4a-SR7.pdf}}\Biggr)-qP\Biggl(\raisebox{-10pt}{\includegraphics[height = .35in]{O4-Proof/O4a-SR8.pdf}}\Biggr)\\&-q^{-1}P\Biggl(\raisebox{-10pt}{\includegraphics[height = .35in]{O4-Proof/O4a-SR9.pdf}}\Biggr)+P\Biggl(\raisebox{-10pt}{\includegraphics[height = .35in]{O4-Proof/O4a-SR6.pdf}}\Biggr)\\
    &= -[n-3]P\Biggl(\raisebox{-10pt}{\includegraphics[height = .35in]{O4-Proof/O4a-SR7.pdf}}\Biggr)+P\Biggl(\raisebox{-10pt}{\includegraphics[height = .35in]{O4-Proof/O4a-SR6.pdf}}\Biggr)-qP\Biggl(\raisebox{-10pt}{\includegraphics[height = .35in]{O4-Proof/O4a-SR8.pdf}}\Biggr)-q^{-1}P\Biggl(\raisebox{-10pt}{\includegraphics[height = .35in]{O4-Proof/O4a-SR9.pdf}}\Biggr).
\end{align*}
Similarly,
\begin{align*}
    P\Biggl(\raisebox{-10pt}{\includegraphics[height = .35in]{Generating-Sets/O4e-1.pdf}}\Biggr)&=q^{n-1}P\Biggl(\raisebox{-10pt}{\includegraphics[height = .35in]{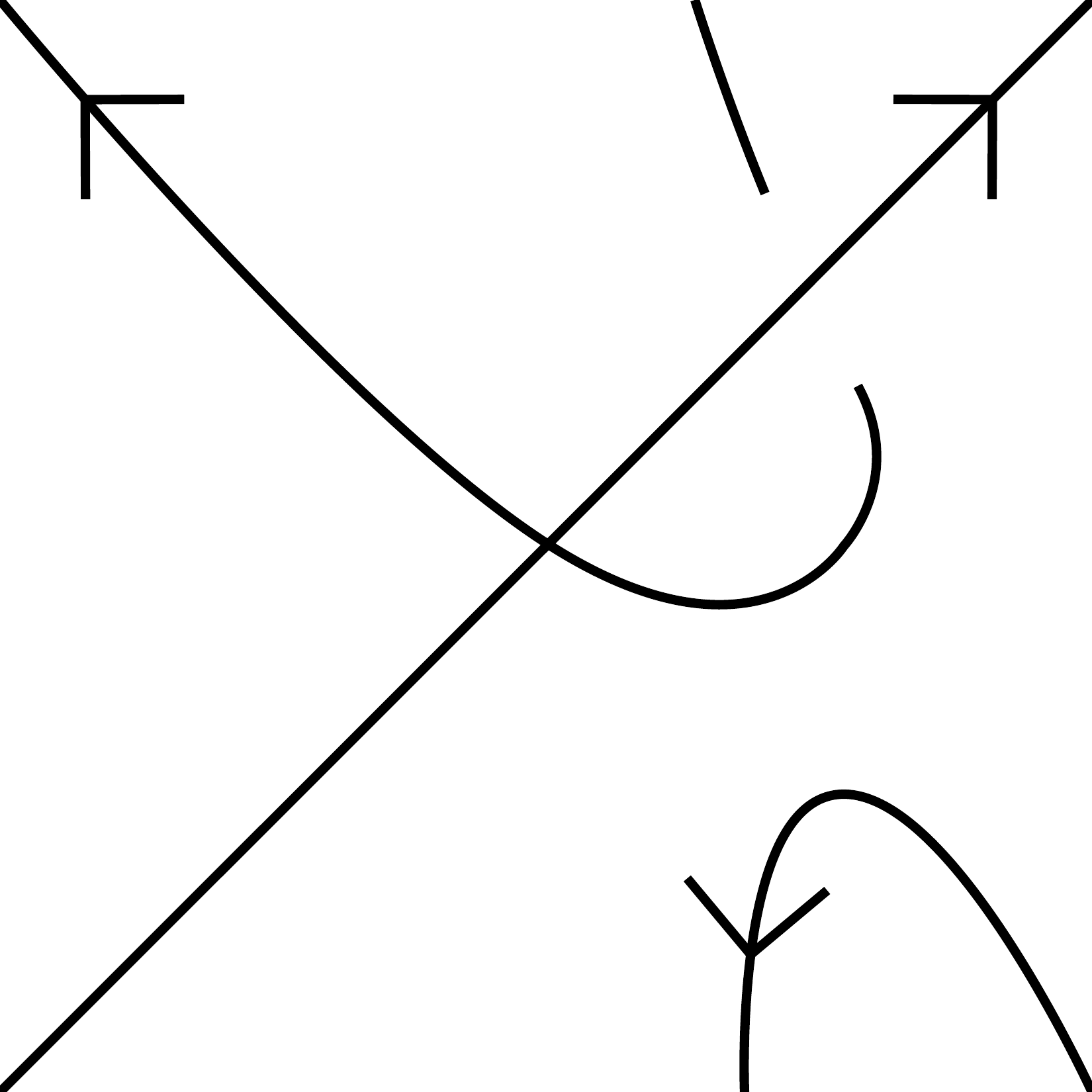}}\Biggr)-q^{n}P\Biggl(\raisebox{-10pt}{\includegraphics[height = .35in]{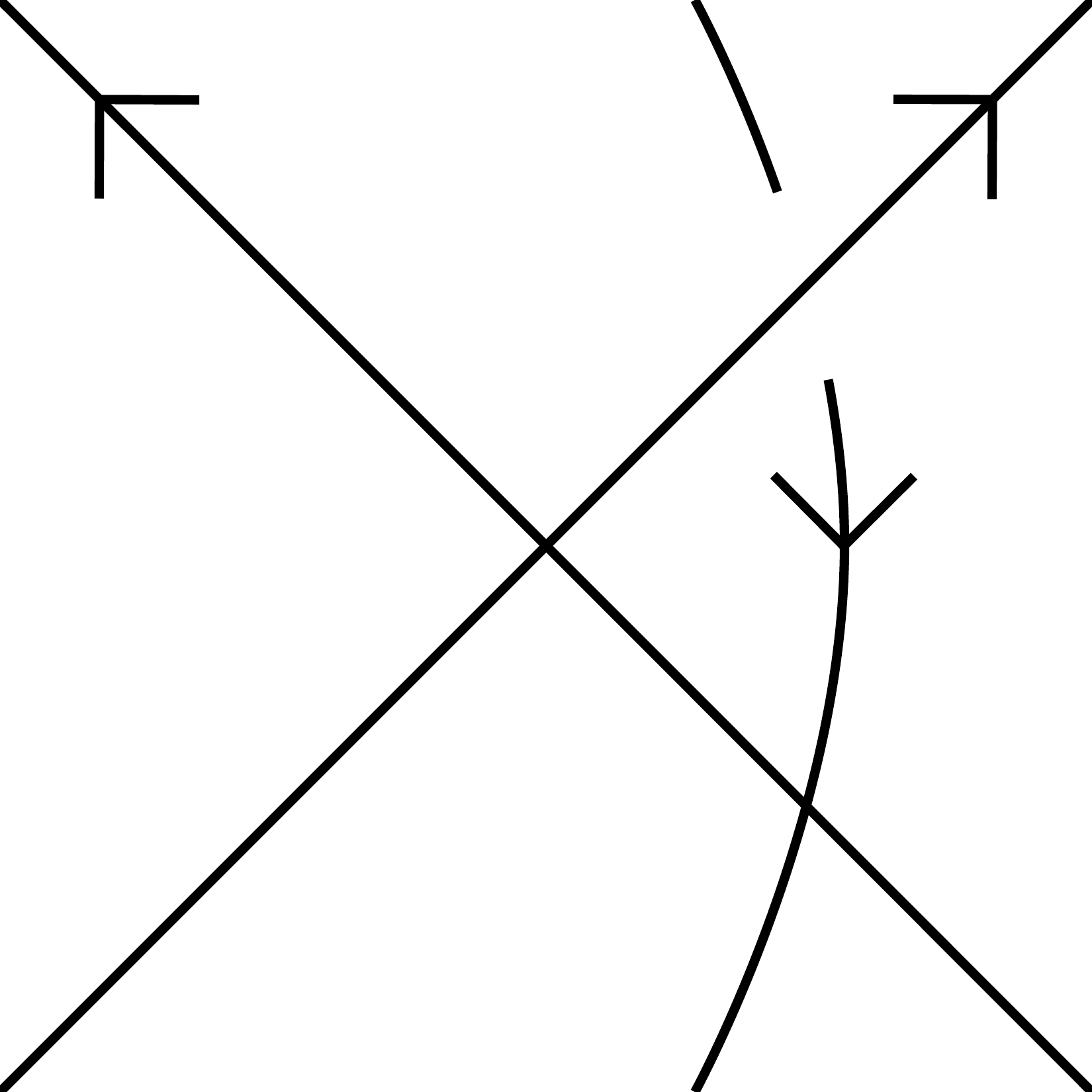}}\Biggr)\\
    &=q^{n-1}\left[q^{1-n}P\Biggl(\raisebox{-10pt}{\includegraphics[height = .35in]{O4-Proof/O4a-SR12-1.pdf}}\Biggr)-q^{-n}P\Biggl(\raisebox{-10pt}{\includegraphics[height = .35in]{O4-Proof/O4a-SR13.pdf}}\Biggr)\right]\\&-q^{n}\left[q^{1-n}P\Biggl(\raisebox{-10pt}{\includegraphics[height = .35in]{O4-Proof/O4a-SR14.pdf}}\Biggr)-q^{-n}P\Biggl(\raisebox{-10pt}{\includegraphics[height = .35in]{O4-Proof/O4a-SR15.pdf}}\Biggr)\right].
    \end{align*}
    Using again the graphical relations for the loop and the oriented bigon, combined with adding like terms, we arrive at the following,
    \begin{align*}
    P\Biggl(\raisebox{-10pt}{\includegraphics[height = .35in]{Generating-Sets/O4e-1.pdf}}\Biggr)&=([n-1]-q^{-1}[n-2]-q[n-2])P\Biggl(\raisebox{-10pt}{\includegraphics[height = .35in]{O4-Proof/O4a-SR16.pdf}}\Biggr)-q^{-1}P\Biggl(\raisebox{-10pt}{\includegraphics[height = .35in]{O4-Proof/O4a-SR9.pdf}}\Biggr)\\&-qP\Biggl(\raisebox{-10pt}{\includegraphics[height = .35in]{O4-Proof/O4a-SR8.pdf}}\Biggr)+P\Biggl(\raisebox{-10pt}{\includegraphics[height = .35in]{O4-Proof/O4a-SR15.pdf}}\Biggr)\\
     &=- [n-3]P\Biggl(\raisebox{-10pt}{\includegraphics[height = .35in]{O4-Proof/O4a-SR16.pdf}}\Biggr)+P\Biggl(\raisebox{-10pt}{\includegraphics[height = .35in]{O4-Proof/O4a-SR15.pdf}}\Biggr)-q^{-1}P\Biggl(\raisebox{-10pt}{\includegraphics[height = .35in]{O4-Proof/O4a-SR9.pdf}}\Biggr)-qP\Biggl(\raisebox{-10pt}{\includegraphics[height = .35in]{O4-Proof/O4a-SR8.pdf}}\Biggr)\\
    &=-[n-3]P\Biggl(\raisebox{-10pt}{\includegraphics[height = .35in]{O4-Proof/O4a-SR7.pdf}}\Biggr)+P\Biggl(\raisebox{-10pt}{\includegraphics[height = .35in]{O4-Proof/O4a-SR6.pdf}}\Biggr)-qP\Biggl(\raisebox{-10pt}{\includegraphics[height = .35in]{O4-Proof/O4a-SR8.pdf}}\Biggr)-q^{-1}P\Biggl(\raisebox{-10pt}{\includegraphics[height = .35in]{O4-Proof/O4a-SR9.pdf}}\Biggr),
\end{align*}
where the last equality above holds due to the last graphical relation in Figure \ref{fig:GSR}. Therefore, 
\[P\Biggl(\raisebox{-10pt}{\includegraphics[height = .35in]{Generating-Sets/O4e-2.pdf}}\Biggr) = P\Biggl(\raisebox{-10pt}{\includegraphics[height = .35in]{Generating-Sets/O4e-1.pdf}}\Biggr). \]

 The invariance of the polynomial $P$ under the move $\Omega5a$ holds, as proved below.
\begin{align*}
    P\Biggl(\raisebox{-10pt}{\includegraphics[height = .35in]{Generating-Sets/O5a-1.pdf}}\Biggr)&=q^{n-1}P\Biggl(\raisebox{-10pt}{\includegraphics[height = .35in]{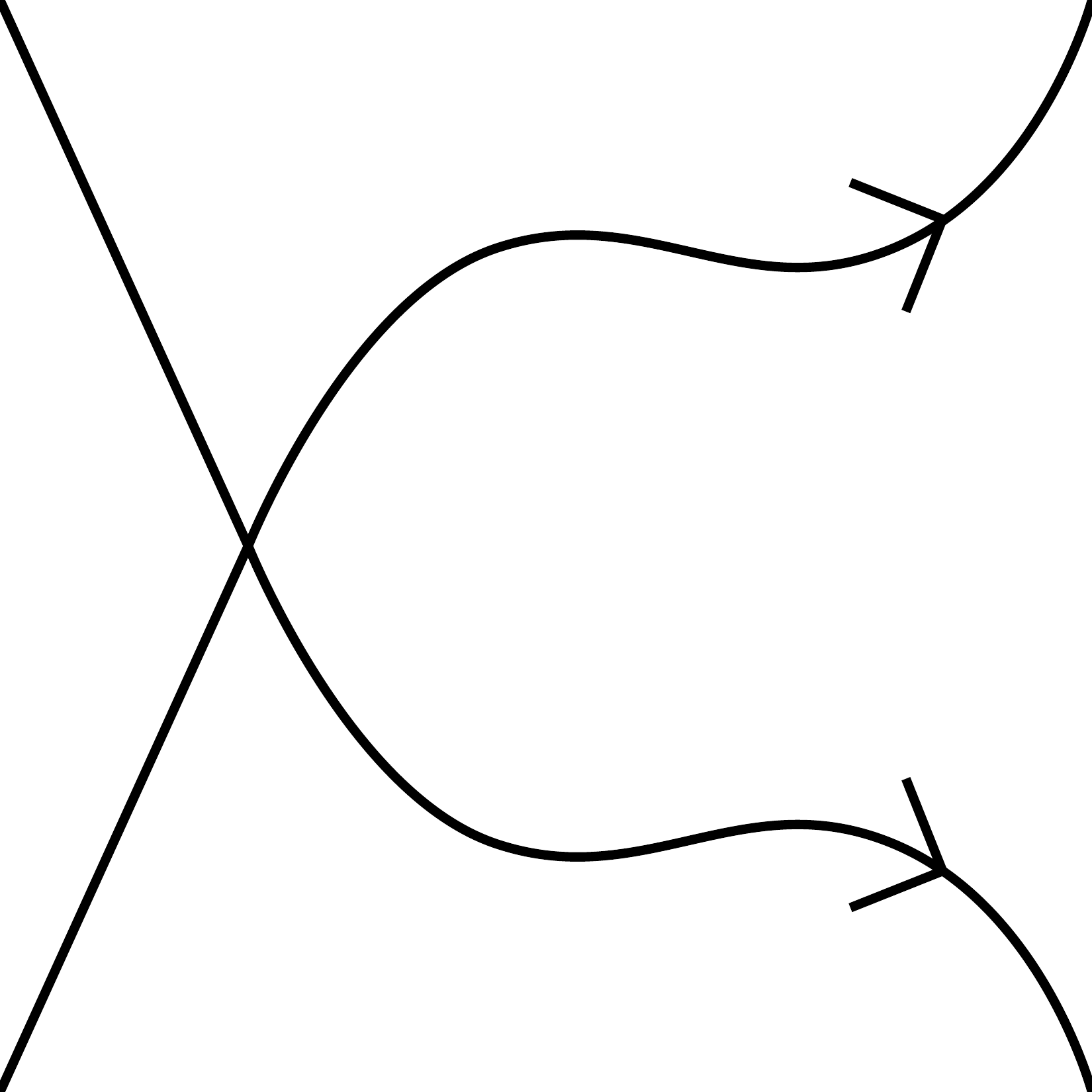}}\Biggr)-q^nP\Biggl(\raisebox{-10pt}{\includegraphics[height = .35in]{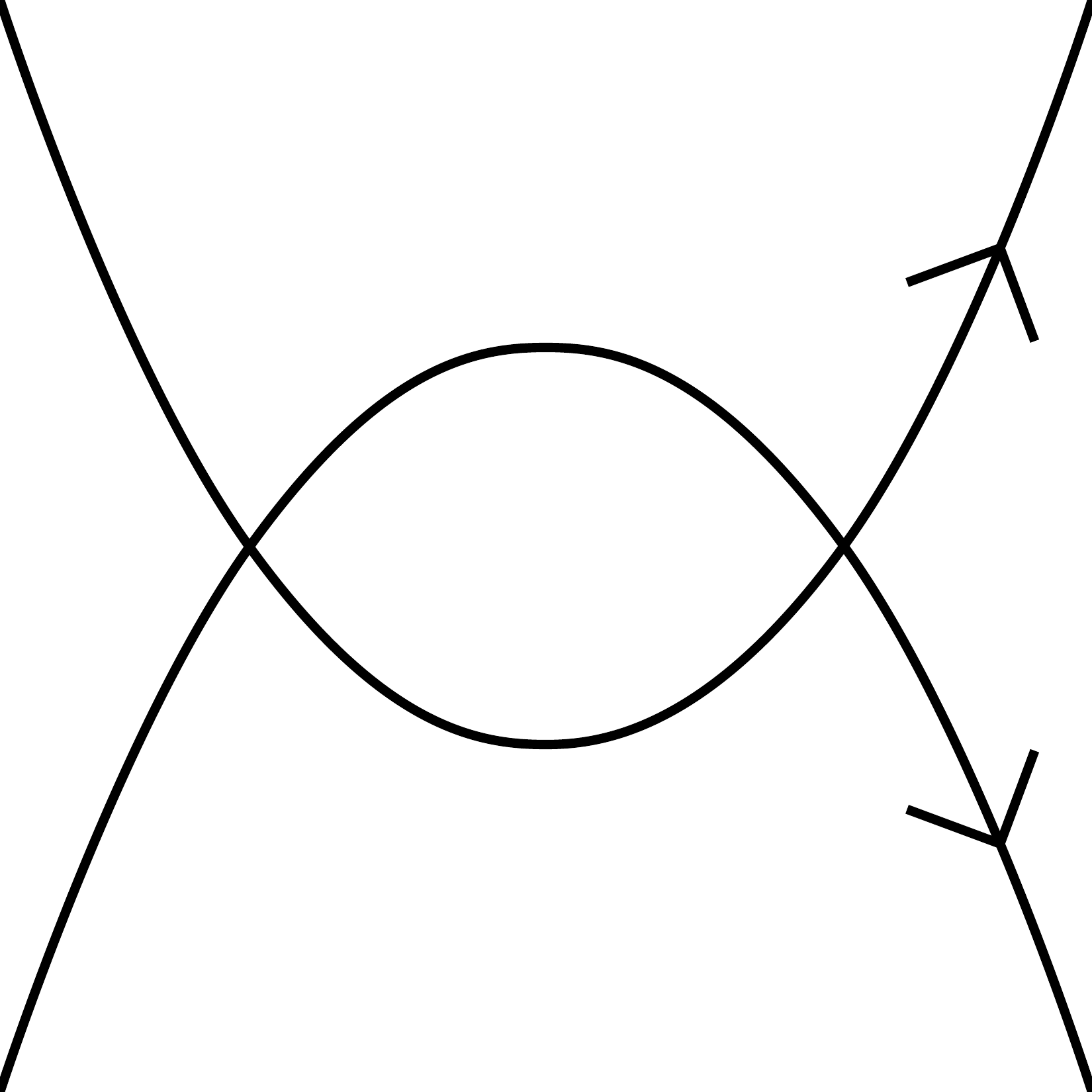}}\Biggr)\\
    &=q^{n-1}P\Biggl(\raisebox{-10pt}{\includegraphics[height = .35in]{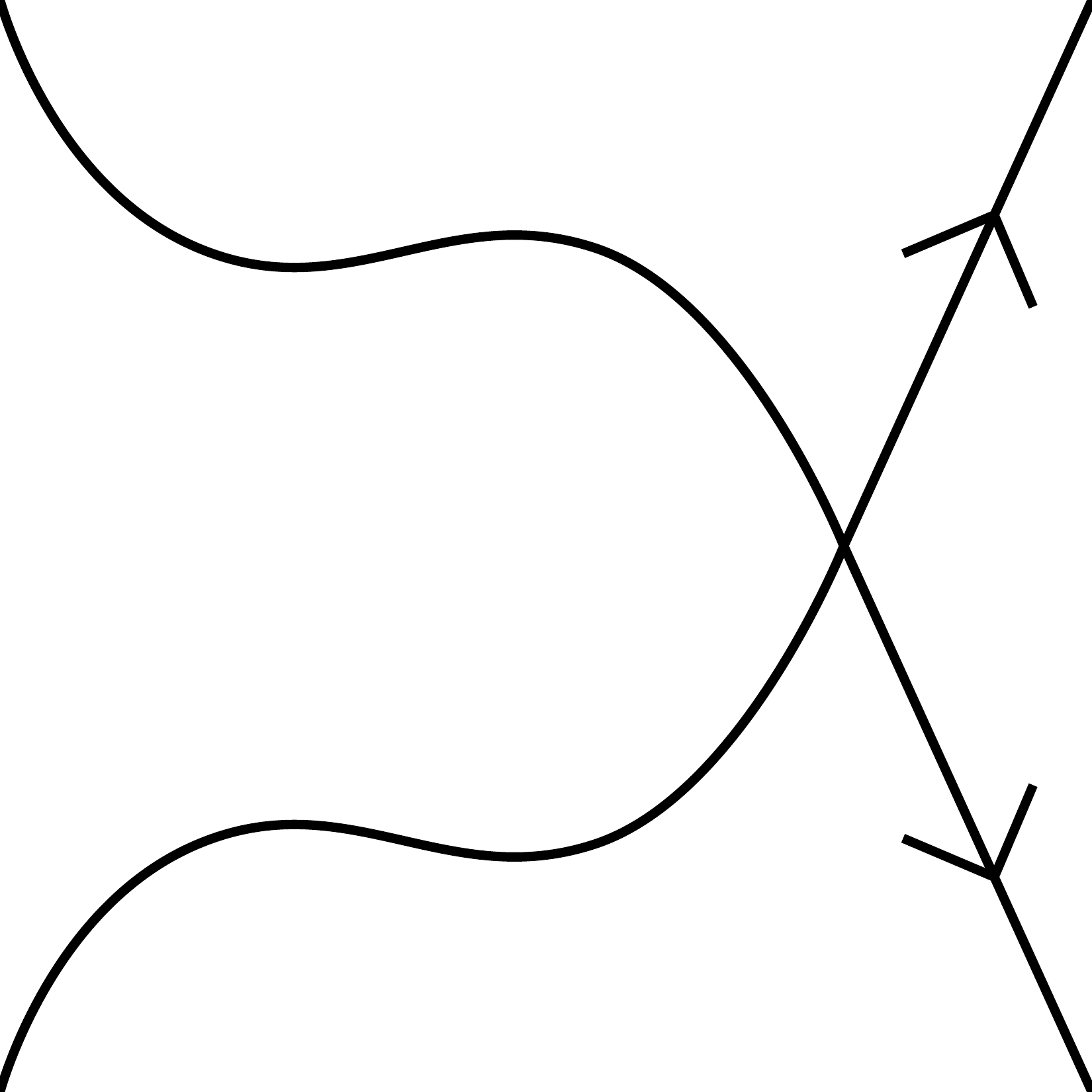}}\Biggr)-q^nP\Biggl(\raisebox{-10pt}{\includegraphics[height = .35in]{O5-Proof/O5a-SR2.pdf}}\Biggr)\\
    &=P\Biggl(\raisebox{-10pt}{\includegraphics[height = .35in]{Generating-Sets/O5a-2.pdf}}\Biggr).
\end{align*}

Since we have proved that $P$ is invariant under the moves $\Omega 1a, \Omega 1b, \Omega 2a$ and $\Omega 3a$, which is a generating set of all oriented versions of the Reidemeister moves, we have that $P$ is invariant under all of the moves $\Omega 1, \Omega 2$ and $\Omega 3$. Moreover, since we also proved that $P$ is invariant under the moves $\Omega 4a, \Omega 4e$ and $\Omega 5a$, we can conclude that $P$ is invariant under all the moves $\Omega 4$ and $\Omega 5$ involving 4-valent vertices of type In-In-Out-Out.

We proceed now to prove the invariance of $P$ under  the moves $\Omega 4$ involving a vertex of type In-Out-In-Out that are in our generating set; these are the moves $\Omega4j$ and $\Omega4l$. We show the proof of the invariance under the move $\Omega 4j$ only, as the proof of invariance under the move $\Omega4l$ follow similarly. In the first step below, we apply the skein relation to resolve the vertex. In the second and third steps we use that $P$ is invariant under the moves $\Omega 2$. In the last equality, we use the skein relation involving a vertex of type In-Out-In-Out, this time reading the relation backwards.
\begin{align*}
    P\Biggl(\raisebox{-10pt}{\includegraphics[height = .35in]{Generating-Sets/O4j-1.pdf}}\Biggr)&=P\Biggl(\raisebox{-10pt}{\includegraphics[height = .35in]{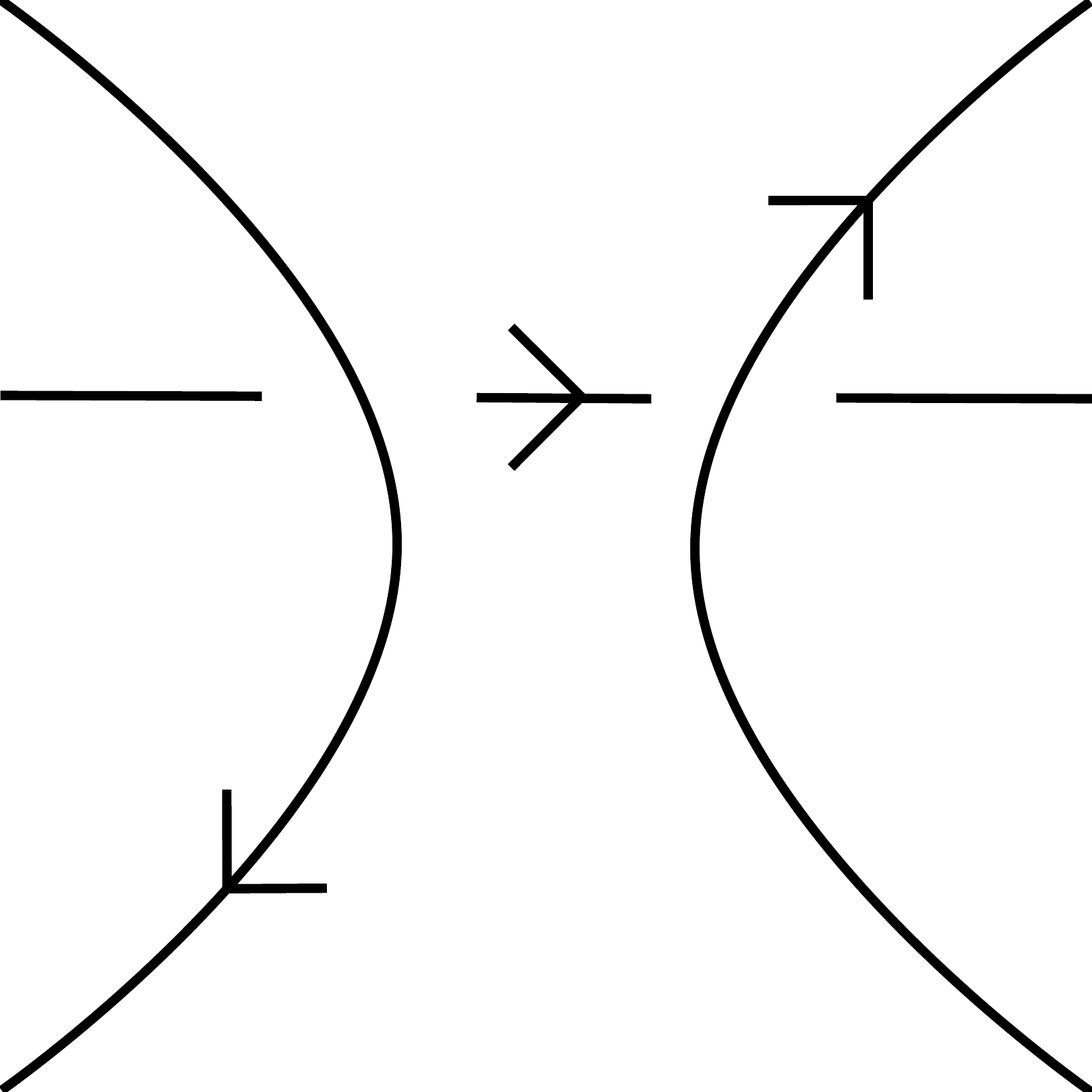}}\Biggr)+P\Biggl(\raisebox{-10pt}{\includegraphics[height = .35in]{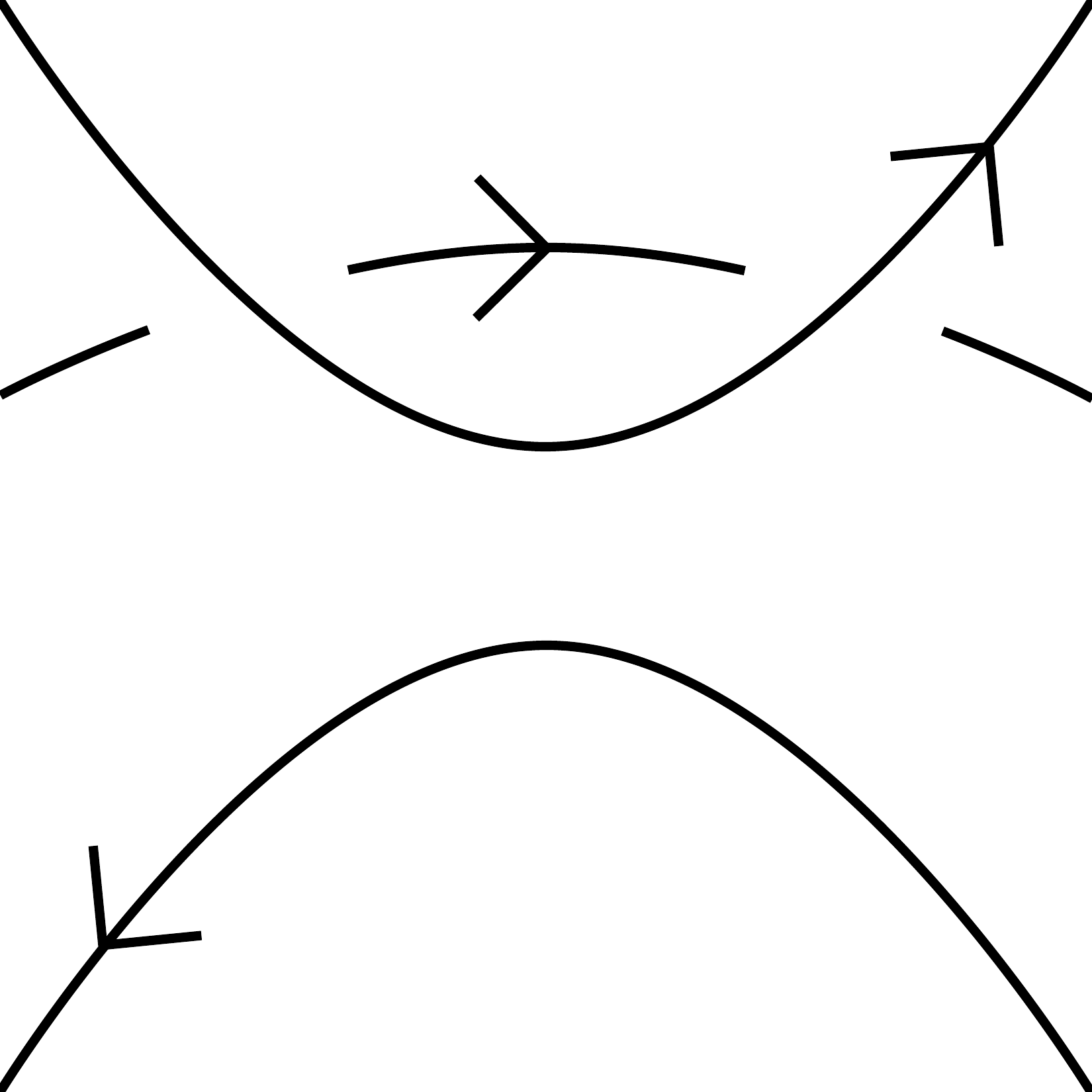}}\Biggr)
    =P\Biggl(\raisebox{-10pt}{\includegraphics[height = .35in]{O4-Proof/O4j-SR1.pdf}}\Biggr)+P\Biggl(\raisebox{-10pt}{\includegraphics[height = .35in]{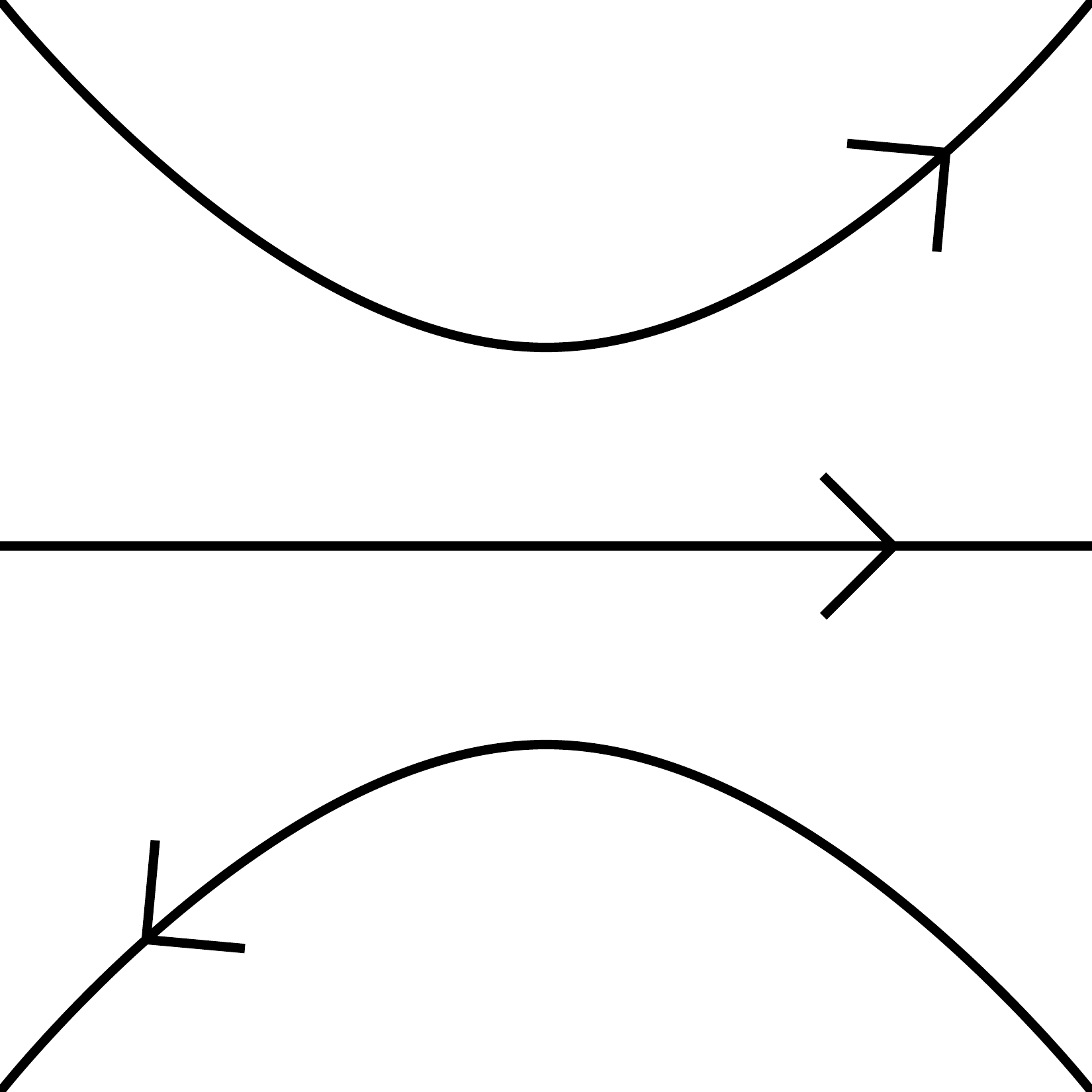}}\Biggr)\\
    &=P\Biggl(\raisebox{-10pt}{\includegraphics[height = .35in]{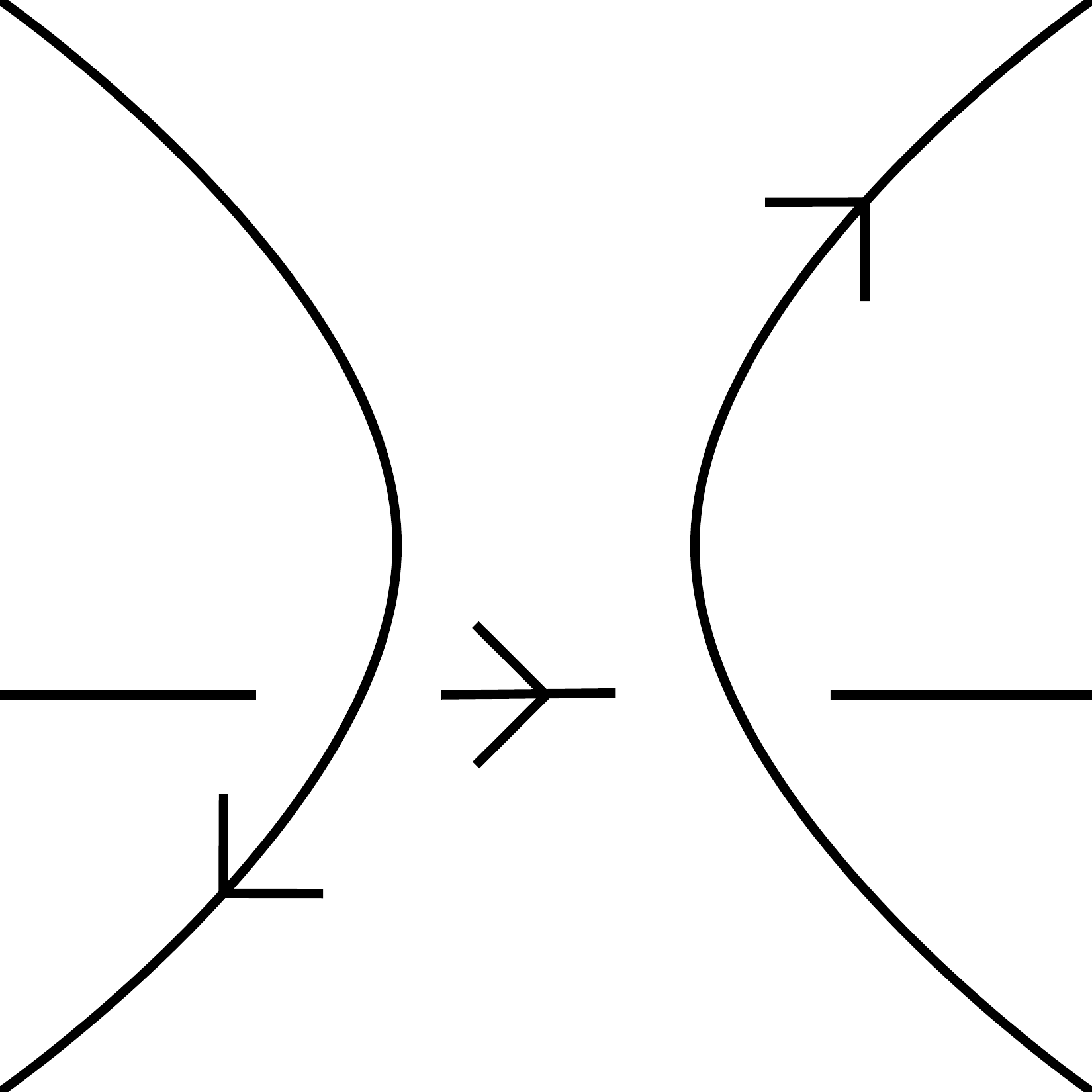}}\Biggr)+P\Biggl(\raisebox{-10pt}{\includegraphics[height = .35in]{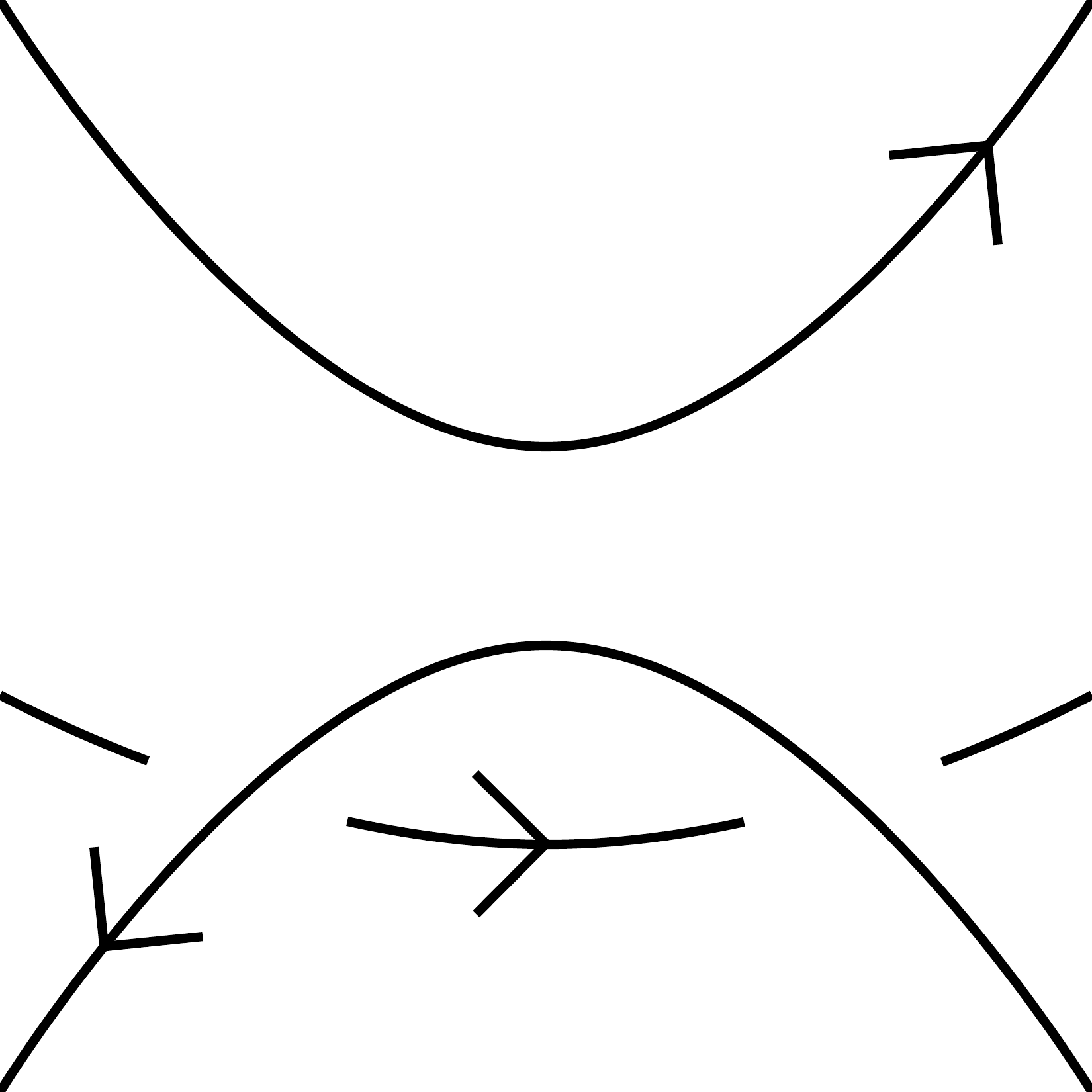}}\Biggr)=P\Biggl(\raisebox{-10pt}{\includegraphics[height = .35in]{Generating-Sets/O4j-2.pdf}}\Biggr).
\end{align*}

It remains to verify the invariance of the polynomial $P$ under the move $\Omega 5g$:

\begin{align*}
    P\Biggl(\raisebox{-10pt}{\includegraphics[height = .35in]{Generating-Sets/O5g-1.pdf}}\Biggr)&=P\Biggl(\raisebox{-10pt}{\includegraphics[height = .35in]{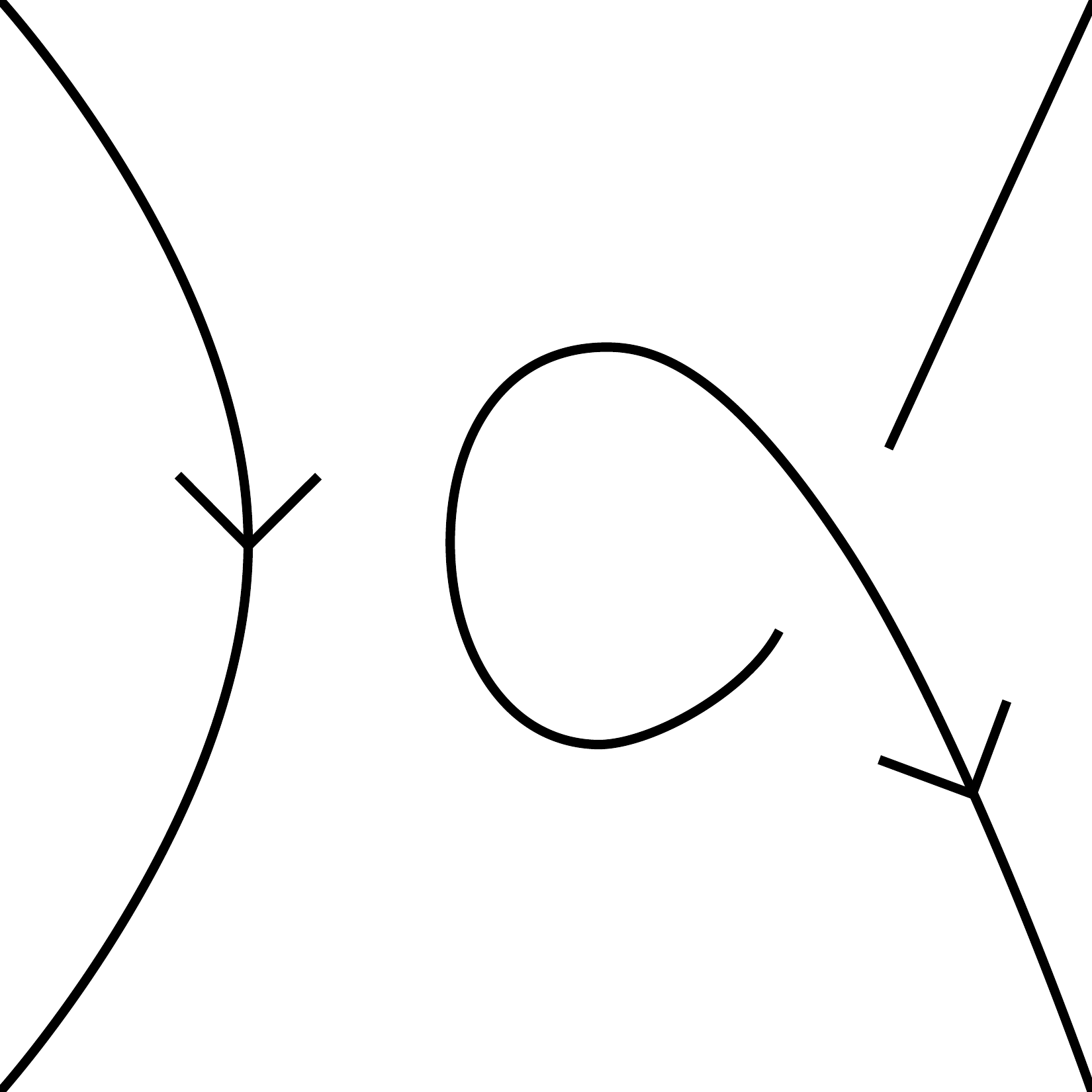}}\Biggr)+P\Biggl(\raisebox{-10pt}{\includegraphics[height = .35in]{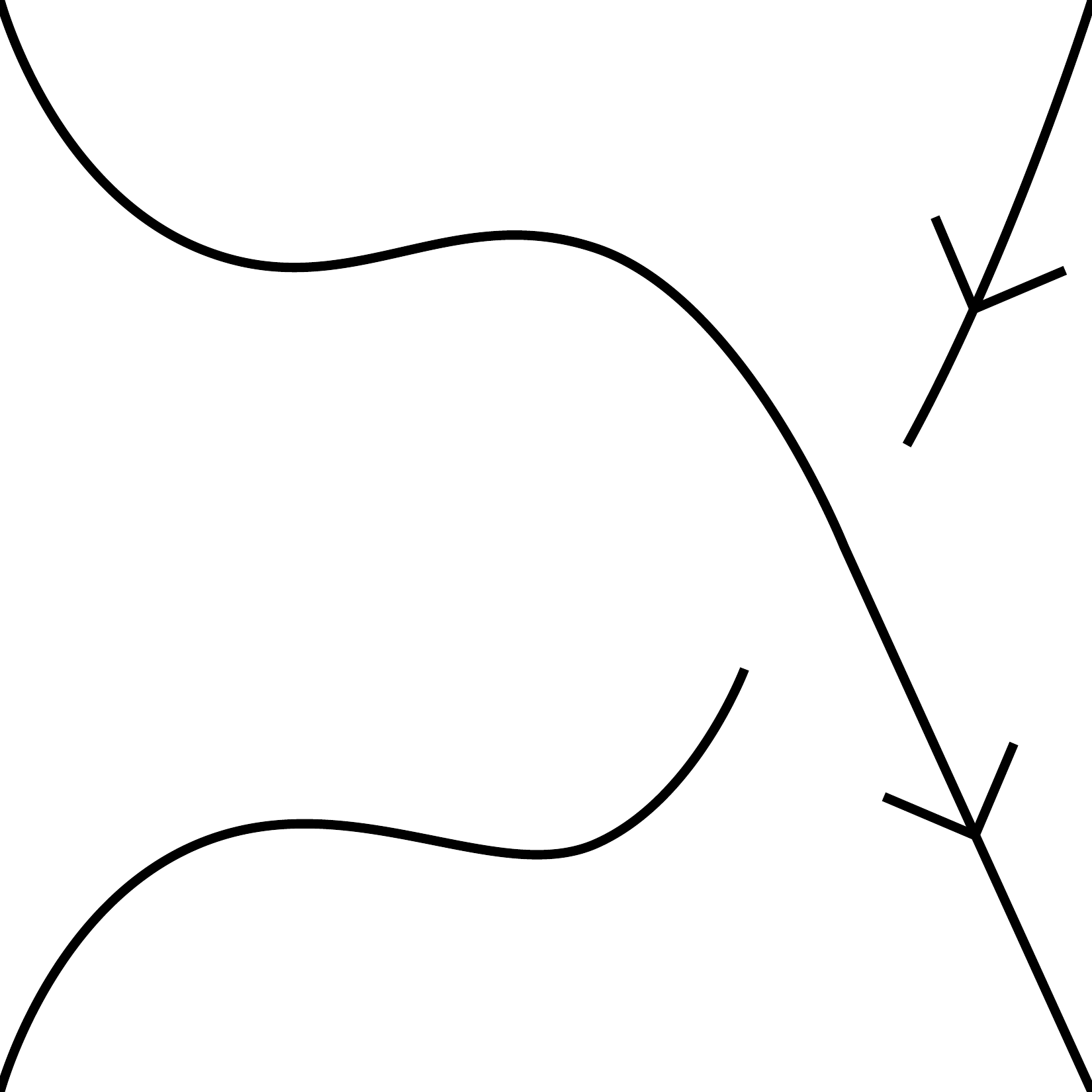}}\Biggr)
    =P\Biggl(\raisebox{-10pt}{\includegraphics[height = .35in]{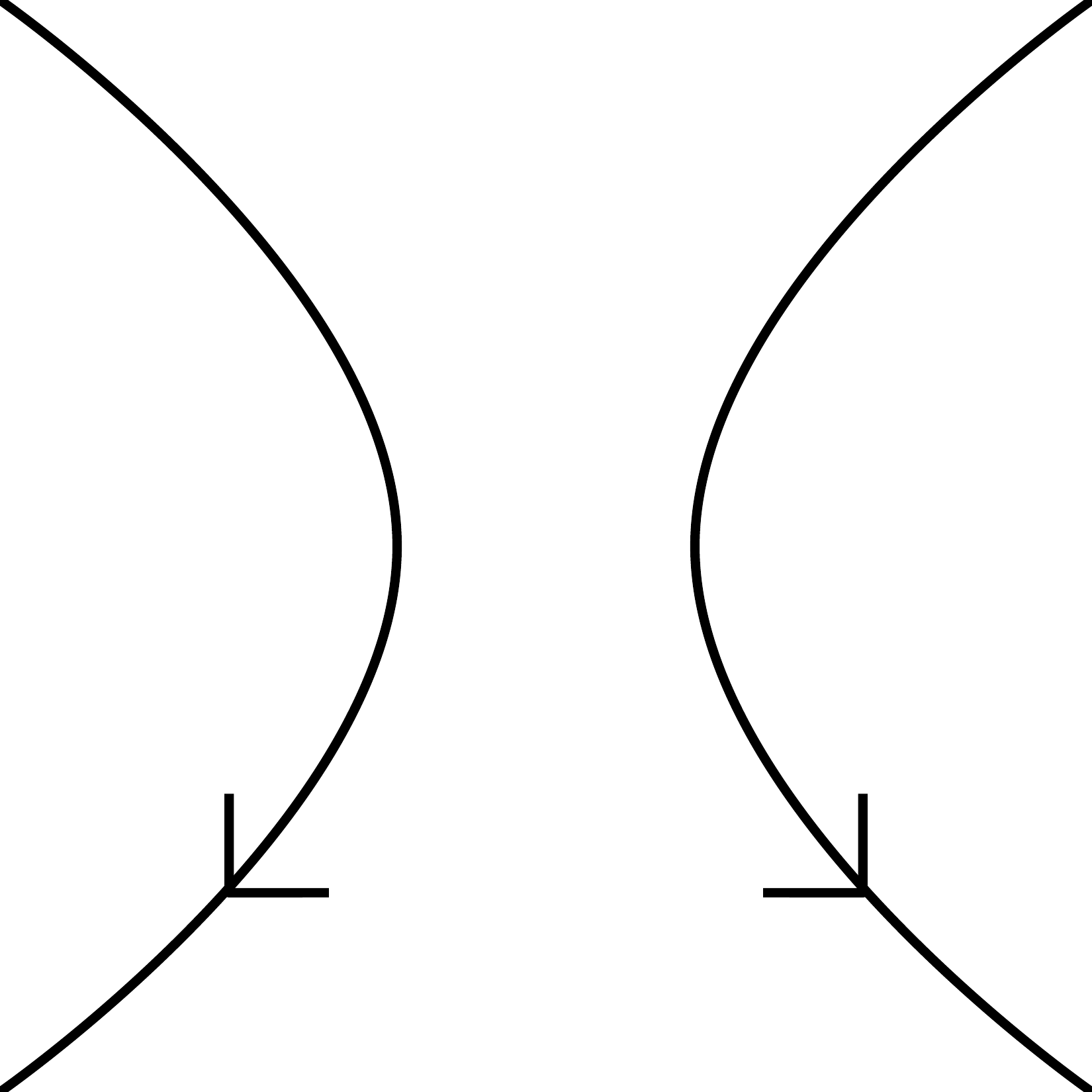}}\Biggr)+P\Biggl(\raisebox{-10pt}{\includegraphics[height = .35in]{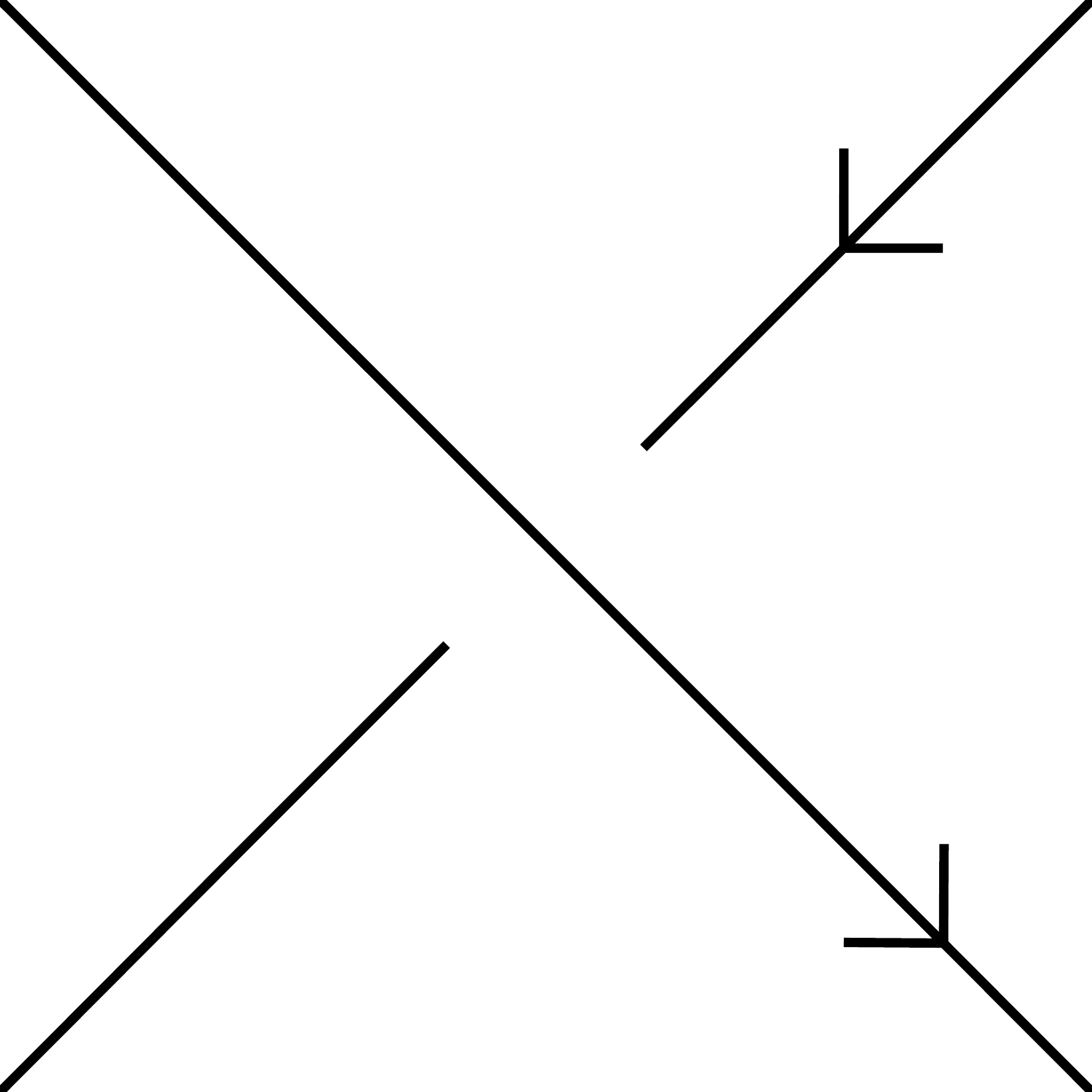}}\Biggr)\\
    &=P\Biggl(\raisebox{-10pt}{\includegraphics[height = .35in]{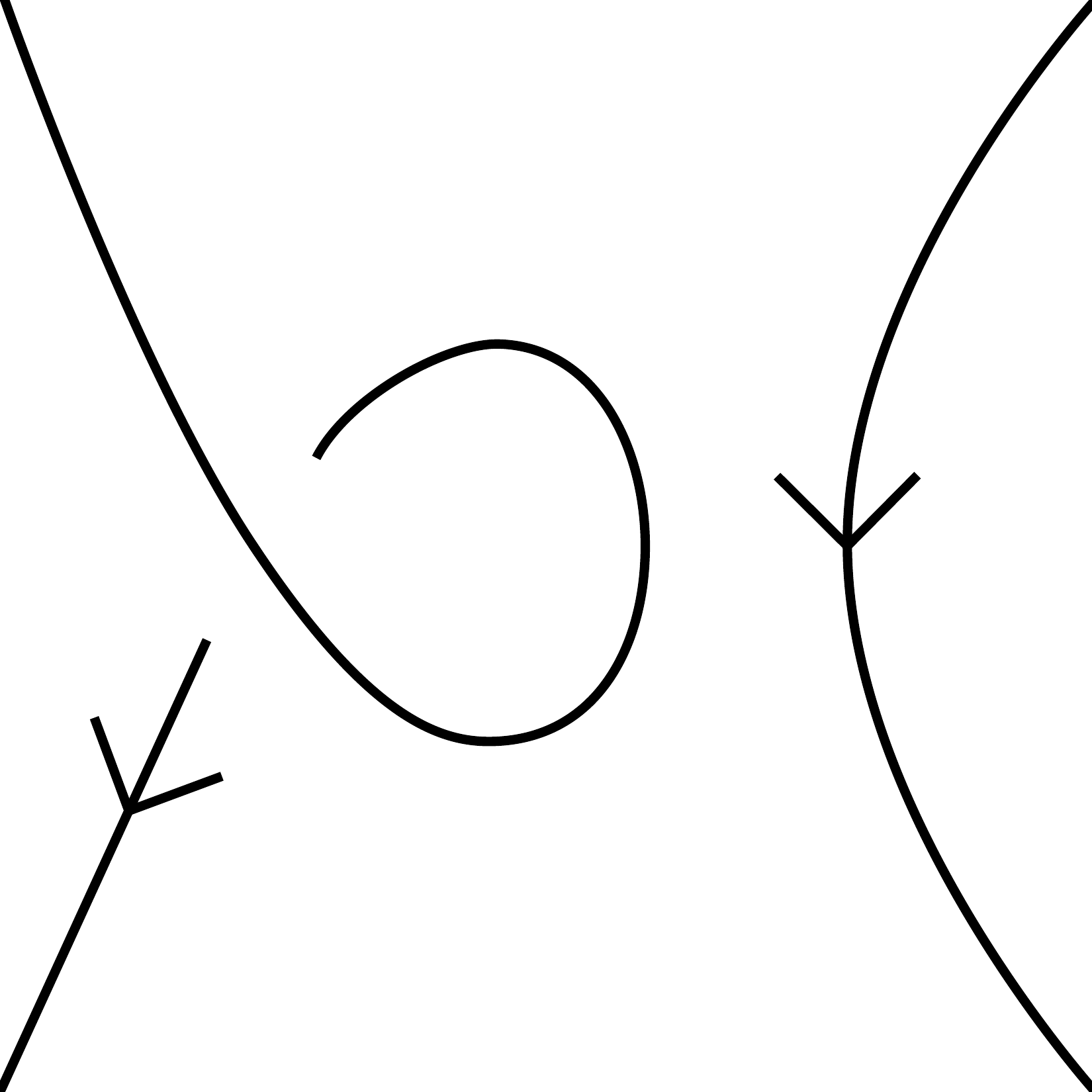}}\Biggr)+P\Biggl(\raisebox{-10pt}{\includegraphics[height = .35in]{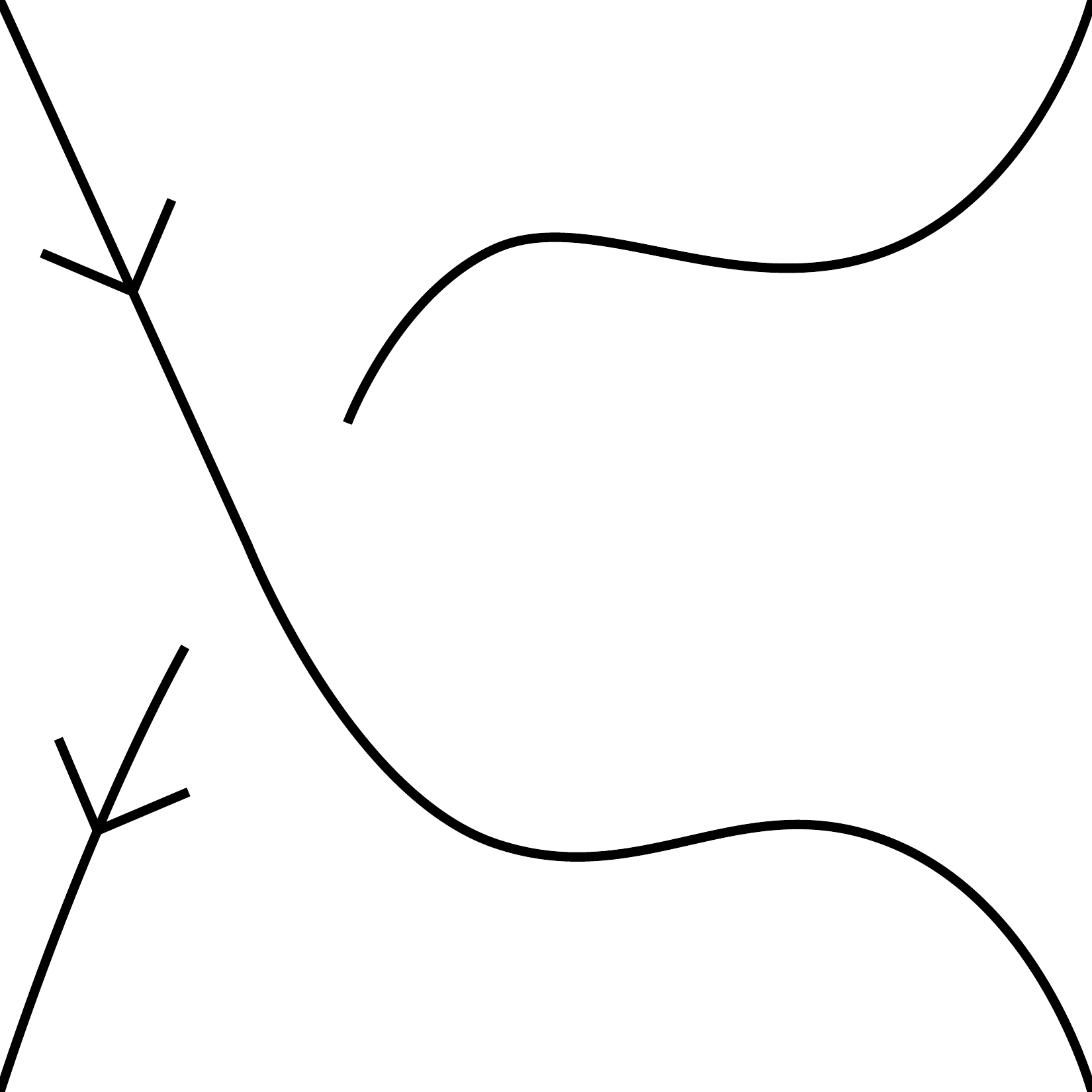}}\Biggr)
    = P\Biggl(\raisebox{-10pt}{\includegraphics[height = .35in]{Generating-Sets/O5g-2.pdf}}\Biggr).
\end{align*}
In the first and fourth equalities above, we used the skein relation for the In-Out-In-Out vertices, while in the second and third equalities, we used that $P$ is invariant under the moves $\Omega 1$.

We have proved the following statement.

\begin{theorem}
If $D$ and $D'$ are 4-valent diagrams representing rigid-vertex isotopic balanced-oriented, knotted 4-valent graphs, then $P(D) = P(D')$. Hence, the polynomial $P$ is a rigid-vertex isotopy invariant for balanced-oriented, knotted 4-valent graphs.
\end{theorem}

With the above result in hand, we can now define the corresponding polynomial for a knotted, balanced-oriented, 4-valent graph $L$, as follows:
\[P(L): = P(D),\]
where $D$ is any diagram representing $L$.

The \textit{mirror-image} of a knotted graph $L$ is obtained by reflecting $L$ across a plane in 3-space. We denote the mirror image of $L$ by $\tilde{L}$. If $D$ is a diagram of $L$, then by changing the over-crossings in $D$ into under-crossings, and changing the under-crossings in $D$ into over-crossings, we obtain a diagram of $\tilde{L}$. Moreover, if $L$ is an oriented knotted graph, the positive/negative crossings in $D$ become negative/positive crossings in $\tilde{D}$, the diagram of $\tilde{L}$. 

\begin{proposition}
The polynomial $P(\tilde{L})$ of the mirror image $\tilde{L}$ of $L$ can be obtained from $P(L)$ by interchanging $q$ and $q^{-1}$.
\end{proposition}

\begin{proof}
The statement follows at once from the skein relations for positive and negative crossings depicted in Figure~\ref{fig:SR Crossings}, and how the diagrams $D$ and $\tilde{D}$ of $L$ and respectively $\tilde{L}$ differ.
\end{proof}


\begin{thebibliography}{00}
\bibitem{Bataineh} K. Bataineh, M. Elhamdadi, M. Hajij, W. Youmans, Generating sets of Reidemeister moves of oriented singular links and quandles, {\it J. Knot Theory Ramifications}, {\bf 27 } (2018), 1850064.
\bibitem{Caprau_2014} C. Caprau, D. Heywood, D. Ibarra, On a state model for the SO(2n) Kauffman polynomial, {\it Involve}, {\bf 7} (2014), 547--563.
\bibitem{CaprauScott} C. Caprau, B. Scott, Minimal generating sets of moves for diagrams of isotopic knots and spatial trivalent graphs,  {\it J. Knot Theory Ramifications}, {\bf 31} (2022), 2250085.
\bibitem{Carpentier} R. P. Carpentier, From planar graphs to embedded graphs - a new approach to Kauffman and Vogel's polynomial,  {\it J. Knot Theory Ramifications}, {\bf 9} (2000), 975--986.
\bibitem{HOMFLY} J. Hoste, A. Ocneanu, K. Millett, P. Freyd, W. B. R. Lickorish, D. Yetter, A new polynomial invariant of knots and links, {\it Bull. Amer. Math. Soc.}, {\bf 12} (1985), 239--246.
\bibitem{Kauffman} L. H. Kauffman, Invariants of Graphs in Three-Space, {\it Trans. Amer. Math. Soc.}, {\bf 311} (1989), 697--710.
\bibitem{MOY} H. Murakami, T. Ohtsuki, S. Yamada, Homfly polynomial via an invariant of colored plane graphs, {\it L'Enseignement Math{\'e}matique}, {\bf 44} (1998), 325--360.
\bibitem{Polyak_2010} M. Polyak, Minimal generating sets of Reidemeister moves, {\it Quantum Topol.}, {\bf 1} (2010), 399--411.
\bibitem{PT} J.H. Przytycki, P. Traczyk, Invariants of links of Conway type, {\it Kobe J. Math.}, {\bf 2} (1987), 115--139.

\end{thebibliography}
\end{document}